\documentclass[11pt, leqno]{amsart}

\usepackage{graphicx}
\usepackage{amssymb,amsmath,amsthm,amscd}
\usepackage{mathrsfs}
\usepackage{lscape}
\usepackage{booktabs}
\usepackage{enumerate}
\usepackage[usenames,dvipsnames]{color}
\usepackage[vcentermath,enableskew]{youngtab}
\usepackage[colorlinks=true,pdfstartview=FitV,linkcolor=blue,citecolor=blue,urlcolor=blue]{hyperref}
\usepackage{tikz}
\tikzset{dynkdot/.style={circle,draw,scale=.38}}

\usepackage[all]{xy}
\CompileMatrices

\usepackage{xspace}
\usepackage{marginnote}
\usepackage{verbatim}
\allowdisplaybreaks[3]

\newcommand{\diagramone}[6]{
\xymatrix@R=6.5ex@C=9ex{
V(\varpi_{#1}) \otimes V(\varpi_{#2})_{-q^{#3}z} \otimes V(\varpi_{#4})_{(-q)^{#5} z}
\ar[rr]^{\qquad \qquad \quad V(\varpi_{#1})\otimes p_{#2,#4}} \ar[d]_{\Runiv_{#1,#2}((-q)^{#3}z) \otimes V(\varpi_#4)_{(-q)^{#5}z}} &&
V(\varpi_{#1}) \otimes V(\varpi_{#6})_{z} \ar[dd]_{\Runiv_{#1,#6}(z)} \\
V(\varpi_{#2})_{(-q)^{#3}z}  \otimes V(\varpi_{#1}) \otimes V(\varpi_{#4})_{(-q)^{#5}z}
\ar[d]_{V(\varpi_{#2})_{(-q)^{#3}z} \otimes \Runiv_{#1,#4}((-q)^{#5}z)} &\\
V(\varpi_{#2})_{(-q)^{#3}z} \otimes V(\varpi_{#4})_{(-q)^{#5}z}  \otimes V(\varpi_{#1})\ar[rr]^{\qquad \qquad \qquad p_{#1,#2}\otimes V(\varpi_3)  }&& V(\varpi_{#6})_{z} \otimes V(\varpi_{#1}).
}}

\newcommand{\diagramtwo}[7]{
\xymatrix@R=3ex@C=9ex{
u_{\varpi_{#1}} \tens u_{\varpi_{#2}} \tens f_{#7}u_{\varpi_{#4}} \ar[rr]  \ar[d] &&
u_{\varpi_{#1}} \tens u_{\varpi_{#6}} \ar[dd] \\
a_{#1,#2}((-q)^{#3}z)u_{\varpi_{#2}} \tens u_{\varpi_{#1}}  \tens f_{#7}u_{\varpi_{#4}}
\ar[d] &\\
a_{{#1},{#2}}((-q)^{#3}z)a_{{#1},{#4}}((-q)^{{#5}}z)u_{\varpi_{#2}} \tens w \ar[rr] &&
a_{{#1},{#6}}(z)u_{\varpi_{#6}}  \tens u_{\varpi_{#1}}
}
}

\newcommand{\nc}{\newcommand}
\numberwithin{equation}{section}

\newenvironment{rouge}{\relax\color{red}}{\relax}
\newenvironment{blue}{\relax\color{blue}}{\hspace*{.5ex}\relax}
\newenvironment{jaune}{\relax\color{Orchid}}{\hspace*{.5ex}\relax}
\nc{\hs}{\hspace*}

\newcommand{\ber}{\begin{rouge}}
\newcommand{\er}{\end{rouge}}
\nc{\bjn}{\begin{jaune}}
\nc{\ejn}{\end{jaune}}
\newcommand{\beb}{\begin{blue}}
\newcommand{\eb}{\end{blue}}

\newcommand{\bers}{\begin{blue}{}\marginnote{\fbox{\scshape\lowercase{O}}}{}}

\theoremstyle{plain}
\newtheorem{lemma}{Lemma}[section]
\newtheorem{proposition}[lemma]{Proposition}
\newtheorem{theorem}[lemma]{Theorem}
\newtheorem{algorithm}[lemma]{Algorithm}

\newtheorem{corollary}[lemma]{Corollary}

\theoremstyle{definition}
\newtheorem{remark}[lemma]{Remark}
\newtheorem{example}[lemma]{Example}

\newtheorem{definition}[lemma]{Definition}

\newcommand{\Lto}{\longrightarrow}

\newcommand{\hd}{{\operatorname{hd}}}
\newcommand{\soc}{{\operatorname{soc}}}

\newcommand{\lr}[1]{ \lan #1 \ran }
\newcommand{\lrt}[1]{ \{ #1 \} }
\newcommand{\SmQ}[2]{S_\mQ {\scriptstyle\sprt{#1}{#2}} }
\newcommand{\VmQ}[2]{V_\mQ {\scriptstyle\sprt{#1}{#2}} }
\newcommand{\SoQ}[2]{S_Q {\scriptstyle\sprt{#1}{#2}} }
\newcommand{\VoQ}[2]{V_Q {\scriptstyle\sprt{#1}{#2}} }
\newcommand{\pprt}[4]{ {\scriptstyle \left( \begin{matrix} #1 \\ #2 \\ #3 \\ #4 \end{matrix}\right)} }

\newcommand{\cmA}{\mathsf{A}}
\newcommand{\rl}{\mathsf{Q}}
\newcommand{\Stom}{\overset{\to}{S}_{[\redez]}(\um)}
\newcommand{\Sgetsm}{\overset{\gets}{S}_{[\redez]}(\um)}

\newcommand{\clfw}{\overline{\Lambda}} 

\newcommand{\prt}[2]{ \left( \begin{matrix} #1 \\ #2 \end{matrix}\right) }
\newcommand{\sprt}[2]{ \fontsize{5}{5}\selectfont \left( \begin{matrix} #1 \\ #2 \end{matrix} \right) \fontsize{11}{11}\selectfont }
\newcommand{\um}{\underline{m}}
\newcommand{\un}{\underline{n}}

\newcommand{\us}{\underline{s}}
\newcommand{\up}{\underline{p}}
\newcommand{\tb}{\mathtt{b}}
\newcommand{\bzero}{\mathbf{0}}

\newcommand{\al}{\alpha}
\newcommand{\be}{\beta}
\newcommand{\ga}{\gamma}
\newcommand{\lan}{\langle}
\newcommand{\ran}{\rangle}
\newcommand{\lf}{[\hspace{-0.3ex}[}
\newcommand{\rf}{]\hspace{-0.3ex}]}

\newcommand{\soplus}{\mathop{\mbox{\normalsize$\bigoplus$}}\limits}
\newcommand{\To}[1][{\hspace{2ex}}]{\xrightarrow{\,#1\,}}

\newcommand{\g}{\mathfrak{g}}
\newcommand{\Z}{\mathbb{Z}}
\newcommand{\A}{\mathbb{A}}
\newcommand{\N}{\mathsf{N}}
\newcommand{\Q}{\mathbb{Q}}
\newcommand{\C}{\mathbb{C}}
\newcommand{\F}{\mathcal{F}}
\newcommand{\Ca}{\mathcal{C}}
\newcommand{\Sym}{\mathfrak{S}}

\newcommand{\mC}{\mathscr{C}}
\newcommand{\mQ}{\mathscr{Q}}

\newcommand{\seteq}{\mathbin{:=}}
\newcommand{\Hom}{\operatorname{Hom}}
\newcommand{\Aut}{\operatorname{Aut}}
\newcommand{\Image}{\operatorname{Im}}
\newcommand{\Rep}{\operatorname{Rep}}
\newcommand{\tens}{\mathop\otimes}

\DeclareMathOperator{\id}{id}  

\newcommand{\len}{\mathrm{len}}
\newcommand{\wt}{\mathrm{wt}}
\newcommand{\cl}{\mathrm{cl}}
\newcommand{\aff}{\mathrm{aff}}

\newcommand{\univ}{\mathrm{univ}}
\newcommand{\norm}{\mathrm{norm}}
\newcommand{\ren}{\mathrm{ren}}
\newcommand{\Mod}{\operatorname{Mod}}
\newcommand{\Modg}{\operatorname{{Mod}_{\mathrm{gr}}}}
\nc{\gmod}{\mbox{-$\mathrm{gmod}$}}
\newcommand{\Rnorm}[1]{R^\norm_{#1}}
\newcommand{\Runiv}[1]{R^\univ_{#1}}
\newcommand{\Rren}[1]{R^\ren_{#1}}
\newcommand{\rmat}[1]{{\mathbf r}_{\mspace{-2mu}\raisebox{-.5ex}{${\scriptstyle{#1}}$}}}

\newcommand{\redex}{\widetilde{w}}
\newcommand{\redez}{{\widetilde{w}_0}}
\newcommand{\rrz}{{\widetilde{\mathsf{w}}_0}}

\newcommand{\ko}{\mathbf{k}}
\newcommand{\conv}{\mathbin{\mbox{\large $\circ$}}}

\newcommand{\qtext}[1]{ \quad \text{#1}\quad }

\newcommand{\iso}{\simeq}

\newcommand{\oO}{\overline{1}}
\newcommand{\oT}{\overline{2}}
\newcommand{\oTh}{\overline{3}}

\newcommand{\PR}{\Phi^+}

\newcommand{\dist}{{\rm dist}}

\newcommand{\Vifp}[1]{V^{(3)}(\varpi_{1})_{-w^2q^{#1}}}
\newcommand{\Vithp}[1]{V^{(3)}(\varpi_{1})_{-wq^{#1}}}
\newcommand{\Viop}[1]{V^{(3)}(\varpi_{1})_{-q^{#1}}}
\newcommand{\Vitp}[1]{V^{(3)}(\varpi_{2})_{-q^{#1}}}
\newcommand{\SQ}[1]{S_Q(\lr{#1})}

\definecolor{darkred}{rgb}{0.7,0,0} 
\newcommand{\defn}[1]{{\color{darkred}\emph{#1}}} 

\usepackage[colorinlistoftodos]{todonotes}

\title[Categorical Langlands duality: Exceptional cases]{
Categorical relations between Langlands dual quantum affine algebras: Exceptional cases}

\author[S.-j.~Oh]{Se-jin Oh}
\thanks{}
\address[S.-j.~Oh]{Ewha Womans University Seoul, 52 Ewhayeodae-gil, Daehyeon-dong, Seodaemun-gu, Seoul, South Korea}
\email{sejin092@gmail.com}
\urladdr{https://sites.google.com/site/mathsejinoh/}

\author[T.~Scrimshaw]{Travis Scrimshaw}
\address[T. Scrimshaw]{School of Mathematics and Physics, The University of Queensland, St.\ Lucia, QLD 4072, Australia}
\email{tcscrims@gmail.com}
\urladdr{https://people.smp.uq.edu.au/TravisScrimshaw/}


\thanks{S.-j.\ Oh was partially supported by the National Research Foundation of Korea(NRF) Grant funded by the Korea government(MSIP) (NRF-2016R1C1B2013135).
T.\ Scrimshaw was partially supported by the Australian Research Council DP170102648.}

\begin{document}

\begin{abstract}
We first compute the denominator formulas for quantum affine
algebras of all exceptional types. Then we prove the isomorphisms
among Grothendieck rings of categories $C_Q^{(t)}$ $(t=1,2,3)$,
$\mC_{\mQ}^{(1)}$ and $\mC_{\mathfrak{Q}}^{(1)}$. These results give Dorey's rule for all exceptional affine types, prove the conjectures of Kashiwara-Kang-Kim and Kashiwara-Oh, and
provides the partial answers of Frenkel-Hernandez on Langlands
duality for finite dimensional representations of quantum affine
algebras of exceptional types.
\end{abstract}

\maketitle

\tableofcontents

\section{Introduction}

Let $U_q'(\g)$ be a quantum affine algebra over an algebraically closed field $\ko$ of characteristic zero. The category $\mathcal{C}_\g$ of finite
dimensional integrable modules over $U_q'(\g)$ has been intensively studied since it is closely related to many branches of mathematics and physics including statistical mechanics, cluster algebras, dynamical systems, etc. We refer the reader to~\cite{CP95,FR99,Her10,HL10,IKT12,Kas02,KNS11,Nakajima04} and references therein.  The category
$\mathcal{C}_\g$ is a $\ko$-linear abelian monoidal category in the sense that it is abelian,  the tensor functor $\otimes$ is $\ko$-bilinear and exact.
Moreover it is rigid in the sense that each module $M$ in $\mathcal{C}_\g$ has a left dual $M^*$ and a right dual ${}^*M$ (see~\eqref{eq: LR dual} for more details).  Inside $\mathcal{C}_\g$, there exists a rigid monoidal subcategory, denoted by $\mathcal{C}^0_\g$, that contains all simple modules in $\mathcal{C}_\g$ up to parameter shifts (see~\cite[Section 10]{HL10} and~\cite[\S 3.1]{KKKOIV}).

The category $\mathcal{C}^0_\g$ contains a countably infinitely set $\mathcal{P} \seteq \{V(\varpi_i)_a\}_{(i \in I,\, a \in \Z[q_s^{\pm 1}])}$ consisting of simple $U_q'(\g)$-modules called the fundamental modules with spectral parameters $a$, whose tensor monomials in a certain order form a basis of the Grothendieck ring $[\mathcal{C}^0_\g]$ of $\mathcal{C}^0_\g$ (\cite{AK97,Chari02,Kas02,VV02}). Thus one can understand the set $\mathcal{P}$ as PBW-type modules of $\mathcal{C}^0_\g$.

For modules $M, N \in \mathcal{C}_\g$, there exists a non-zero $U_q'(\g)$-homomorphism
\[
\rmat{M,N} \colon M \otimes N \to N \otimes M
\]
that satisfies the Yang-Baxter equation.
When $M$ and $N$ are \emph{good} modules in the sense of \cite{Kas02}, there exists a polynomial called the denominator $d_{M,N}(z) \in \ko[z]$. The roots of $d_{M,N}(z)$ determine whether $\rmat{M_a,N_b}$ is an isomorphism or not. Indeed, the subcategory $\mathcal{C}^0_\g$ and the order in the previous paragraph are determined by the denominator formulas $d_{i,j}(z) \seteq d_{V(\varpi_i),V(\varpi_j)}(z)$ between the
fundamental modules $V(\varpi_i)$ and $V(\varpi_j)$. Roughly speaking, one can say that the representation theory for $\mathcal{C}_\g$ is \emph{controlled} by the denominator formulas $d_{i,j}(z)$.

For $\g^{(1)}$ of type $A_n^{(1)}$, $D_n^{(1)}$ or $E_{6,7,8}^{(1)}$ in~\cite{HL11}
and $\g^{(2)}$ of type $A_n^{(2)}$, $D_n^{(2)}$ or $E_6^{(2)}$ in~\cite{KKKOIV,Oh15E},
a tensor subcategory $\mathcal{C}^{(t)}_Q$ of $\mathcal{C}^0_\g$, where $t = 1,2$ and $\g = \g^{(t)}$, that contains
finite number of fundamental modules $\mathcal{P}^{(t)}_Q \seteq \{ V^{(t)}_Q(\be)\seteq V(\varpi_{i_\be})_{a_\be} \ | \ \be \in \PR\} \subset \mathcal{P}$ whose tensor monomials in the induced order (in short, PBW-monomials) form a basis of $[\mathcal{C}^{(1)}_Q]$ and
whose repeated duals of $V^{(t)}_Q(\be)$'s cover $\mathcal{P}$:
\begin{equation} \label{eq: cover}
 \mathcal{P} \ni M \iff \text{ there exists $\be \in \PR$ and $k \in \Z$ such that $M = \bigl(V^{(t)}_Q(\beta)\bigr)^{k*}$.}
\end{equation}
Here $Q$ denotes a Dynkin quiver of the corresponding finite type Dynkin diagram and $\PR$ is the set of positive roots of the same finite type.
Also the pair $(i_\be,a_\be) \in I \times \Z[q_s^{\pm1}]$ for each $V^{(t)}_Q(\be)$ is determined by the coordinate system of the Auslander-Reiten (AR) quiver $\Gamma_Q$ of the path algebra $\C Q$.
Furthermore, it was shown in \cite{HL11,KKKOIV} that $\mathcal{C}^{(t)}_Q$ categorifies the negative half $U_q^-(\mathsf{g})$ of the finite quantum group $U_q(\mathsf{g}) \subset U_q'(\g)$ and PBW-monomials in $\mathcal{P}^{(t)}_Q$ arranged by $d_{i,j}(z)$ categorify the PBW-basis of $U_q^-(\mathsf{g})$ arising from $Q$.
However, for $t = 2$, the result depends on Dorey's rule, which is only conjectured for type $E_6^{(2)}$.

Let us briefly recall a combinatorial property of $\Gamma_Q$ (see \cite{Bedard99}).
For a finite Weyl group $W$ of the same type of $Q$, let $w_0$ be the longest element and $\redez$ be a reduced expression of $w_0$. Then $\Gamma_Q$ is the Hasse diagram of the convex partial order $\prec_{Q}$ on $\PR$. In particular, each $\prec_{Q}$ coincides with the convex partial order $\prec_{[Q]}$ induced from the commutation class $[Q]$ consisting of all reduced expressions adapted to $Q$ and has its unique Coxeter element.
Also all commutation classes adapted to some $Q$ can be grouped into one $r$-cluster point.

In \cite{Oh14D,Oh14A}, the first named author observed that if $V^{(t)}_Q(\al) \otimes V^{(t)}_Q(\be)$ is reducible, then $(\al,\be)$ are comparable with respect to $\prec_Q$ when $\g$ is of type $A_n^{(t)}$ or $D_n^{(t)}$ $(t=1,2)$. The observation means that the order for PBW-monomials in $\mathcal{P}^{(t)}_Q$ $(t=1,2)$ determined by $d_{i,j}(z)$ is equivalent to $\prec_Q$.  Furthermore, the following are proved:
\begin{enumerate}
\item The conditions for
\[
\Hom_{U_q'(\g)}\big( V(\varpi_i)_a \otimes  V(\varpi_j)_b, V(\varpi_k)_c \big) \ne 0
\]
can be interpreted as the coordinates in some $\Gamma_Q$ when $\g$ is of type $A_n^{(t)}$ or $D_n^{(t)}$ and are referred to as Dorey's rule for quantum affine algebras.

\item By defining new statistics on $\Gamma_Q$, the denominator formulas for $A_n^{(t)}$ or $D_n^{(t)}$ can be read from any $\Gamma_Q$.
\end{enumerate}

Here let us review the history of Dorey's rule briefly. In~\cite{Dorey91}, Dorey described relations between three-point couplings in the simply-laced affine
Toda field theories (ATFTs) and Lie theories. More precisely, simply-laced ATFTs on untwisted affine Lie algebras are related to the same type of Lie algebras and simply-laced
ATFTs on twisted affine Lie algebras are related to non-simply laced Lie algebras obtained
by corresponding Dynkin diagram automorphisms.
Generally, in ATFTs, quantum affine algebras appear as quantum symmetry groups and
the fundamental representations of quantum affine algebras correspond to the quantum particles~\cite{BL91}. Interestingly, Dorey's rule was interpreted by Chari and Pressley~\cite{CP96} in the language of representation theory of classical quantum affine algebras by using Coxeter elements for $U_q'(A_n^{(1)})$ and $U_q'(D_n^{(1)})$ and twisted Coxeter elements for $U_q'(B_n^{(1)})$ and $U_q'(C_n^{(1)})$.
Analogous results for $U_q'(E_{6,7,8}^{(1)})$ and classical twisted affine types were shown in~\cite[Section 7.3]{FH15},~\cite{Oh15E} and~\cite{KKKOIV} respectively (see also~\cite{YZ11}).

Motivated by the ideas of Chari and Pressley on Dorey's rule for quantum affine algebras of type $B_n^{(1)}$ (resp.~$C^{(1)}_{n}$) using twisted Coxeter elements
of type $A_{2n-1}$ (resp.~$D_{n+1}$),  the first named author and Suh introduced the folded AR quiver $\widehat{\Upsilon}_{\mQ}$ in \cite{OS15,OS16B,OS16C},
a combinatorial object that can be understood as a twisted analogue of $\Gamma_Q$.  Here $[[\mQ]]$ denotes the $r$-cluster point of a simply-laced finite type $X$ with positive roots $\PR$, which corresponds to any twisted Coxeter element.
In particular, the folded AR quiver has the following properties:
\begin{enumerate}
\item $\widehat{\Upsilon}_{\mQ}$ is the Hasse diagram of $\prec_{[\mQ]}$.
\item We can define some certain subcategory $\mathscr{C}^{(1)}_\mQ$ of $\mathcal{C}_{\g}$ for $\g=B_n^{(1)}$,
$C_n^{(1)}$, $F_4^{(1)}$ and $G_2^{(1)}$ determined by the coordinate system of $\widehat{\Upsilon}_{\mQ}$ (see also the category $\mathscr{C}^-$ in \cite{HL16}).
\item The modules $\mathscr{P}_\mQ \seteq \{ V_\mQ(\be) \in \mathscr{C}^{(1)}_\mQ \mid \be \in \PR\}$ satisfies the property stated in~\eqref{eq: cover}.
\item When $\g=B_n^{(1)}$ or $C_n^{(1)}$, the order for PBW-monomials in $\mathscr{P}_\mQ$  arranged by $d_{i,j}(z)$ is equivalent to $\prec_{[\mQ]}$.
\item For $\g=B_n^{(1)}$ or $C_n^{(1)}$, the Dorey's rule
can be interpreted as \emph{$[\mQ]$-minimal pair} $(\al,\be)$ of $\al+\be \in \PR$
for some $[\mQ]$, and the denominator formulas can be read from any $\widehat{\Upsilon}_{\mQ}$.
\end{enumerate}
However, due to the lack of Dorey's rule and denominator formulas for all exceptional affine types, there were not exceptional affine types analogues of many of the aforementioned results.

The quiver Hecke algebras $R^{\mathsf{g}}$, independently introduced by Khovanov-Lauda \cite{KL09,KL11} and Rouquier \cite{Rou08}, categorify the negative part
$U^-_q(\mathsf{g})$ of quantum groups $U_q(\mathsf{g})$ for all symmetrizable Kac-Moody algebras $\mathsf{g}$. The categorification for (dual) PBW-bases and global bases
of the integral form $U_\A^-(\mathsf{g})$ of $U^-_q(\mathsf{g})$ were very quickly developed since the introduction of quiver Hecke algebras.
Among them, \cite{BKM12,Kato12,KR11,McN15} give the categorification theory for (dual) PBW-bases of $U^-_\A(\mathsf{g})$ associated to finite simple Lie algebra $\mathsf{g}$ by using
convex orders $\prec_{[\redez]}$ on the set of positive roots $\PR$ for any commutation class
$[\redez]$.
As we reviewed, $U^-_q(\mathsf{g})$ for finite type $A_n$, $D_n$ or $E_{6,7,8}$
is also categorified by $\Ca_Q^{(1)}$. Thus it was natural to ask the concrete relationship between the categories $\Rep(R^{\mathsf{g}})$ and $C_Q^{(1)}$ (see~\cite[\S 1.6]{HL11}).

When $\g$ is of affine type $A^{(1)}_n$ and $D^{(1)}_n$, the affirmative answer was given by
Kang, Kashiwara and Kim in~\cite{KKK13B,KKK13A}: For a quantum affine algebra $U_q'(\g)$, we can construct a generalized Schur-Weyl duality functor
\[
\mathcal{F} \colon \Rep(R^J) \to \mathcal{C}_\g
\]
by observing denominator formulas $d_{V,W}(z)$ between good modules
$V,W \in \mathcal{C}_\g$. Here $R^J$ is the quiver Hecke algebra determined by only the choice of good modules in $\mathcal{C}_\g$ and denominator formulas. For
$\g$ of affine type $A^{(1)}_{n}$ and $D^{(1)}_{n}$, by taking $\{ V_Q^{(1)}(\al_i) \}_{i \in I}$
as the chosen good modules, they proved that the quiver Hecke algebra $R^J$ is indeed
$R^{\mathsf{g}}$ by using the denominator formulas $d_{i,j}(z)$ for $U_q'(\g)$. Moreover, by adopting the framework of~\cite{HL11}, it was shown that the induced generalized Schur-Weyl duality functor
\[
\mathcal{F}^{(1)}_Q \colon \Rep(R^{\mathsf{g}}) \to \mathcal{C}^{(1)}_Q
\]
sends simple modules to simple modules bijectively. Thus the PBW-type modules associated to $[Q]$
in both categories can be identified via $\mathcal{F}^{(1)}_Q$. These results were extended to the cases $\Ca_Q^{(2)}$ for $\g=A^{(2)}_{n}$ and $D^{(2)}_{n}$ in~\cite{KKKOIV}. In these cases, however, to show the bijection, the Dorey's rule for $U_q'(\g)$, developed in~\cite{KKKOIV}, was crucially used instead of~\cite{HL11}.

Since there are $\mC_\mQ$, $\widehat{\Upsilon}_\mQ$, Dorey's rule and
denominator formulas for $\g=B_n^{(1)}$ and $C_n^{(1)}$, the above results were extended to
the $U_q'(\g)$ cases~\cite{KO17}. Here, $[\mQ]$ denotes a commutation class of type $A_{2n-1}$ or $D_{n+1}$, respectively, arising from a twisted Coxeter element associated to the corresponding Dynkin diagram folding. More precisely, there exists an exact generalized Schur-Weyl duality functor
\[
\mathscr{F}_\mQ \colon \Rep(R^{\mathsf{g}}) \to \mC^{(1)}_\mQ, \qquad
(\mathsf{g},\g) =
\begin{cases}
(A_{2n-1},B_n^{(1)}), \\
(D_{n+1},C_n^{(1)}),
\end{cases}
\]
which sends simple modules to simple modules bijectively, and give a categorification
theorem for the PBW-bases of $U_q^-(\mathsf{g})$ associated to $[\mQ]$ via the PBW-monomials in $\mathscr{P}_\mQ$.

Hence,  there are subcategories $\mC^{(2)}_Q$ of $U_q'(\g^{(2)})$-modules $(\g^{(2)}=A_{2n-1}^{(2)},D^{(2)}_{n+1} )$ and $\mC_\mQ$ of $U_q'({}^L\g^{(2)})$-modules
$({}^L\g^{(2)}=B_{n}^{(1)},C^{(1)}_{n} )$ whose Grothendieck groups are isomorphic to each other. Such similarities between twisted quantum affine algebra
and its untwisted Langlands dual were observed by many mathematician and are known to be related to the geometric Langlands correspondence
(see \cite{FH11,FH11B,FR98}).

In this paper, we extend these results to all exceptional affine algebras $U_q'(\g)$. As previously mentioned, the main difficulties we had to overcame were lack of denominator formulas and Dorey's rule for exceptional affine cases.
In particular, the denominator formulas for exceptional types are quite difficult to study due to their large dimensions and multiplicities in their classical decompositions.
In order to compute the denominator formulas, we required an explicit basis for $V(\varpi_i)$ with the $U_q'(\g)$-action on that basis.
To the best knowledge of the authors, this was not known.
We began by constructing an explicit $U_q'(\g)$-module structure for fundamental modules associated to minuscule and adjoint representations by using the corresponding crystal graphs.
Using this, we can compute the denominator formulas between these representations and the induced Dorey's type homomorphisms. This allows us to compute all denominator formulas $d_{i,j}(z)$ 
by applying the framework of~\cite{Oh14R} and statistics on $\Gamma_Q$ and $\widehat{\Upsilon}_\mQ$.

From the denominator formulas $d_{i,j}(z)$, we can construct
exact functors
\[
\mathcal{F}_Q^{(t)} \colon \Rep(R^{\mathsf{g}}) \to \mathcal{C}_Q^{(t)} \ (t=1,2,3),
\qquad \mathscr{F}_\mQ \colon \Rep(R^{\mathsf{g}}) \to \mC_\mQ^{(1)}.
\]
This gives an answer for~\cite[Conjecture 4.3.2]{KKK13B}.  This conjecture for $E^{(1)}_{6,7,8}$ is also proved by Fujita~\cite{Fujita18}, announced after completing this paper,
by generalizing Ginzburg-Reshetikhin-Vasserot's geometric realization of the quantum affine Schur-Weyl duality~\cite{GRV94} to generalized Schur-Weyl duality for $A_n^{(1)}$, $D_n^{(1)}$ and $E^{(1)}_{6,7,8}$,
which also implies simpleness of some roots of the denominator formulas in that cases.

By utilizing both categories with the exact functors, we can obtain the complete Dorey's rules for
all exceptional affine cases. In particular, we use the rigidity of $\Ca_\g$ and the exact sequences
for ($[\mQ]$-minimal) pairs $(\al,\be)$ of $\al+\be \in \PR$ in $\Rep(R^{\mathsf{g}})$. With the
results on denominator formulas and Dorey's type homomorphisms at hand, we prove that
\begin{enumerate}
\item $\mathcal{C}_Q^{(t)}$ and $\mC_\mQ^{(1)}$ give categorifications of $U_q^-(\mathsf{g})$,
\item $\mathcal{F}_Q^{(t)}$ and $\mathscr{F}_\mQ$ sends simple modules to simple modules bijectively, and
\item the PBW-monomials in $\mathcal{P}^{(t)}_Q$ and $\mathscr{P}_\mQ$  give categorifications of the PBW bases of $U_q^-(\mathsf{g})$ associated to $[Q]$ and $[\mQ]$ for all exceptional cases.
\end{enumerate}
These results for $E_6^{(t)}$ $(t=1,2)$ and $F_4^{(1)}$ were conjectured in \cite[Appendix]{KO17}.

Now we have the following diagram which are related to our works on exceptional cases:
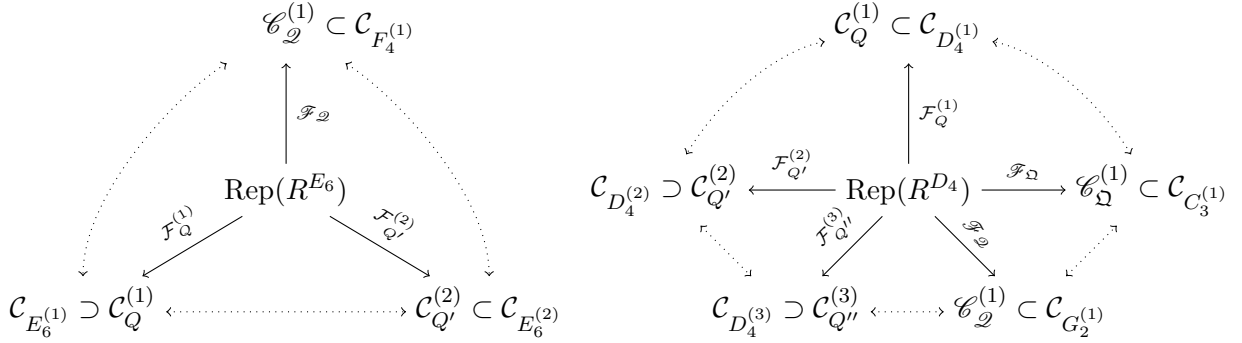
\[
\begin{tikzpicture}[scale=2.7,baseline=0]
\node (CQ) at (0,0.8) {$\ \ \quad \qquad \mC^{(1)}_\mQ  \subset \Ca_{F_4^{(1)}}$};
\node (U) at (0,0) {$\Rep(R^{E_6})$};
\node (C1) at (-1,-0.6) {$ \Ca_{E_6^{(1)}} \supset \mathcal{C}^{(1)}_{Q}$};
\node (C2) at (1,-0.6) {$ \mathcal{C}^{(2)}_{Q'} \subset \Ca_{E_6^{(2)}}$};
\draw[<-] (CQ) -- node[midway,right] {\tiny $\mathscr{F}_\mQ$}  (U);
\draw[<-] (C1) -- node[midway,sloped,above] {\tiny $\mathcal{F}^{(1)}_Q$}  (U);
\draw[<-] (C2) -- node[midway,sloped,above] {\tiny $\mathcal{F}^{(2)}_{Q'}$}  (U);
\draw[<->,dotted] (CQ) .. controls (-0.5,0.5) and (-1,0) .. (C1);
\draw[<->,dotted] (CQ) .. controls (0.5,0.5) and (1,0) .. (C2);
\draw[<->,dotted] (C1) -- (C2);
\end{tikzpicture}
\
\begin{tikzpicture}[scale=2.7,baseline=0]
\node (C1) at (0,0.8) {$\mathcal{C}^{(1)}_{Q} \subset \Ca_{D_4^{(1)}}$};
\node (C2) at (-1.2,0) {$ \Ca_{D_4^{(2)}} \supset\mathcal{C}^{(2)}_{Q'}$};
\node (C3) at (-0.6,-0.6) {$\Ca_{D_4^{(3)}} \supset \mathcal{C}^{(3)}_{Q''}$};
\node (CD) at (1.2,0) {$\mC^{(1)}_{\mathfrak{Q}} \subset \Ca_{C_3^{(1)}} $};
\node (CQ) at (0.6,-0.6) {$\mC^{(1)}_\mQ \subset \Ca_{G_2^{(1)}}$};
\node (U) at (0,0) {$\Rep(R^{D_4})$};
\draw[<-] (C1) -- node[midway,right] {\tiny $\mathcal{F}^{(1)}_Q$}  (U);
\draw[<-] (C2) -- node[midway,sloped,above] {\tiny $\mathcal{F}^{(2)}_{Q'}$}  (U);
\draw[<-] (C3) -- node[midway,sloped,above] {\tiny $\mathcal{F}^{(3)}_{Q''}$}  (U);
\draw[<-] (CD) -- node[midway,sloped,above] {\tiny $\mathscr{F}_{\mathfrak{Q}}$}  (U);
\draw[<-] (CQ) -- node[midway,sloped,above] {\tiny $\mathscr{F}_{\mQ}$}  (U);
\draw[<->,dotted] (C1) .. controls (-0.6,0.7) and (-0.9,0.5) .. (C2);
\draw[<->,dotted] (C2) -- (C3);
\draw[<->,dotted] (C3) -- (CQ);
\draw[<->,dotted] (CQ) -- (CD);
\draw[<->,dotted] (CD) .. controls (0.9,0.5) and (0.6,0.7) .. (C1);
\end{tikzpicture}
\]
These diagrams are closely related to the conjectures of Frenkel-Hernandez (see \cite[Conj.~2.2, Conj.~2.4, Conj.~3.10]{FH11B})
for exceptional affine types similar to classical affine types as shown in~\cite{KO17}:
For a representation $V$ in $\mathcal{C}^{(2)}_{Q}$, one can corresponds to a representation $\widehat{V}$ in $\mC^{(1)}_\mQ$ in a canonical way
via induced functor $\F_\mQ \circ \bigl( F^{(2)}_Q \bigr)^{-1}$ for any $[Q]$ and $[\mQ]$. Furthermore, if $V$ is simple, so is $\widehat{V}$.

Therefore, our results finish this program on categorical similarities between monoidal subcategories of (un)twisted quantum affine algebras and
their Langlands duals for the exceptional affine types, where the nonexceptional affine types were done in~\cite{KKK13B,KKKOIV,KO17} (see also \cite{KKO17}).


This paper is organized as follows.
In Section~\ref{sec1}, we briefly review the previous results on the commutations classes and (folded) AR quivers we require.
In Section~\ref{sec2:Qunatum affine algebras}, we review the necessary theory on quantum affine algebras, Dorey's rule, denominator formulas and (folded) AR quivers. We also give our explicit $U_q'(\g)$-module structure for affine minuscule and adjoint representations.
In Section~\ref{sec4:computations}, we compute the denominator formulas for all exceptional types  with some ambiguity for roots of higher orders
and parts of Dorey's type homomorphisms.
In Section~\ref{sec5:quiver Hecke}, we review the basic materials on quiver Hecke algebras, the generalized Schur-Weyl duality functors and known results on the functors.
In Section~\ref{sec6:application}, we apply the denominator formulas obtained in
Section~\ref{sec4:computations} to the context of the generalized Schur-Weyl duality, which yields our main results for categories and Dorey's rule.
In Section~\ref{sec:Refine} and Appendix~\ref{Sec: Addendum}, we remove all of ambiguities about orders of roots in cases of $U_q'(F_4^{(1)})$, $U_q'(E_6^{(2)})$  and $E_{6,7,8}^{(1)}$ by using the exact functor what have constructed in Section~\ref{sec6:application} and the homomorphisms investigated in~\cite{Oh15E}.

Note that Appendix~\ref{Sec: Addendum} has been included as an addendum to the published version~\cite{OS19}, which removes all of the ambiguities of orders of roots.


\subsection*{Acknowledgements}
The authors would like to thank Masaki Kashiwara for useful discussions.
The authors would like to thank the anonymous referee for useful comments.
T.S.\ would like to thank Ewha Womans University for its hospitality during his stay in June, 2017, where this work began.
S.-j.O.\ would like to thank The University of Queensland for its hospitality during his stay in February, 2018.
This work benefited from computations using \textsc{SageMath}~\cite{sage,combinat}.

\section{\texorpdfstring{$r$}{r}-cluster points and their folded Auslander-Reiten quivers with coordinates} \label{sec1}

\subsection{Cartan datum and quantum groups}

A \defn{Cartan datum} is a sextuple $(\cmA,P,\Pi,P^\vee,\Pi^\vee, (\cdot\,,\,\cdot))$ with index set $I$ consisting of
(a) an integer-valued matrix $\cmA = (a_{ij})_{i,j \in I}$, called \defn{the symmetrizable generalized Cartan matrix},
(b) a free abelian group $P$ called the \defn{weight lattice},
(c) $\Pi= \{ \alpha_i \in P \mid \ i \in I \}$ called the \defn{simple roots},
(d) $P^{\vee}\seteq\Hom(P, \Z)$ called the \defn{co-weight lattice} with the canonical pairing $\langle\cdot\,,\, \cdot \rangle$,
(e) $\Pi^{\vee} = \{ h_i \mid i \in I \}\subset P^{\vee}$ called the \defn{simple coroots},
(f) a symmetric bilinear form $(\cdot\,,\,\cdot) \colon P \times P \to \Q$,
satisfying
(i) $\langle h_i,\alpha_j \rangle = a_{ij}$ for all $i,j \in I$,
(ii) $\Pi$ is linearly independent,
(iii) for each $i \in I$, there exists $\Lambda_i \in P$ such that
           $\langle h_j, \Lambda_i \rangle =\delta_{ij}$ for all $j \in I$,
(iv)  $(\al_i,\al_i)\in\Q_{>0}$ for any $i\in I$,
(v) for any $\lambda\in P $ and $i \in I$,
$\lan h_i,\lambda\ran=
\dfrac{2(\alpha_i,\lambda)}{(\alpha_i,\alpha_i)}.$

We call $\Lambda_i$ the \defn{fundamental weights}. The free abelian group $\rl \seteq \soplus_{i \in I}\Z\al_i$ is called the \defn{root lattice}. Set $\rl^+=\sum_{i \in I}\Z_{\ge 0} \al_i \subset \rl$ and
$\rl^- = \sum_{i \in I}\Z_{\le 0} \al_i \subset \rl$. For $\beta = \sum_{i \in I}n_i\al_i \in\rl^+$, we set $|\beta| = \sum_{i \in I} n_i \in \Z_{\ge 0}$

Let $\gamma$ be the smallest positive integer such that $\gamma \dfrac{(\al_i,\al_i)}{2} \in \Z$ for all $i \in I$. For each $i \in I$, set $q_i = q ^{\frac{(\al_i,\al_i)}{2}} \in \Q(q^{1/\gamma})$.
For $m,n \in \Z_{\ge 0}$ and $i \in I$, we define
\[
[n]_i = \dfrac{q_i^n-q_i^{-n}}{q_i-q_i^{-1}},
\qquad\qquad
[n]_i!=\prod_{k=1}^n [k]_i,
\qquad\qquad
\begin{bmatrix} m \\ n \end{bmatrix}_i = \dfrac{[m]_i!}{[m-n]_i![n]_i!}.
\]

\begin{remark}
The imaginary number $i$ will always be denoted by $\sqrt{-1}$ so that there is no danger of confusion with an indexing variable $i$.
\end{remark}

\begin{definition}
For Cartan datum $(\cmA,P,\Pi,P^\vee,\Pi^\vee, (\cdot\,,\,\cdot))$, the associated \defn{(Drinfel'd-Jimbo) quantum group} $U_q(\mathsf{g})$ is the $\Q(q^{1/\gamma})$-algebra generated by $e_i, \ f_i$ $(i \in I)$ and $q^h$ $(h \in \gamma^{-1}P)$ subject to
the following relations:
\begin{enumerate}[{\rm (i)}]
  \item  $q^0=1, q^{h} q^{h'}=q^{h+h'} $ for $ h,h' \in \gamma^{-1}P^{\vee},$
  \item  $q^{h}e_i q^{-h}= q^{\langle h, \alpha_i \rangle} e_i,
          \ q^{h}f_i q^{-h} = q^{-\langle h, \alpha_i \rangle }f_i$ for $h \in \gamma^{-1}P^{\vee}, i \in I$,
  \item  $e_if_j - f_je_i =  \delta_{ij} \dfrac{K_i -K^{-1}_i}{q_i- q^{-1}_i }, \ \ \text{ where } K_i=q_i^{ h_i},$
  \item  $\displaystyle \sum^{1-a_{ij}}_{k=0}
  (-1)^ke^{(1-a_{ij}-k)}_i e_j e^{(k)}_i =  \sum^{1-a_{ij}}_{k=0} (-1)^k
  f^{(1-a_{ij}-k)}_i f_jf^{(k)}_i=0 \quad \text{ for }  i \ne j, $
\end{enumerate}
where $e_i^{(k)}=e_i^k/[k]_i!$ and $f_i^{(k)}=f_i^k/[k]_i!$. Recall that $\delta_{ij}$ is the Kronecker delta. We set $q_s = q^{1/\gamma}$.
\end{definition}
We denote by $U_q^+(\mathsf{g})$ (resp.\ $U_q^-(\mathsf{g})$) by the subalgebra of $U_q(\mathsf{g})$ generated by $e_i$ (resp.\ $f_i$) $(i \in I)$.
We also denote by $U_\A^+(\mathsf{g})$ (resp.~$U_\A^-(\mathsf{g})$) by the $\A\seteq\Z[q,q^{-1}]$-subalgebra of $U_q(\mathsf{g})$ generated by $e^{(n)}_i$ (resp.\ $f^{(n)}_i$) for all $i \in I$ and $n\in \Z$.

\subsection{Adapted and (triply) twisted adapted points}\label{subsec: adapted points}
Let us denote by $\Delta$ the Dynkin diagrams of finite type, labeled by an index set $I$, and their automorphisms $\sigma$.
By $\sigma$, we can obtain the Dynkin diagrams $\widehat{\Delta}$ of finite type $BCFG$ as orbits of $\sigma$:
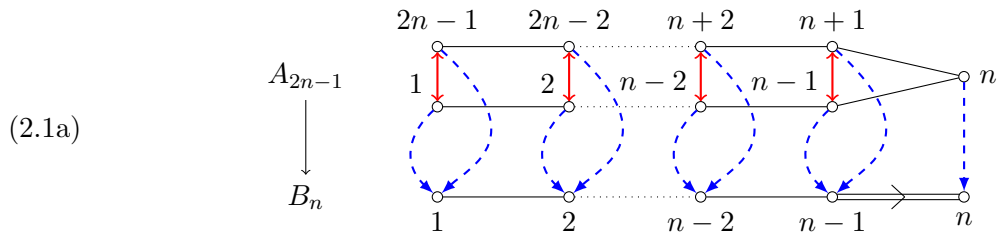
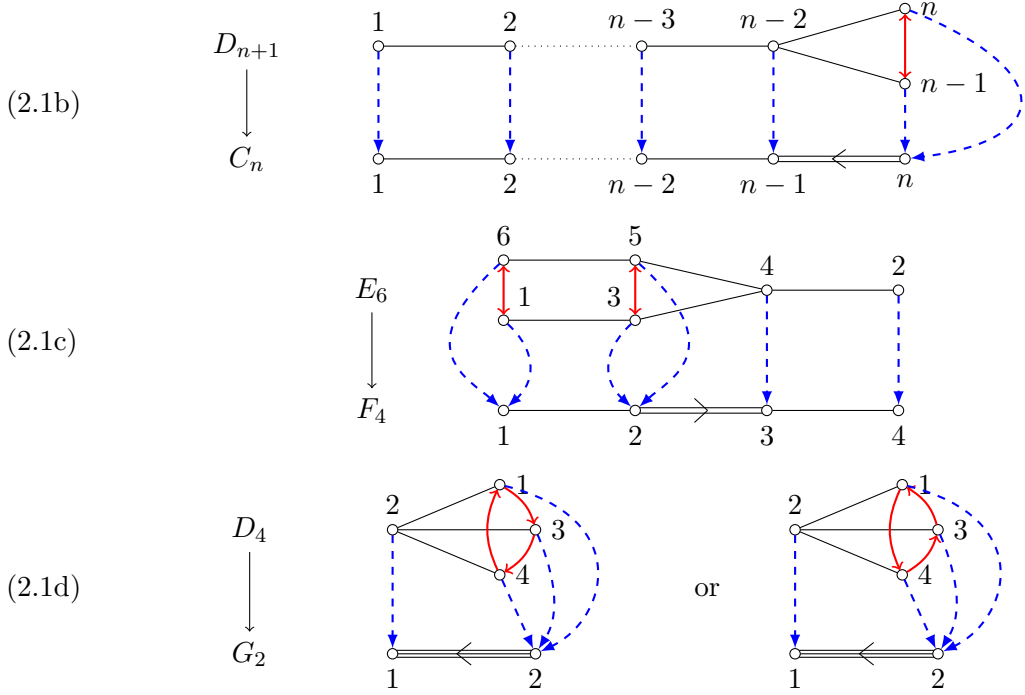
\begin{subequations}
\label{eq: diagram foldings}
\begin{gather}
\begin{tikzpicture}[xscale=1.75,yscale=.8,baseline=0]
\node (A2n1) at (0,1) {$A_{2n-1}$};
\node[dynkdot,label={above:$n+1$}] (A6) at (4,1.5) {};
\node[dynkdot,label={above:$n+2$}] (A7) at (3,1.5) {};
\node[dynkdot,label={above:$2n-2$}] (A8) at (2,1.5) {};
\node[dynkdot,label={above:$2n-1$}] (A9) at (1,1.5) {};
\node[dynkdot,label={above left:$n-1$}] (A4) at (4,0.5) {};
\node[dynkdot,label={above left:$n-2$}] (A3) at (3,0.5) {};
\node[dynkdot,label={above left:$2$}] (A2) at (2,0.5) {};
\node[dynkdot,label={above left:$1$}] (A1) at (1,0.5) {};
\node[dynkdot,label={right:$n$}] (A5) at (5,1) {};
\path[-]
 (A1) edge (A2)
 (A3) edge (A4)
 (A4) edge (A5)
 (A5) edge (A6)
 (A6) edge (A7)
 (A8) edge (A9);
\path[-,dotted] (A2) edge (A3) (A7) edge (A8);
\path[<->,thick,red] (A1) edge (A9) (A2) edge (A8) (A3) edge (A7) (A4) edge (A6);
\def\Foffset{-1}
\node (Bn) at (0,\Foffset) {$B_n$};
\foreach \x in {1,2}
{\node[dynkdot,label={below:$\x$}] (B\x) at (\x,\Foffset) {};}
\node[dynkdot,label={below:$n-2$}] (B3) at (3,\Foffset) {};
\node[dynkdot,label={below:$n-1$}] (B4) at (4,\Foffset) {};
\node[dynkdot,label={below:$n$}] (B5) at (5,\Foffset) {};
\path[-] (B1) edge (B2)  (B3) edge (B4);
\draw[-,dotted] (B2) -- (B3);
\draw[-] (B4.30) -- (B5.150);
\draw[-] (B4.330) -- (B5.210);
\draw[-] (4.55,\Foffset) -- (4.45,\Foffset+.2);
\draw[-] (4.55,\Foffset) -- (4.45,\Foffset-.2);
 \draw[-latex,dashed,color=blue,thick]
 (A1) .. controls (0.75,\Foffset+1) and (0.75,\Foffset+.5) .. (B1);
 \draw[-latex,dashed,color=blue,thick]
 (A9) .. controls (1.75,\Foffset+1) and (1.25,\Foffset+.5) .. (B1);
 \draw[-latex,dashed,color=blue,thick]
 (A2) .. controls (1.75,\Foffset+1) and (1.75,\Foffset+.5) .. (B2);
 \draw[-latex,dashed,color=blue,thick]
 (A8) .. controls (2.75,\Foffset+1) and (2.25,\Foffset+.5) .. (B2);
 \draw[-latex,dashed,color=blue,thick]
 (A3) .. controls (2.75,\Foffset+1) and (2.75,\Foffset+.5) .. (B3);
 \draw[-latex,dashed,color=blue,thick]
 (A7) .. controls (3.75,\Foffset+1) and (3.25,\Foffset+.5) .. (B3);
 \draw[-latex,dashed,color=blue,thick]
 (A4) .. controls (3.75,\Foffset+1) and (3.75,\Foffset+.5) .. (B4);
 \draw[-latex,dashed,color=blue,thick]
 (A6) .. controls (4.75,\Foffset+1) and (4.25,\Foffset+.5) .. (B4);
 \draw[-latex,dashed,color=blue,thick] (A5) -- (B5);
\draw[->] (A2n1) -- (Bn);
\end{tikzpicture}
\label{eq: B_n} \allowdisplaybreaks \\
\begin{tikzpicture}[xscale=1.75,yscale=1.25,baseline=-25]
\node (Dn1) at (0,0) {$D_{n+1}$};
\node[dynkdot,label={above:$1$}] (D1) at (1,0){};
\node[dynkdot,label={above:$2$}] (D2) at (2,0) {};
\node[dynkdot,label={above:$n-3$}] (D3) at (3,0) {};
\node[dynkdot,label={above:$n-2$}] (D4) at (4,0) {};
\node[dynkdot,label={right:$n$}] (D6) at (5,.4) {};
\node[dynkdot,label={right:$n-1$}] (D5) at (5,-.4) {};
\path[-] (D1) edge (D2)
  (D3) edge (D4)
  (D4) edge (D5)
  (D4) edge (D6);
\draw[-,dotted] (D2) -- (D3);
\path[<->,thick,red] (D6) edge (D5);
\def\Coffset{-1.2}
\node (Cn) at (0,\Coffset) {$C_n$};
\foreach \x in {1,2}
{\node[dynkdot,label={below:$\x$}] (C\x) at (\x,\Coffset) {};}
\node[dynkdot,label={below:$n-2$}] (C3) at (3,\Coffset) {};
\node[dynkdot,label={below:$n-1$}] (C4) at (4,\Coffset) {};
\node[dynkdot,label={below:$n$}] (C5) at (5,\Coffset) {};
\draw[-] (C1) -- (C2);
\draw[-,dotted] (C2) -- (C3);
\draw[-] (C3) -- (C4);
\draw[-] (C4.30) -- (C5.150);
\draw[-] (C4.330) -- (C5.210);
\draw[-] (4.55,\Coffset+.1) -- (4.45,\Coffset) -- (4.55,\Coffset-.1);
\path[-latex,dashed,color=blue,thick]
 (D1) edge (C1)
 (D2) edge (C2)
 (D3) edge (C3)
 (D4) edge (C4)
 (D5) edge (C5);
\draw[-latex,dashed,color=blue,thick]
 (D6) .. controls (6.15,-.25) and (6.15,-1) .. (C5);
 \draw[->] (Dn1) -- (Cn);
\end{tikzpicture}
 \label{eq: C_n} \allowdisplaybreaks \\
\begin{tikzpicture}[xscale=1.75,yscale=.8,baseline=0]
\node (E6desc) at (0,1) {$E_6$};
\node[dynkdot,label={above:$2$}] (E2) at (4,1) {};
\node[dynkdot,label={above:$4$}] (E4) at (3,1) {};
\node[dynkdot,label={above:$5$}] (E5) at (2,1.5) {};
\node[dynkdot,label={above:$6$}] (E6) at (1,1.5) {};
\node[dynkdot,label={above left:$3$}] (E3) at (2,0.5) {};
\node[dynkdot,label={above right:$1$}] (E1) at (1,0.5) {};
\path[-]
 (E2) edge (E4)
 (E4) edge (E5)
 (E4) edge (E3)
 (E5) edge (E6)
 (E3) edge (E1);
\path[<->,thick,red] (E3) edge (E5) (E1) edge (E6);
\def\Foffset{-1}
\node (F4desc) at (0,\Foffset) {$F_4$};
\foreach \x in {1,2,3,4}
{\node[dynkdot,label={below:$\x$}] (F\x) at (\x,\Foffset) {};}
\draw[-] (F1.east) -- (F2.west);
\draw[-] (F3) -- (F4);
\draw[-] (F2.30) -- (F3.150);
\draw[-] (F2.330) -- (F3.210);
\draw[-] (2.55,\Foffset) -- (2.45,\Foffset+.2);
\draw[-] (2.55,\Foffset) -- (2.45,\Foffset-.2);
\path[-latex,dashed,color=blue,thick]
 (E2) edge (F4)
 (E4) edge (F3);
\draw[-latex,dashed,color=blue,thick]
 (E1) .. controls (1.25,\Foffset+1) and (1.25,\Foffset+.5) .. (F1);
\draw[-latex,dashed,color=blue,thick]
 (E3) .. controls (1.75,\Foffset+1) and (1.75,\Foffset+.5) .. (F2);
\draw[-latex,dashed,color=blue,thick]
 (E5) .. controls (2.75,\Foffset+1) and (2.25,\Foffset+.5) .. (F2);
\draw[-latex,dashed,color=blue,thick]
 (E6) .. controls (0.25,\Foffset+1) and (0.75,\Foffset+.5) .. (F1);
\draw[->] (E6desc) -- (F4desc);
\end{tikzpicture}
 \label{eq: F_4} \allowdisplaybreaks \\
\begin{tikzpicture}[xscale=1.9,yscale=1.5,baseline=-25]
\node (D4desc) at (0,0) {$D_{4}$};
\node[dynkdot,label={right:$1$}] (D1) at (1.75,.4){};
\node[dynkdot,label={above:$2$}] (D2) at (1,0) {};
\node[dynkdot,label={right:$3$}] (D3) at (2,0) {};
\node[dynkdot,label={right:$4$}] (D4) at (1.75,-.4) {};
\draw[-] (D1) -- (D2);
\draw[-] (D3) -- (D2);
\draw[-] (D4) -- (D2);
\path[->,red,thick]
(D1) edge[bend left=20] (D3)
(D3) edge[bend left=20] (D4)
(D4) edge[bend left=20] (D1);
\def\Coffset{-1.1}
\node (G2desc) at (0,\Coffset) {$G_2$};
\node[dynkdot,label={below:$1$}] (G1) at (1,\Coffset){};
\node[dynkdot,label={below:$2$}] (G2) at (2,\Coffset) {};
\draw[-] (G1) -- (G2);
\draw[-] (G1.40) -- (G2.140);
\draw[-] (G1.320) -- (G2.220);
\draw[-] (1.55,\Coffset+.1) -- (1.45,\Coffset) -- (1.55,\Coffset-.1);
\path[-latex,dashed,color=blue,thick]
 (D2) edge (G1);
\draw[-latex,dashed,color=blue,thick]
 (D1) .. controls (2.7,0.1) and (2.5,-0.8) .. (G2);
\draw[-latex,dashed,color=blue,thick]
 (D3) .. controls (2.1,-0.3) and (2.2,-0.7) .. (G2);
\draw[-latex,dashed,color=blue,thick] (D4) -- (G2);
\draw[->] (D4desc) -- (G2desc);
\end{tikzpicture}
\qquad \text{or} \qquad
\begin{tikzpicture}[xscale=1.9,yscale=1.5,baseline=-25]
\node[dynkdot,label={right:$1$}] (D1) at (1.75,.4){};
\node[dynkdot,label={above:$2$}] (D2) at (1,0) {};
\node[dynkdot,label={right:$3$}] (D3) at (2,0) {};
\node[dynkdot,label={right:$4$}] (D4) at (1.75,-.4) {};
\draw[-] (D1) -- (D2);
\draw[-] (D3) -- (D2);
\draw[-] (D4) -- (D2);
\path[<-,red,thick]
(D1) edge[bend left=20] (D3)
(D3) edge[bend left=20] (D4)
(D4) edge[bend left=20] (D1);
\def\Coffset{-1.1}
\node[dynkdot,label={below:$1$}] (G1) at (1,\Coffset){};
\node[dynkdot,label={below:$2$}] (G2) at (2,\Coffset) {};
\draw[-] (G1) -- (G2);
\draw[-] (G1.40) -- (G2.140);
\draw[-] (G1.320) -- (G2.220);
\draw[-] (1.55,\Coffset+.1) -- (1.45,\Coffset) -- (1.55,\Coffset-.1);
\path[-latex,dashed,color=blue,thick]
 (D2) edge (G1);
\draw[-latex,dashed,color=blue,thick]
 (D1) .. controls (2.7,0.1) and (2.5,-0.8) .. (G2);
\draw[-latex,dashed,color=blue,thick]
 (D3) .. controls (2.1,-0.3) and (2.2,-0.7) .. (G2);
\draw[-latex,dashed,color=blue,thick] (D4) -- (G2);
\end{tikzpicture}
  \label{eq: G_2}
\end{gather}
\end{subequations}

\begin{definition}\label{def: overline sigma}
We denote by $\overline{\sigma}$ the map from vertices of $\Delta$ to the one of $\widehat{\Delta}$ induced by $\sigma$.
\end{definition}

For example, $\overline{\sigma}(4)=3$ when $\Delta$ is of type $E_6$.

Let $\mathsf{W}$ be the Weyl group, generated by simple reflections
$( s_i \mid i \in I)$ corresponding to $\Delta$,
and $w_0$ the longest element of $\mathsf{W}$. We denote by
$\ast$ the involution on $I$ defined by
\begin{align}\label{eq:invI}
w_0(\al_i)=-\al_{i^*}.
\end{align}
 We also denote by  $\Phi$ the set of all roots and by $\Phi^+$ the set of all positive roots. Let $\N = |\PR|$.
\begin{remark}
In this paper, we fix the length $|\alpha_i|$ of the longest root as $1$. Furthermore, $q_s = q_i$, where $\alpha_i$ is a shortest root.
\end{remark}

We say that two reduced expressions
$\redex=s_{i_1}s_{i_2}\cdots s_{i_{\ell}}$ and
$\redex'=s_{j_1}s_{j_2}\cdots s_{j_{\ell}}$ of $w
\in W$ are \defn{commutation equivalent}, denoted by $\widetilde{w} \sim \widetilde{w}'$, if
$s_{j_1}s_{j_2}\cdots s_{j_{\ell}}$ is obtained from
$s_{i_1}s_{i_2}\cdots s_{i_{\ell}}$ by applying the commutation
relations $s_{k}s_{l} = s_{l}s_{k}$.
We denote by $[\redex]$ the commutation equivalence class
of $\redex$.

 The following proposition is well-known:

\begin{proposition}
For $\redez=s_{i_1}s_{i_2}\cdots s_{i_{\N-1}}s_{i_\N}$,
$\redez'=s_{i_\N^*}s_{i_1}s_{i_2}\cdots s_{i_{\N-1}}$ is a reduced expression of $w_0$ and $[\redez'] \ne [\redez]$. Similarly,
$\redez''=s_{i_2}\cdots s_{i_{\N-1}}s_{i_\N}s_{i^*_1}$ is a reduced expression of $w_0$
and $[\redez''] \ne [\redez]$.
\end{proposition}

The action of the reflection functor $r_i$ on $[\redez]$ is defined by
\[
r_i \, [\redez]=
\begin{cases}
[(s_{i^*},s_{i_1}\cdots, s_{i_{\N-1}})] &  \text{if there is $\redez'=(s_{i_1}, \cdots, s_{i_{\N-1}}, s_i)\in [\redez]$,}\\
\ [\redez] &  \text{otherwise.}
\end{cases}
\]

\begin{definition}[\cite{OS15}]  \label{def: ref equi}
 Let  $[\redez]$ and $[\redez']$ be two commutation classes. We say $[\redez]$ and $[\redez']$ are
\defn{reflection equivalent} and write $[\redez]\overset{r}{\sim} [\redez']$ if $[\redez']$ can be obtained from $[\redez]$ by a
sequence of reflection maps. The equivalence class
$\lf \redez \rf \seteq \{\, [\redez] \mid [\redez]\overset{r}{\sim}[\redez']\, \}$
with respect to the reflection equivalence relation  is called an \defn{$r$-cluster point}.
\end{definition}

Let $\sigma \in GL(\C\Phi)$ be a linear transformation of finite order
which preserves $\Pi$. Hence $\sigma$ preserves
$\Phi$ itself and normalizes $\mathsf{W}$ and so $\mathsf{W}$ acts by conjugation on the coset $W\sigma$.
Note that $\sigma$ agrees with the Dynkin diagram automorphism given by~\eqref{eq: diagram foldings}.

Let $\{ \Pi_{i_1}, \ldots,\Pi_{i_k} \}$ be the all orbits of $\Pi$ in
$\Phi$ with respect to $\sigma$. For each $r\in \{1, \cdots, k \}$, choose $\alpha_{i_r} \in \Pi_{i_r}$ arbitrarily,
and let $s_{i_r} \in W$ denote the corresponding reflection. Let $w$
be the product of $s_{i_1}, \ldots , s_{i_k}$ in any order. The
element $w\sigma \in W\sigma$ of $w\in W$ thus obtained is called a \defn{$\sigma$-Coxeter element}.
\begin{itemize}
\item If $\sigma = \id$, then a \defn{$\id$-Coxeter element} is simply called a \defn{Coxeter element}.
\item If $\sigma$ is an involution, then a \defn{$\sigma$-Coxeter element} is also called a \defn{twisted Coxeter element}.
\item If $\sigma$ is of order $3$, then a \defn{$\sigma$-Coxeter element} is also called a \defn{triply twisted Coxeter element}.
\end{itemize}

\begin{example}\label{ex: t coxeter for exceptional} \hfill
\begin{enumerate}
\item For $\sigma$ in~\eqref{eq: F_4}, $s_1s_3s_4s_2 \sigma$ is a twisted Coxeter element of finite type $E_6$.
\item For $\sigma$ in~\eqref{eq: G_2}, $s_2s_1 \sigma$ is a triply twisted Coxeter element of finite type $D_4$.
\end{enumerate}
\end{example}

For a reduced expression $s_{j_1} \cdots s_{j_\ell}$, we define
\[
(s_{j_1} \cdots s_{j_\ell})^{\sigma} \seteq s_{j^{\sigma}_1} \cdots s_{j^{\sigma}_\ell} \text{ and }
(s_{j_1} \cdots s_{j_s})^{k \sigma} \seteq  ( \cdots ((s_{j_1} \cdots s_{j_s} \underbrace{ )^{\sigma} )^{\sigma} \cdots )^{\sigma}}_{ \text{ $k$-times} }.
\]

\begin{proposition}[{\cite{OS16B,OS16C}}]
For each $($triply$)$ twisted Coxeter element $s_{i_1} \cdots s_{i_\ell} \sigma$,
\begin{align} \label{eq: red twisted}
\redez^{\sigma} \seteq \prod_{i=0}^{ (|\PR| / \ell)  -1} (s_{i_1} \cdots s_{i_\ell})^{k \sigma} \text{ is a reduced expression of $w_0$.}
\end{align}
\end{proposition}

\begin{definition} \hfill
\begin{enumerate}
\item For $\sigma$ in \eqref{eq: B_n}, \eqref{eq: C_n}, \eqref{eq: F_4}, $\lf \mQ \rf \seteq \lf \redez^{\sigma} \rf$ is called the \defn{twisted adapted $r$-cluster point} of type $A_{2n-1}$, $D_{n+1}$ or $E_6$.
A commutation class $[\mQ'] \in \lf \mQ \rf$ is called a \defn{twisted adapted class}.
\item For $\sigma$ and $\sigma'$ in \eqref{eq: G_2},  $\lf \mathfrak{Q} \rf \seteq \lf \redez^{\sigma} \rf \sqcup \lf \redez^{\sigma'} \rf $ is called the \defn{triply twisted adapted $r$-cluster point} of $D_4$.
A commutation class $[\mathfrak{Q}'] \in \lf \mathfrak{Q} \rf$ is called a \defn{triply twisted adapted class}.
\end{enumerate}
\end{definition}

\begin{example}
\label{ex: redez for exceptional}
By Example~\ref{ex: t coxeter for exceptional}, we have the following reduced expressions of $w_0$ contained $\lf \mQ \rf $ (resp.~$\lf \mathfrak{Q} \rf$):
\begin{enumerate}
\item $\redez^{\sigma} \seteq (s_1s_3s_4s_2 s_6s_5s_4s_2)^4 s_1s_3s_4s_2 \in \lf \mQ \rf $ of type $E_6$;
\item $\redez^{\sigma} \seteq (s_2s_1 s_2s_3 s_2s_4)^2 \in  \lf \mathfrak{Q} \rf$ of type $D_4$.
\end{enumerate}
\end{example}

For any quiver $\mathtt{Q}$, we say that a vertex $i$ in $\mathtt{Q}$ is a \defn{source} (resp.\ \defn{sink}) if and only if
there are only exiting arrows out of it (resp.\ entering arrows into it).
Let $Q$ be a Dynkin quiver by orienting edges of a Dynkin diagram $\Delta$ of type $ADE$. We denote by $Q^{{\rm rev}}$ the quiver obtained
by reversing all arrows of $Q$.
For any $i \in I$, let $s_iQ$ denote the quiver obtained by $Q$ by reversing the arrows incident with $i$.
We say that a reduced expression $\redex=s_{i_1}s_{i_2}\cdots s_{i_{\ell(w)}}$
 of $w \in W$  is
\defn{adapted} to $Q$ if $i_k$ is a sink of the quiver $s_{i_{k-1}} \cdots s_{i_2}s_{i_1}Q$ for all $1 \le k \le \ell(w)$.

The following results are well-known:

\begin{theorem} [\cite{Bedard99,Lus90}]  \label{thm: Qs} \hfill
\begin{enumerate}[{\rm (1)}]
\item Any reduced expression $\redez$ of $w_0$ is adapted
to at most one Dynkin quiver $Q$.
\item For each Dynkin quiver $Q$, there is a reduced expression $\redez$ of $w_0$ adapted to $Q$. Moreover,
any reduced expression $\redez'$ in $[\redez]$ is adapted to $Q$,
and the commutation equivalence class $[\redez]$
is uniquely determined by $Q$.
We denote by $[Q]$ of the commutation class $[\redez]$.
\item For any Dynkin quivers $Q$ and $Q'$ of a fixed Dynkin diagram, the commutation classes $[Q]$ and $[Q']$ are reflection equivalent.
\end{enumerate}
\end{theorem}

From Theorem~\ref{thm: Qs}(3), for a fixed Dynkin diagram $\Delta$, there exists only a single $r$-cluster point $\lf \Delta \rf$.

\begin{remark}
In~\cite{KO17}, the notation $\lf Q \rf$ was used instead of $\lf \Delta \rf$, where $Q$ is a Dynkin quiver of $\Delta$.
\end{remark}

\begin{theorem}[{\cite{OS16B,OS16C}}]
The number of commutation classes in $\lf \Delta \rf$ is the same as the one of $\lf \mQ \rf$, when their types coincide.
\end{theorem}

\subsection{(Folded) Auslander-Reiten quivers}

Note that for any reduced expression $\redez = s_{i_1} \cdots s_{i_{\N}}$ of $w_0$, we can define a total order $<_{\redez}$ on $\Phi^+$ as follows:
\begin{equation}
\label{compatible reading}
\alpha_{i_1} <_{\redez} s_{i_1}( \alpha_{i_2}) <_{\redez} s_{i_1} s_{i_2}(\alpha_{i_3}) <_{\redez} \cdots <_{\redez} s_{i_1} s_{i_2} \cdots s_{i_{\N-1}}(\alpha_{i_{\N}})
\end{equation}
Denote $\beta^{\redez}_k \seteq s_{i_1} s_{i_2} \cdots s_{i_{k-1}}(\alpha_{i_k})$.
The order $<_{\redez}$ is \defn{convex} in the following sense: For any $\al,\be \in \Phi^+$ satisfying $\al+\be \in \Phi^+$, we have either
\begin{align*}
\al <_{\redez} \al+\be <_{\redez} \be \quad\text{ or }\quad \be <_{\redez} \al+\be <_{\redez} \al.
\end{align*}
By considering  $<_{\redez'}$ for all $\redez' \in [\redez]$, we can obtain the \defn{convex \emph{partial} order} $\prec_{[\redez]}$ on $\Phi^+$ defined as follows:
\begin{align}
\al \prec_{[\redez]}  \be \quad \text{ if } \al <_{\redez'} \be \text{ for all } \redez' \in [\redez].
\end{align}

In~\cite{OS15}, the combinatorial \defn{Auslander-Reiten (AR) quiver}  $\Upsilon_{[\redez]}$ is introduced
for \emph{any} commutation class $[\redez]$ of any finite type.

\begin{algorithm}
\label{Alg_AbsAR}
Let $\redez=s_{i_1}s_{i_2} \cdots s_{i_{\N}}$ be
a reduced expression of the longest element $w_0\in W$.
The quiver $\Upsilon_{\redez}=(\Upsilon^0_{\redez},
\Upsilon^1_{\redez})$ associated to $\redez$ is constructed as follows:
\begin{enumerate}[{\rm (Q1)}]
\item $\Upsilon_{\redez}^0$ consists of $\N$ vertices labeled by $\beta^{\redez}_1, \cdots, \beta^{\redez}_{\N}$ $($see \eqref{compatible reading}$)$.
 \item There is an arrow from $\beta^{\redez}_k$ to $\beta^{\redez}_j$ for $1\leq j<k \leq \N$ if $(i)$ two vertices $i_k$ and $i_j$ are connected in the Dynkin diagram,
$(ii)$ for $j'$ such that $j<j'<k$, we have $i_{j'}\neq i_j, i_k$.
\end{enumerate}
\end{algorithm}

\begin{theorem}[{\cite{OS15}}]
\hfill
\begin{enumerate}[{\rm (1)}]
\item For any $\redez' \sim \redez'' \in [\redez]$, we have $\Upsilon_{\redez'} \iso \Upsilon_{\redez''}$ as quivers with the same labeling. Thus $\Upsilon_{[\redez]}$ is well-defined.
\item $\al \prec_{[\redez]} \be$ if and only if there exists a path from $\be$ to $\al$ in $\Upsilon_{[\redez]}$.
\item When $[\redez]$ coincides with $[Q]$ for some Dynkin quiver $Q$, $\Upsilon_{[\redez]}$ is isomorphic to the AR quiver $\Gamma_Q$ associated to the Dynkin quiver $Q$.
\end{enumerate}
\end{theorem}
We call $\Upsilon_{[\redez]}$ the \defn{twisted Auslander-Reiten (AR) quiver} when $[\redez]$ is (triply) twisted adapted.

By the above theorem, for $\beta \in \PR$ and $[\redez]$, we can assign an index $i$ on $\beta$ satisfying $\beta= \beta^{\redez'}_k$ for some $\redez'=\beta^{\redez}_1, \cdots, \beta^{\redez}_{\N} \in [\redez]$ with $i_k=i$. We call $i$ the \defn{residue} with respect to $[\redez]$.

For the rest of this subsection, we denote by $[\rrz]$ a reduced expression of $w_0$ of type $X$ that is contained in one of $\lf \Delta \rf$, $\lf \mQ \rf$ or $\lf \mathfrak{Q} \rf$, all of which are of simply laced types.
Recall that $[\rrz]$ is associated to a Dynkin diagram automorphism $\sigma$. We denote by $\widehat{X}$ the type of the Dynkin diagram obtained by the automorphism $\sigma$. For instance, if $[\rrz] \in \lf \mQ \rf$ of type $D_5$, $\widehat{X}$
represents the finite type $C_4$. We remark that $X$ is different from $\widehat{X}$ if and only if $\sigma \ne \id$.

\begin{definition}
\label{def: length and d}
Let $\mathsf{a}$ be an arrow between a vertex of residue $i$ and a vertex of residue $j$ in $\Upsilon_{[\rrz]}$.
Recall that $\widehat{I}$ is the index set of the Dynkin diagram $\widehat{\Delta}$ of~\eqref{eq: diagram foldings}. Let $\{\widehat{\alpha}_{\widehat{\imath}} \mid \widehat{\imath} \in \widehat{I} \}$ be the corresponding simple roots of $\widehat{\Delta}$.
We define the \defn{length} $\ell(\mathsf{a})$ as follows:
\[
\ell(\mathsf{a}) \seteq \min\{ |\widehat{\al}_{\widehat{\imath}}|^2, |\widehat{\al}_{\widehat{\jmath}}|^2 \}.
\]
Let $\mathsf{d}$ denote the order of the Dynkin diagram automorphism $\sigma$. 
\end{definition}
Furthermore, the lengths of arrows in a (twisted) AR quiver $\Upsilon_{[\rrz]}$ is equal to $1 / \mathsf{d}$ except for possibly when $\sigma$ is \eqref{eq: B_n} or \eqref{eq: F_4}, in which case the length of an arrow may also be $1$.

Then it is proved in~\cite{OS16B,OS16C} that we can assign $\beta \in \Upsilon_{[\rrz]}$ a coordinate $\Omega_{[\rrz]}(\be) \seteq (i,p) \in I \times \dfrac{1}{\mathsf{d}}\Z$ which is \emph{compatible} to the lengths of arrows.
Here $i$ is the residue of $\beta$ with respect to $[\rrz]$.
Here we give examples instead of giving all the details (see~\cite{OS16B,OS16C} for details). We remark here that the coordinate system on $\Gamma_Q$ was originally given in~\cite{HL10}.

\begin{example}
\label{ex:twisted AR quiver E}
\hfill \\
(1) For the Dynkin quiver
$
 Q =
\begin{tikzpicture}[baseline=5,>=stealth,yscale=0.7,xscale=0.8,font=\tiny]
\node[dynkdot,label={below:1}] (-2) at (-2,0) {};
\node[dynkdot,label={below:3}] (-1) at (-1,0) {};
\node[dynkdot,label={below:4}] (0) at (0,0) {};
\node[dynkdot,label={below:5}] (1) at (1,0) {};
\node[dynkdot,label={below:6}] (2) at (2,0) {};
\node[dynkdot,label={[label distance=-4pt]-30:2}] (t) at (0,1) {};
\draw[->] (-2) -- (-1);
\draw[->] (-1) -- (0);
\draw[->] (0) -- (1);
\draw[->] (1) -- (2);
\draw[->] (t) -- (0);
\end{tikzpicture}
$
of type $E_6$, its AR quiver $\Gamma_Q$ with the coordinate system can be depicted as follows: $\left({\scriptstyle\prt{a_1a_2a_3}{a_4a_5a_6}} \seteq \displaystyle\sum_{i=1}^6 a_i \al_i\right)$
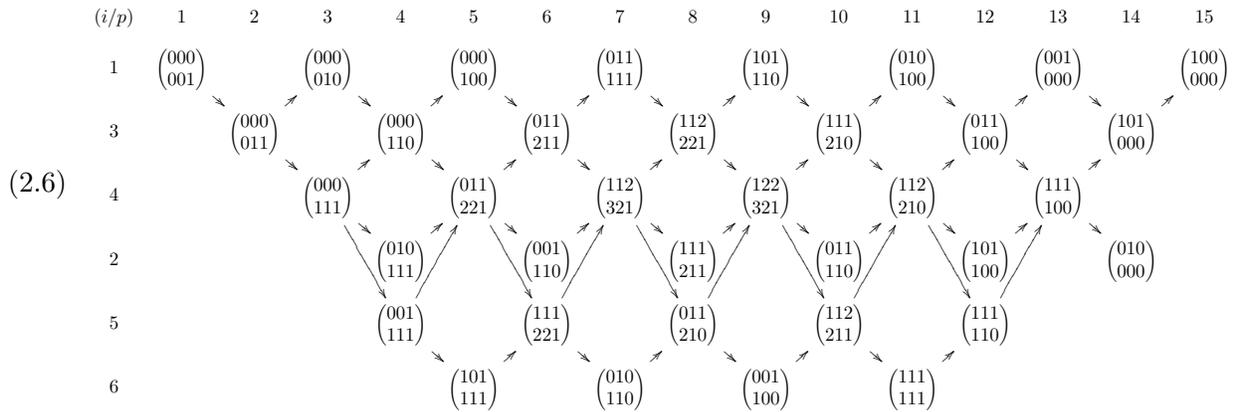
\begin{equation}
\label{eq: E6 AR quiver}
\raisebox{6em}{\scalebox{0.65}{\xymatrix@R=1.0ex@C=1.3ex{
(i/p) & 1 & 2 & 3 & 4 & 5 & 6 & 7 & 8 & 9 & 10 & 11 & 12 & 13 & 14 & 15   \\
1&{\scriptstyle\prt{000}{001}} \ar[dr] && {\scriptstyle\prt{000}{010}} \ar[dr] && {\scriptstyle\prt{000}{100}} \ar[dr]&& {\scriptstyle\prt{011}{111}}\ar[dr]
&&{\scriptstyle\prt{101}{110}}\ar[dr] && {\scriptstyle\prt{010}{100}}\ar[dr] && {\scriptstyle\prt{001}{000}} \ar[dr]&& {\scriptstyle\prt{100}{000}} \\
3&& {\scriptstyle\prt{000}{011}}\ar[dr]\ar[ur] && {\scriptstyle\prt{000}{110}}\ar[dr]\ar[ur] && {\scriptstyle\prt{011}{211}}\ar[dr]\ar[ur]
&& {\scriptstyle\prt{112}{221}}\ar[dr]\ar[ur] && {\scriptstyle\prt{111}{210}}\ar[dr]\ar[ur] && {\scriptstyle\prt{011}{100}}\ar[dr]\ar[ur]
&& {\scriptstyle\prt{101}{000}}\ar[ur] \\ 
4&&& {\scriptstyle\prt{000}{111}}\ar[ddr]\ar[dr]\ar[ur]&& {\scriptstyle\prt{011}{221}}\ar[ddr]\ar[dr]\ar[ur]&& {\scriptstyle\prt{112}{321}}\ar[ddr]\ar[dr]\ar[ur]
&& {\scriptstyle\prt{122}{321}}\ar[ddr]\ar[dr]\ar[ur] && {\scriptstyle\prt{112}{210}}\ar[ddr]\ar[dr]\ar[ur]&& {\scriptstyle\prt{111}{100}}\ar[dr]\ar[ur] \\
2&&&& {\scriptstyle\prt{010}{111}}\ar[ur]&& {\scriptstyle\prt{001}{110}}\ar[ur]&& {\scriptstyle\prt{111}{211}}\ar[ur]&& {\scriptstyle\prt{011}{110}}\ar[ur]
&& {\scriptstyle\prt{101}{100}}\ar[ur]&& {\scriptstyle\prt{010}{000}} \\
5&&&& {\scriptstyle\prt{001}{111}}\ar[uur]\ar[dr]&& {\scriptstyle\prt{111}{221}}\ar[uur]\ar[dr]&& {\scriptstyle\prt{011}{210}}\ar[uur]\ar[dr]
&& {\scriptstyle\prt{112}{211}}\ar[uur]\ar[dr]&& {\scriptstyle\prt{111}{110}}\ar[uur] \\
6&&&&& {\scriptstyle\prt{101}{111}}\ar[ur]&& {\scriptstyle\prt{010}{110}}\ar[ur]&& {\scriptstyle\prt{001}{100}}\ar[ur]&& {\scriptstyle\prt{111}{111}}\ar[ur] \\
}}}
\end{equation}

\noindent
(2) For $[\rrz]=[\mQ]$ of type $E_6$ in (1) of Example~\ref{ex: redez for exceptional}, its twisted AR quiver $\Upsilon_{[\mQ]}$ with the coordinate system can be depicted as follows:
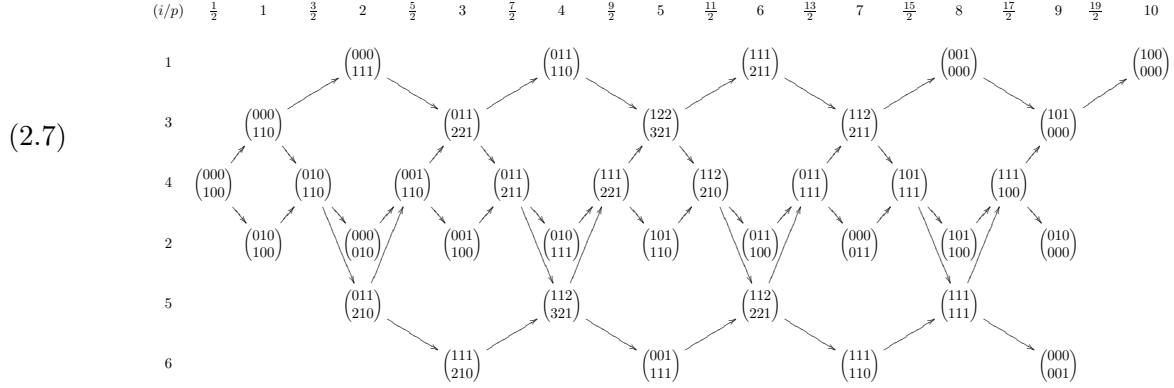
\begin{equation}
\label{eq:unfolded1}
 \raisebox{4.6em}{\scalebox{0.51}{\xymatrix@C=0.1ex@R=2.6ex{
(i/p) & \frac{1}{2}  & 1 & \frac{3}{2} & 2 & \frac{5}{2} & 3 & \frac{7}{2} & 4 & \frac{9}{2} & 5 & \frac{11}{2} & 6 & \frac{13}{2} & 7 & \frac{15}{2} & 8 & \frac{17}{2} & 9 & \frac{19}{2} &&& 10    \\
1 &&&& {\scriptstyle\prt{000}{111}} \ar@{->}[drr] &&&& {\scriptstyle\prt{011}{110}}\ar@{->}[drr] &&&& {\scriptstyle\prt{111}{211}}\ar@{->}[drr]
&&&& {\scriptstyle\prt{001}{000}} \ar@{->}[drr] &&&&&&  {\scriptstyle\prt{100}{000}} \\
3 && {\scriptstyle\prt{000}{110}} \ar@{->}[urr]\ar@{->}[dr] &&&& {\scriptstyle\prt{011}{221}}\ar@{->}[urr]\ar@{->}[dr] &&&&{\scriptstyle\prt{122}{321}}
\ar@{->}[urr]\ar@{->}[dr]
&&&&{\scriptstyle\prt{112}{211}} \ar@{->}[urr]\ar@{->}[dr] &&&& {\scriptstyle\prt{101}{000}}\ar@{->}[urrrr]\\
4 & {\scriptstyle\prt{000}{100}}\ar@{->}[ur]\ar@{->}[dr] && {\scriptstyle\prt{010}{110}} \ar@{->}[dr]\ar@{->}[ddr] && {\scriptstyle\prt{001}{110}} \ar@{->}[ur]\ar@{->}[dr]
&& {\scriptstyle\prt{011}{211}} \ar@{->}[dr]\ar@{->}[ddr] && {\scriptstyle\prt{111}{221}}\ar@{->}[ur]\ar@{->}[dr] &&{\scriptstyle\prt{112}{210}}\ar@{->}[ddr]\ar@{->}[dr]
&&{\scriptstyle\prt{011}{111}}\ar@{->}[ur]\ar@{->}[dr] &&
{\scriptstyle\prt{101}{111}}\ar@{->}[dr]\ar@{->}[ddr] && {\scriptstyle\prt{111}{100}}\ar@{->}[ur]\ar@{->}[dr]\\
2 && {\scriptstyle\prt{010}{100}}\ar@{->}[ur] && {\scriptstyle\prt{000}{010}} \ar@{->}[ur] && {\scriptstyle\prt{001}{100}} \ar@{->}[ur] &&
{\scriptstyle\prt{010}{111}}\ar@{->}[ur] &&{\scriptstyle\prt{101}{110}}\ar@{->}[ur] && {\scriptstyle\prt{011}{100}} \ar@{->}[ur]&& {\scriptstyle\prt{000}{011}}
\ar@{->}[ur] &&{\scriptstyle\prt{101}{100}}\ar@{->}[ur]  &&{\scriptstyle\prt{010}{000}} \\
5 &&&& {\scriptstyle\prt{011}{210}}\ar@{->}[drr]\ar@{->}[uur] &&&& {\scriptstyle\prt{112}{321}}\ar@{->}[drr]\ar@{->}[uur]
&&&& {\scriptstyle\prt{112}{221}}\ar@{->}[drr]\ar@{->}[uur] &&&& {\scriptstyle\prt{111}{111}} \ar@{->}[drr]\ar@{->}[uur]\\
6 &&&&&& {\scriptstyle\prt{111}{210}}\ar@{->}[urr]  &&&& {\scriptstyle\prt{001}{111}} \ar@{->}[urr]
&&&& {\scriptstyle\prt{111}{110}} \ar@{->}[urr]  &&&& {\scriptstyle\prt{000}{001}}
}}}
\end{equation}
\end{example}

\begin{example}
\label{ex:twisted AR quiver D}
\hfill \\
(1) For the Dynkin quiver
$
 Q =
\begin{tikzpicture}[baseline=0,>=stealth,yscale=0.7,xscale=0.8,font=\tiny]
\node[dynkdot,label={below:1}] (1) at (-1,0) {};
\node[dynkdot,label={below:2}] (2) at (0,0) {};
\node[dynkdot,label={right:3}] (3) at (1,0.4) {};
\node[dynkdot,label={right:4}] (4) at (1,-0.4) {};
\draw[->] (2) -- (1);
\draw[->] (2) -- (4);
\draw[->] (3) -- (2);
\end{tikzpicture}
$
of type $D_4$, its AR quiver $\Gamma_Q$ with the coordinate system can be depicted as follows:
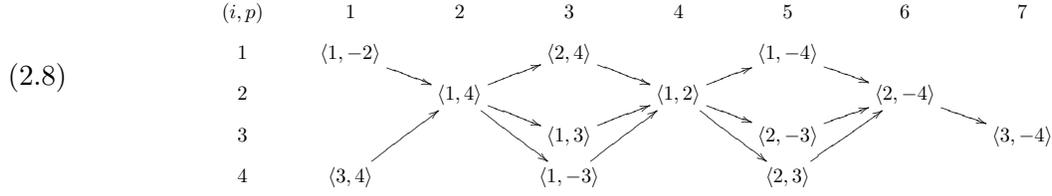
\begin{equation}
\label{eq: D4 AR quiver}
\raisebox{2.5em}{ \scalebox{0.7}{\xymatrix@R=1ex{
(i,p) & 1 & 2 & 3 & 4 & 5 & 6 & 7 \\
1&\lan  1,-2 \ran  \ar@{->}[dr] && \lan  2,4 \ran \ar@{->}[dr] && \lan  1,-4 \ran  \ar@{->}[dr]  \\
2&& \lan  1,4 \ran  \ar@{->}[dr]\ar@{->}[ddr]\ar@{->}[ur] && \lan  1,2 \ran \ar@{->}[ddr]\ar@{->}[dr]\ar@{->}[ur] && \lan  2,-4 \ran  \ar@{->}[dr] \\
3&&& \lan  1,3\ran  \ar@{->}[ur] && \lan  2,-3 \ran  \ar@{->}[ur] && \lan  3,-4 \ran  \\
4& \lan  3,4 \ran  \ar@{->}[uur] , && \lan  1,-3 \ran  \ar@{->}[uur] && \lan  2,3 \ran  \ar@{->}[uur]
}}}
\end{equation}
Here $\lan a, \pm b \ran \seteq \epsilon_a \pm \epsilon_b$.

\noindent
(2) For $[\rrz]=[\mathfrak{Q}]$ of type $D_4$ in (2) of Example~\ref{ex: redez for exceptional}, its twisted AR quiver $\Upsilon_{[\mathfrak{Q}]}$ with the coordinate system can be depicted as follows:
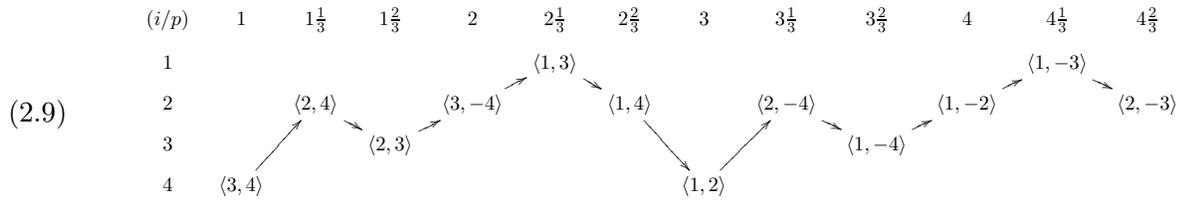
\begin{equation}
\label{eq:unfolded2}
\raisebox{3.5em}{\scalebox{0.7}{\xymatrix@C=2ex@R=1ex{
(i/p) & 1 & 1\frac{1}{3} & 1\frac{2}{3} &2 & 2\frac{1}{3} &2\frac{2}{3} &3 & 3\frac{1}{3} & 3\frac{2}{3} &4 & 4\frac{1}{3} &4\frac{2}{3}  \\
1 &&&&& \lan 1, 3 \ran \ar@{->}[dr] &&&&&& \lan 1, -3 \ran \ar@{->}[dr]\\
2 && \lan 2, 4 \ran \ar@{->}[dr] && \lan 3, -4 \ran\ar@{->}[ur] && \lan 1, 4 \ran \ar@{->}[ddr]&& \lan 2, -4 \ran \ar@{->}[dr]&& \lan 1, -2 \ran \ar@{->}[ur] && \lan 2, -3 \ran\\
3 &&& \lan 2, 3 \ran\ar@{->}[ur] &&&&&& \lan 1, -4 \ran\ar@{->}[ur]\\
4 & \lan 3, 4 \ran \ar@{->}[uur] &&&&&& \lan 1, 2 \ran \ar@{->}[uur]
}}}
\end{equation}
\end{example}

If $\sigma \ne \id$, we can construct the \defn{folded AR quiver} $\widehat{\Upsilon}_{[\rrz]}$ from $\Upsilon_{[\rrz]}$ via the involution $\sigma$ in the following sense:
There are \emph{no} vertices $\al$ and $\be$ in $\Upsilon_{[\mQ]}$ (resp.~$\Upsilon_{[\mathfrak{Q}]}$) such that
\[
(\widehat{\imath},p)=(\widehat{\jmath},p'), \  \text{ where } \ \Omega_{[\rrz]}(\al)=(i,p) \  \text{ and } \  \Omega_{[\rrz]}(\be)=(j,p') \ ([\rrz]=[\mQ] \text{ or } [\mathfrak{Q}]).
\]
Thus we can assign
$\beta \in \widehat{\Upsilon}_{[\rrz]}$ a $\widehat{\Omega}_{[\rrz]}(\be) \seteq  (\widehat{\imath},p) \in \widehat{I} \times \dfrac{1}{\mathsf{d}}\Z$, called the \defn{folded coordinate} of $\beta$.

\begin{example}
\label{ex:folded AR quiver E}
In Example~\ref{ex:twisted AR quiver E}, $\Upsilon_{[\mQ]}$ can be folded into $\widehat{\Upsilon}_{[\mQ]}$ as follows:
\begin{equation}
\label{eq:folded1}
 \raisebox{4.6em}{\scalebox{0.6}{\xymatrix@C=0.1ex@R=2.6ex{
(\widehat{\imath}/p) & \frac{1}{2}  & 1 & \frac{3}{2} & 2 & \frac{5}{2} & 3 & \frac{7}{2} & 4 & \frac{9}{2} & 5 & \frac{11}{2} & 6 & \frac{13}{2} & 7 & \frac{15}{2} & 8 & \frac{17}{2} & 9 & \frac{19}{2} & 10    \\
1 &&&& {\scriptstyle\prt{000}{111}} \ar@{->}[drr] &&{\scriptstyle\prt{111}{210}}\ar@{->}[drr]&& {\scriptstyle\prt{011}{110}}\ar@{->}[drr] && {\scriptstyle\prt{001}{111}} \ar@{->}[drr]
&& {\scriptstyle\prt{111}{211}}\ar@{->}[drr]
&& {\scriptstyle\prt{111}{110}} \ar@{->}[drr]
&& {\scriptstyle\prt{001}{000}} \ar@{->}[drr] &&
{\scriptstyle\prt{000}{001}} && {\scriptstyle\prt{100}{000}} \\
2 && {\scriptstyle\prt{000}{110}} \ar@{->}[urr]\ar@{->}[dr] &&  {\scriptstyle\prt{011}{210}}\ar@{->}[urr]\ar@{->}[dr]   && {\scriptstyle\prt{011}{221}}\ar@{->}[urr]\ar@{->}[dr] && {\scriptstyle\prt{112}{321}}\ar@{->}[urr]\ar@{->}[dr]
&&{\scriptstyle\prt{122}{321}} \ar@{->}[urr]\ar@{->}[dr]
&& {\scriptstyle\prt{112}{221}}\ar@{->}[urr]\ar@{->}[dr]
&&{\scriptstyle\prt{112}{211}} \ar@{->}[urr]\ar@{->}[dr] &&
{\scriptstyle\prt{111}{111}} \ar@{->}[urr]\ar@{->}[dr] && {\scriptstyle\prt{101}{000}}\ar@{->}[urr]\\
3 & {\scriptstyle\prt{000}{100}}\ar@{->}[ur]\ar@{->}[dr] && {\scriptstyle\prt{010}{110}} \ar@{->}[dr]\ar@{->}[ur] && {\scriptstyle\prt{001}{110}} \ar@{->}[ur]\ar@{->}[dr]
&& {\scriptstyle\prt{011}{211}} \ar@{->}[dr]\ar@{->}[ur] && {\scriptstyle\prt{111}{221}}\ar@{->}[ur]\ar@{->}[dr] &&{\scriptstyle\prt{112}{210}}\ar@{->}[ur]\ar@{->}[dr]
&&{\scriptstyle\prt{011}{111}}\ar@{->}[ur]\ar@{->}[dr] &&
{\scriptstyle\prt{101}{111}}\ar@{->}[dr]\ar@{->}[ur] && {\scriptstyle\prt{111}{100}}\ar@{->}[ur]\ar@{->}[dr]\\
4 && {\scriptstyle\prt{010}{100}}\ar@{->}[ur] && {\scriptstyle\prt{000}{010}} \ar@{->}[ur] && {\scriptstyle\prt{001}{100}} \ar@{->}[ur] &&
{\scriptstyle\prt{010}{111}}\ar@{->}[ur] &&{\scriptstyle\prt{101}{110}}\ar@{->}[ur] && {\scriptstyle\prt{011}{100}} \ar@{->}[ur]&& {\scriptstyle\prt{000}{011}}
\ar@{->}[ur] &&{\scriptstyle\prt{101}{100}}\ar@{->}[ur]  &&{\scriptstyle\prt{010}{000}}
}}}
\end{equation}
\end{example}

\begin{example}
\label{ex:folded AR quiver D}
In Example~\ref{ex:twisted AR quiver D}, $\Upsilon_{[\mathfrak{Q}]}$ can be folded into $\widehat{\Upsilon}_{[\mathfrak{Q}]}$ as follows:
\begin{align}\label{eq:folded2}
\raisebox{1.5em}{\scalebox{0.75}{\xymatrix@C=2ex@R=1.5ex{
(\widehat{\imath}/p) & 1 & 1\frac{1}{3} & 1\frac{2}{3} &2 & 2\frac{1}{3} &2\frac{2}{3} &3 & 3\frac{1}{3} & 3\frac{2}{3} &4 & 4\frac{1}{3} &4\frac{2}{3}  \\
1&\lan 3, 4 \ran \ar@{->}[dr]  && \lan 2, 3 \ran \ar@{->}[dr]  && \lan 1, 3 \ran \ar@{->}[dr] &&\lan 1, 2 \ran \ar@{->}[dr]  &&\lan 1, -4 \ran \ar@{->}[dr]  && \lan 1, -3 \ran \ar@{->}[dr]\\
2&& \lan 2, 4 \ran \ar@{->}[ur] && \lan 3, -4 \ran\ar@{->}[ur] && \lan 1, 4 \ran \ar@{->}[ur]&& \lan 2, -4 \ran \ar@{->}[ur]&& \lan 1, -2 \ran \ar@{->}[ur] && \lan 2, -3 \ran
}}}
\end{align}
\end{example}
For simplicity, we also use $\widehat{\Upsilon}_{[\rrz]}$ when $\sigma = \id$; in that case, $[\rrz] =[Q]$ for some Dynkin quiver $Q$, $\widehat{X}=X$ and $\widehat{\Upsilon}_{[\rrz]} = \Gamma_Q$.




\subsection{Statistics on \texorpdfstring{$\widehat{\Upsilon}_{[\redez]}$}{Upsilonhat[redez]}}

Recall the notation $\beta^{\redez}_k$ and the total order $<_{\redez}$ for $\redez$.
Consider a sequence $\um = (\um_1,\um_2,\ldots,\um_\N) \in \Z_{\ge 0}^{\N}$, and we define $\wt_{\redez}(\um) = \sum_{i=1}^\N \um_i\beta^{\redez}_i\in\rl^+$.

\begin{definition}[\cite{McN15,Oh15E}]
We define the partial orders $<^\tb_{\redez}$ and $\prec^\tb_{[\redez]}$ on $\Z_{\ge 0}^{\N}$ as follows:
\begin{enumerate}[{\rm (i)}]
\item $<^\tb_{\redez}$ is the bi-lexicographical partial order induced by $<_{\redez}$. Namely, $\um<^\tb_{\redez}\um'$ if there exist
$j$ and $k$ $(1\le j\le k\le \N)$ such that
\begin{itemize}
\item $\um_s=\um'_s$ for $1\le s<j$ and $\um_j<\um'_j$,
\item $\um_{s}=\um'_{s}$ for $k<s\le \N$ and $\um_k<\um'_k$.
\end{itemize}
\item For sequences $\um$ and $\um'$, we have $\um \prec^\tb_{[\redez]} \um'$ if and only if $\wt_{\redez}(\um)=\wt_{\redez}(\um')$ and
$\un<^\tb_{\redez'} \un'$ for all $\redez' \in [\redez]$,
where $\un$ and $\un'$ are sequences such that
$\un_{\redez'}=\um_{\redez}$ and $\un'_{\redez'}=\um_{\redez}$.
\end{enumerate}
\end{definition}

We give the following definitions from~\cite{McN15,Oh15E} but instead using more of the language of posets.
We say a sequence $\um=(\um_1,\um_2,\ldots,\um_{\N}) \in \Z^{\N}_{\ge 0}$ is \defn{$[\redez]$-simple} if it is minimal with respect to the partial order $\prec^\tb_{[\redez]}$. For a given $[\redez]$-simple sequence $\us=(s_1,\ldots,s_{\N}) \in \Z^{\N}_{\ge 0}$, we call a cover\footnote{Recall that a cover of $x$ in a poset $P$ with partial order $\prec$ is an element $y \in P$ such that $x \prec y$ and there does not exists $y' \in P$ such that $x \prec y' \prec y$.} of $\us$ under $\prec^{\tb}_{[\redez]}$ a \defn{$[\redez]$-minimal sequence of $\us$}. The \defn{$[\redez]$-distance} of a sequence $\um$ is the largest integer $k \geq 0$ such that
\[
\um^{(0)} \prec^\tb_{[\redez]} \cdots \prec^\tb_{[\redez]} \um^{(k)} = \um
\]
and $\um^{(0)}$ is $[\redez]$-simple.

We call a sequence $\um$ a \defn{pair} if $|\um|\seteq \sum_{i=1}^\N m_i=2$ and $m_i \le 1$ for $1\le i\le \N$. We mainly use the notation $\up$ for a pair.
Consider a pair $\up$ such that there exists a unique $[\redez]$-simple sequence $\us$ satisfying
$\us \preceq^\tb_{[\redez]} \up$, we call $\us$ the \defn{$[\redez]$-socle} of $\up$ and denoted it by $\soc_{[\redez]}(\up)$.

\begin{proposition}[{\cite[Lemma 2.6]{BKM12}}]
\label{pro: BKM minimal}
For $\ga \in \PR \setminus \Pi$ and any $\redez$ of $w_0$, a $[\redez]$-minimal sequence of $\ga$ is indeed a pair $(\al,\beta)$ for some $\al,\beta \in \PR$ such that $\al+\beta = \ga$.
\end{proposition}


As we did in the previous subsection, $[\rrz]$ denotes a commutation class in $\lf \Delta \rf$, $\lf \mQ \rf$ or $\lf \mathfrak{Q} \rf$. Recall, for each $[\rrz]$, we can correspond $\widehat{X}$, $\sigma$, $\widehat{I}$ and $\mathsf{d}$.

Following~\cite{Oh15E,OS16B,OS16C}, for a folded AR quiver $\widehat{\Upsilon}_{[\rrz]}$, indices $\widehat{k},\widehat{l} \in \widehat{I}$ and an integer $t \in \Z_{\ge 1}$,
we define the subset $\Phi_{[\rrz]}(\widehat{k},\widehat{l})[t] \subset \PR \times \PR$ as the pairs $(\alpha,\beta) \in \PR \times \PR$ such that $\alpha$ and $\beta$ are comparable under $\prec_{[\rrz]}$ and
\[
\{ \widehat{\Omega}_{[\rrz]}(\al),\widehat{\Omega}_{[\rrz]}(\beta) \} = \{ (\widehat{k},a), (\widehat{l},b)\} \quad \text{ such that } \quad |a-b| = \dfrac{t}{\mathsf{d}}.
\]

\begin{proposition}[{\cite{Oh15E,OS16B,OS16C}}]
\label{prop:dist_theta_defn}
\hfill
\begin{enumerate}[{\rm (1)}]
\item For any $(\al,\beta), \ (\al',\beta')  \in \Phi_{[\rrz]}(\widehat{k},\widehat{l})[t]$, we have
\begin{align} \label{eq: otkl}
\dist_{[\rrz]}(\al,\beta)=\dist_{[\rrz]}(\al',\beta').
\end{align}
Moreover, we denote by $o^{[\rrz]}_t(\widehat{k},\widehat{l}) \seteq \dist_{[\rrz]}(\al,\beta)$ for any $(\al,\beta) \in \Phi_{[\rrz]}(\widehat{k},\widehat{l})[t]$.

\item The integer, defined by
\begin{equation} \label{eq: theta}
\theta^{[\rrz]}_t(\widehat{k},\widehat{l}) = \begin{cases}
\max\left\{ o^{[Q]}_t(\widehat{k},\widehat{l}),o^{[Q^{{\rm rev}}]}_t(\widehat{k},\widehat{l}) \right\} & \text{ if  $[\rrz]=[Q]$ for some $Q$}, \\[2ex]
\left\lceil o^{[\rrz]}_t(\widehat{k},\widehat{l})/\mathsf{d} \right \rceil  & \text{ otherwise},
\end{cases}
\end{equation}
does not depend on the choice of $[\rrz'], [\rrz''] \in \lf \rrz \rf$.
\end{enumerate}
\end{proposition}

From Proposition~\ref{prop:dist_theta_defn}, we can define the (folded) distance polynomial $D^{\widehat{X}}_{\widehat{k},\widehat{l}}(z;-q)$ for $\widehat{k},\widehat{l} \in \widehat{I}$ on $\lf \rrz \rf$ as
\[
D^{\lf \rrz \rf}_{\widehat{k},\widehat{l}}(z;-q) = \prod_{t \in \Z_{\ge 0}} \left( z - (-1)^\kappa q^t \right)^{\theta^{\lf \rrz \rf}_t(\widehat{k},\widehat{l})},
\]
where
\[
\kappa = \begin{cases}
\widehat{k}+\widehat{l} & \text{if $\sigma$ is from~\eqref{eq: B_n} or~\eqref{eq: F_4}}, \\
t & \text{otherwise}.
\end{cases}
\]

\section{Quantum affine algebras}
\label{sec2:Qunatum affine algebras}

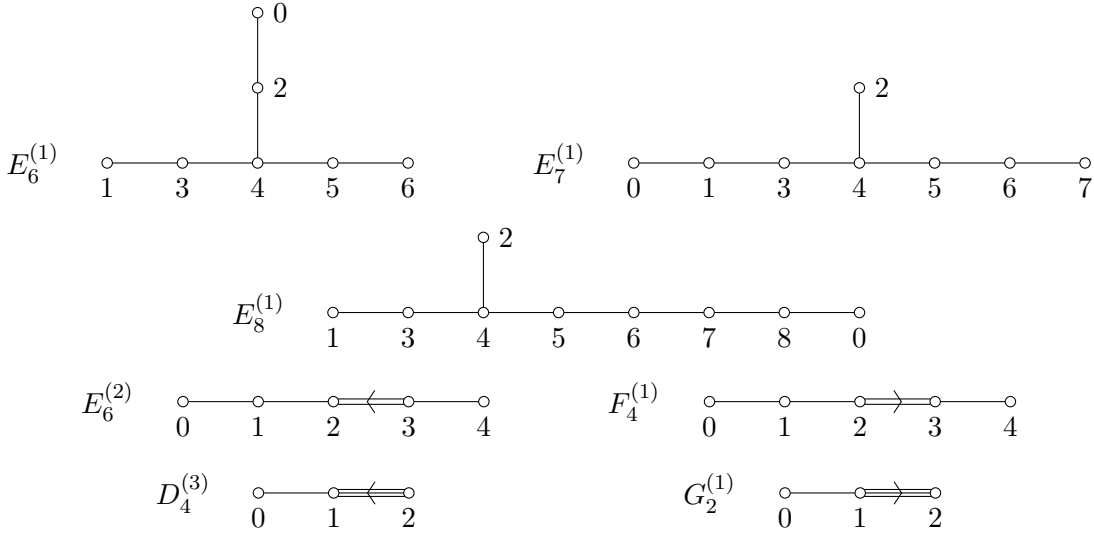
\begin{figure}[t]
\begin{center}
\begin{tikzpicture}[scale=1, baseline=-20]
\draw (-0.5,0) node[anchor=east] {$E_6^{(1)}$};
\node[dynkdot,label={below:$1$}] (E61) at (0,0) {};
\node[dynkdot,label={below:$3$}] (E63) at (1,0) {};
\node[dynkdot,label={below:$4$}] (E64) at (2,0) {};
\node[dynkdot,label={below:$5$}] (E65) at (3,0) {};
\node[dynkdot,label={below:$6$}] (E66) at (4,0) {};
\node[dynkdot,label={right:$2$}] (E62) at (2,1) {};
\node[dynkdot,label={right:$0$}] (E60) at (2,2) {};
\path[-] (E61) edge (E63) (E63) edge (E64) (E64) edge (E65) (E65) edge (E66)
  (E64) edge (E62) (E62) edge (E60);
\begin{scope}[xshift=8cm]
\draw (-1.5,0) node[anchor=east] {$E_7^{(1)}$};
\node[dynkdot,label={below:$1$}] (E1) at (0,0) {};
\node[dynkdot,label={below:$3$}] (E3) at (1,0) {};
\node[dynkdot,label={below:$4$}] (E4) at (2,0) {};
\node[dynkdot,label={below:$5$}] (E5) at (3,0) {};
\node[dynkdot,label={below:$6$}] (E6) at (4,0) {};
\node[dynkdot,label={below:$7$}] (E7) at (5,0) {};
\node[dynkdot,label={right:$2$}] (E2) at (2,1) {};
\node[dynkdot,label={below:$0$}] (E0) at (-1,0) {};
\path[-] (E1) edge (E3) (E3) edge (E4) (E4) edge (E5) (E5) edge (E6) (E6) edge (E7)
  (E4) edge (E2) (E1) edge (E0);
\end{scope}
\end{tikzpicture}
\begin{tikzpicture}[scale=1]
\draw (-0.5,0) node[anchor=east] {$E_8^{(1)}$};
\node[dynkdot,label={below:$1$}] (E1) at (0,0) {};
\node[dynkdot,label={below:$3$}] (E3) at (1,0) {};
\node[dynkdot,label={below:$4$}] (E4) at (2,0) {};
\node[dynkdot,label={below:$5$}] (E5) at (3,0) {};
\node[dynkdot,label={below:$6$}] (E6) at (4,0) {};
\node[dynkdot,label={below:$7$}] (E7) at (5,0) {};
\node[dynkdot,label={below:$8$}] (E8) at (6,0) {};
\node[dynkdot,label={right:$2$}] (E2) at (2,1) {};
\node[dynkdot,label={below:$0$}] (E0) at (7,0) {};
\path[-] (E1) edge (E3) (E3) edge (E4) (E4) edge (E5) (E5) edge (E6) (E6) edge (E7) (E7) edge (E8)
  (E4) edge (E2) (E8) edge (E0);
\end{tikzpicture}

\vspace{5pt}
\begin{tikzpicture}[scale=1]
\node (F4desc) at (-1,0) {$E_6^{(2)}$};
\foreach \x in {0,1,2,3,4}
{\node[dynkdot,label={below:$\x$}] (F\x) at (\x,0) {};}
\path[-] (F0) edge (F1) (F1) edge (F2) (F3) edge (F4);
\draw[-] (F2.30) -- (F3.150);
\draw[-] (F2.330) -- (F3.210);
\draw[-] (2.55,+.15) -- (2.45,0) -- (2.55,0-.15);

\begin{scope}[xshift=7cm]
\node (E6desc) at (-1,0) {$F_4^{(1)}$};
\foreach \x in {0,1,2,3,4}
{\node[dynkdot,label={below:$\x$}] (E\x) at (\x,0) {};}
\path[-] (E0) edge (E1) (E1) edge (E2) (E3) edge (E4);
\draw[-] (E2.30) -- (E3.150);
\draw[-] (E2.330) -- (E3.210);
\draw[-] (2.45,+.15) -- (2.55,0) -- (2.45,-.15);
\end{scope}
\end{tikzpicture}

\vspace{5pt}
\begin{tikzpicture}[scale=1,baseline=0]
\node (G2desc) at (-1,0) {$D_4^{(3)}$};
\node[dynkdot,label={below:$0$}] (G0) at (0,0){};
\node[dynkdot,label={below:$1$}] (G1) at (1,0){};
\node[dynkdot,label={below:$2$}] (G2) at (2,0) {};
\draw[-] (G0) -- (G1) -- (G2);
\draw[-] (G1.40) -- (G2.140);
\draw[-] (G1.320) -- (G2.220);
\draw[-] (1.55,+.15) -- (1.45,0) -- (1.55,0-.15);

\begin{scope}[xshift=7cm]
\node (D2desc) at (-1,0) {$G_2^{(1)}$};
\node[dynkdot,label={below:$0$}] (D0) at (0,0){};
\node[dynkdot,label={below:$1$}] (D1) at (1,0){};
\node[dynkdot,label={below:$2$}] (D2) at (2,0) {};
\draw[-] (D0) -- (D1) -- (D2);
\draw[-] (D1.40) -- (D2.140);
\draw[-] (D1.320) -- (D2.220);
\draw[-] (1.45,+.15) -- (1.55,0) -- (1.45,-.15);
\end{scope}
\end{tikzpicture}
\end{center}
\caption{Dynkin diagrams for the exceptional affine types.}
\label{fig:exceptional_types}
\end{figure}

\subsection{Quantum affine algebra}
In this section, we review the representation theory of finite dimensional integrable modules over quantum affine algebras referring~\cite{AK97,Kas02}.
When we deal with quantum affine algebras, we take a base field $\ko$ which is the algebraic closure of $\C(q)$ in $\cup_{m > 0}\C((q^{1/m}))$.

Let $\cmA$ be a generalized Cartan matrix of affine type. We choose $0 \in I$ as the leftmost vertex in the table~\cite[page 54,55]{Kac} except in the $A_{2n}^{(2)}$ case, in which case we take the longest simple root as $\al_0$.
For the exceptional types, we follow the labeling given in Figure~\ref{fig:exceptional_types}, and otherwise we follow the labeling of~\cite{Kac}.
We denote the \defn{imaginary root} by $\delta = \sum_{i \in I}d_i\al_i$ and the \defn{canonical central element} $c = \sum_{i \in I}c_i h_i$~\cite[Chapter 4]{Kac}.
We normalize the bilinear form $(\cdot\,,\,\cdot)$ by
\[
(\delta, \lambda) = \lan c,\lambda \ran \quad \text{ for any } \lambda \in P.
\]

For the affine Cartan datum $(\cmA,P,$ $\Pi,P^\vee,\Pi^\vee, (\cdot\,,\,\cdot))$ and associated quantum group $U_q(\g)$, let $U'_q(\g)$ denote the \defn{quantum affine algebra} that is the subalgebra of $U_q(\g)$ generated by $e_i,f_i,K_i^{\pm 1}$ $(i \in I)$.\footnote{The quantum group $U_q'(\g)$ also may be considered as the quantum group $U_q(\g')$ associated to the derived subalgebra $\g' \seteq [\g, \g]$.}
Let us denote by $\Mod(U_q'(\g))$ the category of left $U_q'(\g)$-modules.

Let $U_q(\g_0)$ denote the quantum group of the corresponding finite-type given by $I_0 \seteq I \setminus \{0\}$, and note that $U_q(\g_0) \subset U_q(\g)$ as the subalgebra generated by $e_i$, $f_i$, $K_i$ for $i \in I_0$.
Let $\clfw_i$ denote the fundamental weights associated to the Cartan datum of $\g_0$.


Let us denote by $ \ \bar{ } \ $ the involution of $U_q'(\g)$ defined as follows:
\[
e_i \mapsto e_i,
\qquad\qquad
f_i \mapsto f_i,
\qquad\qquad
K_i \mapsto K_i ,
\qquad\qquad
q^{1/\gamma} \to q^{-1/\gamma}.
\]

Set $P_\cl \seteq P / \Z\delta$ and $\cl \colon P \to P_\cl$ as the canonical projection. Then $P_\cl = \soplus_{i \in I}\Z \cl(\Lambda_i)$.
We also define the subset $P_\cl^0$ of $P_\cl$ as follows:
\[
P_\cl^0 \seteq \{ \lambda \in P_\cl \mid \lan c, \lambda \ran =0 \} \subset P_\cl.
\]

We denote by $W_{\aff}$ the affine Weyl group generated by $s_i \in \Aut(P)$ $(i \in I)$, where $s_i(\mu)=\mu-\lan h_i,\mu \ran \al_i$. Then $W_\aff$ acts also on $P_\cl$
and $P_\cl^0$.

We call a $U_q'(\g)$-module $M$ \defn{integrable} if
\begin{enumerate}[(a)]
\item $M$ admits a weight space decomposition
\[
M = \soplus_{\mu \in P_\cl} M_\mu,  \quad \text{ where } M_\mu = \{ u \in M \mid K_iu = q_i^{\lan h_i,\mu \ran} \text{ for all } i \in I \},
\]
\item the actions of $e_i$ and $f_i$ are locally nilpotent for all $i \in I$.
\end{enumerate}

We denote by $\Ca_\g$ the abelian tensor category consisting of finite dimensional integrable $U_q'(\g)$-modules. For a simple module $M$ in
$\Ca_\g$, there exists a non-zero vector $v_{\lambda} \in M$ of weight $\lambda \in P_\cl^0$ such that $\lan h_i,\lambda \ran \ge 0$ for all $i \in I_0$
and $\ko v = M_{\lambda}$.
Such a weight $\lambda$ is called the \defn{dominant extremal weight} of $M$ and a non-zero vector in $M_{\lambda}$
is called a \defn{dominant extremal vector} of $M$.

For an integrable $U_q'(\g)$-module $M$, the \defn{affinization} $M_\aff := \ko[z,z^{-1}] \otimes  M$ of $M$ is considered as a vector space over $\ko$ and is equipped with a $U_q'(\g)$-module structure
\[
e_i(u_z) = z^{\delta_{i,0}}(e_iu)_z,
\qquad\qquad
f_i(u_z) = z^{-\delta_{i,0}}(f_iu)_z,
\qquad\qquad
K_i(u_z) = (K_iu)_z,
\]
for all $i \in I$.
Here $u_z$ denotes $\mathbf{1} \otimes  u \in M_\aff$ for $u \in M$. We denote by the action of $z$ on $M_\aff$ by $z_M$ and sometimes write $M_z$ instead of $M_\aff$.

For $a \in \ko^\times$, we define the \defn{evaluation module of $M$ at $a$} as
\[
M_a \seteq M_z / (z - a) M_z
\]
and call the $a$ the \defn{spectral parameter}. Note that for $U_q'(\g)$-modules $M$ and $N$, we have
\[
(M \otimes N)_a \iso M_a \otimes N_a.
\]

For a $U_q'(\g)$-module $M$, we denote by the $U_q'(\g)$-module $\overline{M}=\{ \bar{u} \mid u \in M \}$ whose module structure is give as follows:
\[
x \overline{u} \seteq \overline{\overline{x} u},
\]
where $x \in U_q'(\g)$.
Then we have
\begin{align} \label{eq: bar struture}
\overline{M_a} \iso \overline{M}_{\overline{a}}, \qquad\qquad \overline{M \otimes N} \iso \overline{N} \otimes \overline{M}.
\end{align}

For each $i \in I_0$, we set
\begin{align} \label{eq: funda}
\varpi_i \seteq {\rm gcd}(c_0,c_i)^{-1}\cl(c_0\Lambda_i-c_i\Lambda_0) \in P^0_\cl,
\end{align}
which is called an $i$-th \defn{level $0$ fundamental weight}. Then $\{ \varpi_i \}_{i \in I_0}$ is a basis of $P_\cl^0$.
Therefore, for each $i \in I_0$, there exists a unique simple $U_q'(\g)$-module $V(\varpi_i) \in \Ca_\g$ whose dominant extremal weight is $\varpi_i$ and satisfies the following properties:
\[
\parbox{85ex}{
We can take $u_\mu \in V(\varpi_i)_\mu$ for each $\mu \in W\varpi_i \subset P_\cl$ such that
\begin{enumerate}[(a)]
\item for $j \in I$ and $\mu \in W\varpi_i$ such that $\lan h_i,\mu \ran \ge 0$, $e_j u_\mu =0$ and $f_j^{(\lan h_i,\mu \ran)}u_\mu=u_{s_j\mu}$,
\item for $j \in I$ and $\mu \in W\varpi_i$ such that $\lan h_i,\mu \ran < 0$, $f_j u_\mu =0$ and $e_j^{(-\lan h_i,\mu \ran)}u_\mu=u_{s_j\mu}$,
\item $M$ is generated by $u_{\varpi_i}$.
\end{enumerate}
}
\]
We call $V(\varpi_i)$ the \defn{fundamental representation}.

We say that a simple $U_q'(\g)$-module $M$ is \defn{good} if it has a \defn{bar involution}, a  \defn{crystal basis} with \defn{simple crystal graph} and a \defn{lower global basis} (see~\cite[\S 8]{Kas02} for precise definitions). Note that the fundamental representation $V(\varpi_i)_x$ is good. For each $M \in \Ca_g$, there exist the right dual ${}^*M$ and the left dual $M^*$ of $M$ in the following sense: we have $U_q'(\g)$-module isomorphisms
\begin{equation} \label{eq: LR dual}
\begin{aligned}
& \Hom(M \otimes  X,Y) \iso \Hom(X,{}^*M \otimes  Y), \  \Hom(X \otimes  {}^*M,Y) \iso \Hom(X, Y \otimes  M), \\
& \Hom(M^* \otimes  X,Y) \iso \Hom(X, M \otimes  Y), \  \Hom(X\otimes  M,Y) \iso \Hom(X,Y \otimes  M^*),
\end{aligned}
\end{equation}
which are functorial in $U_q'(\g)$-modules $X$ and $Y$. The duals of a fundamental representation $V(\varpi_i)$ $(i \in I_0)$ are given as follows:
\[
{}^*V(\varpi_i) \iso V(\varpi_{i^*})_{p^*} \qtext{and} V(\varpi_i)^* \iso V(\varpi_{i^*})_{(p^*)^{-1}},
\]
where $p^* \seteq (-1)^{\lan \rho^\vee,\delta \ran}q^{\lan c,\rho \ran}$. Here $\rho$ (resp.~$\rho^\vee$) denotes an element in $P$ (resp.~$P^\vee$) satisfying
$\lan h_i,\rho \ran=1$ (resp.~$\lan \rho^\vee,\alpha_i \ran=1$) for all $i \in I$, and $i^*$ denotes the index in $I$ determined by
$w_0(\alpha_i)=-\alpha_{i^*}$, where $w_0$ is the longest element of $W_0 = \lan s_i \mid i \in I_0 \ran$.

For the fundamental representations, we can compute the $U_q(\g_0)$-decomposition by using the (virtual) Kleber algorithm~\cite{Kleber98,OSS03II}.

\begin{remark}
When we say a vector $v$ is \emph{the} highest weight vector of an irreducible $V(\lambda)$ component in the $U_q(\g_0)$-decomposition of a $U_q'(\g)$-module, we will always mean ``up to scalar in $\ko$.''
\end{remark}

\subsubsection{Module structure for affine minuscule representations}

Let $\g$ be dual to an untwisted affine type.
Suppose $r \in I_0$ is in the orbit of $0$ under some affine Dynkin diagram automorphism.
Then we say $V(\varpi_r)$ is a \defn{affine minuscule representation} since $V(\varpi_r) \iso V(\clfw_r)$ as $U_q(\g_0)$-modules and $V(\clfw_r)$ is a minuscule representation of $U_q(\g_0)$.

We consider the crystal $B(\varpi_r)$ of the fundamental representation $V(\varpi_r)$~\cite{Kas02,NS03,NS06II}.
Denote on $B(\varpi_r)$ the \defn{Kashiwara operators}~\cite{Kas90,Kas91} by $\widetilde{e}_i, \widetilde{f}_i \colon B(\varpi_r) \to B(\varpi_r) \sqcup \{ \bzero \} $ and the weight function by $\wt \colon B(\varpi_r) \to P_{\cl}$.
We define a $U_q(\g)$-module $\mathbb{V}(\varpi_r) \seteq \ko \{ v_b \mid b \in B(\varpi_r) \}$ by
\[
e_i v_{b} = v_{\widetilde{e}_i b},
\qquad\qquad
f_i v_{b} = v_{\widetilde{f}_i b},
\qquad\qquad
K_i v_{b} = q_i^{\langle h_i, \wt(b) \rangle} v_{b},
\]
where we consider $v_{\bzero } \seteq 0$.

\begin{proposition}
$\mathbb{V}(\varpi_r)$ is a $U_q'(\g)$-module.
\end{proposition}

\begin{proof}
Note that since $V(\clfw_i)$ is a minuscule representation, so from, \textit{e.g.},~\cite{Stembridge01II}, as we can parameterize the weight spaces in $V(\varpi_r)$ by the minimal length coset representatives of $W_0 / \operatorname{Stab}_{W_0}(\clfw_r)$, recalling that the stabilizer of $\clfw_r$ in $W_0$ is a parabolic subgroup of $W_0$.
Hence, we have $\widetilde{e}_i^2 b = \widetilde{f}_i^2 b = 0$ for all $b \in B(\varpi_r)$ and $i \in I$.
Therefore, from~\cite[Prop.~2.1]{Stembridge01II}, we have that $\mathbb{V}(\clfw_r)$ is a $U_q'(\g)$-module.
\end{proof}

\begin{proposition}
\label{prop:minuscule_repr}
We have $\mathbb{V}(\varpi_r) \iso V(\varpi_r)$.
\end{proposition}

\begin{proof}
Let $u_{\varpi_r} \in B(\varpi_r)$ be the unique element such that $e_i u_{\varpi_r} = 0$ for all $i \in I_0$.
For any vector
\[
v = \sum_{b \in B(\varpi_r)} \xi_b v_b  \in \mathbb{V}(\varpi_r)
\]
choose a maximal length sequence $(i_j \in I_0)_{j=1}^{\ell}$ such that there exists $b \in B(\varpi_i)$ that satisfies $\widetilde{e}_{i_1} \cdots \widetilde{e}_{i_{\ell}} b = u_{\clfw_r}$ and $\xi_b \neq 0$.
Therefore, we have
\[
e_{i_1} \cdots e_{i_{\ell}} v = \xi_b v_{u_{\varpi_r}},
\]
and hence $\mathbb{V}(\varpi_r)$ is simple.
Since $\mathbb{V}(\varpi_r)$ has a dominant extremal vector of weight $\varpi_r$, the claim follows.
\end{proof}

From Proposition~\ref{prop:minuscule_repr}, we obtain a new construction of the affine minuscule modules from~\cite{Chari01}.

\subsubsection{Module structure for affine adjoint representations}
\label{sec:adjoint_repr}

Suppose $\g$ is not of type $A_n^{(1)}$.
Suppose $r \in I_0$ is adjacent to $0$ in the Dynkin diagram of $\g$.
Then we say $V(\varpi_r)$ is a \defn{affine adjoint representation} since $V(\varpi_r) \iso V(\clfw_r) \oplus V(0)$ as $U_q(\g_0)$-modules and $V(\clfw_r)$ is the adjoint representation of $U_q(\g_0)$.

We require the statistics
\[
\varepsilon_i(b) = \max \{ k \mid \widetilde{e}_i^k b \neq 0 \},
\qquad\qquad
\varphi_i(b) = \max \{ k \mid \widetilde{f}_i^k b \neq 0 \}.
\]
Recall from~\cite{BFKL06}, we can identify
\[
B(\varpi_r) = \{x_{\beta} \mid \beta \in \Phi \} \sqcup \{ y_i \mid i \in I \}
\]
where the crystal structure is given by
\newcommand{\iarrow}{\xrightarrow[\hspace{20pt}]{i}}
\[
\begin{array}{c@{\hspace{10ex}}c}
x_{\beta} \iarrow x_{\beta-\alpha_i} & (\beta - \alpha_i \in \Phi),
\\
x_{\alpha_i} \iarrow y_i \iarrow x_{-\alpha_i},
\end{array}
\]
with $\alpha_0 = \theta$, the highest root of $\Phi^+$.
Note that $y_0$ was denoted by $\emptyset$ in~\cite{BFKL06}.
Let $i \sim j$ denote $a_{ij} \neq 0$ and $i \neq j$, \textit{i.e.}, the nodes $i$ and $j$ are adjacent in the Dynkin diagram.
We define a $U_q(\g)$-module $\mathbb{V}(\varpi_r) \seteq \ko \{ v_b \mid b \in B(\varpi_r) \}$ by
\begin{gather*}
e_i v_{b} = \begin{cases}
0 & \widetilde{e}_i b = \bzero, \\
[\varphi_i(b) + 1]_i v_{\widetilde{e}_i b} & \wt(\widetilde{e}_i b) \neq 0, \vspace{2pt} \\
\displaystyle v_{y_i} + \sum_{j \sim i} \frac{[-a_{ij}]_i}{[2]_j} v_{y_j} & \text{otherwise},
\end{cases}
\qquad
f_i v_{b} = \begin{cases}
0 & \widetilde{f}_i b = \bzero, \\
[\varepsilon_i(b) + 1]_i v_{\widetilde{f}_i b} & \wt(\widetilde{f}_i b) \neq 0, \vspace{2pt} \\
\displaystyle v_{y_i} + \sum_{j \sim i} \frac{[-a_{ij}]_i}{[2]_j} v_{y_j} & \text{otherwise},
\end{cases}
\\
K_i v_{b} = q_i^{\langle h_i, \wt(b) \rangle} v_{b}.
\end{gather*}
Note that $\varphi_i(b) + 1 = \varphi_i(\widetilde{e}_i b)$ and $\varepsilon_i(b) + 1 = \varepsilon_i(\widetilde{f}_i b)$.

\begin{proposition}
$\mathbb{V}(\varpi_r)$ is a $U_q'(\g)$-module.
\end{proposition}

\begin{proof}
For ease of notation, we denote $v_{\beta} := v_{x_{\beta}}$. We verify the relation
\[
e_i f_j - f_i e_i = \delta_{ij} \frac{K_i - K_i^{-1}}{q_i - q_i^{-1}}
\]
and leave the rest to the reader as the computation is similar.

We consider $\beta \in \PR$ and note that the relation holds for $\beta \in \Phi^-$ by the natural duality $e_i \leftrightarrow f_i$ and $v_{\alpha} \leftrightarrow v_{-\alpha}$.
If $\beta + \alpha_i - \alpha_j \notin \PR \sqcup \{0\}$, then $(e_i f_j - f_j e_i) v_{\beta} = 0$. Note that this occurs only if $i \neq j$, and so
\[
(e_i f_j - f_j e_i) v_{\beta} = 0 = \delta_{ij} \frac{K_i - K_i^{-1}}{q_i - q_i^{-1}} v_{\beta}
\]
as desired.
Henceforth we assume $\beta + \alpha_i - \alpha_j \notin \PR$ (note that it cannot be $0$ by height considerations and $\beta \in \PR$).
If $\beta - \alpha_j \notin \PR \sqcup \{ 0 \}$, then $\varphi_i(x_{\beta}) = 0$ and
\[
(e_i f_j - f_j e_i) v_{\beta} =
-\delta_{ij} [\varepsilon_i(x_{\beta})]_i [\varphi_i(x_{\beta}) + 1]_i  v_{\beta} =  -\delta_{ij} [\varepsilon_i(x_{\beta})]_i  v_{\beta}.
\]
Recall from~\cite{Kas90,Kas91} that $\varphi_i(b) - \varepsilon_i(b) = \langle h_i, \wt(b) \rangle$ for all $b \in B(\varpi_r)$. So we have
\begin{align*}
\delta_{ij} \frac{K_i - K_i^{-1}}{q_i - q_i^{-1}} v_{\beta}
& = \delta_{ij} \frac{q_i^{\langle h_i, \wt(b) \rangle} - q_i^{-\langle h_i, \wt(b) \rangle}}{q_i - q_i^{-1}} v_{\beta}
= \delta_{ij} \frac{q_i^{\varphi_i(x_{\beta})} q_i^{-\varepsilon_i(x_{\beta})} - q_i^{\varepsilon_i(x_{\beta})} q_i^{-\varphi_i(x_{\beta})}}{q_i - q_i^{-1}} v_{\beta}
\\ &= \delta_{ij} \frac{q_i^{-\varepsilon_i(x_{\beta})} - q_i^{\varepsilon_i(x_{\beta})}}{q_i - q_i^{-1}} v_{\beta}
= -\delta_{ij} [\varepsilon_i(x_{\beta})]_i  v_{\beta}.
\end{align*}
Similar if $\beta + \alpha_i \notin \PR \sqcup \{ 0 \}$.
Now suppose $\beta + \alpha_i, \beta - \alpha_j \in \PR \sqcup \{ 0 \}$. Therefore, we must either have $i = j$ or $a_{ij} = a_{ji} = 0$ by examining each irreducible finite root system.
Note that if $i \neq j$, then we have
\[
\varphi_i(\widetilde{f}_j x_{\beta}) = \varphi_i(x_{\beta}),
\qquad\qquad
\varepsilon_j(\widetilde{e}_i x_{\beta}) = \varepsilon_j(x_{\beta}),
\]
and by direct computation, we have
\[
(e_i f_j - f_j e_i) v_{\beta}
 = \Bigl( [\varepsilon_j(x_{\beta}) + 1]_j [\varphi_i(x_{\beta}) + 1]_i - [\varepsilon_j(x_{\beta}) + 1]_j [\varphi_i(x_{\beta}) + 1]_i \Bigr) v_{\beta - \alpha_j + \alpha_i}
 = 0
\]
as desired. Now we consider the case $i = j$, and so we have
 \begin{align*}
 (e_i f_i - f_i e_i) v_{\beta}
 & = \Bigl( [\varepsilon_i(x_{\beta}) + 1]_i [\varphi_i(x_{\beta})]_i - [\varepsilon_i(x_{\beta})]_i [\varphi_i(x_{\beta}) + 1]_i \Bigr)
\\ & = \frac{q_i^{\varphi_i(x_{\beta})} q_i^{-\varepsilon_i(x_{\beta})} - q_i^{\varepsilon_i(x_{\beta})} q_i^{-\varphi_i(x_{\beta})}}{q_i - q_i^{-1}} v_{\beta}
\\ & = \frac{q_i^{\langle h_i, \wt(b) \rangle} - q_i^{-\langle h_i, \wt(b) \rangle}}{q_i - q_i^{-1}} v_{\beta} = \delta_{ij} \frac{K_i - K_i^{-1}}{q_i - q_i^{-1}} v_{\beta}.
\end{align*}
Note that this also covers the case when $\beta = \alpha_i$ as $e_i v_{y_j} = f_i v_{y_i} = 0$ and $\varphi_i(x_{\alpha_i}) = 0$.

Now if we consider $v_{y_i}$, the computation is similar to the above except for when $i = j$, where we will have the extra contribution of
\[
\sum_{k \sim i} \frac{[-a_{ik}]_i}{[2]_k} v_{y_k} - \sum_{k \sim i} \frac{[-a_{ik}]_i}{[2]_k} v_{y_k} = 0.
\]
Thus the claim follows.
\end{proof}

\begin{remark}
We note that we can explicitly verify the specific $\mathbb{V}(\varpi_r)$, in the sense that it is a finite computation to show the relations hold, that we use in this paper are indeed $U_q'(\g)$-modules and do not require the general result.
\end{remark}

\begin{proposition}
\label{prop:adjoint_repr}
We have $\mathbb{V}(\varpi_r) \iso V(\varpi_r)$.
\end{proposition}

\begin{proof}
The proof is similar to Proposition~\ref{prop:minuscule_repr} except we have to consider when the result of applying $e_i$ or $f_i$ to a weight vector results in the $0$ weight space. By the weight space decomposition of $\mathbb{V}(\varpi_r)$ and the proof of Proposition~\ref{prop:minuscule_repr}, it is sufficient to consider vectors of the form
\[
v = \sum_{i \in I} \xi_i v_{\alpha_i} \in \mathbb{V}(\varpi_r)
\]
and show that $f_i v \neq 0$ for all $i \in I$ as the result for $e_i v$ is similar. However, we have
\[
f_i v = \left(\xi_i + \sum_{j \sim i} \xi_j \frac{[-a_{ji}]_j}{[2]_i} \right) v_{y_i} + \xi_i \left( \sum_{j \sim i} \frac{[-a_{ij}]_i}{[2]_j} \right) v_{y_j},
\]
which is non-zero if $\xi_i \neq 0$. Hence, $v \neq 0$ if and only if there exists an $i \in I$ such that $f_i v \neq 0$. Hence $\mathbb{V}(\varpi_r)$ is simple.
\end{proof}

From Proposition~\ref{prop:adjoint_repr}, we obtain a new construction of the affine adjoint modules from~\cite{Chari01}.

\subsection{\texorpdfstring{$R$}{R}-matrices}

In this subsection, we recall the notion of $R$-matrices for quantum affine algebras following~\cite[\S 8]{Kas02}.
For $M, N \in \Ca_\g$, there is a morphism of
$\ko[[z_N/z_M]] \otimes _{\ko[z_N/z_M]} \ko[z_M^{\pm 1},z_N^{\pm 1}] \otimes  U_q'(\g)$-modules, denoted by  $\Runiv{M_{z_M},N_{z_N}}$ and  called the
\defn{universal $R$-matrix}:
\begin{align}
\Runiv{M_{z_M},N_{z_N}} \colon \ko[[z_N/z_M]] \otimes _{\ko[z_N/z_M]} (M_{z_M} \otimes  N_{z_N}) \to \ko[[z_N/z_M]] \otimes _{\ko[z_N/z_M]} (N_{z_N}\otimes  M_{z_M}).
\end{align}

We say that $\Runiv{M_{z_M},N_{z_N}}$ is \defn{rationally renormalizable} if there exist $a \in \ko(z_N/z_M)$ and a $\ko[z_M^{\pm 1},z_N^{\pm 1}] \otimes  U_q'(\g)$-module homomorphism
\[
\Rren{M_{z_M},N_{z_N}} \colon M_{z_M} \otimes  N_{z_N} \to N_{z_N} \otimes  M_{z_M}
\]
such that $\Rren{M_{z_M},N_{z_N}} = a\Runiv{M_{z_M},N_{z_N}}$.
Then we can choose $\Rren{M_{z_M},N_{z_N}}$ so that for any $a_M, a_N \in \ko^\times$, the specialization of $\Rren{M_{z_M},N_{z_N}}$ at $z_M=a_M$, $z_N=a_N$,
\begin{align} \label{eq: renormalized R-matrix}
\Rren{M_{z_M},N_{z_N}}|_{z_M=a_M,z_N=a_N} \colon M_{a_M} \otimes N_{a_N} \to N_{a_N} \otimes M_{a_M}
\end{align}
does not vanish provided that $M$ and $N$ are non-zero $U_q'(\g)$-modules in $\Ca_\g$. Such an $\Rren{}$ is unique up to $\ko[(z_M/z_N)^{\pm 1}]^\times =
\bigsqcup\limits_{n \in \Z}\ko^\times z_M^{n}z_N^{-n}$, and it is called a \defn{renormalized $R$-matrix}. For more details, we refer the reader to~\cite{EFK98}.

We denote by
\[
\rmat{M,N} \seteq \Rren{M_{z_M},N_{z_N}}|_{z_M=1,z_N=1} \colon M \otimes  N \to N \otimes  M
\]
and call it the \defn{$R$-matrix}.
By definition $ \rmat{M,N}$ never vanishes.
When $M = V(\varpi_i)$ and $N = V(\varpi_j)$, we simply denote
\[
\rmat{i,j}(z_{V(\varpi_i)}, z_{V(\varpi_j)}) \seteq \Rren{V(\varpi_i)_{z_{V(\varpi_i)}},V(\varpi_j)_{z_{V(\varpi_y)}}}.
\]

For simple $U_q'(\g)$-modules $M$ and $N$ in $\Ca_\g$, the universal $R$-matrix $\Runiv{M_{z_M},N_{z_N}}$ is rationally renormalizable. Then, for
dominant extremal weight vectors $u_M$ and $u_N$ of $M$ and $N$, there exists $a_{M,N}(z_N/z_M) \in \ko[[z_N/z_M]]^\times$ such that
\begin{align}\label{eq: aMN}
\Runiv{M_{z_M},N_{z_N}} \big( (u_M)_{z_M} \otimes  (u_N)_{z_N}) \big) =  a_{M,N}(z_N/z_M) \big( (u_N)_{z_N} \otimes  (u_M)_{z_M}) \big).
\end{align}
Then the map $\Rnorm{M_{z_M},N_{z_N}} \seteq a_{M,N}(z_N/z_M)^{-1}\Runiv{M_{z_M},N_{z_N}}$ is the unique $\ko(z_M,z_N) \otimes _{\ko[z_M^{\pm1},z_N^{\pm1}]} U_q'(\g)$-module homomorphism sending
$\big( (u_M)_{z_M} \otimes  (u_N)_{z_N} \big)$ to $\big( (u_N)_{z_N} \otimes  (u_M)_{z_M} \big)$.
%
It is known that $\ko(z_M,z_N) \otimes _{\ko[z_M^{\pm1},z_N^{\pm1}]} ( M_{z_M} \otimes  N_{z_N})$ is simple
$\ko(z_M,z_N) \otimes _{\ko[z_M^{\pm1},z_N^{\pm1}]} U_q'(\g)$-module (\cite[Proposition 9.5]{Kas02}). We call
$\Rnorm{M_{z_M},N_{z_N}}$ the \defn{normalized $R$-matrix}.

Let us denote by $d_{M,N}(u) \in \ko[u]$ a monic polynomial of the smallest degree such that the image
$d_{M,N}(z_N/z_M)\Rnorm{M_{z_M},N_{z_N}}$ is contained in $N_{z_N} \otimes  M_{z_M}$. We call $d_{M,N}$
the \defn{denominator} of $\Rnorm{M_{z_M},N_{z_N}}$. Then,
\[
d_{M,N}(z_N/z_M)\Rnorm{M_{z_M},N_{z_N}} \colon M_{z_M} \otimes  N_{z_N} \to N_{z_N} \otimes  M_{z_M}
\]
is a renormalized $R$-matrix, and
\[
\rmat{M,N} \colon M \otimes  N \to N \otimes  M
\]
is equal to $d_{M,N}(z_N/z_M)\Rnorm{M_{z_M},N_{z_N}}|_{z_M=1,z_N=1}$ up to a constant multiple.

Note that the following diagram  is commutative
%
\begin{align}\label{eq: Runiv property}
&\xymatrix@C=15ex{
M \otimes  M' \otimes  N\ar[r]_{M \otimes  R^\univ_{M_2,N}}\ar@/^2pc/[rr]^
{R^\univ_{M\otimes  M', N}}
&M\otimes  N\otimes  M' \ar[r]_{R^\univ_{M,N} \otimes  M'}&
N \otimes  M \otimes  M',
}
\end{align}
for $M,M',N  \in \Ca_\g$. Hence, if $\Runiv{{M}_{z_{M}},N_{z_N}}$ and $\Runiv{{M'}_{z_{M'}},N_{z_N}}$ are rationally renormalizable, then $\Runiv{(M \otimes  M')_z,N_{z_N}}$ is also.
In particular, if $M$, $M'$ and $N$ are simple modules in $\Ca_\g$, then
$R^\univ_{(M \otimes  M')_{z_1},N_{z_2}}$ is  rationally renormalizable.


\begin{lemma}[{\cite[Lemma 7.3]{KO17}}]
\label{lem:simplepole}
Let $U_q'(\g)$ be a quantum affine algebra, and let $V$ and $W$ be  good $U_q'(\g)$-modules.
If the normalized $R$-matrix $\Rnorm{V,W}(z)$ has a simple pole at $z=a$ for some $a \in \ko^\times$, then we have
\[
\Image \big( (z-a) \Rnorm{V,W} |_{z=a} \big) = \operatorname{Ker} \big( \Rnorm{W,V} |_{z=a} \big).
\]
Moreover, the tensor product $V \otimes W_a$  is of composition length $2$.
\end{lemma}

\begin{lemma}[{\cite[Lemma C.15]{AK97}}] \label{lem:dvw avw}
Let $V', V'', V$ and $W$ be simple $U_q'(\g)$-modules in $\Ca_\g$. Assume that we have a surjective $U_q'(\g)$-homomorphism
\[
V'\otimes V'' \twoheadrightarrow V.
\]
Then we have
  \begin{align*}
   \dfrac{ d_{W,V'}(z) d_{W,V''}(z) a_{W,V}(z)}{d_{W,V}(z)a_{W,V'}(z)a_{W,V''}(z)}
   \quad \text{and} &\quad
   \dfrac{ d_{V',W}(z) d_{V'',W}(z) a_{V,W}(z)}{d_{V,W}(z)a_{V',W}(z)a_{V'',W}(z)} \in \ko[z^{\pm 1}].
  \end{align*}
\end{lemma}

For $i,j \in I_0$, we denote by $d_{i,j}(z)$ by the denominator $d_{V(\varpi_i),V(\varpi_j)}$ between the fundamental representations $V(\varpi_i)$ and $V(\varpi_j)$.
We write
\[
d_{i,j}(z)=\prod_{\nu} (z-x_\nu) \qtext{ and } d_{i^*,j}(z)=\prod_{\nu} (z-y_\nu).
\]

For rational functions $f,g \in \ko(z)$, we write $f \equiv g$ if there exists an element $a \in \ko[z^{\pm 1}]^\times$ such that $f = ag$.

\begin{lemma}[{\cite{AK97}}]
\label{Lem: aij and dij}
For $k,l \in I_0$, we have
\begin{subequations}
\label{eq: aij and dij}
\begin{align}
a_{k,l}(z)a_{k^*,l}((p^*)^{-1}z) & \equiv
\dfrac{d_{k,l}(z)}{d_{k^*,l}(p^*z^{-1})},\\
a_{k,l}(z) & = q^{(\varpi_k , \varpi_l)}
 \prod_{\nu}\dfrac{(p^* y_\nu z; p^{*2})_\infty
 (p^* y_\nu^{-1} z; p^{*2})_\infty}
  {(x_\nu z; p^{*2})_\infty (p^{*2} x_\nu^{-1} z; p^{*2})_\infty},
\end{align}
\end{subequations}
where $(z;x)_\infty= \prod_{t=0}^\infty (1-x^tz)$ and $a_{k,l}(z) \seteq a_{V(\varpi_k),V(\varpi_l)}$ in~\eqref{eq: aMN}.
\end{lemma}

\subsection{Denominator formulas and folded AR quiver} In this subsection, we shall review the relationship the folded AR quiver
with the denominator formulas between fundamental representations. We start this subsection with the following theorem telling that
the denominator formulas between fundamental representations play an important roles of representation theory on $\Ca_\g$.

\begin{theorem}[{\cite{AK97,Chari02,Kas02,VV02} (see also~\cite[Theorem 2.2.1]{KKK13A})}]
\label{Thm: basic properties}
\hfill
\begin{enumerate}
\item[{\rm (1)}] For good modules $M,N$, the zeros of $d_{M,N}(z)$ belong to
$\C[[q^{1/m}]]q^{1/m}$ for some $m\in\Z_{>0}$.
\item[{\rm (2)}] Let $M_k$ be a good module
with a dominant extremal vector $u_k$ of weight $\lambda_k$, and
$a_k\in\mathbf{k}^\times$ for $k=1,\ldots, t$.
Assume that $a_j/a_i$ is not a zero of $d_{M_i, M_j}(z) $ for any
$1\le i<j\le t$. Then the following statements hold.
\begin{enumerate}
\item[{\rm (i)}]
 $(M_1)_{a_1}\otimes\cdots\otimes (M_t)_{a_t}$ is generated by
$u_1\otimes\cdots \otimes u_t$.
\item[{\rm (ii)}] The head of
$(M_1)_{a_1}\otimes\cdots\otimes (M_t)_{a_t}$ is simple.
\item[{\rm (iii)}] Any non-zero
submodule of $(M_t)_{a_t}\otimes\cdots\otimes (M_1)_{a_1}$ contains the vector $u_t\otimes\cdots\otimes u_1$.
\item[{\rm (iv)}] The socle of $(M_t)_{a_t}\otimes\cdots\otimes (M_1)_{a_1}$
is simple.
\item[{\rm (v)}]
 Let
$\rmat{} \colon (M_1)_{a_1}\otimes\cdots\otimes (M_t)_{a_t}
\to (M_t)_{a_t}\otimes\cdots\otimes (M_1)_{a_1}$
 be  the specialization of $R^{{\rm norm}}_{M_1,\ldots, M_t}$
at $z_k=a_k$.  Then the image of $\rmat{}$ is simple and
it coincides with the head of
$(M_1)_{a_1}\otimes\cdots\otimes (M_t)_{a_t}$
and also with the socle of $(M_t)_{a_t}\otimes\cdots\otimes (M_1)_{a_1}$.
\end{enumerate}
\item[{\rm (3)}] For a simple integrable $U_q'(\g)$-module M, there exists a unique
finite sequence $\big((i_k,a_k) \in I_0 \times \ko^{\times}\bigr)_{k=1}^t$ (up to permutation)
such that $d_{i_k,i_{k'}}(a_{k'}/a_k)\not=0$ for $1\le k<k'\le t$ and
\[
M \iso \hd \left( \bigotimes_{k=1}^t V(\varpi_{i_k})_{a_k} \right).
\]
\item[{\rm (4)}] $d_{k,l}(z)=d_{l,k}(z)=d_{k^*,l^*}(z)=d_{l^*,k^*}(z)$ for $k,l \in I_0$.
\end{enumerate}
\end{theorem}

The denominator formulas between fundamental representations over classical quantum affine algebras were calculated in~\cite{AK97,DO94,KKK13B,Oh14R}. Interestingly, such formulas can be read from folded AR quiver, which are explained
in the following theorems:

\begin{theorem}[{\cite{Oh14R,Oh15E,OS16B,OS16C}}]
\label{thm:reading deno from AR}
\mbox{}
\begin{enumerate}
\item[{\rm (1)}]
For type $A^{(1)}_{n}$, $B^{(1)}_{n}$, $C^{(1)}_{n}$, and $D^{(1)}_{n}$, their denominator formulas between fundamental representations can be read from folded AR quivers
$(\Gamma_Q$ or $\widehat{\Upsilon}_{[\mQ]})$
directly in the following sense:
\begin{align*}
d^{A^{(1)}_{n}}_{k,l}(z) & = D^{\lf \Delta \rf}_{k,l}(z;-q) \times (z-p^{*})^{\delta_{k,l^*}} && \text{for $\lf \Delta \rf$ of type $A_n$,} \\
d^{D^{(1)}_{n}}_{k,l}(z) & = D^{\lf \Delta \rf}_{k,l}(z;-q) \times (z-p^{*})^{\delta_{k,l^*}} && \text{for $\lf \Delta \rf$ of type $D_n$,} \\
d^{B^{(1)}_{n}}_{k,l}(z) & = D^{\lf \mQ \rf}_{k,l}(z;-q_s) \times (z-p^{*})^{\delta_{k,l}} && \text{for $\lf \mQ \rf$ of type $A_{2n-1}$,} \\
d^{C^{(1)}_{n}}_{k,l}(z) & = D^{\lf \mQ \rf}_{k,l}(z;-q_s) \times
(z-p^{*})^{\delta_{k,l}} && \text{for $\lf \mQ \rf$ of type
$D_{n+1}$.}
\end{align*}
%
\item[{\rm (2)}]  For type $A^{(2)}_{n}$ and $D^{(2)}_{n+1}$, their denominator formulas between fundamental representations
can be obtained from the one of type $A^{(1)}_{n}$ and $D^{(1)}_{n+1}$ respectively. Thus the formulas can be also read from AR quivers:
\begin{align*}
d^{A^{(2)}_{n}}_{k,l}(z) & = D^{\lf \Delta \rf}_{k,l}(z;-q)D^{\lf \Delta \rf}_{n+1-k,l}((-1)^nz;-q) \times (z-p^{*})^{\delta_{k,l}}
\\ & = d^{A^{(1)}_{n}}_{k,l}(z)d^{A^{(1)}_{n}}_{n+1-k,l}((-1)^n z)
\\ d^{D^{(2)}_{n+1}}_{k,l}(z) & = \begin{cases}
D^{\lf \Delta \rf}_{n,n+1}(-z;-q)D^{\lf \Delta \rf}_{n,n}(-z;-q) \times (z-p^{*}) & \text{ if } k=l=n, \\
D^{\lf \Delta \rf}_{k,l}(z^2;-q^2)(z-p^{*})^{\delta_{k,l}} &  \text{ otherwise},
\end{cases}
\end{align*}
up to sign $\pm 1$.
\end{enumerate}
\end{theorem}

\subsection{Dorey's rule, categorification and folded AR quiver}

The morphisms in
\[
\Hom_{U_q'(\g)}\big( V(\varpi_i)_a \otimes  V(\varpi_j)_b, V(\varpi_k)_c \big) \quad \text{ for } i,j,k\in I_0 \text{ and } a,b,c \in \ko^\times
\]
are studied by~\cite{CP96,HL10,KKK13B,KKKOIV,Oh14R,Oh15E,YZ11} and called \defn{Dorey's type} morphisms. In~\cite{Oh14D,Oh14A,Oh15E},
the condition of non-vanishing of the above Hom space
are interpreted the positions of $\al,\beta,\ga \in \PR$ in some $\widehat{\Upsilon}_{[\rrz]}$,
where $\gamma = \alpha + \beta$, and $\g$ is a quantum affine algebra of non-exceptional type or $E_{6,7,8}^{(1)}$.
We start this subsection with the definition of monoidal subcategories of $\mC_\g$, which are closely related to the $r$-cluster points and their (folded) AR-quivers introduced
in the earlier parts in this paper.

Recall the map $\overline{\sigma}$ in Definition \ref{def: overline sigma}. For simplicity, we use the following convention:
\[
\text{For a statement $P$, $\chi(P)$ is $1$ if $P$ is true and $0$ if $P$ is false.}
\]

\begin{definition}[{\cite{HL10,KKKOIV}}]
\label{def: VQ(beta)}
Fix $[Q] \in \lf \Delta \rf$ of finite type $\g=A_n$, $D_n$ or $E_{6,7,8}$,
and a positive root $\beta\in \PR$ with $\widehat{\Omega}_{[Q]}(\beta)=(i,p)$. For each $\g^{(t)}$ $(t=1,2,3)$, we take the Dynkin diagram automorphism $\sigma$ of $\Delta$ of type $\g$ whose order is $t$.
We set $U_q'(\g^{(t)})$-module $V^{(t)}_{Q}(\beta)$ defined as follows$\colon$
\begin{align*}
V^{(1)}_{Q}(\beta) \seteq  V(\varpi_{i})_{(-q)^{p}},
\qquad
V^{(2)}_{Q}(\beta) \seteq  V(\varpi_{i^\star})_{((-q)^{p})^\star},
\qquad
V^{(3)}_{Q}(\beta) \seteq  V(\varpi_{i^\dagger})_{((-q)^{p})^\dagger},
\end{align*}
where
\begin{subequations}
\label{eq: star dagger}
\begin{align}
i^\star \seteq \overline{\sigma}(i), \ \ \bigl( (-q)^{p} \bigr)^\star & \seteq
\begin{cases}
\bigl( \chi\bigl(i \le \sigma(i)\bigr)+ \chi\bigl(i>\sigma(i)\bigr) (-1)^n \bigr)(-q)^{p} & \text{ if } \g^{(2)}=A^{(2)}_n, \\ 
\bigl( \delta_{i,\sigma(i)} (\sqrt{-1})^{n-i} + \chi(n-1 \le i \le n)(-1)^i \bigr)(-q)^{p} & \text{ if } \g^{(2)}=D^{(2)}_n, \\
\bigl( \chi(i < \sigma(i))- \chi(i>\sigma(i))+ \delta_{i,\sigma(i)} \sqrt{-1} \bigr) (-q)^p & \text{ if } \g^{(2)}=E^{(2)}_6, \\
\end{cases} \label{eq:star}\\
i^\dagger \seteq \overline{\sigma}(i), \ \ \bigl( (-q)^{p} \bigr)^\dagger &  \seteq
\bigl(\delta_{i,1}-\delta_{i,2}+\delta_{i,3} \omega + \delta_{i,4} \omega^2\bigr)(-q)^{p} \hspace{6.1ex} \text{ if } \g^{(3)}=D^{(3)}_4.
\label{eq:dagger}
\end{align}
\end{subequations}
Here $\omega$ denotes the primitive third root of unity.

We define the smallest abelian full subcategory
$\mathcal{C}^{(t)}_Q$ $(t=1,2,3)$ of $\mathcal{C}_{\g^{(t)}}$ such that
\begin{enumerate}[{\rm (a)}]
\item it is stable by taking subquotient, tensor product and extension,
\item it contains $V^{(t)}_Q(\beta)$ for all $\beta \in \PR$.
\end{enumerate}
\end{definition}

\begin{theorem}[\cite{CP96,KKKOIV,Oh15E}]
\label{thm: Dorey classical 1}
For $\g^{(1)} = A_n^{(1)}$, $D_n^{(1)}$, $E_{6,7,8}^{(1)}$ or $\g^{(2)} = A_n^{(2)}$, $D_n^{(2)}$,  let $(i_1,x_1)$, $(i_2,x_2)$, $(i_3,x_3) \in I_0 \times \ko^\times$. Then
\[
\dim_{\ko} \left( \Hom_{U_q'(\g^{(t)})}\big( V^{(t)}(\varpi_{i_2})_{x_2} \otimes V^{(t)}(\varpi_{i_1})_{x_1} , V^{(t)}(\varpi_{i_3})_{x_3}  \big) \right) = 1  \qquad (t=1,2)
\]
if and only if there exists a Dynkin quiver $Q$ and $\al,\beta,\ga \in \Phi_{Q}^+$ such that
\begin{enumerate}[{\rm (1)}]
\item $\alpha \prec_{[Q]} \beta$ and $\alpha + \beta = \gamma$,
\item $V^{(t)}(\varpi_{i_2})_{x_2}  = V^{(t)}_{Q}(\beta)_a, \ V^{(t)}(\varpi_{i_1})_{x_1}  = V^{(t)}_{Q}(\al)_a, \ V^{(t)}(\varpi_{i_3})_{x_3}  = V^{(t)}_{Q}(\ga)_a$
for some $a \in \ko^\times$.
\end{enumerate}
\end{theorem}

\begin{theorem}[{\cite{HL11,KKKOIV}}]
\label{thm:categorification1}
Let $Q$ be a Dynkin quiver of finite type $A_n$, $D_n$ $(t = 1,2)$ and $E_{6,7,8}$ $(t=1)$.
\begin{enumerate}[{\rm (1)}]
\item $[\Ca_Q^{(t)}] \iso U_\A(\mathsf{g})^\vee$, where $\mathsf{g}$ denotes the simple Lie algebra of the same type and $[\mathcal{C}_Q^{(t)}]$ denotes the Grothendieck ring of $\mathcal{C}_Q^{(t)}$.
\item The dual PBW-basis associated to $[Q]$ and the upper global basis of
$U^-_{\A}(\mathsf{g})^{\vee}$ are categorified by the modules over $U_q'(\g^{(t)})$ in the following sense:
\begin{enumerate}[{\rm (i)}]
\item The set of all simples in $\mathcal{C}^{(t)}_Q$ corresponds to
the upper global basis of $U^-_{\A}(\mathsf{g})^{\vee}|_{q=1}$.
\item The set
\[
\left\{ V_{Q}^{(t)}(\beta_1)^{\otimes m_1}\otimes\cdots \otimes V_{Q}^{(t)}(\beta_\N)^{\otimes m_\N} \mid \um \in \Z_{\ge 0}^\N \right\}
\]
corresponds to the dual PBW-basis under the isomorphism in {\rm (1)}.
\end{enumerate}
\end{enumerate}
\end{theorem}

\begin{definition} [cf.~\cite{OS16B,OS16C}] \label{def: VmQ(beta)} \hfill
\begin{enumerate}
\item[{\rm (1)}] For any $[\mQ] \in \lf \mQ \rf$
and any positive root $\beta\in \PR$
of type $A_{2n-1}$, $D_{n+1}$ or $E_6$, we set the $U_q'(\g^{(1)})$-module $(\g^{(1)}=B_n^{(1)}$, $C_n^{(1)}$ or $F_4^{(1)})$ $V_{\mQ}(\beta)$ defined as follows:
For $\widehat{\Omega}_{\mQ}(\beta)=(i,p/2)$, we define
\begin{align} \label{eq: VmQ(beta)1}
V_{\mQ}(\beta) \seteq  \begin{cases}
V(\varpi_{i})_{(-1)^i(q^{1/2})^{p}} & \text{ if } \g^{(1)}=B_n^{(1)} \text{ or } F_4^{(1)}, \\
V(\varpi_{i})_{(-q^{1/2})^{p}} & \text{ if } \g^{(1)}=C_n^{(1)}.
\end{cases}
\end{align}
\item[{\rm (2)}] For any $[\mathfrak{Q}] \in \lf \mathfrak{Q} \rf$
and any positive root $\beta\in \PR$
of type $D_{4}$, we set the $U_q'(G_2^{(1)})$-module ($\g^{(1)}=G_2^{(1)}$) $V_{\mathfrak{Q}}(\beta)$ defined as follows:
For $\widehat{\Omega}_{\mathfrak{Q}}(\beta)=(i,p/3)$, we define
\begin{align} \label{eq: VmQ(beta)2}
V_{\mathfrak{Q}}(\beta) \seteq V(\varpi_{i})_{(-q^{1/3})^{p}}.
\end{align}
\end{enumerate}
We define the smallest abelian full subcategory
$\mathscr{C}^{(1)}_\mQ$ of $\mathcal{C}_{\g^{(1)}}$ such that
\begin{itemize}
\item[{\rm (a)}] it is stable by taking subquotient, tensor product and extension,
\item[{\rm (b)}] it contains $V_\mQ(\beta)$ (resp.~$V_{\mathfrak{Q}}(\beta)$) for all $\beta \in \PR$.
\end{itemize}
\end{definition}

\begin{theorem} [\cite{CP96,OS16B,OS16C}] \label{thm: Dorey classical 2}
Let $(i,x)$, $(j,y)$, $(k,z) \in I_0 \times \ko^\times$. Then
\[
\dim_{\ko} \left( \Hom_{U_q'(\g^{(1)})}\big( V(\varpi_{j})_y \otimes V(\varpi_{i})_x , V(\varpi_{k})_z  \big) \right) = 1  \quad \text{ for } \g^{(1)}=B^{(1)}_{n} \ \text{ $($resp.\ $C^{(1)}_{n})$}
\]
if and only if there exists a twisted adapted class $[\mQ]$ of type $A_{2n-1}$ $($resp.\ $D_{n+1})$ and $\al,\beta,\ga \in \Phi_{A_{2n-1}}^+$
$($resp.\ $\Phi_{D_{n+1}}^+)$ such that
\begin{enumerate}[{\rm (1)}]
\item $(\al,\beta)$ is a $[\mQ]$-minimal pair of $\ga$,
\item $V(\varpi_{j})_y  = V_{\mQ}(\beta)_a, \ V(\varpi_{i})_x  = V_{\mQ}(\al)_a, \ V(\varpi_{k})_z  = V_{\mQ}(\ga)_a$
for some $a \in \ko^\times$.
\end{enumerate}
\end{theorem}

\begin{theorem}[{\cite{KO17}}]
\label{thm:categorification2}
Let $[\mQ]$ be a twisted adapted class of finite type $A_{2n-1}$ $($resp.~$D_{n+1})$.
\begin{enumerate}[{\rm (1)}]
\item $[\mathscr{C}_\mQ^{(1)}] \iso U_\A(\mathsf{g})^\vee$, where $\mathsf{g}$ denotes the simple Lie algebra of $A_{2n-1}$ $($resp.~$D_{n+1})$ and $[\mathscr{C}_\mQ^{(1)}]$ denotes the Grothendieck ring of $\mathscr{C}_\mQ^{(1)}$.
\item The dual PBW-basis associated to $[\mQ]$ and the upper global basis of
$U^-_{\A}(\mathsf{g})^{\vee}$ are categorified by the modules over $U_q'(\g^{(1)})$ of affine type $B_n^{(1)}$ $($resp.~$C^{(1)}_{n})$ in the following sense:
\begin{enumerate}[{\rm (i)}]
\item The set of all simples in $\mathscr{C}^{(1)}_Q$ corresponds to
the upper global basis of $U^-_{\A}(\mathsf{g})^{\vee}|_{q=1}$.
\item For each $\um \in \Z_{\ge 0}^\N$, define the  $U_q'(\g^{(1)})$-module
$V_\mQ^{(1)}(\um)$ by $$
V_{\mQ}(\beta_1)^{\otimes m_1}\otimes\cdots \otimes V_{\mQ}(\beta_\N)^{\otimes m_\N}.$$
Then the set $\{ V_\mQ(\um) \mid \um \in \Z_{\ge 0}^\N\}$ corresponds to
the dual PBW-basis under the isomorphism in~{\rm (1)}.
\end{enumerate}
\end{enumerate}
\end{theorem}

In the rest of this paper, we shall prove
\begin{itemize}
\item[(1)] analogues of Theorem~\ref{thm:reading deno from AR} for all exceptional affine types,
\item[(2)] analogues of Theorem~\ref{thm: Dorey classical 1} and Theorem~\ref{thm:categorification1} for both $E^{(2)}_{6}$ and $D^{(3)}_{4}$,
\item[(2')] analogues of Theorem~\ref{thm: Dorey classical 2} and Theorem~\ref{thm:categorification2} for both $F^{(1)}_{4}$ and $G^{(1)}_2$,
\end{itemize}
by analyzing the structure of their fundamental representations and $R$-matrices between them.

\section{Denominator formulas for quantum affine algebras} \label{sec4:computations}
In this section, we compute the denominator formula $d_{i,j}(z)$
between fundamental representations $V(\varpi_i)$ and $V(\varpi_j)$
by employing the framework of~\cite{AK97,KKK13A,Oh14R}. We first do
this for types $G_2^{(1)}$, $D_4^{(3)}$, $E_6^{(2)}$, and $F_4^{(1)}$
because one can compute $d_{i,j}(z)$ for certain $i,j \in I_0$ by hand. Using
such basic formulas and Dorey's type morphisms, we can compute all
$d_{i,j}(z)$'s. This strategy will be applied to all exceptional
affine types; that is, for the remained cases, we implemented the fundamental
modules and the $R$-matrix computation in \textsc{SageMath}~\cite{sage,combinat}.
Using this code, we compute some of the basic formulas, which
we then are able to compute all $d_{i,j}(z)$'s using Dorey's rule as
in the case of $U_q'(G_2^{(1)})$ and $U_q'(D_4^{(3)})$. We have also
verified a number of the computations using our code.

We note here that, using geometric approach, Fujita independently showed the distribution of simple roots for $d_{i,j}(z)$ for $E^{(1)}_{6,7,8}$ (\cite{Fujita18}).

\subsection{\texorpdfstring{$U_q'(G_2^{(1)})$ and $U_q'(D_4^{(3)})$}{Uq'(G2(1)) and Uq'(D4(3))}}
\label{sec:G2_D43}

In this subsection, we shall compute all the denominator formulas and show that the formulas can be read from any folded AR quivers of type $D_4$.

\subsubsection{$U_q'(G_2^{(1)})$}

For type $G_2^{(1)}$, we have $q_0=q_1=q$ and $q_2 = q^{1/3}$.
Furthermore, we have $q_s = q^{1/3}$ and $p^* = q^4$.

Recall that the $U_q(\g_0)$-decompositions of $V(\varpi_1)$ and $V(\varpi_2)$ are given as follows:
\begin{align} \label{eq: classical decomp G21}
V(\varpi_1) \iso V(\clfw_1) \oplus V(0) \qtext{ and } V(\varpi_2) \iso V(\clfw_2).
\end{align}

Let us recall the $U_q'(\g)$-modules structure of $V(\varpi_2)$ whose crystal graph can be depicted by as follows~\cite{MOW12}:\footnote{This crystal can also be constructed from, \textit{e.g.},~\cite{HN06,Kas02,NS05,NS06}.}
\[
\begin{tikzpicture}[>=stealth,baseline=0,yscale=0.7,xscale=1.5]
\node (1) at (0,0) {$\young(1)$};
\node (2) at (1.5,0) {$\young(2)$};
\node (3) at (3,0) {$\young(3)$};
\node (0) at (4.5,0) {$\young(0)$};
\node (b3) at (6,0) {$\young(\oTh)$};
\node (b2) at (7.5,0) {$\young(\oT)$};
\node (b1) at (9,0) {$\young(\oO)$};
\path[->,font=\tiny]
 (1) edge[red] node[above]{$1$} (2)
 (2) edge[blue] node[above]{$2$} (3)
 (3) edge[red] node[above]{$1$} (0)
 (0) edge[red] node[above]{$1$} (b3)
 (b3) edge[blue] node[above]{$2$} (b2)
 (b2) edge[red] node[above]{$1$} (b1);
\path[->,font=\tiny]
 (b1) edge[out=130,in=50] node[below]{$0$} (2);
\path[->,font=\tiny]
 (b2) edge[out=230,in=310] node[above]{$0$} (1);
\end{tikzpicture}
\]
where
\begin{align*}
& e_0 \young(1) = \young(\oT), & e_0\young(2) & = \young(\oO), &f_0 \young(\oT) & = \young(1), & f_0\young(\oO) & = \young(2), \\
& e_1 \young(\oT) = \young(\oTh), &  e_1\young(3) & = \young(2), & f_1 \young(2) & = \young(3), & f_1\young(\oTh) & = \young(\oT), \\
& e_2 \young(\oO) = \young(\oT), & e_2\young(2) & = \young(1), & e_2\young(\oTh) & = \young(0), & e_2\young(0) & = [2]_2\young(3),  \\
& f_2 \young(1) = \young(2), & f_2\young(3) & = \young(0), & f_2 \young(0) & = [2]_2\young(\oTh), & f_2\young(\oT) & = \young(\oO).
\end{align*}
If the action of some basis vector is not written, then we should understand
that it vanishes.

The $U_q(\g_0)$-decomposition of $V(\varpi_2) \otimes V(\varpi_2)$ is given as follows (see also~\cite{KM94}):
\begin{align}\label{eq: class decomp} V(\varpi_2) \otimes V(\varpi_2)  \iso V(2\clfw_2) \oplus V(\clfw_2) \oplus V(\clfw_1) \oplus V(0)
\end{align}
whose $U_q(G_2)$-highest weight vectors are given as follows:
\begin{align*}
u_{2\clfw_2} & \seteq \young(1) \otimes \young(1)\ , \\
u_{\clfw_2} & \seteq \young(1)\otimes \young(0) - q_2^6 \young(0) \otimes \young(1) -q_2^2 [2]_2 \young(2) \otimes \young(3)+q_2^5[2]_2 \young(3) \otimes \young(2)\ , \\
u_{\clfw_1} & \seteq \young(1)\otimes \young(2) + q_2 \young(2) \otimes \young(1)\ , \\
u_{0} & \seteq \young(1) \otimes \young(\oO) + q_2^{10}\young(\oO) \otimes \young(1) - q_2\young(2)\otimes\young(\oT) - q_2^9\young(\oT) \otimes \young(2)
\\ & \hspace{20pt} + q_2^4\young(3) \otimes \young(\oTh) + q_2^6\young(\oTh)\otimes \young(3) - \dfrac{q_2^4}{[2]_2} \young(0) \otimes \young(0)\ .
\end{align*}

Recall that the $R$-matrix $\rmat{2,2}(x,y)$ is the unique $U_q'(\g)$-module isomorphism
\[
\rmat{2,2}(x,y) \colon V(\varpi_2)_x \otimes V(\varpi_2)_y \to V(\varpi_2)_y \otimes V(\varpi_2)_x.
\]
By $U_q'(\g)$-linearity and~\eqref{eq: class decomp}, we have
\begin{align*}
\rmat{2,2}(x,y)(u_{2\clfw_2}) & = a^{2\clfw_2}u_{2\clfw_2}, & \rmat{2,2}(x,y)(u_{\clfw_2}) & = a^{\clfw_2}u_{\clfw_2}, \\
\rmat{2,2}(x,y)(u_{\clfw_1}) & = a^{\clfw_1}u_{\clfw_1}, & \rmat{2,2}(x,y)(u_{0}) &= a^{0}u_{0}.
\end{align*}

By direct computations, we have the followings:

\begin{lemma} In $V(\varpi_2)_x \otimes V(\varpi_2)_y$, we have the following:
\begin{enumerate}
\item[{\rm (1)}] $f_0f_1f_2^{(2)}f_1(u_{\clfw_1}) = (q_2^{-1}y^{-1}-q_2x^{-1})u_{2\clfw_2}$,
\item[{\rm (2)}] $f_0f_1f_2(u_{\clfw_2}) = [2]_2q_{2}^{-2} (y^{-1}-q_2^8x^{-1} )u_{2\clfw_2}$,
\item[{\rm (3)}] $f_0f_1f^{(2)}_2f_1f_0(u_{0}) = q_2^{-4}(y^{-1}-q_2^{12}x^{-1})(y^{-1}-q_2^2x^{-1})u_{2\clfw_2}$.
\end{enumerate}
\end{lemma}

\begin{proposition} \label{prop: d22}
Put $z=x^{-1}y$ and denote $\rmat{2,2}(z) = \rmat{2,2}(x,y)$. Then we have
\begin{enumerate}
\item[{\rm (1)}] $a^{2\clfw_2}(z)=(z-q_2^{2})(z-q_2^{8})(z-q_2^{12})$,
\item[{\rm (2)}] $a^{\clfw_2}(z)=(z-q_2^{2})(1-q_2^{8}z)(z-q_2^{12})$,
\item[{\rm (3)}] $a^{\clfw_1}(z)=(1-q_2^{2}z)(z-q_2^{8})(z-q_2^{12})$,
\item[{\rm (4)}] $a^{0}(z)=(1-q_2^{2}z)(z-q_2^{8})(1-q_2^{12}z)$.
\end{enumerate}
\end{proposition}

Now we can conclude that the denominator $d_{2,2}(z)$ for $U_q'(\g)$ is given as follows:
\begin{equation} \label{eq: d22 D43}
d_{2,2}(z) = (z-q_2^{2})(z-q_2^{8})(z-q_2^{12}).
\end{equation}

\begin{remark}
\label{note: strategy}
Once we know the module structure of $V(\varpi_i)$ and $V(\varpi_j)$ explicitly,
we can compute $d_{i,j}(z)$ (see also $d_{1,n}(z)$ for $D^{(2)}_{n+1}$ in~\cite[Section 4.2]{Oh14R}).
\end{remark}

\begin{remark}\label{rmk: heads}
By the above proposition, we have
\begin{enumerate}
\item[{\rm (i)}] $\Image(\rmat{2,2}(-q_s^{-1},-q_s^{1})) \iso V(\clfw_1)\oplus V(0)$ as $U_q(\g_0)$-modules,
\item[{\rm (ii)}] $\Image(\rmat{2,2}((-q_s)^{-4},(-q_s)^{4})) \iso V(\clfw_2)$ as $U_q(\g_0)$-modules.
\end{enumerate}
\end{remark}

\begin{proposition}\label{prop: Dorey D43 22}
There exists a injective $U'_q(\g)$-homomorphism:
\begin{equation}\label{eq: Dorey 22}
V(\varpi_2) \hookrightarrow V(\varpi_2)_{(-q_s)^4} \otimes V(\varpi_2)_{(-q_s)^{-4}}.
\end{equation}
\end{proposition}

\begin{proof}
By \cite{KKKO14S}, the socle of $V(\varpi_2)_{a(-q_s)^4} \otimes V(\varpi_2)_{a(-q_s)^{-4}}$
is simple.
Then \eqref{eq: bar struture},~\eqref{eq: classical decomp G21} and Remark~\ref{rmk: heads} implies that there exists a injective $U'_q(\g)$-homomorphism
\[
V(\varpi_2) \hookrightarrow V(\varpi_2)_{a(-q_s)^4} \otimes V(\varpi_2)_{a(-q_s)^{-4}}
\]
for $a^2=1$. Here we use the uniqueness and simplicity of the fundamental representations
$V(\varpi_i)$.

If $a=-1$, then $d_{2,2}(z)$ has a root $-q_s^{8}$ by taking the dual:
\[
V(\varpi_2) \otimes V(\varpi_1)_{-q_s^{8}}  \twoheadrightarrow V(\varpi_2)_{-q_s^4}.
\]
However, this can not happen by~\eqref{eq: d22 D43}.
\end{proof}

The following proposition is known as the T-system for type $G_2^{(1)}$ proved in~\cite{Her06}. However, we shall give a proof since the argument of the proof
will be applied to other exceptional types throughout this paper.

\begin{proposition}
\label{prop: Dorey D43 12}
There exists a injective $U'_q(\g)$-homomorphism:
\begin{equation}
\label{eq: Dorey 21}
V(\varpi_1) \hookrightarrow V(\varpi_2)_{(-q_s)^1} \otimes V(\varpi_2)_{(-q_s)^{-1}}.
\end{equation}
\end{proposition}

\begin{proof}
By Remark~\ref{rmk: heads}, there exists a injective $U'_q(\g)$-homomorphism
\[
V(\varpi_1) \hookrightarrow V(\varpi_2)_{a(-q_s)^1} \otimes V(\varpi_2)_{a(-q_s)^{-1}}
\]
for some $a \in \ko^\times$. As in the above proposition, we have (see~\cite[Fig.~1]{Yamane98})
\[
f_2^{(3)} f_1(u_{\clfw_1})=  f_0e_1e_0^{(2)}(u_{\clfw_1}).
\]
in $V(\varpi_2)_{a(-q_s)^1} \otimes V(\varpi_2)_{a(-q_s)^{-1}}$. On the other hand, by the direct computation, we have
\begin{align*}
f_2^{(3)} f_1(u_{\clfw_1}) & = \young(2) \otimes \young(\oTh) - q_s \young(\oTh) \otimes \young(2)\ ,
\qquad\qquad
 f_0e_1e_0^{(2)}(u_{\clfw_1})  = a^{-1}\left(\young(2) \otimes \young(\oTh) - q_s \young(\oTh) \otimes \young(2)\right),
\end{align*}
in $V(\varpi_2)_{a(-q_s)^1} \otimes V(\varpi_2)_{a(-q_s)^{-1}}$. Thus we can conclude that $a=1$.
\end{proof}

Note that $d_{1,1}(z)$ was computed in~\cite{Yamane98} and
we have $d_{2,2}(z)$ by Proposition~\ref{prop: d22}:
\begin{subequations}
\label{eq:d11 d22 G21}
\begin{align}
d_{1,1}(z)&=(z-q_s^{6})(z-q_s^{8})(z-q_s^{10})(z-q_s^{12}), \label{eq:d11 G_21}\\
d_{2,2}(z)&=(z-q_s^{2})(z-q_s^{8})(z-q_s^{12}). \label{eq:d22 G 21}
\end{align}
\end{subequations}
Now we shall compute $d_{1,2}(z)=d_{2,1}(z)$.
By~\eqref{eq:d11 d22 G21}, we have $a_{1,1}(z) \seteq a_{V(\varpi_1),V(\varpi_1)}$ and $a_{2,2}(z)\seteq a_{V(\varpi_2),V(\varpi_2)}$ in~\eqref{eq: aMN}, we compute the following:
\begin{equation}
\label{eq: a11 a22}
a_{1,1}(z)=\dfrac{[20][4][22][2][24][0]}{[8][16][10][14][12]^2} \qtext{ and }
a_{2,2}(z)=\dfrac{[14][10][20][4][24][0]}{[2][22][8][16][12]^2},
\end{equation}
where $[a]=((-q_s)^az;q_s^{24}).$

By the direct computation, we have
\begin{equation}
\label{eq: direct}
\Rnorm{2,2}(z)( \young(1) \otimes {\young(2)}_z ) = \dfrac{1-q_s^2}{z-q_s^2} \ {\young(1)}_z \otimes {\young(2)} +
\dfrac{q_s(z-1)}{z-q_s^2} \ {\young(2)}_z \otimes {\young(1)}\ .
\end{equation}

\begin{theorem} \label{thm: G21 a and d}
We have
\[
a_{1,2}(z)=\dfrac{[19][5][23][1]}{[7][11][17][13]} \qtext{ and } d_{1,2}(z)=(z+q_s^7)(z+q_s^{11}).
\]
\end{theorem}

\begin{proof}
By \eqref{eq: Runiv property}, we have the following commutative diagram:
\begin{align*}
\xymatrix@R=5ex@C=9ex{
V(\varpi_2) \otimes V(\varpi_2)_{-q_s^{-1}z} \otimes V(\varpi_2)_{-q_s z}
\ar[rr]^{\qquad \qquad \quad V(\varpi_2)\otimes p_{2,2}} \ar[d]_{\Runiv{2,2}(-q_s^{-1}z) \otimes V(\varpi_2)_{-q_sz}} &&
V(\varpi_{2}) \otimes V(\varpi_{1})_{z} \ar[dd]_{\Runiv{2,1}(z)} \\
V(\varpi_2)_{-q_s^{-1}z}  \otimes V(\varpi_{2}) \otimes V(\varpi_2)_{(-q_s)z}
\ar[d]_{V(\varpi_2)_{(-q_s)^{-1}z} \otimes \Runiv{2,2}(-q_sz)} &\\
V(\varpi_2)_{-q_s^{-1}z} \otimes V(\varpi_2)_{-q_sz}  \otimes V(\varpi_{2})\ar[rr]^{\qquad \qquad \qquad p_{2,2}\otimes V(\varpi_2) } &&
V(\varpi_{1})_{z} \otimes V(\varpi_{2}).
}
\end{align*}
Then we have
\begin{align*}
\xymatrix@R=3ex@C=9ex{
{\young(1)}\otimes  {\young(1)} \otimes  {\young(2)} \ar[rr]  \ar[d] &&
{\young(1)} \otimes  u_{\varpi_1} \ar[dd] \\
a_{2,2}(-q_s^{-1}z){\young(1)}\otimes  {\young(1)} \otimes  {\young(2)}
\ar[d] &\\
a_{2,2}(-q_sz)a_{2,2}(-q_s^{-1}z){\young(1)}\otimes  w \ar[rr] &&
a_{2,1}(z)u_{\varpi_1}  \otimes  {\young(1)}
}
\end{align*}
where
\[
w= \Rnorm{2,2}(-q_sz)({\young(1)}\otimes  {\young(2)}_z) = \dfrac{1-q_s^2}{-q_sz-q_s^2} \ {\young(1)}_z \otimes {\young(2)} +
\dfrac{q_s((-q_s)z-1)}{-q_sz-q_s^2} \ {\young(2)}_z \otimes {\young(1)}\ .
\]
Since ${\young(1)}\otimes {\young(1)}$ vanishes under the map $p_{2,2}$, we have
\[
a_{2,1}(z) = a_{2,2}(-q_sz)a_{2,2}(-q_s^{-1}z) \dfrac{q_s((-q_s)z-1)}{-q_sz-q_s^2}   \equiv a_{2,2}(-q_sz)a_{2,2}(-q_s^{-1}z) \dfrac{[23]}{[-1]}\dfrac{[1]}{[25]}.
\]
Thus our first assertion follows from~\eqref{eq: a11 a22}. For the homomorphism~\eqref{eq: Dorey 21}, we have a surjective homomorphism
\[
V(\varpi_2)_{-q_s^{-1}} \otimes   V(\varpi_2)_{-q_s} \twoheadrightarrow V(\varpi_1).
\]

Applying Lemma~\ref{lem:dvw avw} to the above surjective homomorphism, we have
\begin{equation}
\label{eq: step1}
\begin{aligned}
& \ko[z^{\pm1}] \ni \dfrac{d_{2,2}(-q_s^{-1}z)d_{2,2}(-q_sz)}{d_{1,2}(z)} \dfrac{a_{1,2}(z)}{a_{2,2}(-q_s^{-1}z)a_{2,2}(-q_sz)} \equiv \\
& \dfrac{(z-(-q_s)^{-1})(z-(-q_s)^{3})(z-(-q_s)^{7})(z-(-q_s)^{9})(z-(-q_s)^{11})(z-(-q_s)^{13})}{d_{1,2}(z)}.
\end{aligned}
\end{equation}
By taking the right dual $V(\varpi_2)_{-q_s^{11}}$ of $V(\varpi_2)_{-q_s^{-1}}$ to \eqref{eq: Dorey 21}, we have
\[
V(\varpi_1)_{-q_s^{-1}} \otimes   V(\varpi_2)_{(-q_s)^{10}} \twoheadrightarrow V(\varpi_2).
\]
Similarly, we have
\begin{equation}
\label{eq: step2}
\begin{aligned}
& \dfrac{d_{2,2}((-q_s)^{10}z)d_{1,2}((-q_s)^{-1}z)}{d_{2,2}(z)} \dfrac{a_{2,2}(z)}{a_{1,2}((-q_s)^{-1}z)a_{2,2}((-q_s)^{10}z)}  \\
\equiv &  \dfrac{(z-(-q_s)^{18})(z-(-q_s)^{22})d_{1,2}((-q_s)^{-1}z)}{(z-(-q_s)^{6})(z-(-q_s)^{10})}  \in \ko[z^{\pm1}]
\end{aligned}
\end{equation}
Thus our second assertion follows from~\eqref{eq: step1} and~\eqref{eq: step2}.
\end{proof}

Now we have a type $G_2^{(1)}$ analogue for Theorem~\ref{thm:reading deno from AR}.

\begin{theorem}
For any $[\mathfrak{Q}]$ of type $D_4$, we have
\begin{align*}
d^{G_2^{(1)}}_{1,1}(z) & = D^{[\mathfrak{Q}]}_{1,1}(z;-q_s) \times (z-q^4)=(z-q_s^{6})(z-q_s^{8})(z-q_s^{10})(z-q_s^{12}), \\
d^{G_2^{(1)}}_{1,2}(z) & = D^{[\mathfrak{Q}]}_{1,2}(z)=(z+q_s^7)(z+q_s^{11}),  \\
d^{G_2^{(1)}}_{2,2}(z) & = D^{[\mathfrak{Q}]}_{2,2}(z;-q_s) \times (z-q^4)=(z-q_s^{2})(z-q_s^{8})(z-q_s^{12}).
\end{align*}
\end{theorem}

\subsubsection{$U_q'(D_4^{(3)})$}

For type $D_4^{(3)}$, we have $q_0=q_1 = q$, $q_2 = q^3$. Furthermore, we have $p^* = q^6$. Note that the Dynkin diagram for $\g_0$ should be reversed from the one in \eqref{eq: G_2}.

Recall that the $U_q(\g_0)$-decompositions of $V(\varpi_1)$ and $V(\varpi_2)$  are given as follows:
\begin{align} \label{eq: classical decomp D43}
V(\varpi_1) \iso V(\clfw_1) \oplus V(0) \qtext{ and } V(\varpi_2) \iso V(\clfw_2) \oplus V(\clfw_1)^{\oplus 2}
\oplus V(0)
\end{align}
Let us recall the $U_q'(D_4^{(3)})$-modules structure of $V(\varpi_1)$ whose crystal graph can be depicted as follows:
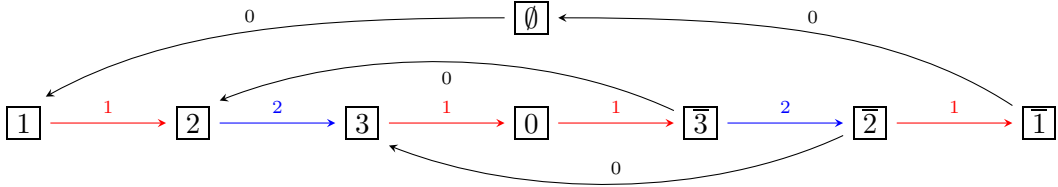
\[
\begin{tikzpicture}[>=stealth,baseline=-4,yscale=0.7,xscale=1.5]
\node (1) at (0,0) {$\young(1)$};
\node (2) at (1.5,0) {$\young(2)$};
\node (3) at (3,0) {$\young(3)$};
\node (0) at (4.5,0) {$\young(0)$};
\node (b3) at (6,0) {$\young(\oTh)$};
\node (b2) at (7.5,0) {$\young(\oT)$};
\node (b1) at (9,0) {$\young(\oO)$};
\node (ep) at (4.5,2) {$\young(\emptyset)$};
\path[->,font=\tiny]
 (1) edge[red] node[above]{$1$} (2)
 (2) edge[blue] node[above]{$2$} (3)
 (3) edge[red] node[above]{$1$} (0)
 (0) edge[red] node[above]{$1$} (b3)
 (b3) edge[blue] node[above]{$2$} (b2)
 (b2) edge[red] node[above]{$1$} (b1)
 (b1) edge[out=130,in=0] node[above]{$0$} (ep)
 (ep) edge[out=180,in=50] node[above]{$0$} (1);
\path[->,font=\tiny]
 (b3) edge[out=140,in=40] node[below]{$0$} (2);
\path[->,font=\tiny]
 (b2) edge[out=220,in=320] node[above]{$0$} (3);
\end{tikzpicture}
\]
where
\begin{align*}
&e_0 {\young(1)}={\young(\emptyset)}+\frac{1}{[2]_1}{\young(0)}, \
 e_0 {\young(2)}={\young(\oTh)}, \
 e_0 {\young(3)}={\young(\oT)}, \
 e_0 {\young(0)}={\young(\oO)}, \
 e_0 {\young(\emptyset)}=\frac{[3]_1}{[2]_1}{\young(\oO)}, \\
&f_0 {\young(\oO)}={\young(\emptyset)}+\frac{1}{[2]_1}{\young(0)}, \
 f_0 {\young(\oT)}={\young(3)}, \
 f_0 {\young(\oTh)}={\young(2)}, \
 f_0 {\young(0)}={\young(1)}, \
 f_0 {\young(\emptyset)}=\frac{[3]_1}{[2]_1}{\young(1)}, \\
&e_1 {\young(2)}={\young(1)}, \
 e_1 {\young(0)}=[2]_1{\young(3)}, \
 e_1 {\young(\oTh)}={\young(0)},  \
 e_1 {\young(\oO)}={\young(\oT)}, \\
&f_1 {\young(\oT)}={\young(\oO)}, \
 f_1 {\young(0)}=[2]{\young(\oTh)}, \
 f_1 {\young(3)}={\young(0)}, \
 f_1 {\young(1)}={\young(2)}, \\
&e_2 {\young(3)}={\young(2)}, \
 e_2 {\young(\oT)}={\young(\oTh)}, \
 f_2 {\young(\oTh)}={\young(\oT)}, \
 f_2 {\young(2)}={\young(3)}.
\end{align*}

In~\cite{KMOY07}, the existence of following Dorey's type homomorphism is proved:
\begin{align}
\label{eq: p11 D43}
p_{1,1} \colon V(\varpi_1)_{-q^{-1}} \otimes V(\varpi_1)_{-q} \twoheadrightarrow V(\varpi_2).
\end{align}
They also computed $d_{1,1}(z)$ also:
\begin{align}
\label{eq:d11 D 43}
d_{1,1}(z)=(z-q^{2})(z-q^{6})(z-\omega q^{4})(z-\omega^2 q^{4}),
\end{align}
where $\omega$ is the primitive third root of unity.

By tensoring the left dual $V(\varpi_1)_{-q^{5}}$ of $V(\varpi_1)_{-q^{-1}}$ on \eqref{eq: p11 D43}, we have
\begin{align}\label{eq: d12 partial}
p'_{2,1} \colon V(\varpi_2) \otimes V(\varpi_1)_{-q^{5}} \twoheadrightarrow V(\varpi_1)_{-q}.
\end{align}
Hence we can notice that $-q^5$ is a root of $d_{1,2}(z)$.

By \eqref{eq:d11 D 43}, we have
\begin{align}\label{eq: a11 D 43}
a_{1,1}(z)=\dfrac{[8][4][12][0]}{[2][10][6]^2} \times \dfrac{\langle 10 \rangle\langle 2 \rangle}{\langle 4 \rangle \langle 8 \rangle},
\end{align}
where
\[
[a]\seteq ((-q)^az;q^{12}), \ \ \langle a \rangle\seteq \prod_{s=0}^{\infty}(1+(-q)^{12s+a}z+(-q)^{24s+2a}z^2), \ \ \{ a \}\seteq [a]\langle a \rangle.
\]

\begin{lemma} \label{lem: a12a22}
$a_{1,2}(z)$ and $a_{2,2}(z)$ are given as follows:
\begin{align*}
a_{1,2}(z)= \dfrac{ \{ 11 \}\{ 1 \}  }{ \{ 5 \}\{ 7 \} },
\qquad\qquad
a_{2,2}(z)= \dfrac{ \{ 10 \}\{ 2 \} \{ 12 \} \{ 0 \}  }{ \{ 4 \}\{ 8 \} \{ 6 \}^2 }.
\end{align*}
\end{lemma}

\begin{proof}
Applying the diagrams in the proof of Theorem~\ref{thm: G21 a and d}, we have
\begin{align} \label{eq: mapsto}
a_{1,1}(-q^{-1}z) a_{1,1}((-q)^{1}z){\young(1)} \otimes w \longmapsto  a_{1,2}(z) (u_{\varpi_2})_z \otimes {\young(1)}\ ,
\end{align}
where
\begin{align} \label{eq: direct D43}
\Rnorm{1,1}(-q z)( \young(1) \otimes {\young(2)}_z ) = \dfrac{1-q^2}{-qz-q^2} \ {\young(1)}_z \otimes {\young(2)} +
\dfrac{q(1-(-q)z)}{-qz-q^2} \ {\young(2)}_z \otimes {\young(1)}\ .
\end{align}

Since $\young(1) \otimes \young(1)$ vanishes under the map $p_{1,1}$, \eqref{eq: mapsto} indicates that
\begin{align*}
a_{1,2}(z) &= a_{1,1}(-q^{-1}z) \ a_{1,1}((-q)^{1}z)
 \dfrac{q(-qz -1 )}{-qz -q^2} \allowdisplaybreaks \\
&  \equiv a_{1,1}(-q^{-1}z) \ a_{1,1}((-q)^{l-1}z)
\dfrac{[1]}{[13]}\dfrac{[11]}{[-1]} \allowdisplaybreaks\\
& = \dfrac{[7][3][11][-1]}{[1][9][5]^2} \dfrac{\langle 9 \rangle\langle 1 \rangle}{\langle 3 \rangle \langle 7 \rangle}
\dfrac{[9][5][13][1]}{[3][11][7]^2} \dfrac{\langle 11 \rangle\langle 3 \rangle}{\langle 5 \rangle \langle 9 \rangle}
\dfrac{[1]}{[13]}\dfrac{[11]}{[-1]} \allowdisplaybreaks\\
& = \dfrac{[1][11]}{[7][5]}\dfrac{\langle 11 \rangle\langle 1 \rangle}{\langle 5 \rangle\langle 7 \rangle}
= \dfrac{ \{11\}\{1\}  }{ \{7\}\{5\} } .
\end{align*}

Recall that $u_{\varpi_2}$ denotes the dominant extremal vector of $V(\varpi_2)$. By direct calculation, one can show that
\[
f_1 (u_{\varpi_2} \otimes \young(1)) =  u_{\varpi_2} \otimes \young(2) \quad\text{ and }\quad f_1 ( \young(1) \otimes u_{\varpi_2})  =  \young(2) \otimes u_{\varpi_2}.
\]
Since $\Rnorm{2,1}$ is a $U_q'(\g)$-homomorphism and sends $u_{\varpi_2} \otimes \young(1)$ to $\young(1) \otimes u_{\varpi_2}$, we have
\[
\Rnorm{2,1}(z)(v_{\varpi_2} \otimes \young(2)) =  \young(2) \otimes  u_{\varpi_2}.
\]

Using the commutative diagram
\begin{align*}
\xymatrix@R=5ex@C=9ex{
V(\varpi_2) \otimes V(\varpi_1)_{-q^{-1}z} \otimes V(\varpi_1)_{-q z}
\ar[rr]^{\qquad \qquad \quad V(\varpi_2)\otimes p_{1,1}} \ar[d]_{\Runiv{2,1}(-q^{-1}z) \otimes V(\varpi_2)_{-qz}} &&
V(\varpi_{2}) \otimes V(\varpi_{2})_{z} \ar[dd]_{\Runiv{2,2}(z)} \\
V(\varpi_1)_{-q^{-1}z}  \otimes V(\varpi_{2}) \otimes V(\varpi_1)_{(-q)z}
\ar[d]_{V(\varpi_1)_{(-q)^{-1}z} \otimes \Runiv{2,1}(-qz)} &\\
V(\varpi_1)_{-q^{-1}z} \otimes V(\varpi_1)_{-qz}  \otimes V(\varpi_{2})\ar[rr]^{\qquad \qquad \qquad p_{1,1}\otimes V(\varpi_2) } &&
V(\varpi_{2})_{z} \otimes V(\varpi_{2}),
}
\end{align*}
we have
\begin{align*}
\xymatrix@R=2ex@C=9ex{
u_{\varpi_2} \otimes  {\young(1)} \otimes  {\young(2)} \ar[rr]  \ar[d] &&
{\young(1)} \otimes  (u_{\varpi_2})_z \ar[dd] \\
a_{2,1}(-q^{-1}z){\young(1)}\otimes  u_{\varpi_2} \otimes  {\young(2)}
\ar[d] &\\
a_{2,1}(-qz)a_{2,1}(-q^{-1}z){\young(1)}\otimes  w \ar[rr] &&
a_{2,2}(z)(u_{\varpi_2})_z \otimes  {\young(1)}
}
\end{align*}
where $w=\Rnorm{2,1}((-q)^{1}z)( u_{\varpi_2} \otimes \young(2) ) =  \young(2) \otimes  u_{\varpi_2}$. Thus we have
\[
a_{2,2}(z) = a_{2,1}(-q^{-1}z) \ a_{2,1}((-q)^{1}z) \equiv \dfrac{ \{10\}\{0\}  }{ \{4\}\{6\} } \dfrac{ \{12\}\{2\}  }{ \{6\}\{8\} }
\equiv \dfrac{ \{10\}\{2\}  }{ \{4\}\{8\} } \dfrac{ \{12\}\{0\}  }{ \{6\}^2}.
\]
\end{proof}

Note that, by~\cite[Eq.~(1.7)]{AK97}, we have a $U_q'(\g)$ isomorphism
\[
V(\varpi_2)_x \iso V(\varpi_2)_y
\]
for $x,y \in \ko^\times$ such that $x^3=y^3$. In~\cite[page 39, arXiv version]{Her10}, it is shown that $q^2$ is a root of $d_{2,2}(z)$.
Then the following theorem can be proved by applying the same argument of~\cite[Theorem 5.8]{Oh14R}.

\begin{theorem}
The denominators of $U_q'(\g)$ are given as follows:
\begin{subequations}
\label{eq: denom}
\begin{align}
d^{D_4^{(3)}}_{1,1}(z) & = (z-q^{2})(z-q^{6})(z-\omega q^{4})(z-\omega^2 q^{4}), \\
d^{D_4^{(3)}}_{1,2}(z) & = (z^3+q^9)(z^3+q^{15}),\\
d^{D_4^{(3)}}_{2,2}(z) & = (z^3-q^6)(z^3-q^{12})^2(z^3-q^{18}).
\end{align}
\end{subequations}
\end{theorem}

\begin{corollary} \label{cor: D41 D43}
The denominators of $U_q'(\g)$ can be read from any $\Gamma_Q$ of type $D_4$:
\begin{align*}
d^{D_4^{(3)}}_{1,1}(z) & = d^{D_4^{(1)}}_{1,1}(z)d^{D_4^{(1)}}_{1,3}(\omega z)d^{D_4^{(1)}}_{1,4}(\omega^2 z) = D^{[Q]}_{1,1}(z;-q)D^{[Q]}_{1,3}(\omega z;-q)D^{[Q]}_{1,4}(\omega^2 z;-q) \times (z-p^*), \\
d^{D_4^{(3)}}_{2,1}(z) & = d^{D_4^{(1)}}_{2,1}(z)d^{D_4^{(1)}}_{2,3}(\omega z)d^{D_4^{(1)}}_{2,4}(\omega^2 z) = D^{[Q]}_{2,1}(z;-q)D^{[Q]}_{2,3}(\omega z;-q)D^{[Q]}_{2,4}(\omega^2 z;-q), \\
d^{D_4^{(3)}}_{2,2}(z) & = d^{D_4^{(1)}}_{2,2}(z)d^{D_4^{(1)}}_{2,2}(\omega z)d^{D_4^{(1)}}_{2,2}(\omega^2 z) = D^{[Q]}_{2,2}(z;-q)D^{[Q]}_{2,2}(\omega z;-q)D^{[Q]}_{2,2}(\omega^2 z;-q) \times (z-p^*).
\end{align*}
\end{corollary}

\subsection{\texorpdfstring{$U_q'(E^{(2)}_{6})$ and $U_q'(F^{(1)}_{4})$}{Uq'(E6(2)) and Uq'(F4(1))}}

We first use results obtained from our \textsc{SageMath} implementation by using the $U_q'(\g)$-module structure of $V(\varpi_i)$ for various indices $i$ given by Proposition~\ref{prop:adjoint_repr}.
Then we shall compute the remained denominator formulas almost completely, and show that the formulas can be read from any folded AR quivers of type $E_6$.

\subsubsection{$U_q'(E^{(2)}_{6})$}

For type $E_6^{(2)}$, we have $p^*=-q^{12}$. Moreover, $\g_0$ is of type $F_4$, but the Dynkin diagram should be reversed from the one in \eqref{eq: F_4}.

The $U_q(\g_0)$-decompositions of  $V(\varpi_1)$ and $V(\varpi_4)$ are given by
\[
V(\varpi_1) \iso V(\clfw_1) \oplus V(0), \qquad V(\varpi_4) \iso V(\clfw_4) \oplus V(\clfw_1) \oplus V(0).
\]
For $V(\varpi_1)$, we use the $U_q'(\g)$-module structure given by Proposition~\ref{prop:adjoint_repr}.

The $U_q(\g_0)$-decomposition of $V(\varpi_1) \otimes V(\varpi_1)$ is given by (see also~\cite{KM94})
\begin{equation}
\label{eq: class decomp E6(2)}
V(\varpi_1) \otimes V(\varpi_1)  \iso V(2\clfw_1) \oplus V(\clfw_2) \oplus V(\clfw_4) \oplus V(\clfw_1)^{\oplus 3} \oplus V(0)^{\oplus 2}
\end{equation}
whose $U_q(\g_0)$-highest weight vectors can be labeled as follows:
\[
u_{2\clfw_1},
\qquad
u_{\clfw_2},
\qquad
u_{\clfw_4},
\qquad
u^{(t)}_{\clfw_1} \ (t=1,2,3),
\qquad
u^{(s)}_{0} \ (s=1,2).
\]
We can compute the following in $V(\varpi_1)_x \otimes V(\varpi_1)_y$, where we are considering an extension of scalars to $\ko(x,y)$:
\begin{subequations}
\label{eq:paths_maximal E62}
\begin{align}
e_2 e_3 e_1 e_2 e_* u_{\clfw_2} & = \frac{(q^4 +1)^2 [2]^{10} [3]^3 (y - q^2 x)}{q^5 (xy)^3} u_{2\clfw_1},
\allowdisplaybreaks\\
e_4 e_3 (e_2 e_1)^2 e_3 e_4 e_0 e_* u_{\clfw_4} & = \frac{(q^4 + 1)^2 [2]^{11} [3]^3 (q^6 x + y) (y - q^2 x)}{q^6 (xy)^4} u_{2\clfw_1},
\allowdisplaybreaks\\
\begin{bmatrix}
e_0 u^{(1)}_{\clfw_1} & e_0 u^{(2)}_{\clfw_1} & e_0 u^{(3)}_{\clfw_1} \\
e_{\times} u^{(1)}_{\clfw_1} & e_{\times} u^{(2)}_{\clfw_1} & e_{\times} u^{(3)}_{\clfw_1} \\
e_{\bullet} u^{(1)}_{\clfw_1} & e_{\bullet} u^{(2)}_{\clfw_1} & e_{\bullet} u^{(3)}_{\clfw_1} \\
\end{bmatrix}
& =
\begin{bmatrix}
\frac{q^{-2} [2] y}{xy} & \frac{[2] x}{xy} & 0 \\
\frac{[2]^3 [3] x}{(xy)^2} & \frac{[2]^3 [3] y}{q^2 (xy)^2} & \frac{\eta [2]^5 [3] (q^2 x - y)}{(xy)^2} \\
\frac{\zeta (q^2 x + 2x + y)}{x^4 y^3}
  & \frac{\zeta (q^2 x + 2q^2 y + y)}{x^3 y^4}
  & \frac{\zeta [2] (q^12 x - y) (q^2 x - y)}{(xy)^4}
\end{bmatrix} u_{2\clfw_1},
\allowdisplaybreaks\\
\begin{bmatrix}
e_0^2 u^{(1)}_0 & e_0^2 u^{(2)}_0 \\
e_0 e_{\times} u^{(1)}_0 & e_0 e_{\times} u^{(2)}_0
\end{bmatrix}
& = \begin{bmatrix}
\frac{q^{19} [2] (q^3 x^2 + [2] y^{2})}{(xy)^2} & \frac{[2]^3}{qxy} \\
\xi &
\frac{[2]^4 [3] (q^2 x^2 + y^2)}{q^2 (xy)^3}
\end{bmatrix}
u_{2\clfw_1},
\end{align}
\end{subequations}
where
\begin{align*}
e_* & = e_2 e_4 e_3 e_1 e_3 e_4 e_2^2 e_1 e_3 e_2^2 e_1 e_0 e_3 e_2^2 e_3 e_1^2 e_4^2 e_3 e_2^2 e_3 e_2 e_3 e_2 e_4 e_1^2 e_0 e_2 e_3 e_2 e_1 e_0^3, \allowdisplaybreaks\\
e_{\bullet} & = e_1 e_0 e_2 e_3 e_1 e_2 e_*, \qquad \qquad
e_{\times}  =  e_1 e_2 e_3 e_4 e_2 e_3 e_1 e_2^2 e_1 e_3 e_4 e_2 e_3 e_2 e_1 e_0^3, \allowdisplaybreaks\\
\zeta & = q^{-6} [2]^{10} [3]^3 (q^4 + 1)^2, \qquad
\eta = q^6 - 2q^4 + 2q^2 - 1, \allowdisplaybreaks\\
\xi & = \frac{[2]^4 [3]  \Bigl( q^{14} \eta x^2 + \bigl( [3][4](q^{15} + q^5) - (4q^{18} + 6q^{14} + 3q^{10} + 6q^6 + 4q^2) \bigr)xy - \eta y^2 \Bigr)}{q (xy)^3}.
\end{align*}

By $U_q'(\g)$-linearity and~\eqref{eq: class decomp E6(2)}, we have
\begin{gather*}
\rmat{1,1}(x,y)(u_{2\clfw_1})=a^{2\clfw_1}u_{2\clfw_1},
\quad
\rmat{1,1}(x,y)(u_{\clfw_2})=a^{\clfw_2}u_{\clfw_2},
\quad
\rmat{1,1}(x,y)(u_{\clfw_4})=a^{\clfw_4}u_{\clfw_4},\\
\rmat{1,1}(x,y)(u^{(i)}_{\clfw_1})=\sum_{j=1}^{3}a^{\clfw_1}_{ij}u^{(j)}_{\clfw_1} \ (i=1,2,3),
\qquad
\rmat{1,1}(x,y)(u^{(i)}_{0})= \sum_{j=1}^2 a^{0}_{ij}u^{(j)}_{0} \ (i=1,2).
\end{gather*}

By using~\eqref{eq:paths_maximal E62} and that $\rmat{1,1}(x,y)$ is a $U_q'(\g)$-isomorphism, we obtain the following.

\begin{proposition}\label{prop: E62 a}
Put $z=x^{-1}y$ and denote $\rmat{1,1}(z)=\rmat{1,1}(x,y)$. Then we have
\begin{align*}
a^{2\clfw_1}(z) & = (z+q^{12})(z-q^8)(z+q^6)(z-q^2), \allowdisplaybreaks \\
a^{\clfw_2}(z) & = (z+q^{12})(z-q^8)(z+q^6)(q^2z-1), \allowdisplaybreaks\\
a^{\clfw_4}(z) & = (z+q^{12})(z-q^8)(q^6z+1)(q^2z-1), \allowdisplaybreaks\\
\begin{bmatrix}
a_{11}^{\clfw_1}(z) \\ a_{21}^{\clfw_1}(z) \\ a_{31}^{\clfw_1}(z)
\end{bmatrix} & =
\begin{bmatrix}
  -{\left(q^{12} - q^{10} + q^{8} - q^{4} + q^{2} - z\right)} {\left(q^{12} + z\right)} {\left(q^{4} - 1\right)} z \\
  q^2 {\left( - z^{2} + (q^{8} - q^{6} + q^{2} - 1) z + q^{12}\right)} {\left(q^{12} + z\right)} {\left(z - 1\right)} \\
  q^{3} (q^4 - 1) (q^6 + 1) (q^{12} + 1) {\left(q^{12} + z\right)} {\left(q^{2} z - 1\right)} {\left(z - 1\right)}
\end{bmatrix},
\allowdisplaybreaks \\
\begin{bmatrix}
a_{12}^{\clfw_1}(z) \\ a_{22}^{\clfw_1}(z) \\ a_{32}^{\clfw_1}(z)
\end{bmatrix} & =
\begin{bmatrix}
  q^2 {\left(- z^{2} + (q^{12} - q^{10} + q^{6} - q^{4}) z  + q^{12} \right)} {\left(q^{12} + z\right)} {\left(z - 1\right)} \\
  -{\left(q^{12} - q^{10} z + q^{8} z - q^{4} z + q^{2} z - z\right)} {\left(q^{12} + z\right)} {\left(q^{4} - 1\right)} z \\
  -q (q^4 - 1) (q^6 + 1) (q^{12} + 1) {\left(q^{12} + z\right)} {\left(q^{2} z - 1\right)} {\left(z - 1\right)} z
\end{bmatrix},
\allowdisplaybreaks\\
\begin{bmatrix}
a_{13}^{\clfw_1}(z) \\ a_{23}^{\clfw_1}(z) \\ a_{33}^{\clfw_1}(z)
\end{bmatrix} & =
\begin{bmatrix}
  q (q^2 - 1) {\left(q^{12} + z\right)} {\left(z - 1\right)} z \\
  q (q^2 - 1) {\left(q^{12} + z\right)} {\left(z - 1\right)} z \\
  -{(q^{12} z^{2} + (q^2 - 1) (q^{12} + q^6 + 1) z - q^{2})} {\left(q^{12} + z\right)} {\left(q^{2} z - 1\right)}
\end{bmatrix},
\allowdisplaybreaks\\
a_{11}^{0}(z) & = -q^{26} z^4 + q^{16} (q^4 - 1) (q^6 + 1) z^3 + (-q^{28} + q^{26} + q^{20} - q^{18} + 2 q^{14} - q^{10} + q^{8} + q^{2} - 1) z^2 \\
  & \hspace{11pt} - q^2 (q^4 - 1) (q^6 + 1)  z - q^{2}, \allowdisplaybreaks\\
a_{21}^{0}(z) & = (q^2 + 1)^2 (q^2 - 1) q {\left(z + 1\right)} {\left(z - 1\right)} z, \allowdisplaybreaks\\
a_{12}^{0}(z) & = {\left(q^{2} - 1\right)} q {\left(z + 1\right)} {\left(z - 1\right)} \allowdisplaybreaks\\
  & \hspace{20pt} \times \Bigl( -q^{14} (q^4 - q^2 + 1) (q^{12} + 1) (z^2 + 1) \allowdisplaybreaks\\
  & \hspace{40pt} + (q^{20} - 2q^{18} + 3q^{16} - 3q^{14} + 2q^{12} - q^{10} + 2q^8 - 3q^6 + 3q^4 - 2q^2 + 1) q^{12} [13] z \Bigr), \allowdisplaybreaks \\
a_{22}^{0} & = a_{11}^{0}\bigr|_{z \to z^{-1}} \cdot z^4. 
\end{align*}
\end{proposition}

Thus we have a denominator formulas
\begin{equation}
\label{eq:d11 E6(2)}
d_{1,1}(z) = (z+q^{12})(z-q^8)(z+q^6)(z-q^2).
\end{equation}

\begin{remark} \label{rem: kernel E6(2)}
By substituting $z$ with a root of $d_{1,1}(z)$ into $a(z)$'s in Proposition \ref{prop: E62 a}, we can observe the following isomorphisms $U_q(F_4)$-modules:
\begin{align*}
\Image(\rmat{1,1}((-q)^{-3},q^{3})) & \iso V(\clfw_4) \oplus V(\clfw_1) \oplus V(0), \\
\Image(\rmat{1,1}((-q)^{-4},(-q)^{4})) & \iso  V(\clfw_1) \oplus V(0).
\end{align*}
\end{remark}

\begin{proposition}
\label{prop:Dorey E62 begin}
We have
\begin{subequations}
\begin{align}
V(\varpi_1)_{-q^{-1}}\otimes V(\varpi_1)_{-q^{1}} & \twoheadrightarrow V(\varpi_2), \label{eq: E62 112} \allowdisplaybreaks \\
V(\varpi_4)_{-q^{-1}}\otimes V(\varpi_4)_{-q^{1}} & \twoheadrightarrow V(\varpi_3) \iso V(\varpi_3)_{-1}, \label{eq: E62 443} \allowdisplaybreaks\\
V(\varpi_1)_{-q^{-3}}\otimes V(\varpi_1)_{q^{3}} & \twoheadrightarrow V(\varpi_4)_{\sqrt{-1}}\iso  V(\varpi_4)_{-\sqrt{-1}}, \label{eq: E62 114} \allowdisplaybreaks \\
V(\varpi_1)_{(-q)^{-4}}\otimes V(\varpi_1)_{(-q)^{4}} & \twoheadrightarrow V(\varpi_1)_{-1}. \label{eq: E62 111}
\end{align}
\end{subequations}
\end{proposition}

\begin{proof}
In~\cite[page 39, arXiv version]{Her10}, the Dorey's type morphisms~\eqref{eq: E62 112} and~\eqref{eq: E62 443} for $U_q'(E_6^{(2)})$ are given.
Also, \eqref{eq: E62 112} and~\eqref{eq: E62 443} are known as the twisted $T$-system of type $E_6^{(2)}$.

As in Proposition~\ref{prop: Dorey D43 22}, Remark~\ref{rem: kernel E6(2)} implies that~\eqref{eq: E62 114} has to be
a Dorey's type homomorphism of the form
\[
V(\varpi_1)_{(-q)^{-3}} \otimes V(\varpi_1)_{q^{3}} \twoheadrightarrow V(\varpi_4)_a
\]
for $a^2=-1$. Thus our assertion follows from the fact that
\[
V(\varpi_4)_x \iso  V(\varpi_4)_y  \quad\text{ for } x,y \in \ko^\times \text{ such that } x^2=y^2.
\]

Similarly, \eqref{eq: bar struture} implies that \eqref{eq: E62 111} has to be
a Dorey's type homomorphism of the form
\[
V(\varpi_1)_{(-q)^{-4}}\otimes V(\varpi_1)_{(-q)^{4}} \twoheadrightarrow V(\varpi_1)_{a}
\]
for $a^2=1$. If $a=1$, then $d_{1,1}(z)$ has a root $-q^{8}$, by taking dual:
\[
V(\varpi_1)_{(-q)^{4}} \hookrightarrow V(\varpi_1)_{1} \otimes  V(\varpi_1)_{-q^{-8}}.
\]
However, this can not happen by \eqref{eq:d11 E6(2)}.
\end{proof}

When we compute the remained denominator formulas, we need to know the image of some vectors under the $\Rnorm{i,j}$. Here we collect some computations:

\begin{subequations}
\label{eq: E6(2)R_parts}
\begin{align}
\Rnorm{1,1}(u_{\varpi_1} \otimes f_1 u_{\varpi_1}) & = \dfrac{(1-q^2)}{(z-q^2)}(u_{\varpi_1} \otimes f_1 u_{\varpi_1})+
\dfrac{q(z - 1)}{(z - q^2)}(f_1 u_{\varpi_1} \otimes u_{\varpi_1}), \label{eq: E6(2)R112}\allowdisplaybreaks \\
\Rnorm{1,1}(u_{\varpi_1} \otimes f_* u_{\varpi_1}) & =  \dfrac{(q^4 + z)q^2(z - 1)}{(q^6 + z)(z - q^2)}
 f_* u_{\varpi_1} \otimes u_{\varpi_1} + \text{ other terms },  \label{eq: E6(2)R114}\allowdisplaybreaks\\
\Rnorm{4,4}(u_{\varpi_4} \otimes f_4 u_{\varpi_4}) & = \dfrac{(1-q^4)}{(z^2-q^4)}(u_{\varpi_4}\otimes f_4 u_{\varpi_4})+
\dfrac{q^2(z^2 - 1)}{(z^2 - q^4)}(f_4 u_{\varpi_4} \otimes u_{\varpi_4}).  \label{eq: E6(2)R443}
\end{align}
\end{subequations}
Here $f_*$ denotes $f_1 f_2 f_3 f_2 f_1$.

By Lemma~\ref{Lem: aij and dij} and~\eqref{eq:d11 E6(2)}, we have
\[
a_{1,1}(z) \equiv \dfrac{\lr{10}\lr{14}[6][18]\lr{4}\lr{20}[24][0]}{[2][22]\lr{6}\lr{18}[8][16]\lr{12}^2},
\]
where
\[
[a] \seteq ((-q)^az;q^{24})_{\infty}, \quad \lr{a} \seteq (-(-q)^az;q^{24})_{\infty}, \quad \lrt{a} \seteq [a]\lr{a}.
\]
Also, using~\eqref{eq: E6(2)R112} and $a_{1,1}(z)$, we can obtain
\begin{align*}
a_{1,2}(z) & \equiv \dfrac{\lr{9}\lr{15}[5][19]\lr{3}\lr{21}[1][23]}{[3][21]\lr{7}\lr{17}[9][15]\lr{11}\lr{13}}, \\
a_{2,2}(z) & =\dfrac{[6][18]\lr{4}\lr{20}\lr{2}\lr{22}[24][0]}{\lr{6}\lr{18}[8][16][10][14]\lr{12}^2}, \label{eq: a12 a22 E6(2)}\\
d_{1,2}(z) & =(z+q^3)(z-q^5)(z-q^7)(z+q^7)(z+q^{9})(z-q^{11}), 
\end{align*}
by applying the same arguments in Section~\ref{sec:G2_D43}.

\begin{proposition}
\label{prop: d22 E6(2)}
We have
\[
d_{2,2}(z)= (z-q^2)(z-q^4)(z-q^6)(z-q^8)^2(z-q^{10})(z+q^4)^\epsilon(z+q^6)^2(z+q^8)^\epsilon(z+q^{10})(z+q^{12})
\]
for some $\epsilon  \in \{ 0,1 \}.$
\end{proposition}

\begin{proof}
Applying Lemma~\ref{lem:dvw avw} to \eqref{eq: E62 112}, we have
\begin{align*}
\dfrac{d_{1,2}(-q^{-1}z)d_{1,2}(-qz)}{d_{2,2}(z)} \dfrac{a_{2,2}(z)}{a_{1,2}(-q^{-1}z)a_{1,2}(-qz)} \in \ko[z^{\pm 1}],
\end{align*}
and hence
\[
\dfrac{(z-q^2)(z-q^4)(z-q^6)(z-q^8)^2(z-q^{10})(z+q^4)(z+q^6)^2(z+q^8)(z+q^{10})(z+q^{12})}{d_{2,2}(z)}
\]
is in $\ko[z^{\pm 1}]$.

Let us consider the surjective homomorphism
\begin{equation}
\label{eq:p121}
V^{(2)}(\varpi_1)_{-q^{-10}} \otimes V^{(2)}(\varpi_2)_{-q} \to  V^{(2)}(\varpi_1),
\end{equation}
which can be obtained from \eqref{eq: E62 112}.

Applying Lemma~\ref{lem:dvw avw} to \eqref{eq:p121}, we have
\begin{align} \label{eq: step2 d22}
\dfrac{d_{1,2}(-q^{-10}z)d_{2,2}(-qz)}{d_{1,2}(z)} \dfrac{a_{1,2}(z)}{a_{1,2}(-q^{-10}z)a_{2,2}(-qz)} \in \ko[z^{\pm 1}].
\end{align}
Note that
\[
\dfrac{a_{1,2}(z)}{a_{1,2}(-q^{-10}z)a_{2,2}(-qz)} \equiv \dfrac{(z-q^7)(z+q^{3})(z-q^{1})(z+q^{-1})}{(z-q^{5})(z+q^{7})(z-q^{9})(z+q^{1})}.
\]
Hence \eqref{eq: step2 d22} can be expressed as follows:
\[
\dfrac{(z+q^{12})(z+q^{18})(z-q^{16})(z-q^{26})(z-q^{1})(z+q^{-1})d_{2,2}(-qz)}{(z-q^{5})^2(z+q^{7})^2(z-q^{9})(z+q^{1})(z+q^{9})(z-q^{11})}\in \ko[z^{\pm 1}].
\]
Thus, we have $(z+q^{6})^2(z-q^{8})^2(z+q^{10})(z-q^{2})(z-q^{10})(z+q^{12})|d_{2,2}(z)$. Since $(z-q^{2}) | d_{2,2}(z)$ by~\cite[page 39; arxiv version]{Her10},~\eqref{eq: aij and dij} implies our assertion.
\end{proof}
We will refine the denominator
formula $d_{2,2}(z)$ by using generalized Schur-Weyl duality. More precisely, we will prove that $\epsilon=1$.

Now we shall compute $a_{1,4}(\sqrt{-1}z)$, which needs to be treated more carefully:

\begin{lemma}
We have
\[
a_{1,4}(\sqrt{-1}z)=\dfrac{\lrt{7}\lrt{17}\lrt{3}\lrt{21}}{\lrt{5}\lrt{19}\lrt{9}\lrt{15}}.
\]
\end{lemma}

\begin{proof}
Note that the $U_q(\g_0)$ highest weight vector $u_{\clfw_4}$ in $V(\varpi_1) \otimes V(\varpi_1)$ contains a summand
$u_{\varpi_1} \otimes f_1f_2f_3f_2f_1u_{\varpi_1}$. Now we set $f_* = f_1 f_2 f_3 f_2 f_1$ for brevity. By~\eqref{eq: Runiv property},
we have the following commutative diagrams:
\begin{align*}
\xymatrix@R=5ex@C=9ex{
V(\varpi_1) \otimes V(\varpi_1)_{-q^{-3}z} \otimes V(\varpi_1)_{q^{3} z}
\ar[rr]^{\qquad \qquad \quad V(\varpi_1)\otimes \eqref{eq: E62 114}} \ar[d]_{\Runiv{1,1}(-q^{-3}z) \otimes V(\varpi_1)_{q^{3}z}} &&
V(\varpi_{1}) \otimes V(\varpi_{4})_{\sqrt{-1}z} \ar[dd]_{\Runiv{1,4}(iz)} \\
V(\varpi_1)_{-q^{-3}z}  \otimes V(\varpi_{1}) \otimes V(\varpi_1)_{q^{3}z}
\ar[d]_{V(\varpi_1)_{(-q)^{-3}z} \otimes \Runiv{1,1}(q^3 z)} &\\
V(\varpi_1)_{-q^{-3}z} \otimes V(\varpi_1)_{q^{3}z}  \otimes V(\varpi_{1})\ar[rr]^{\qquad \qquad \qquad \eqref{eq: E62 114}\otimes V(\varpi_1) } &&
V(\varpi_{4})_{\sqrt{-1}z} \otimes V(\varpi_{1}).
}
\end{align*}
sending
\begin{align*}
\xymatrix@R=2ex@C=9ex{
u_{\varpi_1} \otimes  u_{\varpi_1} \otimes  f_*u_{\varpi_1} \ar[rr]  \ar[d] &&
u_{\varpi_1} \otimes  (u_{\varpi_4})_{\sqrt{-1}} \ar[dd] \\
a_{1,1}(-q^{-3}z)u_{\varpi_1} \otimes  u_{\varpi_1}  \otimes  f_*u_{\varpi_1}
\ar[d] &\\
a_{1,1}(q^3z)a_{1,1}(-q^{-3}z)u_{\varpi_1} \otimes  w \ar[rr] &&
a_{1,4}(\sqrt{-1}z)(u_{\varpi_4})_{\sqrt{-1}}  \otimes  u_{\varpi_1}
}
\end{align*}
where $w$ is the image of $u_{\varpi_1}  \otimes  f_*u_{\varpi_1}$ under $\Rnorm{1,1}(q^3z)$. Since the remained terms in
\eqref{eq: E6(2)R114} vanish under \eqref{eq: E62 114}, we have
\begin{align*}
a_{1,4}(\sqrt{-1}z) \equiv a_{1,1}(q^3z)a_{1,1}(-q^{-3}z)  \dfrac{[21][-1]\lr{3}\lr{25}}{[-3][23]\lr{27}\lr{1}}
=\dfrac{[7][17]\lr{7}\lr{17} [3][21]\lr{3}\lr{21}  }{[5][19] \lr{5}\lr{19} [9][15]\lr{9}\lr{15}}
\end{align*}
since
\[
\dfrac{(q^4 + z)q^2(z - 1)}{(q^6 + z)(z - q^2)} \equiv \dfrac{[21][-1]\lr{3}\lr{25}}{[-3][23]\lr{27}\lr{1}}.
\]
Hence, the claim follows.
\end{proof}

Now using~\eqref{eq: E6(2)R_parts} 
and the $a_{i,j}(z)$ values we have computed, we can compute the remaining values of $a_{i,j}(z)$:
\begin{subequations}
\begin{align}
a_{4,4}(z) & =\dfrac{\lrt{10}\lrt{14}\lrt{4}\lrt{20}\lrt{24}\lrt{0}}{\lrt{2}\lrt{22}\lrt{8}\lrt{16}\lrt{12}^2},
 & a_{2,4}(\sqrt{-1}z) & = \dfrac{\lrt{2}\lrt{22}}{\lrt{10}\lrt{14}}, \label{eq: a44 a24 E6(2)}\\
a_{3,4}(z) & \equiv \dfrac{\lrt{5}\lrt{19}\lrt{1}\lrt{23}}{\lrt{7}\lrt{17}\lrt{11}\lrt{13}},
  & a_{3,3}(z) & = \dfrac{\lrt{4}\lrt{20}\lrt{2}\lrt{22}\lrt{0}\lrt{24}}{\lrt{8}\lrt{16}\lrt{10}\lrt{14}\lrt{12}^2},  \label{eq: a34 a33 E6(2)}\\
a_{2,3}(\sqrt{-1}z) & \equiv \dfrac{\lrt{3}\lrt{21}\lrt{1}\lrt{23}}{\lrt{9}\lrt{15}\lrt{11}\lrt{13}},
  & a_{1,3}(\sqrt{-1}z) & = \dfrac{\lrt{2}\lrt{22}}{\lrt{10}\lrt{14}}.  \label{eq: a23 a13 E6(2)}
\end{align}
\end{subequations}

Also, we can compute all $d_{i,j}(z)$ with some ambiguities (as in Proposition~\ref{prop: d22 E6(2)}) will be refined later
(except $\epsilon'$ in~\eqref{eq: d33 E6(2)}):
\begin{subequations}
\begin{align}
& d_{1,4}(z) = (z^2+q^{10})(z^2+q^{18}), \ d_{4,4}(z) =(z^2-q^4)(z^2-q^{12})(z^2-q^{16})(z^2-q^{24}),  \label{eq: d14 d44 E6(2)}\\
& d_{2,4}(z) = (z^2+q^8)(z^2+q^{12})(z^2+q^{16})(z^2+q^{20}), \label{eq: d24 E6(2)} \\
&d_{3,4}(z) = (z^2-q^6)^\epsilon(z^2-q^{10})(z^2-q^{14})^2(z^2-q^{18})^\epsilon(z^2-q^{22}),\label{eq: d34 E6(2)}\\%
&d_{3,3}(z)=(z^2-q^4)(z^2-q^8)^{1+\epsilon}(z^2-q^{12})^{2+\epsilon'}(z^2-q^{16})^{2+\epsilon}(z^2-q^{20})^2(z^2-q^{24}),\label{eq: d33 E6(2)}\\
&d_{2,3}(z)= (z^2+q^6)(z^2+q^{10})^{1+\epsilon}(z^2+q^{14})^{1+\epsilon}(z^2+q^{18})^2(z^2+q^{22}), \label{eq: d23 E6(2)}\\
&d_{1,3}(z) = (z^2+q^8)(z^2+q^{12})(z^2+q^{16})(z^2+q^{20}), \label{eq: d13 E6(2)}
\end{align}
\end{subequations}
for some $\epsilon,\epsilon'  \in \{ 0,1 \}$.

\subsubsection{$U_q'(F^{(1)}_{4})$}

For type $F_4^{(1)}$, we have $q_0=q_1=q_2=q$ and $q_3=q_4=q^{1/2}$. Hence, we have $q_s = q^{1/2}$ and $p^*=q^{9}$.
The classical decompositions of $V(\varpi_1)$ and $V(\varpi_4)$ as $U_q(\g_0)$-modules are given as follows:
\begin{equation}\label{eq: class decomp F4(1)}
 V(\varpi_1) \iso V(\clfw_1) \oplus V(0) \qquad V(\varpi_4) \iso V(\clfw_4).
\end{equation}

For $V(\varpi_1)$ and $V(\varpi_4)$, we use the $U_q'(\g)$-module structure given by Proposition~\ref{prop:adjoint_repr}.
Note that for $V(\varpi_4)$, the corresponding $U_q(\g_0)$-representation $V(\clfw_4)$ corresponds to the ``little'' adjoint representation as $\clfw_4$ is the highest short root. Hence the affine adjoint representation structure still applies.

By our \textsc{SageMath} implementation, we can compute the denominator formulas $d_{1,1}(z)$, $d_{1,4}(z)$ and $d_{1,4}(z)$ since we know their module structure (see Remark~\ref{note: strategy}):
\begin{subequations}
\begin{align}
d_{1,1}(z)& =(z-q_s^{4})(z-q_s^{10})(z-q_s^{12})(z-q_s^{18}), \label{eq: d11 F4(1)} \\
d_{1,4}(z)& =(z+q_s^{8})(z+q_s^{14}), \label{eq: d14 F4(1)} \\
d_{4,4}(z)& =(z-q_s^{2})(z-q_s^{8})(z-q_s^{12})(z-q_s^{18}). \label{eq: d44 F4(1)}
\end{align}
\end{subequations}

Let us see $\rmat{4,4}(x,y)$ more concretely. The $U_q(\g_0)$-decomposition of $V(\varpi_4) \otimes V(\varpi_4)$ is given as follows:
\[
V(\varpi_4) \otimes V(\varpi_4) \iso V(2\clfw_4) \oplus V(\clfw_3) \oplus V(\clfw_1)\oplus V(\clfw_4) \oplus V(0)
\]
whose $U_q(\g_0)$-highest weight vectors can be labeled as
$
u_{2\clfw_4},
u_{\clfw_4},
u_{\clfw_3},
u_{\clfw_1},
u_{0}.
$
We can compute the following in $V(\varpi_4)_x \otimes V(\varpi_4)_y$, where we are considering an extension of scalars to $\ko(x,y)$:
\begin{subequations}
\label{eq:paths_maximal F41}
\begin{align}
e_4 e_3 e_2 e_1 e_0 e_* u_{\clfw_4} & = \frac{ [2]_3^5 [2]_4^4 [2]_2^3 [2]_1^2 [2]_0 [3]_3^2 (q_s^4 + 3 q_s^2 + 1) (q_s^{12} x - y) (q_s^2 x - y) }{q_s^4 (xy)^4} u_{2\clfw_4},
\\
e_* u_{\clfw_3} & = \frac{ [2]_3^5 [2]_4^3 [2]_2^3 [2]_1^2 [2]_0 [3]_3^2 (q_s^4 + 3 q_s^2 + 1) (y - q_s^2 x) }{q_s^3 (xy)^3} u_{2\clfw_4},
\\
e_1 e_2 e_3^2 e_2 e_1 e_0 u_{\clfw_1} & = [2]_3 \frac{y - q_s^8 x}{q_s^2 xy} u_{2\clfw_4},
\\
e_0 e_1 e_2 e_3^2 e_2 e_1 e_0 u_{0} & = [2]_3 \frac{(q_s^{18} x - y)(q_s^8 x - y)}{q_s^4 (xy)^2} u_{2\clfw_4},
\end{align}
\end{subequations}
where
\begin{align*}
e_* & =  e_3 e_4 e_2 e_1 e_3^2 e_4 e_2 e_3 e_2 e_1 e_3^2 e_2^2 e_1 e_3 e_4^2 e_0 e_3 e_4 e_2 e_3^2 e_1 e_2^2
\\ & e_1 e_3 e_4 e_3 e_4 e_2 e_3^2 e_4 e_2 e_3 e_1 e_2 e_3 e_4 e_0^2 e_1^2 e_2^2 e_3^3 e_2 e_3 e_1 e_0 e_2 e_1 e_0.
\end{align*}

By $U_q'(F_4^{(1)})$-linearity and \eqref{eq: class decomp F4(1)}, we have
\begin{align*}
\rmat{4,4}(x,y)(u_{2\clfw_4}) & = a^{2\clfw_4}u_{2\clfw_4}, &
\rmat{4,4}(x,y)(u_{\clfw_3}) & = a^{\clfw_3}u_{\clfw_3}, &
\rmat{4,4}(x,y)(u_{\clfw_1}) & = a^{\clfw_1}u_{\clfw_1}, \\
\rmat{4,4}(x,y)(u_{\clfw_4}) & =a^{\clfw_4}u_{\clfw_4}, &
\rmat{4,4}(x,y)(u_{0}) & = a^{0}u_{0}.
\end{align*}

By using~\eqref{eq:paths_maximal F41} and that $\rmat{1,1}(x,y)$ is a $U_q'(\g)$-isomorphism, we obtain the following.

\begin{proposition}
\label{prop: a computation F4(1)}
Put $z=x^{-1}y$ and denote $\rmat{4,4}(z)=\rmat{4,4}(x,y)$. Then we have
\begin{enumerate}[{\rm (1)}]
\item $a^{2\clfw_4}(z)=(z-q_s^{2})(z-q_s^{8})(z-q_s^{12})(z-q_s^{18})$,
\item $a^{\clfw_3}(z)=(q_s^{2}z-1)(z-q_s^{8})(z-q_s^{12})(z-q_s^{18})$,
\item $a^{\clfw_1}(z)=(z-q_s^{2})(q_s^{8}z-1)(z-q_s^{12})(z-q_s^{18})$,
\item $a^{\clfw_4}(z)=(q_s^{2}z-1)(z-q_s^{8})(q_s^{12}z-1)(z-q_s^{18})$,
\item $a^{0}(z)=(z-q_s^{2})(z-q_s^{8})(z-q_s^{12})(q_s^{18}z-1)$.
\end{enumerate}
\end{proposition}

\begin{remark}
By substituting $z$ with roots of $d_{4,4}(z)$ into $a(z)$'s in Proposition~\ref{prop: a computation F4(1)}, we can observe the following:
\begin{subequations}
\label{rem: kernel F4(1)}
\begin{align}
\Image(\rmat{4,4}((-q_s)^{-6},(-q_s)^{6})) & \iso  V(\clfw_4) \text{ as $U_q(F_4)$-modules}, \label{eq: Kernel F4(1)a}\\
\Image(\rmat{4,4}((-q_s)^{-4},(-q_s)^{4})) & \iso  V(\clfw_1) \oplus V(0)  \text{ as $U_q(F_4)$-modules.}
\end{align}
\end{subequations}
\end{remark}

\begin{proposition} We have the following Dorey's type $U_q'(F_4^{(1)})$-homomorphisms:
\begin{subequations}
\begin{align}
V(\varpi_1)_{-q_{s}^{-2}} \otimes V(\varpi_1)_{-q_{s}^{2}} & \twoheadrightarrow V(\varpi_2), \label{eq: her112} \\
V(\varpi_4)_{-q_{s}^{-1}} \otimes V(\varpi_4)_{-q_{s}^{1}} & \twoheadrightarrow V(\varpi_3), \label{eq: her443} \\
V(\varpi_4)_{q_{s}^{-6}} \otimes V(\varpi_4)_{q_{s}^{6}} & \twoheadrightarrow V(\varpi_4), \label{eq: OS444} \\
V(\varpi_4)_{q_{s}^{-4}} \otimes V(\varpi_4)_{q_{s}^{4}} & \twoheadrightarrow V(\varpi_1)_{-1}. \label{eq: OS441}
\end{align}
\end{subequations}
\end{proposition}

\begin{proof}
Note that \eqref{eq: her112} and \eqref{eq: her443} are known as T-system (\cite[Section 7]{Her06}). By Proposition~\ref{prop: a computation F4(1)} and \eqref{eq: Kernel F4(1)a}, \eqref{eq: bar struture} implies that there exists
a Dorey's type homomorphism
\begin{equation}
\label{eq: d444 F4(1)}
V(\varpi_4)_{q_{s}^{-6}} \otimes V(\varpi_4)_{q_{s}^{6}} \twoheadrightarrow V(\varpi_4)_a
\end{equation}
for $a \in \{ 1,-1\}$.

If $a=-1$, then $d_{4,4}(z)$ has a root $-q_s^{12}$ by taking the dual of~\eqref{eq: d444 F4(1)}:
\[
V(\varpi_4)_{q_{s}^{-6}} \hookrightarrow V(\varpi_4)_{-1} \otimes  V(\varpi_4)_{q_{s}^{-12}}.
\]
However, this can not happen by~\eqref{eq: d11 F4(1)}. Similarly, once can check \eqref{eq: OS441}.
\end{proof}

Using our \textsc{SageMath} implementation, we also have that
\begin{equation}
\label{eq: a11 2 F4(1)}
\rmat{1,1}(x,y)(u'_{\clfw_2}) = (q_s^4 z - 1) (z - q_s^{10}) (z - q_s^{12}) (z - q_s^{18}) u'_{\clfw_2},
\end{equation}
where $u'_{\clfw_2}$ is the $U_q(\g_0)$-highest weight vector in the $U_q(\g_0)$-decomposition
\[
V(\varpi_1) \otimes V(\varpi_1) \iso V(2\clfw_1) \oplus V(\clfw_2) \oplus V(2\clfw_4) \oplus V(\clfw_1)^{\oplus 3} \oplus V(0)^{\oplus 2}.
\]
Using Proposition~\ref{prop: a computation F4(1)} and~\eqref{eq: a11 2 F4(1)}, we can compute
\begin{subequations}
\begin{align}
\Rnorm{4,4}(u \otimes f_4 u) & =  \dfrac{(1-q_s^2)}{(z-q_s^2)} (u \otimes f_4 u)+\dfrac{q_s(z - 1)}{(z - q_s^2)}(f_4 u \otimes u), \label{eq: Rnorm computation 44 F4(1)} \\
\Rnorm{1,1}(u \otimes f_1 u) & =  \dfrac{(1-q_s^4)}{(z-q_s^4)} (u \otimes f_1 u)+\dfrac{q_s^2(z - 1)}{(z - q_s^4)}(f_1 u \otimes u). \label{eq: Rnorm computation 11 F4(1)}
\end{align}
\end{subequations}

Now we can compute all $a_{i,j}(z)$'s clearly, and compute $d_{i,j}(z)$'s with some ambiguities (will be all refined later) by using the same arguments in previous subsections.
We record here the $d_{i,j}(z)$'s:
\begin{subequations}
\begin{align}
d_{1,2}(z)& = (z+q_s^{6})(z+q_s^{8})(z+q_s^{10})(z+q_s^{12})(z+q_s^{14})(z+q_s^{16}),  \label{eq: d12 F4(1)}\\
d_{1,3}(z)& =(z-q_s^{7})(z-q_s^{9})(z-q_s^{13})(z-q_s^{15}),   \label{eq: d13 F4(1)}\\
d_{2,2}(z)&=(z-q_s^{4})(z-q_s^{6})(z-q_s^{8})^{\epsilon}(z-q_s^{10})^{\epsilon}(z-q_s^{12})^{2} (z-q_s^{14})^2(z-q_s^{16})(z-q_s^{18}),  \label{eq: d22 F4(1)}\\
d_{2,3}(z)&=(z+q_s^{5})(z+q_s^{7})^{\epsilon}(z+q_s^{9})(z+q_s^{11})^{\epsilon+1}(z+q_s^{13})(z+q_s^{15})(z+q_s^{17}),  \label{eq: d23 F4(1)}\\
d_{2,4}(z)& =(z-q_s^{6})(z-q_s^{10})(z-q_s^{12})(z-q_s^{16}),   \label{eq: d24 F4(1)}\\
d_{3,3}(z)& =(z-q_s^{2})(z-q_s^{6})(z-q_s^{8})^{1}(z-q_s^{10})^{1}(z-q_s^{12})^{2}(z-q_s^{16})(z-q_s^{18}), \label{eq: d33 F4(1)}\\
d_{3,4}(z)& =(z+q_s^{3})(z+q_s^{7})(z+q_s^{9})^{\epsilon'}(z+q_s^{11})(z+q_s^{13})(z+q_s^{17}). \label{eq: d34 F4(1)} 
\end{align}
\end{subequations}
for some $\epsilon,\epsilon' \in \{ 0,1\}$.
By our \textsc{SageMath} implementation, we have shown that $\epsilon' = 0$ in~\eqref{eq: d34 F4(1)}.
Therefore, we can obtain $d_{3,3}(z)$ in~\eqref{eq: d33 F4(1)}.

\subsection{\texorpdfstring{$U_q'(E^{(1)}_{6,7,8})$}{Uq'(E678(1))}}

Comparing with the quantum affine algebras have dealt with in the previous subsections,
the categories $\Ca_{Q}$ for $U_q'(E^{(1)}_{6,7,8})$ are relatively well-understood (see~\cite{HL10,Oh15E}). For instance, we know Dorey's rule.

For a while, we shall review another statistics on $\Gamma_Q$ .

\begin{definition} \label{def: Q-adjacent}
Let $Q$ be a Dynkin quiver of type $A_n$, $D_n$ or $E_{6,7,8}$.
For pairs $\up'=(\al^{(1)},\be^{(1)}) \prec^\tb_{[Q]} \up=(\al^{(2)},\be^{(2)})$ in $\Z^{\N}_{\ge 0}$, we say that they are \defn{good adjacent neighbors} if
\begin{enumerate}
\item[{\rm (i)}] there exists $\eta \in \PR$ satisfying one of the following conditions:
\begin{enumerate}
\item[{\rm (a)}] $\eta+\beta^{(2)}=\beta^{(1)}, \ \eta+\al^{(1)}=\al^{(2)}$ and $\dist_{[Q]}(\eta,\beta^{(2)}),\dist_{[Q]}(\eta,\al^{(1)})<\dist_{[Q]}(\up)$,
\item[{\rm (b)}] $\beta^{(1)}+\eta=\beta^{(2)}, \ \al^{(2)}+\eta=\al^{(1)}$ and $\dist_{[Q]}(\beta^{(1)},\eta),\dist_{[Q]}(\al^{(2)},\eta)<\dist_{[Q]}(\up)$,
\end{enumerate}
\item[{\rm (ii)}] there exists no $\up''\prec^\tb_{[Q]} \up$ such that it satisfies the conditions {\rm (i)} or
\[
\up' \prec^\tb_{[Q]} \up'' \prec^\tb_{[Q]} \up.
\]
\end{enumerate}
\end{definition}

\begin{definition} \label{def: good neighbor}
For a non $[Q]$-simple pair $\up,$
the \defn{$[Q]$-length} of the pair $\up\in \Z^{\N}_{\ge 0}$, denoted by $\len_{[Q]}(\up)$, is the integer which counts the number of all non $[Q]$-simple pairs
$\up'\in \Z^{\N}_{\ge 0}$ satisfying the following property:
\begin{itemize}
\item $\up' \prec^\tb_{[Q]} \up$ and there exists a sequence of pairs
\[
\up^{(0)}=\up' \prec^\tb_{[Q]} \up^{(1)} \prec^\tb_{[Q]} \up^{(2)} \prec^\tb_{[Q]} \cdots \prec^\tb_{[Q]} \up^{(k)}=\up \quad  (k \in \Z_{\ge 1})
\]
such that $\up^{(i)}, \up^{(i+1)}$ are good adjacent neighbor for all $0 \le i \le k-1$.
\end{itemize}
We call the pairs $\up',\up$ \defn{good neighbors}.
\end{definition}

\begin{proposition}[{\cite[Theorem 6.9 (3)]{Oh15E}}]
\label{prop: [Q]-simple pair}
For a $[Q]$-simple pair $\up=(\al,\be)$, $V_Q(\al) \otimes V_Q(\be)$ is simple.
\end{proposition}

\begin{corollary}
\label{cor: order larger than 2}
For a pair $\up=(\al,\be)$ with $\widehat{\Omega}_Q(\al)=(i,p)$ and $\widehat{\Omega}_Q(\be)=(j,p')$,
$(-q)^{|p-p'|}$ is not a root of $d_{i,j}(z)$ if and only if $\up$ is $[Q]$-simple.
\end{corollary}
\begin{proof}
Our assertion follows from Theorem~\ref{Thm: basic properties} and Proposition~\ref{prop: [Q]-simple pair}.
\end{proof}

\begin{proposition}[{\cite[Theorem 5.23]{Oh15E}}]
\label{prop: comp length}
For a non $[Q]$-simple pair $\up=(\al,\be)$, the composition length of $V_Q(\al) \otimes V_Q(\be)$ is larger than or equal to $\len_{[Q]}(\up)+2$.
\end{proposition}

\begin{corollary}
\label{cor: order 0} For a non $[Q]$-simple pair $\up=(\al,\be)$
with $\len_{[Q]}(\up) \ge 3$, the order of root $(-q)^{|p-p'|}$ of
$d_{i,j}(z)$ is strictly larger than $1$, when
$\widehat{\Omega}_Q(\al)=(i,p)$ and
$\widehat{\Omega}_Q(\be)=(j,p')$.
\end{corollary}

\begin{proof}
Our assertion follows from Lemma~\ref{lem:simplepole} and Proposition~\ref{prop: comp length}.
\end{proof}

\begin{remark}
Using Corollary~\ref{cor: order larger than 2} and~\ref{cor: order 0}, we can determine whether the give root $(-q)^{|p-p'|}$ of $d_{i,j}(z)$ is one or larger than one,
by observing the composition length. Here, $(-q)^{|p-p'|}$ arises from a non $[Q]$-simple pair $(\al,\be)$ with $\widehat{\Omega}_Q(\al)=(i,p)$ and $\widehat{\Omega}_Q(\be)=(j,p')$.
\end{remark}

For $U_q'(E^{(1)}_{6})$, $U_q'(E^{(1)}_{7})$ and $U_q'(E^{(1)}_{8})$, we shall use the $\Gamma_Q$'s given in~\eqref{eq: E6 AR quiver}, Appendix~\ref{Sec:Dynkin E7}, and Appendix~\ref{Sec:Dynkin E8}, respectively,
to apply Corollary~\ref{cor: order larger than 2} and~\ref{cor: order 0}.

\subsubsection{$E_6^{(1)}$}

In our \textsc{SageMath} implementation, we use the $U_q'(\g)$-module structure on $V(\varpi_1)$ and $V(\varpi_6)$ given by Proposition~\ref{prop:minuscule_repr} and $V(\varpi_2)$ given by Proposition~\ref{prop:adjoint_repr}.
Therefore, we obtain denominator formulas for $d_{1,1}(z)=d_{6,6}(z)$, $d_{1,6}(z)$, $d_{1,2}(z)=d_{2,6}(z)$ and $d_{2,2}(z)$ as follows:
\begin{align*}
d_{1,1}(z) & = d_{6,6}(z) = (z-q^2)(z-q^8), & d_{1,6}(z)& = (z-q^6)(z-q^{12}), \\
d_{1,2}(z) & = d_{2,6}(z) = (z+q^5)(z+q^9), & d_{2,2}(z) & = (z-q^2)(z-q^6)(z-q^8)(z-q^{12}).
\end{align*}

Additionally, we can obtain the following computation by our \textsc{SageMath} implementation: For $i=1,2$ or $6$, we have
\begin{align}
\label{eq: Rnorm computation E6(1)}
\Rnorm{i,i}\left( u_{\varpi_i}  \otimes f_i(u_{\varpi_i})_z  \right)=  \dfrac{1-q^2}{z-q^2} \ (u_{\varpi_i})_z  \otimes  f_iu_{\varpi_i} +
\dfrac{q(z-1)}{z-q^2} \ f_i(u_{\varpi_i})_z \otimes u_{\varpi_i}.
\end{align}

The Dorey's type morphism
\begin{align}
\label{eq:p11 E6(1)}
V(\varpi_1)_{-q^{-1}} \otimes V(\varpi_1)_{-q} \to  V(\varpi_3)
\end{align}
given by Theorem~\ref{thm: Dorey classical 1} can be observed using $\Gamma_Q$ given in~\eqref{eq: E6 AR quiver}.

\begin{proposition}
\label{prop: d13 E6(1)}
We have $d_{1,3}(z) =d_{3,6}(z) = (z+q^{3})(z+q^{7})(z+q^{9})$.
\end{proposition}

\begin{proof}
By applying the technique have applied, we have
\[
a_{1,1}(z) \equiv \dfrac{[6][18][24][0]}{[2][8][16][22]}, \ \ a_{1,3}(z) \equiv \dfrac{[1][5][19][23]}{[3][9][15][21]},
\]
where $[a]\seteq ((-q)^az;q^{24})$. Applying Lemma~\ref{lem:dvw avw} to~\eqref{eq:p11 E6(1)} with $W=V(\varpi_1)$, we have
\[
 \dfrac{(z+q^{-1})(z+q^{3})(z+q^{7})(z+q^{9})}{d_{1,3}(z)} \in \ko[z^{\pm 1}].
 \]
By Corollary~\ref{cor: order 0},
we can observe $\left( \sprt{000}{110} ,\sprt{000}{001} \right)$,  $\left( \sprt{112}{221},\sprt{000}{001}  \right)$ and $\left( \sprt{111}{221},\sprt{000}{001}  \right)$ in $\Gamma_Q$
are minimal pairs of $[Q]$-simples $\left( \sprt{000}{111} \right)$, $\left( \sprt{011}{111},\sprt{101}{111} \right) $ and $\left( \sprt{111}{211} \right)$, respectively. Thus Corollary~\ref{cor: order larger than 2}
and Corollary~\ref{cor: order 0} tell that $(z+q^{3})(z+q^{7})(z+q^{9})$ divides $d_{1,3}(z)$.
Since $d_{1,3}(z)$ has roots in $\C[[q^{1/m}]]q^{1/m}$ for some $m \in \Z_{>0}$, we have
\[
d_{1,3}(z) = (z+q^{3})(z+q^{7})(z+q^{9}),
\]
as we desired.
\end{proof}

\begin{proposition}
\label{prop: d33 E6(1)}
We have
\[
d_{3,3}(z) = d_{5,5}(z) = (z-q^2)(z-q^4)(z-q^6)(z-q^8)^2(z-q^{10}).
\]
\end{proposition}

\begin{proof}
Applying Lemma~\ref{lem:dvw avw} to~\eqref{eq:p11 E6(1)} with $W=V(\varpi_3)$, we have
\begin{equation}
\begin{aligned}
& a_{3,3}(z) \equiv a_{1,3}(-qz)a_{1,3}(-q^{-1}z) = \dfrac{[0][6][18][24]}{[8][10][14][16]}, \\
& \dfrac{(z-q^{2})(z-q^{4})(z-q^{6})(z-q^{8})^2(z-q^{10})}{d_{3,3}(z)} \in \ko[z^{\pm1}].
\end{aligned}
\end{equation}
As in the previous proposition, one can check that $(z-q^{2})(z-q^{4})(z-q^{6})(z-q^{8})(z-q^{10})$ divides $d_{3,3}(z)$.
Furthermore, for the pair $\up=\left( \sprt{111}{210},\sprt{000}{011}  \right)$, we can check
\[
\begin{matrix}
\up^{(1)} = \left( \sprt{101}{110},\sprt{010}{111}  \right)  \\
\up^{(2)} = \left( \sprt{111}{211},\sprt{000}{010}  \right)
\end{matrix}
     \prec^\tb_Q \up= \left( \sprt{111}{210},\sprt{000}{011}  \right)
\]
and $\up^{(i)}$ $(i=1,2)$ and $\up$ are good adjacent neighbors.
Thus $(z-q^{8})^2$ divides $d_{3,3}(z)$, which yields our assertion.
\end{proof}

\begin{proposition}
\label{prop:E61_ambiguity}
We have
\begin{align}
d_{4,4}(z)=(z-q^2)(z-q^4)^2(z-q^6)^{2+\epsilon}(z-q^8)^3(z-q^{10})^2(z-q^{12}),
\end{align}
for some $\epsilon \in \{0,1\}$.
\end{proposition}

\begin{proof}
Applying Lemma~\ref{lem:dvw avw} to
\[
V(\varpi_3)_{-q^{-1}} \otimes V(\varpi_1)_{q^2} \to  V(\varpi_4)
\]
with $W=V(\varpi_3)$, we have
\begin{align}
& a_{4,4}(z) \equiv a_{3,4}(-q^{-1}z)a_{1,4}(q^{2}z) = \dfrac{[0][2][4][20][22][24]}{[8][10][12]^2[14][16]}, \label{eq: a44 E6(1)}\\
&\dfrac{(z-q^{2})(z-q^{4})^2(z-q^{6})^3(z-q^{8})^3(z-q^{10})^2(z-q^{12})}{d_{4,4}(z)} \in \ko[z^{\pm1}].
\end{align}
As in the previous proposition, we can check
$$\dfrac{d_{4,4}(z)}{(z-q^{2})(z-q^{4})^2(z-q^{6})^2(z-q^{8})^2(z-q^{10})^2(z-q^{12})} \in \ko[z^{\pm1}].$$
If the multiplicity of $q^8$ is 2, we have contradiction for \eqref{eq: aij and dij} and \eqref{eq: a44 E6(1)}. Hence our assertion follows.
\end{proof}

Applying the techniques in Proposition~\ref{prop: d13 E6(1)} and Proposition~\ref{prop: d33 E6(1)}, we can obtain the $d_{i,j}(z)$'s:
\begin{subequations}
\begin{align}
& d_{1,4}(z)=d_{4,6}(z)=(z-q^4)(z-q^6)(z-q^8)(z-q^{10}), \\
& d_{1,5}(z)=d_{3,6}(z)=(z+q^5)(z+q^7)(z+q^{11}),\\
& d_{3,4}(z)=d_{4,5}(z)=(z+q^3)(z+q^5)^2(z+q^7)^2(z+q^9)^2(z+q^{11}), \\
& d_{3,5}(z)=(z-q^4)(z-q^6)^2(z-q^8)(z-q^{10})(z-q^{12}), \\
& d_{3,2}(z)=d_{2,5}(z)=(z-q^4)(z-q^6)(z-q^8)(z-q^{10}), \\
& d_{4,4}(z)=(z-q^2)(z-q^4)^2(z-q^6)^{2+\epsilon}(z-q^8)^3(z-q^{10})^2(z-q^{12}),  \label{eq: d44 E6(2)} \\
& d_{2,4}(z)=(z+q^3)(z+q^5)(z+q^7)^2(z+q^9)(z+q^{11}).
\end{align}
\end{subequations}

The $\epsilon$ is indeed $1$, which is proved in an addendum to our publication~\cite{OS19} (see Appendix~\ref{Sec: Addendum}).


\subsubsection{$E_7^{(1)}$}

Using our \textsc{SageMath} implementation for $V(\varpi_1)$ and $V(\varpi_7)$ with the $U_q'(\g)$-module structure by Proposition~\ref{prop:minuscule_repr}, we can obtain denominator formulas for $d_{1,1}(z)$, $d_{1,7}(z)$, $d_{7,7}(z)$ as follows:
\begin{subequations}
\begin{align}
d_{1,1}(z) & = (z-q^{2})(z-q^{8})(z-q^{12})(z-q^{18}), \allowdisplaybreaks\\
d_{1,7}(z) & = (z+q^{7})(z+q^{13}), \allowdisplaybreaks\\
d_{7,7}(z) & = (z-q^{2})(z-q^{10})(z-q^{18}).
\end{align}
\end{subequations}

Using the Dorey's type morphisms obtained from Theorem~\ref{thm: Dorey classical 1}, we can compute $d_{i,j}(z)$'s almost completely. See Appendix~\ref{app:E7_denominators},
as in the $E_6^{(1)}$ case.

\subsubsection{$E_8^{(1)}$}

By our \textsc{SageMath} implementation using the $U_q'(\g)$-module structure on $V(\varpi_8)$ given by Proposition~\ref{prop:adjoint_repr}, we can obtain denominator formulas for $d_{8,8}(z)$ as follows:
\[
d_{8,8}(z) = (z-q^2)(z-q^{12})(z-q^{20})(z-q^{30}).
\]
Moreover, using the Dorey's type surjection
\[
V(\varpi_8)_{q^{-6}} \otimes V(\varpi_8)_{q^6} \twoheadrightarrow V(\varpi_1)
\]
to define the $U_q'(\g)$-module structure of $V(\varpi_1)$, we can use our \textsc{SageMath} implementation to compute
\[
d_{1,1}(z) = (z - q^2) (z - q^8) (z - q^{12}) (z - q^{14}) (z - q^{18}) (z - q^{20}) (z - q^{24}) (z - q^{30}).
\]

Note that for $i = 1,8$, there is a unique $j \sim i$, and that
\begin{subequations}
\label{eq:terminal node}
\begin{align}
u_{\clfw_j} & = u_{\varpi_i} \otimes f_i u_{\varpi_i} - q u_{\varpi_i} \otimes f_i u_{\varpi_i},
\\ \frac{\rmat{i,i}(x,y)(u_{\clfw_j})}{d_{i,i}(z)} & = \frac{q^2 z - 1}{z - q^2} u_{\clfw_j},
\end{align}
\end{subequations}
the former is the $U_q(\g_0)$-highest weight vector in $V(\varpi_i) \otimes V(\varpi_i)$ and the latter of which is computed by our \textsc{SageMath} implementation.
Thus, we perform the following computations:
\begin{subequations}
\label{eq:Rnorm E8}
\begin{align}
\Rnorm{8,8}\left( u_{\varpi_8}  \otimes f_8(u_{\varpi_8})_z  \right) & =  \dfrac{1-q^2}{z-q^2} \ (u_{\varpi_8})_z  \otimes f_8u_{\varpi_8} +
\dfrac{q(z-1)}{z-q^2} \ f_8(u_{\varpi_8})_z \otimes u_{\varpi_8}, \label{eq: Rnorm computation E8(1) 88} \allowdisplaybreaks\\
\Rnorm{1,1}\left( u_{\varpi_1}  \otimes f_1(u_{\varpi_1})_z  \right) & =  \dfrac{1-q^2}{z-q^2} \ (u_{\varpi_1})_z  \otimes f_1u_{\varpi_1} +
\dfrac{q(z-1)}{z-q^2} \ f_1(u_{\varpi_1})_z \otimes u_{\varpi_1}, \label{eq: Rnorm computation E8(1) 11} \allowdisplaybreaks\\
\Rnorm{8,8}(z)(u_{\varpi_8}  \otimes  f_*u_{\varpi_8}) & = \dfrac{(z-q^{10})(z-1)}{(z-q^{2})(z-q^{12})}f_* u \otimes u + \text{ other terms},
\label{eq: Rnorm computation E8(1) 88 sp}
\end{align}
\end{subequations}
where $f_* = f_8 f_7 f_6 f_5 f_4 f_2 f_3 f_4 f_5 f_6 f_7 f_8$. We note that we used our \textsc{SageMath} implementation to compute the last one, but it should be possible to compute by hand using $\rmat{8,8}(x,y)$ on classically highest weight vectors.

\begin{remark}
We note that all computations of $\rmat{i,i}(u_{\varpi_i} \otimes f_i u_{\varpi_i})$ for degree 1 vertices $i$ in the Dynkin diagram all have the same form of~\eqref{eq:terminal node}. This is reflected in the fact that, \textit{e.g.},~\eqref{eq: Rnorm computation E8(1) 88} and~\eqref{eq: Rnorm computation E8(1) 11} are also of the same form as~\eqref{eq: Rnorm computation E6(1)}. Thus, there appears to be some common structure among all of these modules and $R$-matrices.
\end{remark}

Using the Dorey's type morphisms obtained from Theorem~\ref{thm: Dorey classical 1} with the above computations, we can compute $d_{i,j}(z)$'s almost completely. See Appendix~\ref{app:E8_denominators}.

\subsection{Conclusion: Folded AR quiver and denominator formulas}

Now we can conclude that the denominator formulas for exceptional types can be read for their corresponding folded AR quivers in the following sense:
\begin{theorem} \label{thm: simply laced minimal denom} \hfill
\begin{enumerate}[{\rm (1)}]
\item $V(\varpi_i)_{p} \otimes V(\varpi_j)_{p'}$ is reducible if and only if there exists a folded AR quiver $\widehat{\Upsilon}_{[\rrz]}$ and a pair $(\al,\be)$ such that
 \begin{enumerate}[{\rm (i)}]
\item $(\al,\be)$ is non $[\rrz]$-simple,
\item $V(\varpi_i)_{p} \iso V_{[\rrz]}(\al)_a$ and $V(\varpi_j)_{p'} \iso V_{[\rrz]}(\be)_a$ for some $a \in \ko^\times$,
\end{enumerate}
\item $\pm q^{t/\mathsf{d}}$ is a root of $d_{i,j}(z)$ of order larger than equal to $2$ if and only if $\theta^{\lf\rrz\rf}_t(i,j) \ge 2$. In particular,
if $(\al,\be)$ is a $[\rrz]$-minimal pair of $\al+\be \in \PR$ and contained in $\Phi_{[\rrz]}(i,j)[t]$ for some $[\rrz]$, then $\pm q^{t/\mathsf{d}}$ is a root of $d_{i,j}(z)$ of order $1$ .
\end{enumerate}
Furthermore, when $\g$ is of type $G^{(1)}_{2}$ and $F^{(1)}_{4}$,
\[
d^{\g}_{i,j}(z) = D^{\lf\rrz\rf}_{i,j}(z;-q^{1/\mathsf{d}}) \times (z-p^*)^{\delta_{i,j}}.
\]
\end{theorem}

We have seen that the denominator formulas for $U'_q(D_4^{(3)})$ can be obtained from the ones for $U_q(D_4^{(1)})$ (Corollary~\ref{cor: D41 D43})  and hence $U_q(D_4^{(2)})$ (Theorem~\ref{thm:reading deno from AR}).
The following can be understood as a $E_6$-analogues: We can observe (up to $\pm 1$)
\begin{align*}
&d^{E_6^{(2)}}_{1,1}(z) = d^{E_6^{(1)}}_{1,1}(z)d^{E_6^{(1)}}_{1,6}(-z), \  d^{E_6^{(2)}}_{1,2}(z) = d^{E_6^{(1)}}_{1,3}(z)d^{E_6^{(1)}}_{1,5}(-z), \ d^{E_6^{(2)}}_{4,4}(z) = d^{E_6^{(1)}}_{2,2}(z)d^{E_6^{(1)}}_{2,2}(-z), \allowdisplaybreaks\\
&d^{E_6^{(2)}}_{2,2}(z) = d^{E_6^{(1)}}_{3,3}(z)d^{E_6^{(1)}}_{3,5}(-z), \  d^{E_6^{(2)}}_{3,4}(z) = d^{E_6^{(1)}}_{2,4}(z)d^{E_6^{(1)}}_{2,4}(-z), \allowdisplaybreaks\\
&d^{E_6^{(2)}}_{1,3}(z) = d^{E_6^{(1)}}_{1,4}(\sqrt{-1}z)d^{E_6^{(1)}}_{1,4}(-\sqrt{-1}z), \ d^{E_6^{(2)}}_{1,4}(z) = d^{E_6^{(1)}}_{1,2}(\sqrt{-1}z)d^{E_6^{(1)}}_{1,2}(-\sqrt{-1}z), \allowdisplaybreaks\\
&d^{E_6^{(2)}}_{2,4}(z) = d^{E_6^{(1)}}_{2,3}(\sqrt{-1}z)d^{E_6^{(1)}}_{2,3}(-\sqrt{-1}z), \  d^{E_6^{(2)}}_{2,3}(z) = d^{E_6^{(1)}}_{3,4}(\sqrt{-1}z)d^{E_6^{(1)}}_{3,4}(-\sqrt{-1}z), \allowdisplaybreaks
\end{align*}
and we expect
\begin{align}
d^{E_6^{(2)}}_{3,3}(z) = d^{E_6^{(1)}}_{4,4}(z)d^{E_6^{(1)}}_{4,4}(-z).
\end{align}

\begin{remark}
In \eqref{eq: d44 E6(2)}, we have shown the existence of a root of $d_{4,4}(z)$ whose order is $3$.
Such roots of order at least $3$ can also be found in the denominator formulas for $E^{(1)}_{7}$ and $E^{(1)}_{8}$ in Appendix~\ref{app:E7_denominators} and Appendix~\ref{app:E8_denominators} respectively.
To the best knowledge of the authors, this is the first such observation of a root of order strictly larger than $2$.
\end{remark}

\section{Quiver Hecke algebras and generalized Schur-Weyl duality} \label{sec5:quiver Hecke}

In this section, we shall review the representation theory of quiver Hecke algebras, introduced by Khovanov-Lauda and Rouquier independently, and
generalized Schur-Weyl duality constructed in~\cite{KKK13A,KKK13B}.

\subsection{Quiver Hecke algebras}  We recall the definition of quiver Hecke algebra associated with a symmetrizable Cartan datum $(\cmA,P,\Pi,P^\vee,\Pi^\vee,(\cdot\,,\,\cdot))$ satisfying $(\al_i,\al_i) \in 2\Z$ for all
$i \in I$.  Let us take a family of polynomials $(Q_{i,j})_{i,j\in I}$ in $\ko[u,v]$ satisfying the following condition:
\begin{align}\label{eq: Qij}
Q_{i,j}(u,v) = (1 - \delta_{i,j}) \sum_{\substack{ p.q \in \Z_{\ge 0} \\ (\al_i,\al_i)p+(\al_j,\al_j)q=-2(\al_i,\al_j) }} t_{i,j;p,q}u^pv^q,
\end{align}
with $t_{i,j;p,q} \in \ko$, $t_{i,j;p,q}=t_{j,i;q,p}$ and $t_{i,j;-a_{ij},0} \in \ko^\times$. Then we have
\[
Q_{i,j}(u,v) = Q_{j,i}(v,u) \quad \text{ for all } i,j \in I.
\]

We denote by $\Sym_n=\lan s_i \mid 1 \le i \le n \ran$ the symmetric group on $n$-letters, where $s_i$ denote the transposition $(i,i+1)$ of $i$ and $i+1$.
Then $\Sym_n$ acts on $I^n$ by place permutations.

For $\beta \in \rl^+$ with $|\beta|=n$, we set
\[
I^\beta = \{ \nu=(\nu_1,\ldots,\nu_n) \in I^n \mid \al_{\nu_1}+\cdots+\al_{\nu_n}=\beta \}.
\]

\begin{definition}[{\cite{KL09,KL11,Rou08}}]
For $\beta \in \rl^+$ with $|\beta|=n$, the \defn{quiver Hecke algebra} $R(\be)$ at $\beta$ associated with a Cartan datum $(\cmA,P,\Pi,P^\vee,\Pi^\vee,(\cdot\,,\,\cdot))$
and a matrix $(Q_{i,j})_{i,j\in I}$ is the $\ko$-algebra generated by
the elements $\{ e(\nu) \}_{\nu \in  I^{\beta}}$, $ \{x_k \}_{1 \le
k \le n}$, $\{ \tau_m \}_{1 \le m \le n-1}$ satisfying the following
defining relations:
\begin{align*}
& e(\nu) e(\nu') = \delta_{\nu, \nu'} e(\nu), \ \
\sum_{\nu \in  I^{\beta} } e(\nu) = 1, \ \  x_{k} x_{m} = x_{m} x_{k}, \ \ x_{k} e(\nu) = e(\nu) x_{k}, \allowdisplaybreaks\\
& \tau_{m} e(\nu) = e(s_{m}(\nu)) \tau_{m}, \ \ \tau_{k} \tau_{m} =
\tau_{m} \tau_{k} \ \ \text{if} \ |k-m|>1, \ \  \tau_{k}^2 e(\nu) = Q_{\nu_{k}, \nu_{k+1}} (x_{k}, x_{k+1})
e(\nu), \allowdisplaybreaks\\
& (\tau_{k} x_{m} - x_{s_k(m)} \tau_{k}) e(\nu) = \begin{cases}
-e(\nu) \ \ & \text{if} \ m=k, \nu_{k} = \nu_{k+1}, \\
e(\nu) \ \ & \text{if} \ m=k+1, \nu_{k}=\nu_{k+1}, \\
0 \ \ & \text{otherwise},
\end{cases} \allowdisplaybreaks\\
& (\tau_{k+1} \tau_{k} \tau_{k+1}-\tau_{k} \tau_{k+1} \tau_{k}) e(\nu) = \delta_{\nu_k,\nu_{k+2}} \dfrac{Q_{\nu_{k}, \nu_{k+1}}(x_{k},
x_{k+1}) - Q_{\nu_{k}, \nu_{k+1}}(x_{k+2}, x_{k+1})} {x_{k} -
x_{k+2}}e(\nu).
\end{align*}
\end{definition}

Let us denote  by
$\Mod(R(\beta))$ the category of  $R(\beta)$-modules and by
$\Modg (R(\beta))$ the category of graded  $R(\beta)$-modules.
The category of graded $R(\beta)$-modules which are finite dimensional over $\ko$ is denoted by $R(\beta)\gmod$. In this paper, an $R(\beta)$-module means a graded $R(\beta)$-module, unless stated otherwise.

For a graded $R(\be)$-module $M=\soplus_{k \in \Z} M_k$, we define $qM=\soplus_{k \in \Z}(qM)_k$, where
$(qM)_k = M_{k-1}$. We call $q$ the \defn{grading shift functor} on $\Modg (R(\beta))$.

For $\be,\gamma\in \rl^+$ with $|\be|=m$ and $|\ga|=n$, we set the idempotent $e(\be,\ga)$ in $R(\be+\ga)$:
\[
e(\be,\ga) = \sum_{\substack{\nu \in I^{\be+\ga} \\ (\nu_1,\cdots,\nu_m) \in I^\beta} }e(\nu).
\]
Then we have the $\ko$-algebra homomorphism
\[
R(\be) \otimes  R(\ga) \to  e(\be,\ga)R(\be+\ga)e(\be,\ga)
\]
sending
\begin{align*}
e(\mu)\otimes  e(\nu) & \mapsto e(\mu*\nu)  & &&& (\mu\in I^{\beta}, \ \nu \in I^\ga), \\
x_k\otimes  1 & \mapsto x_ke(\beta,\gamma), &   1\otimes  x_s & \mapsto x_{m+s}e(\beta,\gamma) && (1\le k\le m, \ 1\le s\le n), \\
\tau_k\otimes  1& \mapsto \tau_ke(\beta,\gamma), & 1\otimes  \tau_s & \mapsto \tau_{m+s}e(\beta,\gamma) && (1\le k< m, \ 1\le s< n),
\end{align*}
where  $\mu*\nu$ is the concatenation of $\mu$ and $\nu$; \textit{i.e.}, $\mu*\nu=(\mu_1,\ldots,\mu_m,\nu_1,\ldots,\nu_n)$.

For an $R(\be)$-module $M$ and an $R(\ga)$-module $N$, we define the \defn{convolution product} $M \conv N$ by
\[
M \conv N \seteq R(\be+\ga)e(\be,\ga) \otimes _{R(\be)\otimes  R(\ga)} (M \otimes  N).
\]
Then the category $R\gmod \seteq  \soplus_{\be \in \rl^+} R(\beta)\gmod$ becomes a tensor category in the sense of~\cite[Appendix A.1]{KKK13A} induced by the convolution product.

For an $R(\beta)$-module $M$ and an $R(\gamma)$-module $N$,
we define the \defn{convolution product} $M\conv N$ by
\[
M\conv N = R(\beta + \gamma) e(\beta,\gamma) \otimes _{R(\beta )\otimes R( \gamma)}(M\otimes N).
\]

For $M \in R(\beta)\gmod$, the dual space
\[
M^* \seteq \Hom_{\ko}(M, \ko)
\]
admits an $R(\beta)$-module structure via
\begin{align*}
(r \cdot  f)(u) \seteq f(\psi(r) u) \qquad (r \in R( \beta), \ u \in M),
\end{align*}
where $\psi$ denotes the $\ko$-algebra anti-involution on $R(\beta)$ which fixes the generators $e(\nu)$, $x_m$ and $\tau_k$ for $\nu
\in I^{\beta}, 1 \leq m \leq  |\beta|$ and $1 \leq k<|\beta|$.

A simple module $M$ in $R \gmod$ is called \defn{self-dual} if $M^* \iso M$.
Every simple module is isomorphic to a grading shift of a self-dual simple module (\cite[\S 3.2]{KL09}).

\begin{theorem}[{\cite{KL09,Rou08}}]
\label{Thm:categorification}
For a symmetrizable Cartan datum $(\cmA,P,\Pi,P^\vee,\Pi^\vee,(\cdot\,,\,\cdot))$, let us denote by
$U_q(\mathsf{g})$ the corresponding quantum group and by $R$ the corresponding quiver Hecke algebra
with $(Q_{i,j})_{i,j \in I}$ in \eqref{eq: Qij}. Then there exists an $\A$-algebra isomorphism
\begin{align}\label{eq:KLRU}
\Psi: U^-_\A(\mathsf{g})^\vee \iso K(R\gmod),
\end{align}
where $K(R\gmod)$ denotes the Grothendieck ring of $R\gmod$ and
\[
U^-_\A(\mathsf{g})^\vee \seteq \{ x \in U_q^-(\mathsf{g}) \mid (x,y) \in \A \text{ for all } y \in U_\A^-(\mathsf{g}) \}.
\]
\end{theorem}

\begin{definition} A quiver Hecke algebra $R(\beta)$ is \defn{symmetric} if $Q_{i,j}(u,v)$ is a polynomial in $\ko[u-v]$ for all $i,j \in {\rm supp}(\beta).$ In particular,
we say that the quiver Hecke algebra $R$ is \defn{symmetric} if $\cmA$ is symmetric and
$\Q_{ij}(u,v)$ is a polynomial in $u-v$ for all $i,j\in I$.
\end{definition}

\begin{theorem}[{\cite{Rou12,VV11}}]
\label{thm:categorification 2}
Assume that the quiver Hecke algebra $R$ is symmetric and
$\ko$ is of characteristic zero.
 Then under the isomorphism
in {\rm Theorem~\ref{Thm:categorification}},
the dual canonical/upper global basis $\mathbf{B}^{{\rm up}}$ of $U^-_{\A}(\mathsf{g})^{\vee}$ corresponds to
the set of the isomorphism classes of \defn{self-dual} simple $R$-modules.
\end{theorem}

\subsection{\texorpdfstring{$R$}{R}-matrices for quiver Hecke algebras}

For $|\beta|=n$ and $1 \le k <n$, define the \defn{intertwiner} $\varphi_k \in R(\beta)$ by
\begin{align}\label{def:int}
\varphi_k e(\nu) = \begin{cases}
(\tau_kx_k-x_k\tau_k)e(\nu) & \text{ if } \nu_k=\nu_{k+1}, \\
\tau_k (\nu) & \text{ otherwise.}
\end{cases}
\end{align}
Since the intertwiners $\{ \varphi_k \}_{1 \le k \le n-1}$ satisfies the braid relation, we can define the element $\varphi_w$ for each $w \in \Sym_n$.
For $m,n \in \Z_\ge 0$, we set $w[m,n] \in \Sym_{m+n}$ such that
\[
w[m,n](k) = \begin{cases} k+n & \text{ if } 1 \le k \le m, \\ k-m & \text{ if } m < k \le m+n. \end{cases}
\]

We define another symmetric bilinear form $( \,\cdot, \,\cdot)_{{\rm n}}$ on $\rl$ by
$
(\al_i,\al_j)_{\rm n} = \delta_{ij}.
$

For an $R(\be)$-module $M$ and an $R(\ga)$-module with $\be,\ga \in \rl^+$ such that $|\be|=m$, $|\ga|=n$, the map
\begin{align*}
M \otimes  N & \to q^{(\be,\ga)-2(\be,\ga)_{{\rm n}}} N \conv M \ \  \text{ given by }  \ \   u \otimes  v  \mapsto \varphi_{w[n,m]} (v \otimes  u)
\end{align*}
is an $R(\be) \otimes  R(\ga)$-homomorphism by~\cite[Lemma 1.3.1]{KKK13A} and it extends to an $R(\beta+\ga)$-homomorphism
\[
R_{M,N} \colon M \conv N \to   q^{(\be,\ga)-2(\be,\ga)_{{\rm n}}} N \conv M.
\]

\emph{From now on, we assume that quiver Hecke algebras are symmetric}. Let $z$ be an indeterminate that is homogeneous of degree $2$, and $\psi_z \colon R(\beta) \longrightarrow \ko[z] \otimes  R(\beta)$ be the algebra
given by
\[
\psi_z(x_k) = x_k+z,
\qquad\qquad
\psi_z(\tau_k) = \tau_k ,
\qquad\qquad
\psi_z\bigl(e(\nu)\bigr) = e(\nu).
\]

For an $R(\beta)$-module $M$, we denote by $M_z$ the $(\ko[z]\otimes  R(\be))$-module $\ko[z] \otimes  M$ with the action of $R(\be)$ twisted by $\psi_z$:
\[
e(\nu)(a \otimes  u) = a \otimes  e(\nu)u, \qquad x_k(a \otimes  u) = (za) \otimes  u + a \otimes  (x_k u), \qquad  \tau_k(a \otimes  u) = a \otimes  (\tau_ku),
\]
for $\nu \in I^\be$, $a \in \ko[u]$ and $u \in M$.

For a non-zero $R(\be)$-module $M$ and a non-zero $R(\ga)$-module $N$,
\begin{eqnarray} &&
\parbox{80ex}{let $s$ be the order of zero of $R_{M_z,N_{z'}} \colon M_z \conv N_{z'} \to q^{(\be,\ga)-2(\be,\ga)_{{\rm n}}} N_{z'} \conv M_z$, \textit{i.e.}, the largest non-negative integer such that
the image of $R_{M_z,N_{z'}}$ is contained in $(z'-z)^s q^{(\be,\ga)-2(\be,\ga)_{{\rm n}}} N_{z'} \conv M_z$.
}\label{eq: degree s}
\end{eqnarray}

For a non-zero $R(\be)$-module $M$ and a non-zero $R(\ga)$-module $N$, we set
\[
\Lambda(M,N) \seteq -(\be,\ga)+2(\be,\ga)_{{\rm n}}-2s, \quad \Rren{M_z,N_z'}\seteq (z'-z)^{-s}R_{M_z.N_{z'}}
\]
and define
\[
\rmat{M,N}\seteq (\Rren{M_z,N_z'})|_{z=z'=0} \colon M \conv N \to q^{-\Lambda(M,N)}N \conv M.
\]
By~\cite[Lemma 1.4.4 {\rm (ii)}]{KKK13A}, $\rmat{M,N}$ does not vanish.

%

We say that a simple $R(\be)$-module $M$ is \defn{real} if $M\conv M$ is simple.

\begin{remark}
We shall work always in the category of finite dimensional graded modules.
For the sake of simplicity,
\emph{we simply say that $M$ is an $R$-module instead of saying
that $M$ is a finite dimensional graded $R(\beta)$-module for $\beta\in\rl^+$.}
We also sometimes ignore grading shifts if there is no danger of confusion.
Hence, for $R$-modules $M$ and $N$, \emph{we sometimes say that
$f \colon M \to N$ is a homomorphism if
$f \colon q^aM \to N$ is a morphism in $R\gmod$ for some $a\in\Z$.}
If we want to emphasize that $f \colon q^aM \to N$ is a morphism in $R\gmod$, we say so.
\end{remark}

\begin{theorem}[{\cite[Theorem 3.2]{KKKO14S}}]
\label{thm: socle head}
Let $M$ be a simple $R(\al)$-module and $N$ be a simple $R(\be)$-module, one of which is real. Then there exist unique non-zero homomorphisms (up to constant)
\[
\rmat{M,N} \colon M \conv N \Lto N \conv M \quad \text{and} \quad \rmat{N,M} \colon N \conv M \Lto M \conv N,
\]
satisfying the following properties:
\begin{enumerate}[{\rm (i)}]
\item $M \conv N$ has a simple socle and a simple head.
\item The socle and head of $M \conv N$ are distinct and appear once in the composition series of $M \conv N$ unless $M \conv N \iso N \conv M$ is simple.
\end{enumerate}
\end{theorem}

\begin{proposition}[{\cite[Corollary 3.11]{KKKO14S}}]
\label{prop: not vanishing}
Let $M_k$ be a finite dimensional graded $R$-module ($k=1,2,3$), and let $\varphi_1 \colon L \to M_1 \conv M_2$ and
$\varphi_2 \colon M_2 \conv M_3 \to L'$ be non-zero homomorphisms. Assume further that $M_2$ is simple. Then the composition
\begin{align} \label{eq: not vanishing}
L\conv M_3\To[\varphi_1\conv M_3] M_1\conv M_2\conv M_3\To[M_1\conv\varphi_2]
M_1\conv L'
\end{align}
does not vanish. Similarly, for non-zero homomorphisms $\varphi_1 \colon L \to M_2 \conv M_3$,
$\varphi_2 \colon M_1 \conv M_2 \to L'$ and an assumption that $M_2$ is simple, we have a non-zero composition
\begin{align} \label{eq: not vanishing 2}
M_1 \conv L \To[M_1\conv \varphi_1] M_1\conv M_2\conv M_3 \To[\varphi_2 \conv M_3]
L '\conv M_3.
\end{align}
\end{proposition}

\begin{theorem}[{\cite{BKM12,Kato12,KR11,McN15,Oh15E}}]
\label{thm: BkMc}
For a finite simple Lie algebra $\mathsf{g}$, we choose a reduced expression $\redez$ of the longest element $w_0$.
Let $R$ be the quiver Hecke algebra corresponding to $\mathsf{g}$. For each positive root $\beta \in \PR$, there exists a self-dual simple module $S_{[\redez]}(\beta)$
satisfying the following properties$:$
\begin{enumerate}[{\rm (a)}]
\item $S_{[\widetilde{w}_0]}(\beta)^{\conv m}$ is a real simple $R(m\be)$-module.
\item For every $\um_{\redez} \in \Z_{\ge 0}^{\N}$, there exists a unique non-zero $R$-module homomorphism (up to constant)
\[
\Stom \overset{\rmat{\um}}{\Lto} \Sgetsm
\]
where
\begin{align*}
\Stom & \seteq S_{[\redez]}(\beta_1)^{\conv m_1}\conv\cdots \conv S_{[\redez]}(\beta_\N)^{\conv m_\N}, \  \   \Sgetsm  \seteq  S_{[\redez]}(\beta_\N)^{\conv m_\N}\conv\cdots \conv S_{[\redez]}(\beta_1)^{\conv m_1}
\end{align*}
such that  $\Image(\rmat{\um}) \iso \hd\left(\Stom\right) \iso \soc\left(\Sgetsm\right)$ is simple.
\item For any sequence $\um_{\redez} \in \Z_{\ge 0}^{\ell(w_0)}$, we have
\[
  [\Stom] \in [\Image(\rmat{\um})] + \displaystyle\sum_{ \um' \prec^{\tb}_{[\redez]} \um } \Z_{\ge 0}[\Image(\rmat{\um'})].
\]
\item For distinct $\um, \um' \in \Z_{\ge 0}^{\ell(w_0)}$, the simple modules $\Image(\rmat{\um})$ and $\Image(\rmat{\um'})$ are distinct.
\item For every simple $R$-module $M$, there exists a unique sequence $\um \in \Z_{\ge 0}^{\ell(w_0)}$ such that
$M \iso \Image(\rmat{\um}) \iso \operatorname{hd}\big(\Stom \big).$
\item For any $[\redez]$-minimal pair $(\beta^\redez_k,\beta^\redez_l)$ of $\beta^\redez_j=\beta^\redez_k+\beta^\redez_l$, there exists an exact sequence
\begin{equation} \label{eq: 6 ses}
0 \to S_{[\redez]}(\beta_j)
\to S_{[\redez]}(\beta_k) \conv S_{[\redez]}(\beta_l) \xrightarrow{\mathbf{r}_{e_k+e_l}} S_{[\redez]}(\beta_l) \conv S_{[\redez]}(\beta_k) \to
S_{[\redez]}(\beta_j) \to 0.
\end{equation}
\end{enumerate}
Moreover, under the isomorphism $\Psi$ in \eqref{eq:KLRU}, each isomorphism class $[S_{[\redez]}(\beta)]$ of $S_{[\redez]}(\beta)$ $(\beta \in \PR)$
is mapped onto the \defn{dual root vector} $\mathbf{F}_{[\redez]}^{{\rm up}}(\beta)$
of the dual PBW basis $P_{[\redez]}$ of $U^-_\A(\g_0)^\vee$ $($up to $q^{\Z})$,
 which is associated to the commutation class $[\redez]$ of $w_0$ (see~\cite[\S 4]{KKK13B} for more details).
\end{theorem}

\subsection{Generalized Schur-Weyl duality functors}

In this subsection, we recall the generalized Schur-Weyl duality functors constructed in~\cite{KKK13A}.

Let $U_q'(\g)$ be a quantum affine algebra over $\ko$. Assume that we are given an index set $J$, a family of good modules $\{ V_j \}_{j \in J}$ and a map $X \colon J \to \ko^\times$. We associate a
quiver $\Gamma^J$ from the datum $(J,X,\{ V_j\}_{j \in J})$ as follows:
\begin{enumerate}[(1)]
\item Take $J$ as the set of vertices.
\item Put $d_{ij}$-many arrows from $i$ to $j$, where $d_{ij}$ denotes the order of zero of $d_{V_i,V_j}(z_j/z_i)$ at $z_j/z_i=X(j)/X(i)$.
\end{enumerate}
Note that $d_{ij}d_{ji}=0$ for $i,j \in J$.

We define a symmetric Cartan matrix $\cmA^J=(a^J_{ij})$ and a family of polynomials $(P^J_{ij})_{i,j\in J}$ by
\[
a^J_{ij} = \begin{cases}
\qquad 2 & \text{ if } i=j, \\
-d_{ij}-d_{ji} & \text{ if } i \ne j,
\end{cases}
\quad \text{ and } \quad
P^J_{i,j}(u,v)= (u-v)^{d_{ij}}.
\]
We denote by
$\rl_J^+$ the root lattice and $\{ \al_i^J \}_{i \in J}$ the set of simple roots associated to $\cmA^J$.

Let us denote by $R^J(\beta)$ the symmetric quiver Hecke algebra associated with the Cartan matrix $\cmA^J$ and a matrix $(Q^J_{i,j})_{i,j}$ in \eqref{eq: Qij}, where
\[
Q^J_{i,j} = (1 - \delta_{i,j}) P^J_{i,j}(u,v)P^J_{j,i}(v,u).
\]

For each $i \in J$, let $L(i)$ be the $1$-dimensional
$R^J(\alpha_i)$-module  generated by a nonzero vector $u(i)$ with
relation $x_1 u(i) = 0$ and $e(j) u(i) = \delta_{j,i} u(i)$ for $j \in
J$. The space $L(i)_z :=  \ko[z] \otimes L(i)$ admits an
$R^J(\alpha_i)$-module structure as follows:
\begin{align*}
x_1\bigl( a \otimes u(i) \bigr) = (za) \otimes u(i), \qquad e(j)\bigl( a\otimes u(i) \bigr) = \delta_{j,i}\bigl( a \otimes u(i) \bigr).
\end{align*}

\begin{theorem}~\cite[Theorem 3.3]{KKK13A}\label{thm: F}
For each $\be \in \rl^+_J$, there exists a $(U_q'(\g),R^J(\be))$-bimodule $\widehat{V}^{\tens \be}$ which induces the functor
\begin{align*}
\F_{ \beta } \colon \Mod(R^J(\beta)) &\rightarrow \Mod(U_q'(\g))
\quad \text{ given by }  \quad
\F_{ \beta } (M) \seteq \widehat{V}^{\otimes \be} \otimes_{R^J({ \beta })} M.
\end{align*}
In particular, for any $i \in J$, we have
\begin{align}
\F(L(i))\iso (V_i)_{X(i)} \quad \text{for } i \in J.
\end{align}
\end{theorem}

\begin{theorem}[{\cite{KKK13A}}]
\label{thm:duality}
When we restrict the functor $\mathcal{F}$ to $R^J\gmod$, we have
\[
\mathcal{F} \colon \soplus_{\be \in \rl^+_J} R^{J}(\be)\gmod  \rightarrow \mathcal{C}_\g.
\]
Moreover, $\mathcal{F}$ satisfies the following properties:
\begin{enumerate}[{\rm (a)}]
\item $\mathcal{F}$ is a tensor functor; that is, for any $M_1, M_2 \in R^{J}\gmod$, we have
\[
\mathcal{F}(R^J(0)) \iso \ko \quad \text{ and } \quad \mathcal{F}(M_1 \conv M_2) \iso \mathcal{F}(M_1) \otimes  \mathcal{F}(M_2).
\]
\item If the quiver $\Gamma^J$ is a Dynkin quiver of finite type $A_n$, $D_n$ or $E_{6,7,8}$, then $\mathcal{F}$ is exact.
\end{enumerate}
\end{theorem}
We call the functor $\mathcal{F}$ the \defn{generalized Schur-Weyl duality functor}.

\begin{theorem}[\cite{KKK13B,KKKOIV}]
\label{thm:gQASW duality}
Let $U_q'(\g^{(t)})$ be a quantum affine algebra of type $A^{(t)}_{n}$  $($resp.\ $D^{(t)}_{n})$
and let $Q$ be a Dynkin quiver of finite type $A_n$  $($resp.\ $D_n)$ for $t=1,2$.
Take $J$ and $\mathcal{S}$ as the set of simple roots $\Pi$ associated to $Q$.
We define two maps
\[
s \colon \Pi \to \{ V(\varpi_i) \mid i \in I_0 \} \quad \text{ and } \quad X \colon \Pi \to  \ko^\times
\]
as follows:
for $\al \in \Pi$ with $\widehat{\Omega}_Q(\al) = (i,p)$, we define
\[
s(\al) = \begin{cases} V(\varpi_i) & \text{if $\g^{(1)}=A^{(1)}_{n}$ or $D^{(1)}_{n}$,} \\
V(\varpi_{i^\star}) & \text{if $\g^{(2)}=A^{(2)}_{n}$ or $D^{(2)}_{n}$,}
\end{cases}
\qquad X(\al)=
\begin{cases} (-q)^p  & \text{if $\g^{(1)}=A^{(1)}_{n}$ or $D^{(1)}_{n}$,} \\
((-q)^p)^\star & \text{if $\g^{(2)}=A^{(2)}_{n}$ or $D^{(2)}_{n}$.} \end{cases}
\]
Then we have the following:
\begin{enumerate}[{\rm (1)}]
\item The underlying graph of $\Gamma^{J}$ coincides with the one of $Q$. Hence the functor
\[
\F^{(t)}_Q \colon  R^{J}\gmod \rightarrow \mathcal{C}^{(t)}_Q \quad (t =1,2) \text{ in {\rm Theorem~\ref{thm:duality}} is exact.}
\]
\item The functor $\F^{(t)}_Q$ induces a bijection from
the set of the isomorphism classes of simple objects of
$R^{J}\gmod$ to that of $\mathcal{C}^{(t)}_Q$. In particular,
$\F^{(t)}_Q$ sends $S_Q(\beta) \seteq S_{[Q]}(\beta)$ to
$V^{(t)}_Q(\beta)$. 
\item The functors $\F^{(1)}_Q$ and $\F^{(2)}_Q$ induce the ring isomorphisms in {\rm Theorem~\ref{thm:categorification1}}. Moreover, the induced functor between
$\mathcal{C}^{(1)}_Q$ and $\mathcal{C}^{(2)}_Q$:
\[
\begin{tikzpicture}[baseline=0]
\node (C1) at (-4,0) {$\mathcal{C}^{(1)}_Q$};
\node (C2) at (4,0) {$\mathcal{C}^{(2)}_{Q}$};
\node (RJ) at (0,0) {$R^J\gmod$};
\draw[->] (RJ) -- node[midway,below] {$\F^{(1)}_Q$} (C1);
\draw[->] (RJ) -- node[midway,below] {$\F^{(2)}_Q$} (C2);
\path[<->,dotted] (C1) edge[out=25,in=155](C2);
\end{tikzpicture}
\]
 preserves dimensions and sends simples to simples, bijectively.
\end{enumerate}
\end{theorem}

\begin{remark}
\label{remark: En(1) conjecture}
It was pointed out in~\cite{KKK13B} that when $\g^{(1)}$ is of type $E_{6,7,8}^{(1)}$, $Q$ is a Dynkin quiver of finite type $E_{6,7,8}$ and we choose $(J,\mathcal{S},s,X)$ as in Theorem~\ref{thm:gQASW duality},
one can prove that (i) $\F_Q^{(1)}$ is exact (ii) $\F_Q^{(1)}$ sends $S_Q(\be)$ to $V_Q(\be)$ ($\be \in \PR$) and (iii) $\F^{(1)}_Q$ induces a bijection from
the set of the isomorphism classes of simple objects of
$R^{J}\gmod$ to that of $\mathcal{C}^{(1)}_Q$ \emph{provided that $\Gamma^{J}$ is of finite type $E_{6,7,8}$},
by applying the same arguments in~\cite[\S 4]{KKK13B} (see also~\cite[Conjecture 4.3.2]{KKK13B}).
\end{remark}

\begin{theorem}[\cite{KO17}] \label{thm:gQASW duality 2}
Let $U_q'(\g^{(1)})$ be a quantum affine algebra of type $B^{(1)}_{n}$  $($resp.\ $C^{(1)}_{n})$
and let $[\mQ]$ be a twisted adapted class of finite type $A_{2n-1}$  $($resp.\ $D_{n+1})$.
Take $J$ and $\mathcal{S}$ as the set of simple roots $\Pi$ of finite type $A_{2n-1}$  $($resp.\ $D_{n+1})$.
We define two maps
\[
s \colon \Pi \to \{ V(\varpi_{i}) \mid i \in I_0 \} \quad \text{ and } \quad X \colon \Pi \to  \ko^\times
\]
as follows: for $\al \in \Pi$ with $\widehat{\Omega}_Q(\al) = (i,p/\mathsf{d})$, we define
\[
s(\al)= V(\varpi_{i}),
\qquad\qquad
X(\al)=
\begin{cases} (-1)^i(q^{1/\mathsf{d}})^p  & \text{if $\g^{(1)}=B^{(1)}_{n}$,} \\
(-q^{1/\mathsf{d}})^p & \text{if $\g^{(1)}=C^{(1)}_{n}$.} \end{cases}
\]
Then we have the following:
\begin{enumerate}[{\rm (1)}]
\item The underlying graph of $\Gamma^{J}$ is a Dynkin diagram of finite type $A_{2n-1}$  $($resp.\ $D_{n+1})$. Hence the functor
\[
\mathscr{F}_\mQ \colon  R^{J}\gmod \rightarrow \mathscr{C}^{(1)}_\mQ \text{ in {\rm Theorem~\ref{thm:duality}} is exact.}
\]
\item The functor $\mathscr{F}^{(1)}_\mQ$ induces a bijection from
the set of the isomorphism classes of simple objects of $R^{J}\gmod$
to that of $\mathscr{C}^{(1)}_\mQ$. In particular, $\mathscr{F}_\mQ$
sends $S_\mQ(\beta) \seteq S_{[\mQ]}(\beta)$ to $V_\mQ(\beta)$.
\item[{\rm (3)}] The functors $\mathscr{F}_\mQ$ induces the ring isomorphism in {\rm Theorem~\ref{thm:categorification2}}. Moreover, the induced functors among
$\mathcal{C}^{(1)}_Q$, $\mathcal{C}^{(2)}_{Q'}$ and $\mathscr{C}^{(1)}_\mQ$
\[
\begin{tikzpicture}[>=latex,scale=3]
\node (C1) at (-1,0) {$\mathcal{C}^{(1)}_Q$}; \node (C2) at (1,0)
{$\mathcal{C}^{(2)}_{Q'}$}; \node (sC) at (0,-0.75)
{$\mathscr{C}^{(1)}_{\mQ}$}; \node (RJ) at (0,0) {$R^J\gmod$};
\draw[<->,dotted] (C1) to[out=30,in=150] (C2); \draw[<->,dotted]
(C2) to[out=-120,in=10] (sC); \draw[<->,dotted] (sC)
to[out=170,in=-60] (C1); \draw[->] (RJ) -- node[midway,below]
{\scriptsize $\F_Q^{(2)}$} (C1); \draw[->] (RJ) --
node[midway,below] {\scriptsize $\F_{Q'}^{(2)}$} (C2); \draw[->]
(RJ) -- node[midway,right] {\scriptsize $\mathscr{F}_\mQ$} (sC);
\end{tikzpicture}
\]
send simples to simples, bijectively.
\end{enumerate}
\end{theorem}

\section{Application from generalized Schur-Weyl duality}
\label{sec6:application}

In this section, we shall show the applications of the denominator formulas
that can be obtained from the generalized Schur-Weyl duality functors.


\subsection{\texorpdfstring{$U_q'(E_{6,7,8}^{(1)})$}{Uq'(E678(1))}}

In this subsection, $Q$ always denotes a Dynkin quiver of finite type $E_{6,7,8}$ and $U_q'(\g)$ denotes the quantum affine algebra of type $E_{6,7,8}^{(1)}$.

The following proposition can be checked with the denominator formulas given in the previous section.


\begin{proposition}
\label{prop: particular E_n}
For the Dynkin quiver $Q$ given in Example~\ref{ex:folded AR quiver E} and Appendix~\ref{Sec:Dynkin E7},~\ref{Sec:Dynkin E8}, take $J$ and $\mathcal{S}$ as the set of simple roots $\Pi$ associated to $Q$.
We define two maps
\[
s \colon \Pi \to \{ V(\varpi_i) \mid i \in I_0 \} \quad \text{ and } \quad X \colon \Pi \to  \ko^\times
\]
as follows:
for $\al \in \Pi$ with $\widehat{\Omega}_Q(\al)=(i,p)$, we define
\[
s(\al) = V(\varpi_i), \qquad\qquad X(\al)= (-q)^p.
\]
Then the underlying graph of $\Gamma^{J}$ coincides with the one of $Q$.
\end{proposition}

\begin{proof}
Note that the pairs of simple roots $(\al_i,\al_j)$ are $[Q]$-minimal if $i,j$ are adjacent in $\Delta$ and $[Q]$-simple, otherwise. Thus the underlying graph of $\Gamma^J$
coincides with the one of $Q$ by Theorem~\ref{thm: simply laced minimal denom}.
\end{proof}

\begin{example}
Let us see the $E_6$ case with an example. By forgetting the arrows of $\Gamma_Q$ in \eqref{eq: E6 AR quiver}, the underlying graph of $\Gamma^{J}$ can be described as follows:
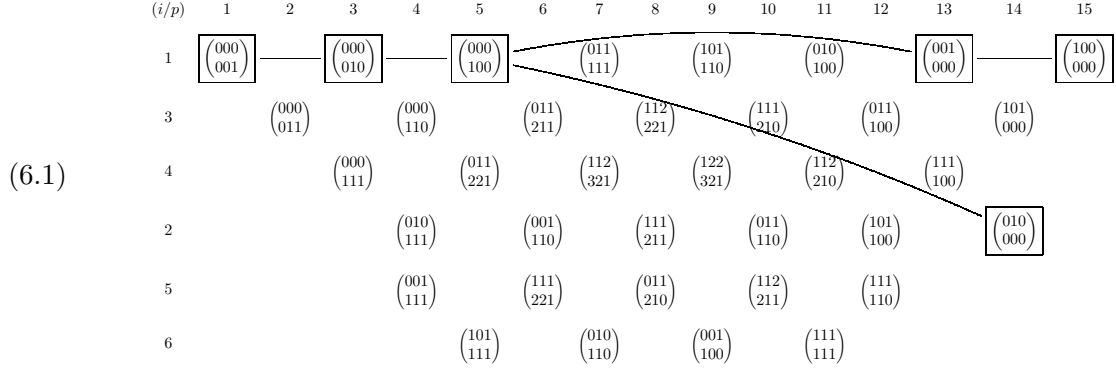
\begin{align} \label{eq: E6 AR quiver Gamma_J} \raisebox{6em}{\scalebox{0.55}{\xymatrix@R=1.0ex@C=0.5ex{
(i/p) & 1 & 2 & 3 & 4 & 5 & 6 & 7 & 8 & 9 & 10 & 11 & 12 & 13 & 14 & 15   \\
1&{\scriptstyle\boxed{\prt{000}{001}}} \ar@{-}[rr] && {\scriptstyle\boxed{\prt{000}{010}}} \ar@{-}[rr]  && {\scriptstyle\boxed{\prt{000}{100}}} \ar@{-}@/^1pc/[dddrrrrrrrrr]
\ar@{-}@/^1.5pc/[rrrrrrrr] && {\scriptstyle\prt{011}{111}}
&&{\scriptstyle\prt{101}{110}} && {\scriptstyle\prt{010}{100}} && {\scriptstyle\boxed{\prt{001}{000}}} \ar@{-}[rr]  && {\scriptstyle\boxed{\prt{100}{000}}} \\
3&& {\scriptstyle\prt{000}{011}} && {\scriptstyle\prt{000}{110}} && {\scriptstyle\prt{011}{211}}
&& {\scriptstyle\prt{112}{221}} && {\scriptstyle\prt{111}{210}} && {\scriptstyle\prt{011}{100}}
&& {\scriptstyle\prt{101}{000}} \\ 
4&&& {\scriptstyle\prt{000}{111}}&& {\scriptstyle\prt{011}{221}}&& {\scriptstyle\prt{112}{321}}
&& {\scriptstyle\prt{122}{321}} && {\scriptstyle\prt{112}{210}}&& {\scriptstyle\prt{111}{100}} \\
2&&&& {\scriptstyle\prt{010}{111}}&& {\scriptstyle\prt{001}{110}}&& {\scriptstyle\prt{111}{211}}&& {\scriptstyle\prt{011}{110}}
&& {\scriptstyle\prt{101}{100}}&& {\scriptstyle\boxed{\prt{010}{000}}} \\
5&&&& {\scriptstyle\prt{001}{111}}&& {\scriptstyle\prt{111}{221}}&& {\scriptstyle\prt{011}{210}}
&& {\scriptstyle\prt{112}{211}}&& {\scriptstyle\prt{111}{110}} \\
6&&&&& {\scriptstyle\prt{101}{111}}&& {\scriptstyle\prt{010}{110}}&& {\scriptstyle\prt{001}{100}}&& {\scriptstyle\prt{111}{111}} \\
}}}
\end{align}
since
\begin{align*}
& d_{1,1}(z)=(z-q^2)(z-q^8) \quad \text{and} \quad d_{1,2}(z)=(z+q^5)(z+q^9).
\end{align*}
\end{example}

\begin{theorem}
\label{thm: simplicity type E}
The above proposition holds for any $Q$ of type $E_{6,7,8}$. Hence for any $Q$ of type $E_{6,7,8}$, there exists an exact functor
\[
\F_Q^{(1)} \colon R^{E_{6,7,8}}\gmod \to \Ca_Q^{(1)},
\]
which sends $S_Q(\be)$ to $V_Q(\be)$ $(\be \in \PR)$ and $\F^{(1)}_Q$ induces a bijection from
the set of the isomorphism classes of simple objects of
$R^{E_{6,7,8}}\gmod$ to that of $\mathcal{C}^{(1)}_Q$.
\end{theorem}
\begin{proof}
Our assertion follows from Remark~\ref{remark: En(1) conjecture}.
\end{proof}

\subsection{\texorpdfstring{$U_q'(D_4^{(3)})$ and $U_q'(E_6^{(2)})$}{Uq'(D4(3)) and Uq'(E6(2))}}

In this subsection, we shall prove the $U_q'(D_4^{(3)})$ and $U_q'(E_6^{(2)})$-analogues of
Theorem~\ref{thm: Dorey classical 1}, Theorem~\ref{thm:categorification1} and Theorem~\ref{thm:gQASW duality}. Since the frameworks for $U_q'(D_4^{(3)})$ and $U_q'(E_6^{(2)})$
are the same, we shall give proofs for $U_q'(D_4^{(3)})$-case only.

\begin{lemma} \label{lem: Q well}
For Dynkin quivers $Q$ and $Q'$ of the same finite type $A_n$, $D_n$ or $E_{6,7,8}$, assume that there exists a pair
$(\al,\be)$ of $\gamma$ contained in $\Phi_{[Q]}(i,j)(t)$, then there exists a pair $(\al',\be')$ such that $\al'+\be' \in \PR$ and
$(\al',\be')$ is contained in $\Phi_{[Q']}(i,j)(t)$ or  $\Phi_{[Q']}(i^*,j^*)(t)$.
\end{lemma}

\begin{proof}
The assertion follows from well-definedness of the distance polynomial $D^{\lf \Delta \rf}_{i,j}(z)$ and the fact that
$\Gamma_{Q'}$ can be obtained from $\Gamma_Q$ by applying the reflection functors properly.
\end{proof}

\begin{proposition} \label{prop: particular D43 E62}
For any Dynkin quiver $Q$ of finite type $D_4$ $(t=3)$ $($resp.~$E_6$  $(t=2))$, take $J$ and $\mathcal{S}$ as the set of simple roots $\Pi$ associated to $Q$.
We define two maps
\[
s \colon \Pi \to \{ V^{(t)}(\varpi_i) \mid i \in I_0 \} \quad \text{ and } \quad X \colon \Pi \to  \ko^\times
\]
as follows:
for $\al \in \Pi$ with $\widehat{\Omega}_Q(\al)=(i,p)$, we define
\[
s(\al) = \begin{cases}
V^{(3)}(\varpi_{i^{\dagger}}) & \text{ if $Q$ is of type $D_4$}, \\
V^{(2)}(\varpi_{i^{\star}}) & \text{ if $Q$ is of type $E_6$} \end{cases}
 \quad X(\al)= \begin{cases}  ((-q)^p)^{\dagger}& \text{ if $Q$ is of type $D_4$}, \\
((-q)^p)^{\star}& \text{ if $Q$ is of type $E_6$},
\end{cases}
\]
where $i^{\dagger}$ and $((-q)^p)^{\dagger}$ are defined in \eqref{eq:dagger} $($resp.~$i^{\star}$ and $((-q)^p)^{\star}$ are defined in \eqref{eq:star}$)$.
Then the underlying graph of $\Gamma^{J}$ coincides with the Dynkin diagram of finite type $D_4$
$($resp.~$E_6)$. Hence we have an exact functor
\[
\F_Q^{(3)} \colon R^{D_4}\gmod \to \Ca^{(3)}_{Q} \quad \text{$($resp.\ }\F_Q^{(2)} \colon R^{E_6}\gmod \to \Ca^{(2)}_{Q}).
\]
\end{proposition}

\begin{proof}
Our assertion follows from Theorem~\ref{thm:gQASW duality} on $U_q'(D^{(1)}_4)$ and Corollary~\ref{cor: D41 D43}.
\end{proof}

\begin{example}
By forgetting the arrows of $\Gamma_Q$ in \eqref{eq: E6 AR quiver}, the underlying graph of $\Gamma^{J}$ can be described as follows:
\begin{align}
\raisebox{3.5em}{ \scalebox{0.7}{\xymatrix@R=0.5ex{
(i,p) & 1 & 2 & 3 & 4 & 5 & 6 & 7 \\
1&{\boxed{\lan  1,-2 \ran}} \ar@{-}[ddrrrr]     && \lan  2,4 \ran  && \lan  1,-4 \ran    \\
2&& \lan  1,4 \ran   && \lan  1,2 \ran  && \lan  2,-4 \ran   \\
3&&& \lan  1,3\ran   &&  {\boxed{\lan  2,-3 \ran}} \ar@{-}[rr]   && {\boxed{\lan  3,-4 \ran}}  \\
4& {\boxed{\lan  3,4 \ran}}\ar@{-}[urrrr]   , && \lan  1,-3 \ran   && \lan  2,3 \ran
}}}
\end{align}
\end{example}

Now we shall prove the following theorem:

\begin{theorem} \label{thm: D43 S to V}
For any Dynkin quiver $Q$ of finite type $D_4$ $(t=3)$ or $E_6$ $(t=2)$, and $\gamma \in \PR$, we have
\[
\F_Q^{(t)}(S_Q(\ga)) \iso V^{(t)}_Q(\gamma).
\]
\end{theorem}

\begin{proof}
We shall prove the claim for the AR quiver $\Gamma_Q$ in \eqref{eq: D4 AR quiver} as the rest of the cases are similar. By Theorem~\ref{thm: F}, we have
$\F_Q^{(3)}(S_Q(\al_i)) \iso V^{(3)}_Q(\al_i)$.
Hence,
\begin{align*}
\F_Q^{(3)}(S(\al_1)) & \iso V^{(3)}(\varpi_1)_{-q},
  & \F_Q^{(3)}(S(\al_2)) & \iso V^{(3)}(\varpi_1)_{-\omega q^5},  \\
\F_Q^{(3)}(S(\al_3)) & \iso V^{(3)}(\varpi_1)_{-\omega q^7},
  & \F_Q^{(3)}(S(\al_4)) & \iso V^{(3)}(\varpi_1)_{-\omega^2q}.
\end{align*}
Thus, we have an exact sequence
\begin{align} \label{eq: D43 al2+al3}
 0 \to S_Q(\al_2+\al_3) \to S_Q(\al_3) \conv S_Q(\al_2) \To[\rmat{}] S_Q(\al_2) \conv S_Q(\al_3) \to S_Q(\al_2+\al_3) \to 0.
\end{align}
Applying the exact functor $\F_Q^{(3)}$, we have
\begin{align*}
0 \to \F_Q^{(3)}(S_Q(\al_2+\al_3)) & \to V^{(3)}(\varpi_1)_{-\omega q^7} \otimes   V^{(3)}(\varpi_1)_{-\omega q^5} \\
& \To[\F_Q^{(3)}(\rmat{})] V^{(3)}(\varpi_1)_{-\omega q^5} \otimes  V^{(3)}(\varpi_1)_{-\omega q^7} \to \F_Q^{(3)}(S_Q(\al_2+\al_3))\to 0.
\end{align*}

By the Dorey's type homomorphism and the argument of~\cite[Theorem 7.3]{KO17}, we have
\[
\F^{(3)}_Q(S_Q(\al_2+\al_3)) \iso V^{(3)}(\varpi_2)_{-q^6}\iso V^{(3)}(\varpi_2)_{-\omega^{t} q^6} \qquad\qquad (t=1,2).
\]
Applying the same argument to the minimal pairs $(\al_1,\al_2+\al_3)$ and $(\al_4,\al_2+\al_3)$, we have
\begin{align*}
\F^{(3)}_Q(S_Q(\lan 1,-4 \ran)) & \iso V^{(3)}(\varpi_1)_{-q^{5}},
& \F^{(3)}_Q(S_Q(\lan 2,3 \ran)) & \iso V^{(3)}(\varpi_1)_{-w^2q^{5}}.
\end{align*}
By~\cite[page 39; arxiv version]{Her10}, we have the six-term exact sequence
\begin{equation} \label{eq: Her22 D43}
\begin{aligned}
0 & \to \Viop{5}\otimes\Vithp{5}\otimes\Vifp{5} \to \Vitp{6}\otimes\Vitp{4} \\
& \to \Vitp{4}\otimes \Vitp{6} \to \Viop{5}\otimes\Vithp{5}\otimes\Vifp{5} \to 0
\end{aligned}
\end{equation}
Note that we have an exact sequence
\begin{equation} \label{eq: SESD4}
\begin{aligned}
0 & \to \SQ{1,-4}\conv\SQ{2,-3}\conv\SQ{2,3} \to  \SQ{2,-4}\conv\SQ{1,2} \\
& \to \SQ{1,2}\conv\SQ{2,-4} \to \SQ{1,-4}\conv\SQ{2,-3}\conv\SQ{2,3} \to 0.
\end{aligned}
\end{equation}
Furthermore,
$
 \Viop{5}\otimes\Vithp{5}\otimes\Vifp{5} \text{ is simple}.
$
By applying the exact functor $\F^{(3)}_Q$ to \eqref{eq: SESD4}, we have
\begin{align*}
0 & \to \Viop{5}\otimes\Vithp{5}\otimes\Vifp{5} \to \Vitp{6}\otimes\F^{(3)}_Q(\SQ{1,2}) \\
&  \to \F^{(3)}_Q(\SQ{1,2})\otimes \Vitp{6} \to \Viop{5}\otimes\Vithp{5}\otimes\Vifp{5} \to 0.
\end{align*}
In particular, we have
\[
\F^{(3)}_Q(\SQ{1,2})\otimes \Vitp{6} \twoheadrightarrow \Viop{5}\otimes\Vithp{5}\otimes\Vifp{5}.
\]
Since every module in $\Ca_\g$ has a dual, by taking the dual $\Vitp{0}$ of $\Vitp{6}$, we have
\[
\F^{(3)}_Q(\SQ{1,2}) \iso \Vitp{4}
\]
with Theorem~\ref{thm: socle head}.

By~\cite[page 39; arxiv version]{Her10}, we have the six-term exact sequence
\begin{align*}
0 \to \Vitp{4} \to \Viop{5}\otimes\Viop{3}  \to \Viop{3}\otimes \Viop{5} \to \Vitp{4} \to 0.
\end{align*}
Note that we have an exact sequence
\begin{equation} \label{eq: SESD4 2}
\begin{aligned}
0 \to \SQ{1,2}   \to \SQ{1,-4}\conv\SQ{2,3} \to \SQ{2,3}\conv\SQ{1,-4} \to \SQ{1,2} \to 0.
\end{aligned}
\end{equation}
By applying the exact functor $\F^{(3)}_Q$ to \eqref{eq: SESD4 2}, we have
\begin{align*}
0 \to \Vitp{4} & \to \Viop{5} \otimes \F^{(3)}_Q(\SQ{2,4}) \\
& \to \F^{(3)}_Q(\SQ{2,4})\otimes \Viop{5}  \to \Vitp{4} \to 0
\end{align*}
In particular, we have
\[
\F^{(3)}_Q(\SQ{2,4})\otimes \Viop{5} \twoheadrightarrow \Vitp{4}
\]
By taking the dual $\Viop{-1}$ of $\Viop{5}$, we have $\F^{(3)}_Q(\SQ{2,4}) \iso \Viop{3}.$

Since $\al_2$ and $\al_4$ is a pair for $\al_2 + \al_4 = \lr{2,4}$, the fact that $ \F^{(3)}_Q(\SQ{2,4}) \iso \Viop{3}$ implies
\begin{align} \label{eq: 11p1}
\Vifp{1} \otimes \Vithp{5} \twoheadrightarrow \Viop{3}.
\end{align}
Using \eqref{eq: 11p1}, we can obtain,
\begin{align*}
\F^{(3)}_Q(\SQ{1,-3}) \iso \Vifp{3}, \ \ \F^{(3)}_Q(\SQ{1,3}) \iso \Vithp{3}.
\end{align*}
Finally, we can obtain $\F^{(3)}_Q(\SQ{1,4}) \iso \Vitp{2}$ by using \eqref{eq: Her22 D43}.

Now we have shown that for the Dynkin quiver $Q$
\begin{itemize}
\item[{\rm (a)}] $ \F^{(3)}_Q(S_Q(\be)) \iso V^{(3)}_Q(\be)$ for all $\be \in \PR$ and
\item[{\rm (b)}] $ V^{(3)}_Q(\be) \otimes V^{(3)}_Q(\al) \twoheadrightarrow V^{(3)}_Q(\gamma)$ for every $[Q]$-minimal pair $(\al,\be)$ of $\gamma \in \PR$.
\end{itemize}
Furthermore, for any pair of $(\al,\be)$ of $\gamma \in \PR$, it is proved in~\cite[Theorem 5.20]{Oh15E} that
there exists a six-term exact sequence
\[
0 \to S_Q(\gamma) \to S_Q(\al) \conv S_Q(\be) \to S_Q(\be) \conv S_Q(\al) \to S_Q(\gamma) \to 0.
\]
Hence we can prove that
\[
V^{(3)}_Q(\be) \otimes V^{(3)}_Q(\al) \twoheadrightarrow V^{(3)}_Q(\gamma)
\]
for any pair $(\al,\be)$ of $\gamma \in \PR$, by applying the same argument.

We have shown that our assertion holds for the $Q$. By Lemma~\ref{lem: Q well},
we can apply the induction argument on $|\gamma|$ for any $Q'$ of finite type $D_4$.
\end{proof}

\begin{corollary} \label{cor: Dorey D43 E63 completely}
 For $U_q'(D^{(3)}_4)$ $(t=3)$ or $U_q'(E^{(2)}_6)$ $(t=2)$,
let $(i,x)$, $(j,y)$, $(k,z) \in I_0 \times \ko^\times$. Then
\[
\Hom_{U_q'(\g^{(t)})}\big( V^{(t)}(\varpi_{j})_y \otimes V^{(t)}(\varpi_{i})_x , V^{(t)}(\varpi_{k})_z  \big) \ne 0
\]
if and only if there exists a Dynkin quiver $Q$ and $\al,\beta,\ga \in \Phi_{Q}^+$ such that
\begin{enumerate}
\item[{\rm (1)}] $\alpha \prec_{[Q]} \beta$ and $\alpha + \beta = \gamma$,
\item[{\rm (2)}] $V(\varpi_{j})_y  = V^{(t)}_{Q}(\beta)_a, \ V(\varpi_{i})_x  = V^{(t)}_{Q}(\al)_a, \ V(\varpi_{k})_z  = V^{(t)}_{Q}(\ga)_a$
for some $a \in \ko^\times$.
\end{enumerate}
\end{corollary}

For the rest of this subsection we shall take $\mathsf{g}$ as the simple Lie algebra of type $D_4$ when $Q$ is of finite type $D_4$ and
 $\mathsf{g}$ as the simple Lie algebra of type $E_6$ when $Q$ is of finite type $E_6$.

Now, we can apply the same arguments of~\cite{KKKOIII,KKKOIV} to obtain the following theorem.
\begin{theorem}
\label{thm: D43 E62}
For a Dynkin quiver $Q$ of finite type $D_4$ $(t=3)$ or $E_6$ $(t=2)$, we have the followings:
\begin{enumerate}[{\rm (1)}]
\item The functor $\F^{(t)}_Q$ sends a simple module to a simple module. Moreover, the functor $\F^{(t)}_Q$ induces a bijection from
the set of simple modules in $R^\mathsf{g}\gmod$ to the set of simple modules in $\Ca^{(t)}_Q$.
\item For any $[Q],[Q'] \in \lf \Delta \rf$,
there exist the following isomorphisms induced by $\F^{(i)}_Q$ and $\F^{(i')}_{Q'}$:
\begin{align} \label{eq: isomQ}
\begin{cases}
[\Ca^{(i)}_Q] \iso U^-_\A(\g)^\vee|_{q=1} \iso [\Ca^{(i')}_{Q'}]  \ (i,i' \in \{1,2,3\}) & \text{ if $Q, Q'$ are of type $D_4$}, \\
[\Ca^{(1)}_Q] \iso U^-_\A(\g)^\vee|_{q=1} \iso [\Ca^{(2)}_{Q'}] & \text{ if $Q,Q'$ are of type $E_6$}.
\end{cases}
\end{align}
\item A dual PBW-basis associated to $[Q]$ and the upper global basis of
$U^-_{\A}(\mathsf{g})^{\vee}$
are categorified by the modules over $U_q'(\g^{(t)})$ in the following sense:
\begin{enumerate}[{\rm (i)}]
\item The set of simple modules in $\Ca^{(t)}_Q$ corresponds to
the upper global basis of $U^-_{\A}(\mathsf{g})^{\vee}|_{q=1}$.
\item The set
\[
\left\{ V_{Q}^{(t)}(\beta_1)^{\otimes m_1}\otimes\cdots \otimes V_{Q}^{(t)}(\beta_\N)^{\otimes m_\N} \mid \um \in \Z_{\ge 0}^\N \right\}
\]
corresponds to the dual PBW-basis associated to $[Q]$ under the isomorphism in~\eqref{eq: isomQ}.
\end{enumerate}
\item The induced ring isomorphisms
\[
[\F_Q^{(i')}] \circ [{\F_Q^{(i)}}]^{-1} \colon [\Ca^{(i)}_Q]  \overset{\sim}{\Lto}  [\Ca^{(i')}_Q]
\]
send simples to simples and preserve dimensions. Here $i,i' \in \{1,2,3\}$ if $Q$ is of finite type $D_4$, and $i,i' \in \{1,2\}$ if $Q$ is of finite type $E_6$.
\end{enumerate}
\end{theorem}

\subsection{\texorpdfstring{$U_q'(F_4^{(1)})$ and $U_q'(G_2^{(1)})$}{Uq'(F4(1)) and Uq'(G2(1))}}

In this subsection, we shall prove the $U_q'(F_4^{(1)})$ and $U_q'(G_2^{(1)})$-analogues of
Theorem~\ref{thm: Dorey classical 2}, Theorem~\ref{thm:categorification2} and Theorem~\ref{thm:gQASW duality 2}. Since the frameworks for $U_q'(F_4^{(1)})$ and $U_q'(G_2^{(1)})$
are the same, we shall give proofs for $U_q'(F_4^{(1)})$-case only, which is more involved.

\begin{lemma} \label{lem:mQ well}
 For (resp.\ triply) twisted adapted class $[\rrz]$ and $[\rrz']$ of type $A_{2n-1}$ (resp.\ $D_{4}$),
$(\al,\be)$ of $\gamma \in \PR$ contained in $\Phi_{[\rrz]}(i,j)(t)$, then there exists a pair $(\al',\be')$ such that $\al'+\be' \in \PR$ and
$(\al',\be')$ is contained in $\Phi_{[\rrz']}(i,j)(t)$.
\end{lemma}

\begin{proof}
The assertion follows from well-definedness of the distance polynomial $D^{\lf \rrz \rf}_{i,j}(z)$ and the fact that
$\widehat{\Upsilon}_{[\rrz']}$ can be obtained from $\widehat{\Upsilon}_{[\rrz]}$ by applying the reflection functors properly.
\end{proof}

In this subsection, $U_q(\g^{(1)})$ denotes $U_q(F_4^{(1)})$ if $[\rrz]$ is a twisted adapted class of finite type $E_6$, and
$U_q(G_2^{(1)})$ if $[\rrz]$ is a triply twisted adapted class of finite type $D_4$.

\begin{proposition} \label{prop: particular F G}
For a Dynkin diagram $\Delta$ of type $E_6$ (resp.\ $D_4$), for any (resp.\ triply) twisted adapted class $[\rrz]$, take $J$ and $\mathcal{S}$ as the set of simple roots $\Pi$ associated to $\Delta$.
We define two maps
\[
s \colon \Pi \to \{ V^{(1)}(\varpi_i) \mid i \in I_0 \} \quad \text{ and } \quad X \colon \Pi \to  \ko^\times
\]
as follows:
for $\al \in \Pi$ with $\widehat{\Omega}_{[\rrz]}(\al)=(i,p/\mathsf{d})$, we define
\[
s(\al) = \begin{cases}
V^{(1)}(\varpi_{i^{\star}}) & \text{if $\widehat{\Delta}$ is of type $F_4$}, \\
V^{(1)}(\varpi_{i^{\dagger}}) & \text{if $\widehat{\Delta}$ is of type $G_2$},
\end{cases}
\qquad
X(\al)= \begin{cases}  (-1)^i(-q^{1/2})^p & \text{if $\widehat{\Delta}$ is of type $F_4$}, \\
(-q^{1/3})^p& \text{if $\widehat{\Delta}$ is of type $G_2$}.
\end{cases}
\]
Then the underlying graph of $\Gamma^{J}$ coincides with $\Delta$ of finite type $X_n$.
Hence we have an exact functor
\[
\mathscr{F}_{[\rrz]} \colon R^{X_n}\gmod \to \mC^{(1)}_{[\rrz]}.
\]
\end{proposition}

\begin{proof}
Note that $(\al_i,\al_j)$ are $[\rrz]$-minimal if $i,j$ are adjacent in $\Delta$ and $[\rrz]$-simple, otherwise. Thus the underlying graph of $\Gamma^J$
coincides with $\Delta$ by Theorem~\ref{thm: simply laced minimal denom}.
\end{proof}

\begin{example}
By forgetting the arrows of $\widehat{\Upsilon}_{[\mQ]}$ in \eqref{eq:folded1}, the underlying graph of $\Gamma^{J}$ can be described as follows:
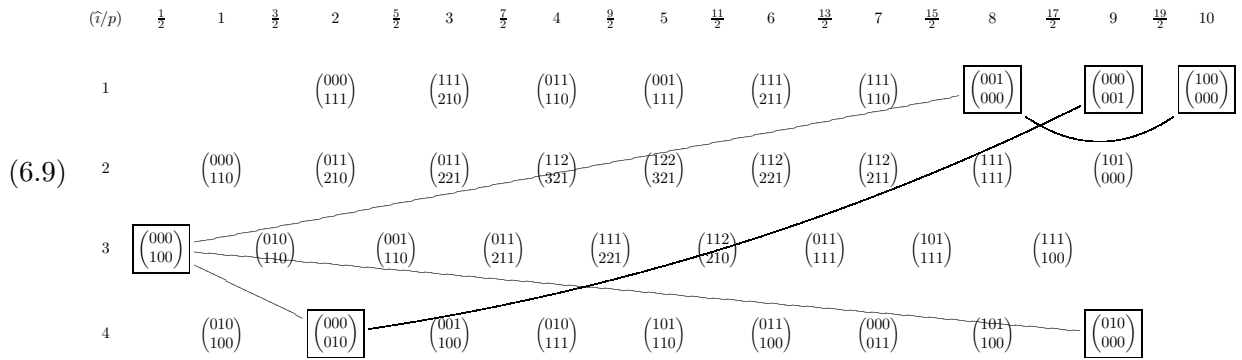
\begin{align} \label{eq:folded1 GammaJ}
 \raisebox{5.6em}{\scalebox{0.55}{\xymatrix@C=0.1ex@R=4.0ex{
(\widehat{\imath}/p) & \frac{1}{2}  & 1 & \frac{3}{2} & 2 & \frac{5}{2} & 3 & \frac{7}{2} & 4 & \frac{9}{2} & 5 & \frac{11}{2} & 6 & \frac{13}{2} & 7 & \frac{15}{2} & 8 & \frac{17}{2} & 9 & \frac{19}{2} & 10    \\
1 &&&& {\scriptstyle\prt{000}{111}}  &&{\scriptstyle\prt{111}{210}}&& {\scriptstyle\prt{011}{110}} && {\scriptstyle\prt{001}{111}}
&& {\scriptstyle\prt{111}{211}}
&& {\scriptstyle\prt{111}{110}}
&& {\scriptstyle\boxed{\prt{001}{000}}} \ar@{-}@/_3pc/[rrrr]  &&
{\scriptstyle\boxed{\prt{000}{001}}} && {\scriptstyle\boxed{\prt{100}{000}}} \\
2 && {\scriptstyle\prt{000}{110}}  &&  {\scriptstyle\prt{011}{210}}   && {\scriptstyle\prt{011}{221}} && {\scriptstyle\prt{112}{321}}
&&{\scriptstyle\prt{122}{321}}
&& {\scriptstyle\prt{112}{221}}
&&{\scriptstyle\prt{112}{211}}  &&
{\scriptstyle\prt{111}{111}}  && {\scriptstyle\prt{101}{000}}\\
3 & {\scriptstyle\boxed{\prt{000}{100}}}  \ar@{-}[drrr] \ar@{-}[uurrrrrrrrrrrrrrr]\ar@{-}[drrrrrrrrrrrrrrrrr]  && {\scriptstyle\prt{010}{110}}  && {\scriptstyle\prt{001}{110}}
&& {\scriptstyle\prt{011}{211}}  && {\scriptstyle\prt{111}{221}} &&{\scriptstyle\prt{112}{210}}
&&{\scriptstyle\prt{011}{111}} &&
{\scriptstyle\prt{101}{111}} && {\scriptstyle\prt{111}{100}}\\
4 && {\scriptstyle\prt{010}{100}} && {\scriptstyle\boxed{\prt{000}{010}}} \ar@{-}@/_2pc/[uuurrrrrrrrrrrrrr] && {\scriptstyle\prt{001}{100}}  &&
{\scriptstyle\prt{010}{111}} &&{\scriptstyle\prt{101}{110}} && {\scriptstyle\prt{011}{100}} && {\scriptstyle\prt{000}{011}}
 &&{\scriptstyle\prt{101}{100}}  &&{\scriptstyle\boxed{\prt{010}{000}}}
}}}
\end{align}
since
\begin{align*}
d_{1,1}(z)& =(z-q_s^{4})(z-q_s^{10})(z-q_s^{12})(z-q_s^{18}), \  d_{1,3}(z) =(z-q_s^{7})(z-q_s^{9})(z-q_s^{13})(z-q_s^{15}),   \\
d_{1,4}(z)& =(z+q_s^{8})(z+q_s^{14}), \ d_{4,4}(z) =(z-q_s^{2})(z-q_s^{8})(z-q_s^{12})(z-q_s^{18}),\\
d_{3,4}(z)& =(z+q_s^{3})(z+q_s^{7})(z+q_s^{9})^{\epsilon'}(z+q_s^{11})(z+q_s^{13})(z+q_s^{17}).
\end{align*}

\end{example}

\begin{theorem} \label{thm: F41 S to V}
For any (triply) twisted adapted class $[\rrz]$ of finite type $X=E_6$ or $X=D_4$, and $\gamma \in \PR$, we have
\[
\mathscr{F}_{[\rrz]}(S_{[\rrz]}(\ga)) \iso V_{[\rrz]}(\gamma).
\]
\end{theorem}

At first, we shall prove the above theorem for the folded AR quiver $\widehat{\Upsilon}_{[\mQ]}$ in \eqref{eq:folded1}. By~\ref{thm: F}, we have
$\mathscr{F}_\mQ(S_\mQ(\al_i)) \iso V_\mQ(\al_i)$.
Hence,
\begin{align*}
&\mathscr{F}_\mQ(S_\mQ(\al_1)) \iso V(\varpi_1)_{-q_s^{20}}, \ \mathscr{F}_\mQ(S_\mQ(\al_3)) \iso V(\varpi_1)_{-q_s^{16}}, \ \mathscr{F}_\mQ(S_\mQ(\al_6)) \iso V(\varpi_1)_{-q_s^{18}}, \  \\
& \mathscr{F}_\mQ(S_\mQ(\al_2)) \iso V(\varpi_4)_{q_s^{18}}, \ \mathscr{F}_\mQ(S_\mQ(\al_4)) \iso V(\varpi_3)_{-q_s^{1}}, \ \mathscr{F}_\mQ(S_\mQ(\al_5)) \iso V(\varpi_4)_{q_s^{4}}. \
 \end{align*}
Thus we have an exact sequence
\begin{align} \label{eq: F43 al2+al4}
 0 \to S_\mQ(\al_2+\al_4) \to S_\mQ(\al_4) \conv S_\mQ(\al_2) \overset{\rmat{}}{\Lto} S_\mQ(\al_2) \conv S_\mQ(\al_4) \to S_\mQ(\al_2+\al_4) \to 0.
\end{align}
Applying the exact functor $\mathscr{F}_\mQ$ to \eqref{eq: F43 al2+al4}, we have
\begin{align*}
0 \to \mathscr{F}_\mQ(S_\mQ(\al_2+\al_4)) & \to V(\varpi_4)_{q_s^{18}} \otimes   V(\varpi_3)_{-q_s^1} \\ & \To[\mathscr{F}_\mQ(\rmat{})] V(\varpi_3)_{-q_s^1} \otimes  V(\varpi_4)_{q_s^{18}}
\to \mathscr{F}_\mQ(S_\mQ(\al_2+\al_4))\to 0.
\end{align*}
Then \eqref{eq: her443} implies that
\[
\mathscr{F}_\mQ(S_\mQ(\al_2+\al_4)) \iso V(\varpi_3)_{-q_s^2}.
\]
Similarly, we have $\mathscr{F}_\mQ(S_\mQ(\al_1+\al_3)) \iso V(\varpi_2)_{q_s^{18}}$.

Since $\sprt{000}{010}$ and $\sprt{010}{100}$ is a $[\mQ]$-minimal pair for $\sprt{010}{110}$, we also have $\mathscr{F}_\mQ\left(\SmQ{010}{110}\right) \iso V(\varpi_3)_{-q_s^{3}}.$

Note that we have an exact sequence
\begin{equation} \label{eq: exact start F4}
\begin{aligned}
&0 \to \SmQ{000}{110} \conv \SmQ{010}{100} \to \SmQ{010}{110} \conv \SmQ{000}{100}
 \to \SmQ{000}{100} \conv\SmQ{010}{110}  \to \SmQ{000}{110} \conv \SmQ{010}{100} \to 0.
\end{aligned}
\end{equation}
where $\SmQ{000}{110} \conv \SmQ{010}{100}$ is simple.

By applying the exact functor $\mathscr{F}_\mQ$ to \eqref{eq: exact start F4}, we have 
\begin{align*}
&0 \to \mathscr{F}_\mQ(\SmQ{000}{110}) \otimes  V(\varpi_4)_{q_s^{2}} \to V(\varpi_3)_{-q_s^{3}} \otimes  V(\varpi_3)_{-q_s^{1}}   \\
& \hspace{15ex} \to V(\varpi_3)_{-q_s^{1}} \otimes  V(\varpi_3)_{-q_s^{3}}  \to V(\varpi_4)_{q_s^{2}} \otimes  \mathscr{F}_\mQ(\SmQ{000}{110}) \to 0.
\end{align*}

By~\cite[Page 23]{Her06}, $V(\varpi_3)_{-q_s^{1}} \otimes  V(\varpi_3)_{-q_s^{3}} $ has a simple head $V(\varpi_4)_{q_s^{2}} \otimes  V(\varpi_2)_{q_s^{2}}$. Hence we have
\[
\mathscr{F}_\mQ(\SmQ{000}{110}) \iso V(\varpi_2)_{q_s^{2}}
\]
by the same argument in the previous subsection.

Since $\sprt{000}{010}$ and $\sprt{000}{100}$ is another $[\mQ]$-minimal pair for $\sprt{000}{110}$, we have
\begin{align} \label{eq: p342 F4}
V(\varpi_3)_{-q_s^{1}} \otimes V(\varpi_4)_{q_s^{4}} \to V(\varpi_2)_{q_s^{2}}.
\end{align}

Now we have the following results with the Dorey's type
homomorphisms \eqref{eq: her112}, \eqref{eq: her443}, \eqref{eq:
OS444} and \eqref{eq: OS441} sequentially:
\begin{itemize}
\item $\left( \sprt{000}{001},\sprt{000}{110} \right)$ is a $[\mQ]$-minimal pair of $\sprt{000}{111}$, \eqref{eq: her112} implies
$ \mathscr{F}_\mQ\left(\SmQ{000}{111}\right) \iso  V\left(\varpi_1\right)_{-q_s^{4}},$
\item  $\left(\sprt{000}{001},\sprt{000}{010}\right)$ is a $[\mQ]$-minimal pair of $\sprt{000}{011}$, \eqref{eq: OS441} implies
$ \mathscr{F}_\mQ\left(\SmQ{000}{011}\right) \iso  V\left(\varpi_4\right)_{q_s^{14}},$
\item $\left(\sprt{001}{000},\sprt{010}{100}\right)$ is a $[\mQ]$-minimal pair of $\sprt{011}{100}$, \eqref{eq: OS441} implies
$ \mathscr{F}_\mQ\left(\SmQ{000}{011}\right) \iso  V\left(\varpi_4\right)_{q_s^{12}},$
\item  $\left(\sprt{000}{011},\sprt{011}{100}\right)$ is a $[\mQ]$-minimal pair of $\sprt{011}{111}$, \eqref{eq: her443} implies
$ \mathscr{F}_\mQ\left(\SmQ{011}{111}\right) \iso  V\left(\varpi_3\right)_{-q_s^{13}},$
\item $\left(\sprt{011}{100},\sprt{000}{010}\right)$ is a $[\mQ]$-minimal pair of $\sprt{011}{110}$, \eqref{eq: OS441} implies
$ \mathscr{F}_\mQ\left(\SmQ{011}{110}\right) \iso  V\left(\varpi_1\right)_{-q_s^{18}},$
\item $\left(\sprt{000}{011},\sprt{010}{100}\right)$ is a $[\mQ]$-minimal pair of $\sprt{010}{111}$, \eqref{eq: OS444} implies
$ \mathscr{F}_\mQ\left(\SmQ{010}{111}\right) \iso  V\left(\varpi_4\right)_{q_s^{8}}.$
\end{itemize}

Since $\sprt{000}{001}$ and $\sprt{010}{110}$ is another $[\mQ]$-minimal pair for $\sprt{010}{111}$, we have
\begin{align} \label{eq: p314 F4}
V(\varpi_3)_{-q_s^{3}} \otimes V(\varpi_1)_{-q_s^{18}} \to V(\varpi_4)_{q_s^{8}}.
\end{align}

Then we have the followings
\begin{itemize}
\item  $\left(\sprt{000}{100},\sprt{001}{000}\right)$ is a $[\mQ]$-minimal pair of $\sprt{001}{100}$, \eqref{eq: p314 F4} implies
$ \mathscr{F}_\mQ\left(\SmQ{001}{100}\right) \iso  V\left(\varpi_4\right)_{q_s^{6}},$
\item $\left(\sprt{001}{100},\sprt{000}{010}\right)$ is a $[\mQ]$-minimal pair of $\sprt{001}{110}$, \eqref{eq: her443} implies
$ \mathscr{F}_\mQ\left(\SmQ{001}{110}\right) \iso  V\left(\varpi_3\right)_{-q_s^{5}},$
\item  $\left(\sprt{010}{111},\sprt{001}{100}\right)$ is a $[\mQ]$-minimal pair of $\sprt{011}{211}$, \eqref{eq: her443} implies
$ \mathscr{F}_\mQ\left(\SmQ{011}{211}\right) \iso  V\left(\varpi_3\right)_{-q_s^{7}},$
\item   $\left(\sprt{001}{100},\sprt{010}{110}\right)$ is a $[\mQ]$-minimal pair of $\sprt{011}{210}$, \eqref{eq: p342 F4}  implies
$ \mathscr{F}_\mQ\left(\SmQ{011}{210}\right) \iso  V\left(\varpi_2\right)_{q_s^{4}},$
\item   $\left(\sprt{010}{111},\sprt{001}{110}\right)$ is a $[\mQ]$-minimal pair of $\sprt{011}{221}$, \eqref{eq: p342 F4}  implies
$ \mathscr{F}_\mQ\left(\SmQ{011}{221}\right) \iso  V\left(\varpi_2\right)_{q_s^{6}},$
\item   $\left(\sprt{000}{011},\sprt{001}{100}\right)$ is a $[\mQ]$-minimal pair of $\sprt{001}{111}$, \eqref{eq: OS441} implies
$ \mathscr{F}_\mQ\left(\SmQ{001}{111}\right) \iso  V\left(\varpi_1\right)_{-q_s^{10}},$
\item $\left(\sprt{001}{110},\sprt{100}{000}\right)$ is a $[\mQ]$-minimal pair of $\sprt{101}{110}$, \eqref{eq: p314 F4} implies
$ \mathscr{F}_\mQ\left(\SmQ{101}{110}\right) \iso V\left(\varpi_4\right)_{q_s^{10}},$
\item  $\left(\sprt{100}{000},\sprt{001}{100}\right)$ is a $[\mQ]$-minimal pair of $\sprt{101}{100}$, \eqref{eq: OS441} implies
$ \mathscr{F}_\mQ\left(\SmQ{101}{100}\right) \iso V\left(\varpi_4\right)_{q_s^{16}},$
\end{itemize}
sequentially.

Now we can say that
\begin{itemize}
\item[{\rm (a)}] $\mathscr{F}_\mQ(S_\mQ(\be)) \iso V_\mQ(\be)$ for all $\be \in \PR$ by applying the above arguments,
\item[{\rm (b)}] $V_\mQ(\be) \otimes V_\mQ(\al) \twoheadrightarrow V_\mQ(\gamma)$ for every $[\mQ]$-minimal pair $(\al,\be)$ of $\gamma \in \PR$.
\end{itemize}

\noindent\textit{Proof of theorem~\ref{thm: F41 S to V}}.
We have shown that our assertion holds for the $[\mQ]$. By Lemma~\ref{lem:mQ well},
we can apply the induction argument on $|\gamma|$ for any $[\mQ]$ of finite type $E_6$.
\qed

\begin{corollary} \label{cor: Dorey F41 G21 completely}
 For $U_q'(F^{(1)}_4)$ $(X=E_6)$ or $U_q'(G^{(1)}_2)$ $(X=D_4)$,
let $(i,x)$, $(j,y)$, $(k,z) \in I_0 \times \ko^\times$. Then
\[
\dim_{\ko} \left( \Hom_{U_q'(\g^{(1)})}\big( V^{(t)}(\varpi_{j})_y \otimes V^{(t)}(\varpi_{i})_x , V^{(t)}(\varpi_{k})_z  \big) \right) = 1
\]
if and only if there exists a (triply) twisted adapted class $[\rrz]$ and $\al,\beta,\ga \in \Phi_{X}^+$ such that
\begin{enumerate}[{\rm (1)}]
\item $(\al,\beta)$ is a $[\rrz]$-minimal pair of $\ga$,
\item $V(\varpi_{j})_y  = V_{\mQ}(\beta)_a, \ V(\varpi_{i})_x  = V_{\mQ}(\al)_a, \ V(\varpi_{k})_z  = V_{\mQ}(\ga)_a$
for some $a \in \ko^\times$.
\end{enumerate}
\end{corollary}

For the rest of this subsection we shall take $\mathsf{g}$ as the simple Lie algebra of type $E_6$ when $[\rrz]$ is a twisted adapted class of finite type $E_6$ and
 $\mathsf{g}$ as the simple Lie algebra of type $D_4$ when $[\rrz]$ is a triply twisted adapted class of finite type $D_4$.

\medskip
Now, we can apply the same arguments of~\cite{KKKOIII,KKKOIV,KO17} to obtain the following theorem.

\begin{theorem} \label{thm: F4(1) G2(1)}
For a $($triply$)$ twisted adapted class $[\rrz]=[\mQ]$ of finite type $E_6$ or $[\rrz]=[\mathfrak{Q}]$ of finite type $D_4$, we have the followings:
\begin{enumerate}[{\rm (1)}]
\item The functor $\mathscr{F}_{[\rrz]}$ sends a simple module to a simple module. Moreover, the functor $\mathscr{F}_{[\rrz]}$ induces a bijection from
the set of simple modules in $R^\mathsf{g}\gmod$ to the set of simple modules in $\mathscr{C}^{(1)}_{[\rrz]}$.
\item For any $[\rrz] \in \lf \rrz \rf$,
there exist isomorphisms induced by the functors $\F^{(i)}_Q$, $\mathscr{F}_\mQ$ and $\mathscr{F}_{\mathfrak{Q}}$:
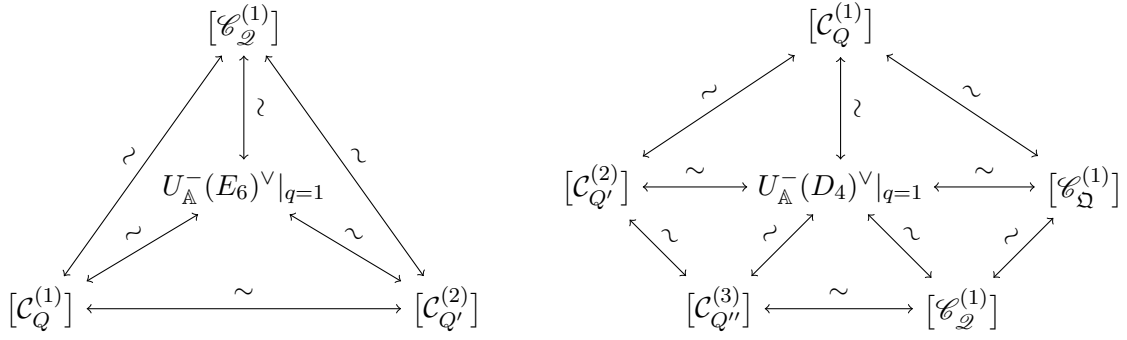
\[
\begin{tikzpicture}[scale=2.7,baseline=0]
\node (CQ) at (0,0.8) {$\bigl[ \mC^{(1)}_\mQ \bigr]$};
\node (U) at (0,0) {$U^-_\A(E_6)^\vee|_{q=1}$};
\node (C1) at (-1,-0.6) {$\bigl[\mathcal{C}^{(1)}_{Q}\bigr]$};
\node (C2) at (1,-0.6) {$\bigl[\mathcal{C}^{(2)}_{Q'}\bigr]$};
\draw[<->] (CQ) -- node[midway,sloped,above] {\small $\sim$}  (U);
\draw[<->] (C1) -- node[midway,sloped,above] {\small $\sim$}  (U);
\draw[<->] (C2) -- node[midway,sloped,above] {\small $\sim$}  (U);
\draw[<->] (CQ) -- node[midway,sloped,above] {\small $\sim$}  (C1);
\draw[<->] (CQ) -- node[midway,sloped,above] {\small $\sim$}  (C2);
\draw[<->] (C1) -- node[midway,sloped,above] {\small $\sim$}  (C2);
\end{tikzpicture}
\qquad
\begin{tikzpicture}[scale=2.7,baseline=0]
\node (C1) at (0,0.8) {$\bigl[\mathcal{C}^{(1)}_{Q}\bigr]$};
\node (C2) at (-1.2,0) {$\bigl[\mathcal{C}^{(2)}_{Q'}\bigr]$};
\node (C3) at (-0.6,-0.6) {$\bigl[\mathcal{C}^{(3)}_{Q''}\bigr]$};
\node (CD) at (1.2,0) {$\bigl[\mC^{(1)}_{\mathfrak{Q}}\bigr]$};
\node (CQ) at (0.6,-0.6) {$\bigl[\mC^{(1)}_\mQ\bigr]$};
\node (U) at (0,0) {$U^-_\A(D_4)^\vee|_{q=1}$};
\draw[<->] (C1) -- node[midway,sloped,above] {\small $\sim$}  (U);
\draw[<->] (C2) -- node[midway,sloped,above] {\small $\sim$}  (U);
\draw[<->] (C3) -- node[midway,sloped,above] {\small $\sim$}  (U);
\draw[<->] (CD) -- node[midway,sloped,above] {\small $\sim$}  (U);
\draw[<->] (CQ) -- node[midway,sloped,above] {\small $\sim$}  (U);
\draw[<->] (C1) -- node[midway,sloped,above] {\small $\sim$}  (C2);
\draw[<->] (C2) -- node[midway,sloped,above] {\small $\sim$}  (C3);
\draw[<->] (C3) -- node[midway,sloped,above] {\small $\sim$}  (CQ);
\draw[<->] (CQ) -- node[midway,sloped,above] {\small $\sim$}  (CD);
\draw[<->] (CD) -- node[midway,sloped,above] {\small $\sim$}  (C1);
\end{tikzpicture}
\]
\item A dual PBW-basis associated to $[\rrz]$ and the upper global basis of
$U^-_{\A}(\mathsf{g})^{\vee}$
are categorified by the modules over $U_q'(\g)$ in the following sense:
\begin{enumerate}[{\rm (i)}]
\item The set of simple modules in $\mC^{(1)}_{[\rrz]}$ corresponds to
the upper global basis of $U^-_{\A}(\mathsf{g})^{\vee}|_{q=1}$.
\item The set
\[
\left\{ V_{[\rrz]}(\beta_1)^{\otimes  m_1}\otimes \cdots \otimes  V_{[\rrz]}(\beta_\N)^{\otimes  m_\N} \mid \um \in \Z_{\ge 0}^\N \right\}
\]
corresponds to the dual PBW-basis associated to $[\rrz]$ under the isomorphism in~\eqref{eq: isomQ}.
\end{enumerate}
\item Any induced ring isomorphisms in {\rm (2)} send simples to simples.
\end{enumerate}
\end{theorem}

\section{Refining denominator formulas} \label{sec:Refine}

As we appointed, we shall refined the denominator formulas for
$U_q'(F_4^{(1)})$ and $U_q'(E_6^{(2)})$ by using the generalized
Schur-Weyl duality functor for them, investigated in the previous
section.

\subsection{\texorpdfstring{$U_q'(F_4^{(1)})$}{Uq'(F4(1))}}

To refine the denominator formulas for $U_q'(F_4^{(1)})$, we shall use $\widehat{\Upsilon}_{\mQ}$ in \eqref{eq:folded1}.

\begin{proposition}
We can refine $d_{3,4}(z)$ as follows:
\[
d^{F_4^{(1)}}_{3,4}(z) =(z+q_s^{3})(z+q_s^{7})(z+q_s^{11})(z+q_s^{13})(z+q_s^{17}).
\]
Hence we have
\[
d^{F_4^{(1)}}_{3,3}(z) =(z-q_s^{2})(z-q_s^{6})(z-q_s^{8})(z-q_s^{10})(z-q_s^{12})^{2}(z-q_s^{16})(z-q_s^{18}).
\]
\end{proposition}

\begin{proof}
Note that the pair $\left( \sprt{000}{100},\sprt{101}{110} \right)$ is $[\mQ]$-simple and hence $\SmQ{000}{100} \conv \SmQ{101}{110}$ is simple.
By Theorem~\ref{thm: F4(1) G2(1)}, $V(\varpi_3)_{-q_s^{1}} \otimes V(\varpi_4)_{-q_s^{10}} $ is also simple and hence our assertion follows.
\end{proof}

\begin{proposition}
We can refine $d_{1,3}(z)$ as follows:
\[
d^{F_4^{(1)}}_{1,3}(z) =(z-q_s^{7})(z-q_s^{9})(z-q_s^{13})(z-q_s^{15}).
\]
\end{proposition}

\begin{proof}
Note that we have
$\SmQ{101}{100} \hookrightarrow \SmQ{100}{000} \conv \SmQ{001}{100}$,
and $\SmQ{001}{100} \conv \SmQ{111}{110} \twoheadrightarrow \SmQ{112}{210}$.
Then we have
\begin{align*}
& \SmQ{101}{100} \conv \SmQ{111}{110}  \hookrightarrow \SmQ{100}{000} \conv \SmQ{001}{100} \conv \SmQ{111}{110}\twoheadrightarrow \SmQ{100}{000} \conv \SmQ{112}{210}
\end{align*}
where the composition is non-zero by Proposition~\ref{prop: not vanishing}.
Hence $\SmQ{100}{000} \conv \SmQ{112}{210}$ is not simple by Theorem~\ref{thm: BkMc}.  By applying the functor $\mathscr{F}_\mQ$, we can say that $q_s^9$ is a root of $d_{1,3}(z)$ by Theorem~\ref{Thm: basic properties}.
\end{proof}

\begin{proposition} \label{prop: F41 refine}
We have
\[
d^{F_4^{(1)}}_{2,3}(z)=(z+q_s^{5})(z+q_s^{7})(z+q_s^{9})(z+q_s^{11})^{2}(z+q_s^{13})(z+q_s^{15})(z+q_s^{17}).
\]
\end{proposition}

\begin{proof}
Note that we have
$\SmQ{000}{111} \hookrightarrow \SmQ{000}{001} \conv \SmQ{000}{110}, \ \ \SmQ{000}{010} \conv \SmQ{000}{001} \twoheadrightarrow \SmQ{000}{011}$.
Then we have
\begin{align*}
& \SmQ{000}{010} \conv \SmQ{000}{111}  \hookrightarrow \SmQ{000}{010} \conv \SmQ{000}{001} \conv \SmQ{000}{110} \twoheadrightarrow \SmQ{000}{011} \conv \SmQ{000}{110}
\end{align*}
whose composition is non-zero by Proposition~\ref{prop: not vanishing}.

Similarly, we have the homomorphisms $\SmQ{111}{221} \hookrightarrow \SmQ{000}{011} \conv \SmQ{111}{210}$.
Then we have
\begin{align*}
& \SmQ{000}{110} \conv \SmQ{111}{221}  \hookrightarrow \SmQ{000}{110} \conv \SmQ{000}{011} \conv \SmQ{111}{210} \twoheadrightarrow \SmQ{000}{111} \conv \SmQ{000}{010}\conv \SmQ{111}{210}
\end{align*}
where the composition is non-zero by Proposition~\ref{prop: not vanishing} again.

Since $\SmQ{000}{111} \conv \SmQ{000}{010}\conv \SmQ{111}{210}$ is simple, we can conclude that $\SmQ{000}{110} \conv \SmQ{111}{221} $ is not simple by Theorem~\ref{thm: BkMc}.
By applying the functor $\mathscr{F}_\mQ$, we can say that $-q_s^7$ is a root of $d_{2,3}(z)$.
\end{proof}

\begin{proposition}
We have
\[
d^{F_4^{(1)}}_{2,2}(z)=(z-q_s^{4})(z-q_s^{6})(z-q_s^{8})^{2}(z-q_s^{10})^{2}(z-q_s^{12})^{2} (z-q_s^{14})^2(z-q_s^{16})(z-q_s^{18}).
\]
\end{proposition}

\begin{proof}
Assume the order of the root $q_s^8$ is $1$. Then, by Lemma~\ref{lem:simplepole}, the composition length
of
$
\VmQ{000}{110} \otimes  \VmQ{122}{321}
$
is $2$. By applying the functor $\mathscr{F}_\mQ$, the composition length
of $\SmQ{000}{110} \conv \SmQ{122}{321}$ is $2$, too.

Applying the same argument of Proposition~\ref{prop: F41 refine}, we have the following non-zero homomorphisms:
\begin{align*}
& \SmQ{011}{110} \conv \SmQ{000}{111} \conv\SmQ{111}{210} \to \SmQ{122}{321} \conv \SmQ{000}{110},  \ \  \SmQ{111}{221} \conv \SmQ{001}{100} \conv \SmQ{010}{110} \to \SmQ{122}{321} \conv \SmQ{000}{110}.
\end{align*}
Thus the composition length of $\SmQ{000}{110} \conv \SmQ{122}{321}$ can not be $2$ by Theorem~\ref{thm: BkMc}.
\end{proof}

\subsection{\texorpdfstring{$U_q'(E_6^{(2)})$}{Uq'(E6(2))}}

To refine the denominator formulas for $U_q'(E_6^{(2)})$, we shall use $\Gamma_{Q}$ in \eqref{eq: E6 AR quiver}.

\begin{proposition}
We have
\[
d^{E_6^{(2)}}_{2,2}(z)= (z-q^2)(z-q^4)(z-q^6)(z-q^8)^2(z-q^{10})(z+q^4)(z+q^6)^2(z+q^8)(z+q^{10})(z+q^{12}).
\]
\end{proposition}

\begin{proof}
Note that we have $\SoQ{000}{011} \hookrightarrow \SoQ{000}{010} \conv \SoQ{000}{001}, \ \ \SoQ{000}{001} \conv \SoQ{010}{110} \twoheadrightarrow \SoQ{010}{111}$.
Then we have
\begin{align*}
& \SoQ{000}{011} \conv \SoQ{010}{110}  \hookrightarrow
\SoQ{000}{010} \conv \SoQ{000}{001} \conv \SoQ{010}{110}
\twoheadrightarrow \SoQ{000}{010} \conv \SoQ{010}{111}.
\end{align*}
where the composition is non-zero by Proposition~\ref{prop: not vanishing}.
Since $\SoQ{111}{221} \hookrightarrow \SoQ{010}{110} \conv \SoQ{101}{111}$,
we have
\begin{align*}
& \SoQ{000}{011} \conv \SoQ{111}{221}  \hookrightarrow \SoQ{000}{011} \conv \SoQ{010}{110} \conv \SoQ{101}{111} \twoheadrightarrow \SoQ{000}{010} \conv \SoQ{010}{111} \conv \SoQ{101}{111}
\end{align*}
where the composition is non-zero by Proposition~\ref{prop: not vanishing}, again.
By applying the functor $\F^{(2)}_Q$, we can say that $-q^4$ is a root of $d_{2,2}(z)$ by Theorem~\ref{Thm: basic properties}.
\end{proof}

\begin{proposition}
We have 
\[
d^{E_6^{(2)}}_{2,3}(z)= (z^2+q^6)(z^2+q^{10})^2(z^2+q^{14})^2(z^2+q^{18})^2(z^2+q^{22}).
\]
\end{proposition}

\begin{proof}
Assume the root $\sqrt{-1} q^5$ is of order $1$. Then, by Lemma~\ref{lem:simplepole}, the composition length
of
$
\VoQ{112}{321} \otimes  \VoQ{000}{011}
$
is $2$.
By Theorem~\ref{thm: D43 E62}, the composition length
of $\SoQ{112}{321} \conv \SoQ{000}{011}$ is $2$, too.

In a similar way of the proof of the previous proposition, we have non-zero homomorphisms
\begin{align*}
& \SoQ{000}{011} \conv \SoQ{112}{321}  \hookrightarrow \SoQ{000}{011} \conv \SoQ{101}{110} \conv \SoQ{011}{211} \twoheadrightarrow \SoQ{000}{010} \conv \SoQ{101}{111} \conv \SoQ{011}{211}, \\
& \SoQ{000}{011} \conv \SoQ{112}{321}  \hookrightarrow \SoQ{000}{011} \conv \SoQ{001}{100} \conv \SoQ{111}{221} \twoheadrightarrow  \SoQ{001}{111} \conv \SoQ{111}{221}.
\end{align*}
By taking dual, one can prove that the composition length of $\SoQ{112}{321} \conv \SoQ{000}{011}$ can not be $2$  by Theorem~\ref{thm: BkMc}.
\end{proof}

\begin{proposition}
We have
\[
d^{E_6^{(2)}}_{3,4}(z) = (z^2-q^6)(z^2-q^{10})(z^2-q^{14})^2(z^2-q^{18})^2(z^2-q^{22}).
\]
\end{proposition}

\begin{proof}
Assume that $q^3$ is not a root of $d_{3,4}(z)$. Theorem~\ref{Thm: basic properties} tells that
$
\VoQ{000}{111} \otimes  \VoQ{001}{110}
$
is simple.
By Theorem~\ref{thm: D43 E62}, it implies that
$\SoQ{000}{111} \conv \SoQ{001}{110}$ is simple, too.

However, we can obtain a non-zero homomorphism
\[
\SoQ{000}{111} \conv \SoQ{001}{110}  \hookrightarrow \SoQ{000}{110} \conv \SoQ{000}{001} \conv \SoQ{001}{110} \twoheadrightarrow \SoQ{000}{110} \conv \SoQ{001}{111},
\]
which yields contradiction by Theorem~\ref{thm: BkMc}.
\end{proof}

\begin{proposition}
\label{prop:E62_ambiguity}
We can refine $d_{3,3}(z)$ as follows: For some $\epsilon' \in \{0,1\}$, we have
\[
d_{3,3}(z)=
(z^2-q^4)(z^2-q^8)^{2}(z^2-q^{12})^{2+\epsilon'}(z^2-q^{16})^{3}(z^2-q^{20})^2(z^2-q^{24}).
\]
\end{proposition}

\begin{proof}
Assume the order of the root $q^4$ is $1$. Then, by Lemma~\ref{lem:simplepole}, the composition length
of
$
\VoQ{000}{111} \otimes  \VoQ{112}{321}
$
is $2$.
By Theorem~\ref{thm: D43 E62}, the composition length of $\SoQ{000}{111} \conv \SoQ{112}{321}$ is $2$, too.

In a similar way of the proof of the previous proposition, we have non-zero homomorphisms
\begin{align*}
& \SoQ{000}{111} \conv \SoQ{112}{321}  \hookrightarrow \SoQ{000}{111} \conv \SoQ{101}{110} \conv \SoQ{011}{211} \twoheadrightarrow \SoQ{000}{110} \conv \SoQ{101}{111} \conv \SoQ{011}{211}, \\
& \SoQ{000}{111} \conv \SoQ{112}{321}  \hookrightarrow \SoQ{000}{111} \conv \SoQ{001}{100} \conv \SoQ{111}{221} \twoheadrightarrow \SoQ{000}{100} \conv \SoQ{001}{111} \conv \SoQ{211}{231}.
\end{align*}
By taking duals, one can notice that $\SoQ{000}{111} \conv \SoQ{112}{321}$ cannot be $2$  by Theorem~\ref{thm: BkMc}.
\end{proof}
The $\epsilon'$ is indeed $1$, which is proved in an addendum to our publication~\cite{OS19} (see Appendix~\ref{Sec: Addendum}).

\appendix

\section{Dynkin quiver \texorpdfstring{$Q$}{Q} of type \texorpdfstring{$E_7$}{E7} and its AR quiver \texorpdfstring{$\Gamma_Q$}{GammaQ}} \label{Sec:Dynkin E7}

Let us consider the Dynkin quiver
$
 Q =
\begin{tikzpicture}[baseline=5,>=stealth,yscale=0.7,xscale=0.8,font=\tiny]
\node[dynkdot,label={below:1}] (-2) at (-2,0) {};
\node[dynkdot,label={below:3}] (-1) at (-1,0) {};
\node[dynkdot,label={below:4}] (0) at (0,0) {};
\node[dynkdot,label={below:5}] (1) at (1,0) {};
\node[dynkdot,label={below:6}] (2) at (2,0) {};
\node[dynkdot,label={below:7}] (3) at (3,0) {};
\node[dynkdot,label={[label distance=-4pt]-30:2}] (t) at (0,1) {};
\draw[->] (-2) -- (-1);
\draw[->] (-1) -- (0);
\draw[->] (0) -- (1);
\draw[->] (1) -- (2);
\draw[->] (2) -- (3);
\draw[->] (t) -- (0);
\end{tikzpicture}
$

Its AR quiver $\Gamma_Q$ can be drawn as follows: 
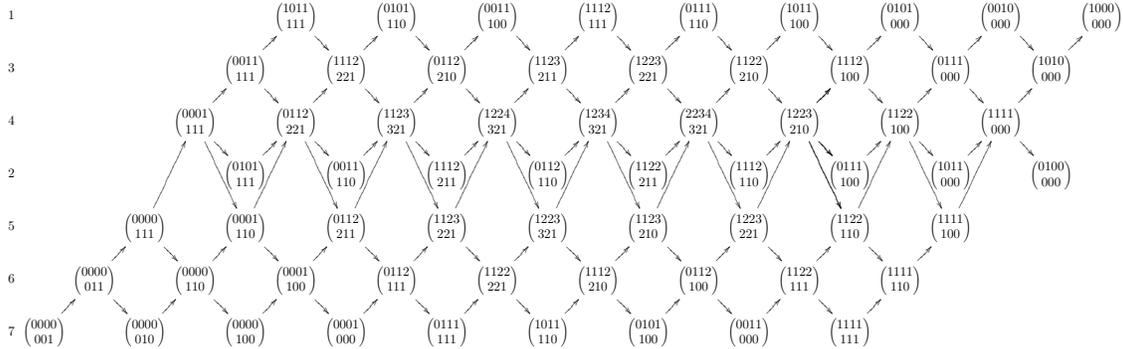
\[
\scalebox{0.45}{\xymatrix@R=2.5ex@C=0.1ex{
1&&&&&&{\scriptstyle\prt{1011}{111}} \ar[dr]&&{\scriptstyle\prt{0101}{110}} \ar[dr]&&{\scriptstyle\prt{0011}{100}} \ar[dr]&&{\scriptstyle\prt{1112}{111}} \ar[dr]
&&{\scriptstyle\prt{0111}{110}} \ar[dr] &&{\scriptstyle\prt{1011}{100}} \ar[dr]&&{\scriptstyle\prt{0101}{000}} \ar[dr]
&&{\scriptstyle\prt{0010}{000}} \ar[dr]&&{\scriptstyle\prt{1000}{000}} \\
3&&&&&{\scriptstyle\prt{0011}{111}} \ar[ur]\ar[dr]&&{\scriptstyle\prt{1112}{221}}\ar[ur]\ar[dr]&&{\scriptstyle\prt{0112}{210}}\ar[ur]\ar[dr]
&&{\scriptstyle\prt{1123}{211}}\ar[ur]\ar[dr]&&{\scriptstyle\prt{1223}{221}}\ar[ur]\ar[dr]
&&{\scriptstyle\prt{1122}{210}}\ar[ur]\ar[dr]&&{\scriptstyle\prt{1112}{100}}\ar[ur]\ar[dr]&&{\scriptstyle\prt{0111}{000}}\ar[ur]\ar[dr]&&{\scriptstyle\prt{1010}{000}}\ar[ur]\\
4&&&&{\scriptstyle\prt{0001}{111}}\ar[ur]\ar[ddr]\ar[dr]&&{\scriptstyle\prt{0112}{221}}\ar[ur]\ar[ddr]\ar[dr]
&&{\scriptstyle\prt{1123}{321}}\ar[ur]\ar[ddr]\ar[dr]&&{\scriptstyle\prt{1224}{321}}\ar[ur]\ar[ddr]\ar[dr]&&{\scriptstyle\prt{1234}{321}}\ar[ur]\ar[ddr]\ar[dr]
&&{\scriptstyle\prt{2234}{321}}\ar[ur]\ar[ddr]\ar[dr]&&{\scriptstyle\prt{1223}{210}}\ar[ur]\ar[ddr]\ar[dr]\ar[ur]\ar[ddr]\ar[dr]
&&{\scriptstyle\prt{1122}{100}}\ar[ur]\ar[ddr]\ar[dr]&&{\scriptstyle\prt{1111}{000}}\ar[ur]\ar[dr]
\\
2&&&&&{\scriptstyle\prt{0101}{111}}\ar[ur]&&{\scriptstyle\prt{0011}{110}}\ar[ur]&&{\scriptstyle\prt{1112}{211}}\ar[ur]&&{\scriptstyle\prt{0112}{110}}\ar[ur]
&&{\scriptstyle\prt{1122}{211}}\ar[ur]&&{\scriptstyle\prt{1112}{110}}\ar[ur]&&{\scriptstyle\prt{0111}{100}}\ar[ur]&&{\scriptstyle\prt{1011}{000}}\ar[ur]
&&{\scriptstyle\prt{0100}{000}}
\\
5&&&{\scriptstyle\prt{0000}{111}}\ar[uur]\ar[dr]&&{\scriptstyle\prt{0001}{110}}\ar[uur]\ar[dr]&&{\scriptstyle\prt{0112}{211}}\ar[uur]\ar[dr]
&&{\scriptstyle\prt{1123}{221}}\ar[uur]\ar[dr]&&{\scriptstyle\prt{1223}{321}}\ar[uur]\ar[dr]&&{\scriptstyle\prt{1123}{210}}\ar[uur]\ar[dr]
&&{\scriptstyle\prt{1223}{221}}\ar[uur]\ar[dr]&&{\scriptstyle\prt{1122}{110}}\ar[uur]\ar[dr]&&{\scriptstyle\prt{1111}{100}}\ar[uur]
\\
6&&{\scriptstyle\prt{0000}{011}}\ar[ur]\ar[dr]&&{\scriptstyle\prt{0000}{110}}\ar[ur]\ar[dr]&&{\scriptstyle\prt{0001}{100}}\ar[ur]\ar[dr]
&&{\scriptstyle\prt{0112}{111}}\ar[ur]\ar[dr]&&{\scriptstyle\prt{1122}{221}}\ar[ur]\ar[dr]
&&{\scriptstyle\prt{1112}{210}}\ar[ur]\ar[dr]&&{\scriptstyle\prt{0112}{100}}\ar[ur]\ar[dr]&&{\scriptstyle\prt{1122}{111}}\ar[ur]\ar[dr]
&&{\scriptstyle\prt{1111}{110}}\ar[ur]
\\
7& {\scriptstyle\prt{0000}{001}}\ar[ur] &&{\scriptstyle\prt{0000}{010}}\ar[ur]&&{\scriptstyle\prt{0000}{100}}\ar[ur]&&{\scriptstyle\prt{0001}{000}}\ar[ur]
&&{\scriptstyle\prt{0111}{111}}\ar[ur] &&{\scriptstyle\prt{1011}{110}}\ar[ur]&&{\scriptstyle\prt{0101}{100}}\ar[ur]
&&{\scriptstyle\prt{0011}{000}}\ar[ur]&&{\scriptstyle\prt{1111}{111}}\ar[ur]
}}
\]

\section{Dynkin quiver \texorpdfstring{$Q$}{Q} of type \texorpdfstring{$E_8$}{E8} and its AR quiver \texorpdfstring{$\Gamma_Q$}{GammaQ}} \label{Sec:Dynkin E8}

Let us consider the Dynkin quiver
$
 Q =
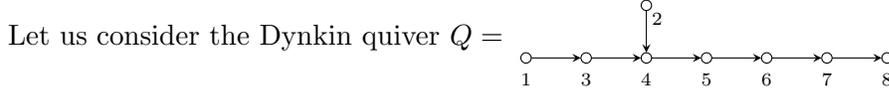
\begin{tikzpicture}[baseline=5,>=stealth,yscale=0.7,xscale=0.8,font=\tiny]
\node[dynkdot,label={below:1}] (-2) at (-2,0) {};
\node[dynkdot,label={below:3}] (-1) at (-1,0) {};
\node[dynkdot,label={below:4}] (0) at (0,0) {};
\node[dynkdot,label={below:5}] (1) at (1,0) {};
\node[dynkdot,label={below:6}] (2) at (2,0) {};
\node[dynkdot,label={below:7}] (3) at (3,0) {};
\node[dynkdot,label={below:8}] (4) at (4,0) {};
\node[dynkdot,label={[label distance=-4pt]-30:2}] (t) at (0,1) {};
\draw[->] (-2) -- (-1);
\draw[->] (-1) -- (0);
\draw[->] (0) -- (1);
\draw[->] (1) -- (2);
\draw[->] (2) -- (3);
\draw[->] (3) -- (4);
\draw[->] (t) -- (0);
\end{tikzpicture}
$

Its AR quiver $\Gamma_Q$ can be drawn as follows:
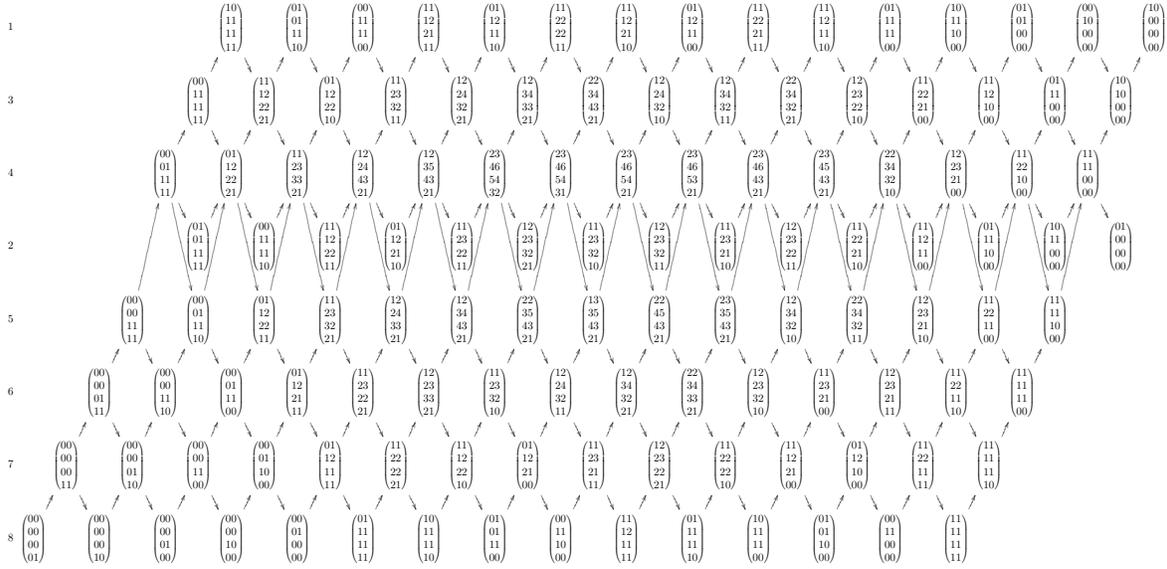
\[
\scalebox{0.38}{\xymatrix@R=3ex@C=0.1ex{
1&&&&&&&\pprt{10}{11}{11}{11}\ar[dr]&&\pprt{01}{01}{11}{10}\ar[dr]&&\pprt{00}{11}{11}{00}\ar[dr]&&\pprt{11}{12}{21}{11}\ar[dr]&&\pprt{01}{12}{11}{10}\ar[dr]
&&\pprt{11}{22}{22}{11}\ar[dr]&&\pprt{11}{12}{21}{10}\ar[dr]&&\pprt{01}{12}{11}{00}\ar[dr]&&\pprt{11}{22}{21}{11}\ar[dr]&&\pprt{11}{12}{11}{10}\ar[dr]
&&\pprt{01}{11}{11}{00}\ar[dr]&&\pprt{10}{11}{10}{00}\ar[dr]&&\pprt{01}{01}{00}{00}\ar[dr]&&\pprt{00}{10}{00}{00}\ar[dr]&&\pprt{10}{00}{00}{00}
\\
3&&&&&&\pprt{00}{11}{11}{11}\ar[dr]\ar[ur]&&\pprt{11}{12}{22}{21}\ar[dr]\ar[ur]&&\pprt{01}{12}{22}{10}\ar[dr]\ar[ur]&&\pprt{11}{23}{32}{11}\ar[dr]\ar[ur]
&&\pprt{12}{24}{32}{21}\ar[dr]\ar[ur]&&\pprt{12}{34}{33}{21}\ar[dr]\ar[ur]&&\pprt{22}{34}{43}{21}\ar[dr]\ar[ur]&&\pprt{12}{24}{32}{10}\ar[dr]\ar[ur]
&&\pprt{12}{34}{32}{11}\ar[dr]\ar[ur]&&\pprt{22}{34}{32}{21}\ar[dr]\ar[ur]&&\pprt{12}{23}{22}{10}\ar[dr]\ar[ur]&&\pprt{11}{22}{21}{00}\ar[dr]\ar[ur]
&&\pprt{11}{12}{10}{00}\ar[dr]\ar[ur]&&\pprt{01}{11}{00}{00}\ar[dr]\ar[ur]&&\pprt{10}{10}{00}{00}\ar[ur]
\\
4&&&&&\pprt{00}{01}{11}{11}\ar[dr]\ar[ur]\ar[ddr]&&\pprt{01}{12}{22}{21}\ar[dr]\ar[ur]\ar[ddr]&&\pprt{11}{23}{33}{21}\ar[dr]\ar[ur]\ar[ddr]&&
\pprt{12}{24}{43}{21}\ar[dr]\ar[ur]\ar[ddr]&&\pprt{12}{35}{43}{21}\ar[dr]\ar[ur]\ar[ddr]&&\pprt{23}{46}{54}{32}\ar[dr]\ar[ur]\ar[ddr]
&&\pprt{23}{46}{54}{31}\ar[dr]\ar[ur]\ar[ddr]&&\pprt{23}{46}{54}{21}\ar[dr]\ar[ur]\ar[ddr]&&\pprt{23}{46}{53}{21}\ar[dr]\ar[ur]\ar[ddr]
&&\pprt{23}{46}{43}{21}\ar[dr]\ar[ur]\ar[ddr]&&\pprt{23}{45}{43}{21}\ar[dr]\ar[ur]\ar[ddr]&&\pprt{22}{34}{32}{10}\ar[dr]\ar[ur]\ar[ddr]
&&\pprt{12}{23}{21}{00}\ar[dr]\ar[ur]\ar[ddr]&&\pprt{11}{22}{10}{00}\ar[dr]\ar[ur]\ar[ddr]&&\pprt{11}{11}{00}{00}\ar[dr]\ar[ur]
\\
2&&&&&&\pprt{01}{01}{11}{11}\ar[ur]&&\pprt{00}{11}{11}{10}\ar[ur]&&\pprt{11}{12}{22}{11}\ar[ur]&&\pprt{01}{12}{21}{10}\ar[ur]&&\pprt{11}{23}{22}{11}\ar[ur]
&&\pprt{12}{23}{32}{21}\ar[ur]&&\pprt{11}{23}{32}{10}\ar[ur]&&\pprt{12}{23}{32}{11}\ar[ur]&&\pprt{11}{23}{21}{10}\ar[ur]&&\pprt{12}{23}{22}{11}\ar[ur]
&&\pprt{11}{22}{21}{10}\ar[ur]&&\pprt{11}{12}{11}{00}\ar[ur]&&\pprt{01}{11}{10}{00}\ar[ur]&&\pprt{10}{11}{00}{00}\ar[ur]&&\pprt{01}{00}{00}{00}
\\
5&&&&\pprt{00}{00}{11}{11}\ar[uur]\ar[dr]&&\pprt{00}{01}{11}{10}\ar[uur]\ar[dr]&&\pprt{01}{12}{22}{11}\ar[uur]\ar[dr]&&\pprt{11}{23}{32}{21}\ar[uur]\ar[dr]
&&\pprt{12}{24}{33}{21}\ar[uur]\ar[dr]&&\pprt{12}{34}{43}{21}\ar[uur]\ar[dr]&&\pprt{22}{35}{43}{21}\ar[uur]\ar[dr]&&\pprt{13}{35}{43}{21}\ar[uur]\ar[dr]
&&\pprt{22}{45}{43}{21}\ar[uur]\ar[dr]&&\pprt{23}{35}{43}{21}\ar[uur]\ar[dr]&&\pprt{12}{34}{32}{10}\ar[uur]\ar[dr]&&\pprt{22}{34}{32}{11}\ar[uur]\ar[dr]
&&\pprt{12}{23}{21}{10}\ar[uur]\ar[dr]&&\pprt{11}{22}{11}{00}\ar[uur]\ar[dr]&&\pprt{11}{11}{10}{00}\ar[uur]
\\
6&&&\pprt{00}{00}{01}{11}\ar[ur]\ar[dr]&&\pprt{00}{00}{11}{10}\ar[ur]\ar[dr]&&\pprt{00}{01}{11}{00}\ar[ur]\ar[dr]&&\pprt{01}{12}{21}{11}\ar[ur]\ar[dr]&&\pprt{11}{23}{22}{21}\ar[ur]\ar[dr]
&&\pprt{12}{23}{33}{21}\ar[ur]\ar[dr]&&\pprt{11}{23}{32}{10}\ar[ur]\ar[dr]&&\pprt{12}{24}{32}{11}\ar[ur]\ar[dr]&&\pprt{12}{34}{32}{21}\ar[ur]\ar[dr]
&&\pprt{22}{34}{33}{21}\ar[ur]\ar[dr]&&\pprt{12}{23}{32}{10}\ar[ur]\ar[dr]&&\pprt{11}{23}{21}{00}\ar[ur]\ar[dr]&&\pprt{12}{23}{21}{11}\ar[ur]\ar[dr]
&&\pprt{11}{22}{11}{10}\ar[ur]\ar[dr]&&\pprt{11}{11}{11}{00}\ar[ur]
\\
7&&\pprt{00}{00}{00}{11}\ar[ur]\ar[dr]&&\pprt{00}{00}{01}{10}\ar[ur]\ar[dr]&&\pprt{00}{00}{11}{00}\ar[ur]\ar[dr]&&\pprt{00}{01}{10}{00}\ar[ur]\ar[dr]
&&\pprt{01}{12}{11}{11}\ar[ur]\ar[dr]&&\pprt{11}{22}{22}{21}\ar[ur]\ar[dr]&&\pprt{11}{12}{22}{10}\ar[ur]\ar[dr]
&&\pprt{01}{12}{21}{00}\ar[ur]\ar[dr]&&\pprt{11}{23}{21}{11}\ar[ur]\ar[dr]&&\pprt{12}{23}{22}{21}\ar[ur]\ar[dr]&&\pprt{11}{22}{22}{10}\ar[ur]\ar[dr]
&&\pprt{11}{12}{21}{00}\ar[ur]\ar[dr]&&\pprt{01}{12}{10}{00}\ar[ur]\ar[dr]&&\pprt{11}{22}{11}{11}\ar[ur]\ar[dr] &&\pprt{11}{11}{11}{10}\ar[ur]
\\
8&\pprt{00}{00}{00}{01}\ar[ur]&&\pprt{00}{00}{00}{10}\ar[ur]&&\pprt{00}{00}{01}{00}\ar[ur]&&\pprt{00}{00}{10}{00}\ar[ur]
&&\pprt{00}{01}{00}{00}\ar[ur]&&\pprt{01}{11}{11}{11}\ar[ur]&&\pprt{10}{11}{11}{10}\ar[ur]&&\pprt{01}{01}{11}{00}\ar[ur]
&&\pprt{00}{11}{10}{00}\ar[ur]&&\pprt{11}{12}{11}{11}\ar[ur]&&\pprt{01}{11}{11}{10}\ar[ur]
&&\pprt{10}{11}{11}{00}\ar[ur]&&\pprt{01}{01}{10}{00}\ar[ur]&&\pprt{00}{11}{00}{00}\ar[ur]&&\pprt{11}{11}{11}{11}\ar[ur]
}}
\]

\section{Denominator formulas for type \texorpdfstring{$E_7^{(1)}$}{E7(1)}}
\label{app:E7_denominators}
We have determined the following denominator formulas:
\begin{align*}
d_{1,1}(z) & = (z-q^{2})(z-q^{8})(z-q^{12})(z-q^{18}), \allowdisplaybreaks\\
d_{1,2}(z)&=(z+q^{5})(z+q^{9})(z+q^{11})(z+q^{15}),   \allowdisplaybreaks \\
d_{1,3}(z)&=(z+q^{3})(z+q^{7})(z+q^{9})(z+q^{11})(z+q^{13})(z+q^{17}),   \allowdisplaybreaks \\
d_{1,4}(z)&=(z-q^{4})(z-q^{6})(z-q^{8})(z-q^{10})^2(z-q^{12})(z-q^{14})(z-q^{16}),   \allowdisplaybreaks \\
d_{1,5}(z)&=(z+q^{5})(z+q^{7})(z+q^{9})(z+q^{11})(z+q^{13})(z+q^{15}),  \allowdisplaybreaks \\
d_{1,6}(z)&=(z-q^{6})(z-q^{8})(z-q^{12})(z-q^{14}),   \allowdisplaybreaks \\
d_{1,7}(z) & = (z+q^{7})(z+q^{13}), \allowdisplaybreaks\\
d_{2,2}(z)&=(z-q^{2})(z-q^{6})(z-q^{8})(z-q^{10})(z-q^{12})(z-q^{14})(z-q^{18}),  \allowdisplaybreaks \\
d_{2,3}(z)&=(z-q^{4})(z-q^{6})(z-q^{8})(z-q^{10})^2(z-q^{12})(z-q^{14})(z-q^{16}),      \allowdisplaybreaks \\
d_{2,4}(z)&=(z+q^{3})(z+q^{5})(z+q^{7})^2(z+q^{9})^2(z+q^{11})^2(z+q^{13})^2(z+q^{15})(z+q^{17}),  \allowdisplaybreaks \\
d_{2,5}(z)&=(z-q^{4})(z-q^{6})(z-q^{8})^2(z-q^{10})(z-q^{12})^2(z-q^{14})(z-q^{16}), \allowdisplaybreaks \\
d_{2,6}(z)&=(z+q^{5})(z+q^{7})(z+q^{9})(z+q^{11})(z+q^{13})(z+q^{15}),     \allowdisplaybreaks \\
d_{2,7}(z)&=(z-q^{6})(z-q^{10})(z-q^{14}),   \allowdisplaybreaks \\
d_{3,3}(z)&=(z-q^{2})(z-q^{4})(z-q^{6})(z-q^{8})^2(z-q^{10})^2(z-q^{12})^2(z-q^{14})(z-q^{16})(z-q^{18}),   \allowdisplaybreaks \nonumber \\
d_{3,4}(z)&=(z+q^{3}) (z+q^{5})^2 (z+q^{7})^2 (z+q^{9})^{3} (z+q^{11})^3 (z+q^{13})^2(z+q^{15})^2(z+q^{17}), \allowdisplaybreaks \nonumber \\
d_{3,5}(z)&=(z-q^{4})(z-q^{6})^2(z-q^{8})^2(z-q^{10})^2(z-q^{12})^2(z-q^{14})^2(z-q^{16}), \allowdisplaybreaks \\
d_{3,6}(z)&=(z+q^{5})(z+q^{7})^2(z+q^{9})(z+q^{11})(z+q^{13})^2 (z+q^{15}), \allowdisplaybreaks \\
d_{3,7}(z)&=(z-q^{6})(z-q^{8})(z-q^{12})(z-q^{14}),    \allowdisplaybreaks \\
d_{4,4}(z) & = (z-q^{2})(z-q^{4})^2(z-q^{6})^{3}(z-q^{8})^{4}(z-q^{10})^{4}(z-q^{12})^{4}(z-q^{14})^3(z-q^{16})^2(z-q^{18}), \allowdisplaybreaks \\
d_{4,5}(z)&=(z+q^{3})(z+q^{5})^2(z+q^{7})^{3}(z+q^{9})^{3}(z+q^{11})^{3}(z+q^{13})^3(z+q^{15})^2(z+q^{17}),\allowdisplaybreaks \nonumber\\
d_{4,6}(z)&=(z-q^{4})(z-q^{6})^2(z-q^{8})^2(z-q^{10})^2(z-q^{12})^2(z-q^{14})^2(z-q^{16}),  \allowdisplaybreaks \\
d_{4,7}(z)&=(z+q^{5})(z+q^{7})(z+q^{9})(z+q^{11})(z+q^{13})(z+q^{15}), \allowdisplaybreaks \\
d_{5,5}(z)&=(z-q^{2})(z-q^{4})(z-q^{6})^2(z-q^{8})^2(z-q^{10})^3(z-q^{12})^2(z-q^{14})^2(z-q^{16})(z-q^{18}),  \allowdisplaybreaks \nonumber \\
d_{5,6}(z)&=(z+q^{3})(z+q^{5})(z+q^{7})(z+q^{9})^2(z+q^{11})^2(z+q^{13})(z+q^{15})(z+q^{17}),  \allowdisplaybreaks \\
d_{5,7}(z)&=(z-q^{4})(z-q^{8})(z-q^{10})(z-q^{12})(z-q^{16}),\allowdisplaybreaks \\
d_{6,6}(z)&=(z-q^{2})(z-q^{4})(z-q^{8})(z-q^{10})^2(z-q^{12})(z-q^{16})(z-q^{18}), \allowdisplaybreaks \\
d_{6,7}(z)&=(z+q^{3})(z+q^{9})(z+q^{11})(z+q^{17}),\allowdisplaybreaks \\
d_{7,7}(z) & = (z-q^{2})(z-q^{10})(z-q^{18}).
\end{align*}

\section{Denominator formulas for type \texorpdfstring{$E_8^{(1)}$}{E8(1)}}
\label{app:E8_denominators}
We have determined the following denominator formulas:
\begin{align*}
d_{1,1}(z)& = (z - q^2) (z - q^8) (z - q^{12}) (z - q^{14}) (z - q^{18}) (z - q^{20}) (z - q^{24}) (z - q^{30}),\allowdisplaybreaks\\
d_{1,2}(z)&=(z+q^{5})(z+q^{9})(z+q^{11})(z+q^{13})(z+q^{15})(z+q^{17})\allowdisplaybreaks\\
& \hspace{10ex} \times (z+q^{19})(z+q^{21})(z+q^{23})(z+q^{27}), \allowdisplaybreaks \nonumber\\
d_{1,3}(z)&=(z+q^{3})(z+q^{7})(z+q^{9})(z+q^{11})(z+q^{13})^2(z+q^{15})(z+q^{17})\allowdisplaybreaks \\
& \hspace{10ex} \times (z+q^{19})^2(z+q^{21})(z+q^{23})(z+q^{25})(z+q^{29}), \allowdisplaybreaks \nonumber\\
d_{1,4}(z)&= (z-q^{4})(z-q^{6})(z-q^{8})(z-q^{10})^2(z-q^{12})^2(z-q^{14})^2(z-q^{16})^2\allowdisplaybreaks\\
& \hspace{10ex} \times (z-q^{18})^2(z-q^{20})^2(z-q^{22})^2(z-q^{24})(z-q^{26})(z-q^{28}),  \allowdisplaybreaks \nonumber\\
d_{1,5}(z)& =(z+q^{5})(z+q^{7})(z+q^{9})(z+q^{11})^2(z+q^{13})(z+q^{15})^2 \allowdisplaybreaks\\
& \hspace{10ex} \times (z+q^{17})^2(z+q^{19})(z+q^{21})^2(z+q^{23})(z+q^{25})(z+q^{27}),   \allowdisplaybreaks \nonumber\\
d_{1,6}(z)&= (z-q^{6})(z-q^{8})(z-q^{10})(z-q^{12})(z-q^{14})(z-q^{16})^2(z-q^{18})\allowdisplaybreaks \\
& \hspace{10ex} \times (z-q^{20})(z-q^{22})(z-q^{24})(z-q^{26}),  \allowdisplaybreaks \nonumber\\
d_{1,7}(z)&= (z+q^{7})(z+q^{9})(z+q^{13})(z+q^{15})(z+q^{17})(z+q^{19})(z+q^{23})(z+q^{25}),\allowdisplaybreaks \nonumber\\
d_{1,8}(z)&= (z-q^{8})(z-q^{14})(z-q^{18})(z-q^{24}), \allowdisplaybreaks \\
d_{2,2}(z)&=(z-q^{2})(z-q^{6})(z-q^{8})(z-q^{10})(z-q^{12})^2(z-q^{14})\allowdisplaybreaks \\
& \hspace{10ex} \times  (z-q^{16})^2(z-q^{18})(z-q^{20})^2(z-q^{22})(z-q^{24})(z-q^{26})(z-q^{30}),  \allowdisplaybreaks \nonumber\\
d_{2,3}(z)&=(z-q^{4})(z-q^{6})(z-q^{8})(z-q^{10})^2(z-q^{12})^2(z-q^{14})^2 \allowdisplaybreaks\\
& \hspace{10ex} \times (z-q^{16})^2(z-q^{18})^2(z-q^{20})^2(z-q^{22})^2(z-q^{24})(z-q^{26})(z-q^{28}), \allowdisplaybreaks \nonumber\\
d_{2,4}(z)& = (z+q^{3})(z+q^{5})(z+q^{7})^{2}(z+q^{9})^{2}(z+q^{11})^{3}(z+q^{13})^{3} (z+q^{15})^{3}(z+q^{17})^{3}\allowdisplaybreaks \\
& \hspace{10ex} \times (z+q^{19})^{3}(z+q^{21})^{3} (z+q^{23})^{2}(z+q^{25})^2(z+q^{27})(z+q^{29}), \nonumber\allowdisplaybreaks \\
d_{2,5}(z)& =  (z-q^{4})(z-q^{6})(z-q^{8})^2(z-q^{10})^2(z-q^{12})^2(z-q^{14})^{3}(z-q^{16})^2 \allowdisplaybreaks\\
& \hspace{10ex} \times  (z-q^{18})^{3}(z-q^{20})^2(z-q^{22})^2(z-q^{24})^2(z-q^{26})(z-q^{28}), \allowdisplaybreaks \nonumber\\
d_{2,6}(z) & =(z+q^{5})(z+q^{7})(z+q^{9})^2(z+q^{11})(z+q^{13})^2(z+q^{15})^2\allowdisplaybreaks \\
& \hspace{10ex} \times (z+q^{17})^2(z+q^{19})^2(z+q^{21})(z+q^{23})^2(z+q^{25})(z+q^{27}), \allowdisplaybreaks \nonumber\\
d_{2,7}(z)&=(z-q^{6})(z-q^{8})(z-q^{10})(z-q^{12})(z-q^{14})(z-q^{16})^2 \allowdisplaybreaks\\
& \hspace{10ex} \times  (z-q^{18})(z-q^{20})(z-q^{22})(z-q^{24})(z-q^{26}) \allowdisplaybreaks ,\nonumber\\
d_{2,8}(z)&= (z+q^{7})(z+q^{11})(z+q^{15})(z+q^{17})(z+q^{21})(z+q^{25}), \allowdisplaybreaks \\
d_{3,3}(z)&=(z-q^{2})(z-q^{4})(z-q^{6})(z-q^{8})^2(z-q^{10})^2(z-q^{12})^{3}(z-q^{14})^3(z-q^{16})^2 \allowdisplaybreaks \\
& \hspace{10ex} \times (z-q^{18})^{3}(z-q^{20})^3 (z-q^{22})^2(z-q^{24})^2(z-q^{26})(z-q^{28})(z-q^{30}),  \allowdisplaybreaks \nonumber\\
d_{3,4}(z)& = (z+q^{3})(z+q^{5})^{2}(z+q^{7})^{2}(z+q^{9})^{3}(z+q^{11})^{4} (z+q^{13})^{4}(z+q^{15})^{4}\allowdisplaybreaks \\
& \hspace{10ex} \times (z+q^{17})^{4}(z+q^{19})^{4}(z+q^{21})^{4}(z+q^{23})^{3}(z+q^{25})^2(z+q^{27})^2(z+q^{29}), \allowdisplaybreaks \nonumber \\
d_{3,5}(z)&=(z-q^{4})(z-q^{6})^{2}(z-q^{8})^{2}(z-q^{10})^{3}(z-q^{12})^{3} (z-q^{14})^{3}(z-q^{16})^{4}\allowdisplaybreaks \\
& \hspace{10ex} \times  (z-q^{18})^{3}(z-q^{20})^{3}(z-q^{22})^{3}(z-q^{24})^{2}(z-q^{26})^2(z-q^{28}), \allowdisplaybreaks \nonumber \\
d_{3,6}(z)&=(z+q^{5})(z+q^{7})^{2}(z+q^{9})^{2}(z+q^{11})^{2}(z+q^{13})^{2}(z+q^{15})^{3}(z+q^{17})^{3} \allowdisplaybreaks \\
& \hspace{10ex} \times (z+q^{19})^{2}(z+q^{21})^{2}(z+q^{23})^{2}(z+q^{25})^{2}(z+q^{27}), \allowdisplaybreaks \nonumber \\
d_{3,7}(z)&=(z-q^{6})(z-q^{8})^2(z-q^{10})(z-q^{12})(z-q^{14})^2(z-q^{16})^{2} \allowdisplaybreaks\\
& \hspace{10ex} \times  (z-q^{18})^2(z-q^{20})(z-q^{22})(z-q^{24})^2(z-q^{26}),\allowdisplaybreaks \nonumber\\
d_{3,8}(z)&=(z+q^{7})(z+q^{9})(z+q^{13})(z+q^{15})(z+q^{17})(z+q^{19})(z+q^{23})(z+q^{25}),\allowdisplaybreaks \\
d_{4,4}(z)&= (z-q^{2})(z-q^{4})^{2}(z-q^{6})^{3}(z-q^{8})^{4}(z-q^{10})^{5}(z-q^{12})^{6}(z-q^{14})^{6}(z-q^{16})^{6}\allowdisplaybreaks  \\
& \hspace{10ex} \times  (z-q^{18})^{6}(z-q^{20})^{6}(z-q^{22})^{5}(z-q^{24})^{4}(z-q^{26})^3(z-q^{28})^2(z-q^{30}), \allowdisplaybreaks \nonumber \\
d_{4,5}(z)&=(z+q^{3})(z+q^{5})^{2}(z+q^{7})^{3}(z+q^{9})^{4}(z+q^{11})^{4} (z+q^{13})^{5}(z+q^{15})^{5}\allowdisplaybreaks\\
& \hspace{10ex} \times  (z+q^{17})^{5}(z+q^{19})^{5} (z+q^{21})^{4}(z+q^{23})^{4}(z+q^{25})^3(z+q^{27})^2(z+q^{29}), \allowdisplaybreaks\nonumber  \\
d_{4,6}(z)&= (z-q^{4})(z-q^{6})^{2}(z-q^{8})^{3}(z-q^{10})^{3}(z-q^{12})^{3}(z-q^{14})^{4}(z-q^{16})^{4}  \allowdisplaybreaks   \\
& \hspace{10ex} \times (z-q^{18})^{4} (z-q^{20})^{3}(z-q^{22})^{3}(z-q^{24})^{3}(z-q^{26})^2(z-q^{28}), \allowdisplaybreaks \nonumber \\
d_{4,7}(z) & = (z+q^{5})(z+q^{7})^2(z+q^{9})^2(z+q^{11})^2(z+q^{13})^2 (z+q^{15})^{3}(z+q^{17})^3 \allowdisplaybreaks\\
& \hspace{10ex} \times (z+q^{19})^2(z+q^{21})^2(z+q^{23})^2(z+q^{25})^2(z+q^{27}), \allowdisplaybreaks \nonumber\\
d_{4,8}(z)&=(z-q^{6})(z-q^{8})(z-q^{10})(z-q^{12})(z-q^{14})(z-q^{16})^2  \allowdisplaybreaks \\
& \hspace{10ex} \times  (z-q^{18})(z-q^{20})(z-q^{22})(z-q^{24})(z-q^{26}),\allowdisplaybreaks \nonumber\\
d_{5,5}(z)&= (z-q^{2})(z-q^{4})(z-q^{6})^{2}(z-q^{8})^{3}(z-q^{10})^{3}(z-q^{12})^{4}(z-q^{14})^{4}  \allowdisplaybreaks \\
& \hspace{10ex} \times (z-q^{16})^{4}(z-q^{18})^{4}(z-q^{20})^{4}(z-q^{22})^{3}(z-q^{24})^3(z-q^{26})^2(z-q^{28}), \allowdisplaybreaks \nonumber \\
d_{5,6}(z)&=(z+q^{3})(z+q^{5})(z+q^{7})^2(z+q^{9})^2(z+q^{11})^{3}(z+q^{13})^{3}(z+q^{15})^{3}(z+q^{17})^{3}   \allowdisplaybreaks \\
& \hspace{10ex} \times (z+q^{19})^{3}(z+q^{21})^3(z+q^{23})^2(z+q^{25})^2(z+q^{27})(z+q^{29}),
\allowdisplaybreaks \\
d_{5,7}(z)&  =   (z-q^{4})(z-q^{6})(z-q^{8})(z-q^{10})^2(z-q^{12})^2(z-q^{14})^2\allowdisplaybreaks\\
& \hspace{10ex} \times  (z-q^{16})^2(z-q^{18})^2(z-q^{20})^2(z-q^{22})^2(z-q^{24})(z-q^{26})(z-q^{28}),  \allowdisplaybreaks \nonumber\\
d_{5,8}(z)&=(z+q^{5})(z+q^{9})(z+q^{11})(z+q^{13})(z+q^{15})\allowdisplaybreaks\\
& \hspace{10ex} \times  (z+q^{17})(z+q^{19})(z+q^{21})(z+q^{23})(z+q^{27}),\allowdisplaybreaks \nonumber\\
d_{6,6}(z)& = (z-q^{2})(z-q^{4})(z-q^{6})(z-q^{8})(z-q^{10})^2(z-q^{12})^3(z-q^{14})^2 (z-q^{16})^2\allowdisplaybreaks \\
& \hspace{10ex} \times (z-q^{18})^2(z-q^{20})^3 (z-q^{22})^2(z-q^{24})(z-q^{26})(z-q^{28})(z-q^{30}),     \allowdisplaybreaks \nonumber\\
d_{6,7}(z)&=(z+q^{3})(z+q^{5})(z+q^{9})(z+q^{11})^2(z+q^{13})^2(z+q^{15})(z+q^{17}) \allowdisplaybreaks\\
& \hspace{10ex} \times (z+q^{19})^2(z+q^{21})^2(z+q^{23})(z+q^{27})(z+q^{29}),\allowdisplaybreaks \nonumber\\
d_{6,8}(z)&=(z-q^{4})(z-q^{10})(z-q^{12})(z-q^{14})(z-q^{18})(z-q^{20})(z-q^{22})(z-q^{28}),\allowdisplaybreaks \\
d_{7,7}(z)&=(z-q^{2})(z-q^{4})(z-q^{10})(z-q^{12})^2(z-q^{14})(z-q^{18})(z-q^{20})^2\allowdisplaybreaks \\
& \hspace{10ex} \times  (z-q^{22})(z-q^{28})(z-q^{30}),\allowdisplaybreaks \nonumber\\
d_{7,8}(z)&= (z+q^{3})(z+q^{11})(z+q^{13})(z+q^{19})(z+q^{21})(z+q^{29}),\allowdisplaybreaks \nonumber\\
d_{8,8}(z) &= (z-q^2)(z-q^{12})(z-q^{20})(z-q^{30}). \allowdisplaybreaks
\end{align*}


\section{Addendum -- Removing Ambiguities}\label{Sec: Addendum}

The following has been included as an addendum to our publication~\cite{OS19}, where we remove the ambiguities from Proposition~\ref{prop:E61_ambiguity}, Proposition~\ref{prop:E62_ambiguity} and that were previously in Appendix~\ref{app:E7_denominators} and Appendix~\ref{app:E8_denominators} in the published version.

\begin{proposition}[{\cite[Proposition~6.8]{Oh15E}}]
\label{prop: R}
 Let $Q$ be a Dynkin quiver of finite simply-laced type.
For a $[Q]$-minimal pair $(\al,\be)$ of a simple sequence $\us=(\al_1,\ldots,\al_k)$, $V^{(1)}_Q(\al_1) \otimes \cdots \otimes V^{(1)}_Q(\al_k)$
is simple and there exists a surjective homomorphism
\[
V_Q^{(1)}(\be) \otimes  V_Q^{(1)}(\al) \twoheadrightarrow V^{(1)}_Q(\al_1) \otimes \cdots \otimes V^{(1)}_Q(\al_k).
\]
\end{proposition}

We first resolve the ambiguity in Proposition~\ref{prop:E61_ambiguity}.

\begin{proposition}
For type $E_6^{(1)}$, we have
\[
d_{4,4}(z) = (z-q^2)(z-q^4)^2(z-q^6)^3(z-q^8)^3(z-q^{10})^2(z-q^{12}).
\]
\end{proposition}

\begin{proof}
Note that, from the AR-quiver of type $E_6$ in~\eqref{eq: E6 AR quiver}, we can read that
\[
\text{$( {\scriptstyle\sprt{111}{100}},{\scriptstyle\sprt{101}{110}} )$ is a $[Q]$-minimal pair of the simple sequence $( {\scriptstyle\sprt{101}{100}},{\scriptstyle\sprt{111}{110}} )$.}
\]
Hence Proposition~\ref{prop: R} states that we have a homomorphism
\begin{align}\label{eq: morphism E61}
V(\varpi_1)_{(-q)^{-3}}  \otimes V(\varpi_4)_{(-q)^1}  \twoheadrightarrow V(\varpi_2)\otimes V(\varpi_5).
\end{align}
Then we have
\begin{align}\label{eq: computation}
\dfrac{ d_{1,4}((-q)^{-3}z)d_{4,4}((-q)^{1}z)}{d_{2,4}(z)d_{5,4}(z)} \times
\dfrac{ a_{2,4}(z)d_{5,4}(z)}{ a_{1,4}((-q)^{-3}z)a_{4,4}((-q)^{1}z)}\in \ko[z^{\pm 1}].
\end{align}
by~\cite[Lemma C.15]{AK97}.
From Lemma~\ref{Lem: aij and dij}, one can compute that
\begin{align*}
& a_{2,4}(z) \equiv \dfrac{[19][5][23][1]}{[7][17][11][13]} &&  a_{3,4}(z) \equiv a_{4,5}(z) \equiv \dfrac{[1][3][21][23]}{[9][11][13][15]} \\
& a_{1,4}(z) \equiv \dfrac{[2][22]}{[10][14]}  && a_{4,4}(z) \equiv \dfrac{[0][2][4][20][22][24]}{[8][10][12]^2[14][16]}.
\end{align*}
Since we have computed $d_{1,4}(z)$, $d_{2,4}(z)$ and $d_{5,4}(z)$, ~\eqref{eq: computation} can be written as follows:
\begin{align*}
&\dfrac{ d_{1,4}((-q)^{-3}z)d_{4,4}((-q)^{1}z)}{d_{2,4}(z)d_{5,4}(z)}
= \dfrac{  (z+q^1)(z+q^3)^2(z+q^5)^{2\text{ or } 3}(z+q^7)^4(z-q^{9})^3(z-q^{11})^2 (z+q^{13}) }
{ (z+q^3)^2(z+q^5)^3(z+q^7)^4(z+q^9)^3(z+q^{11})^2}
\end{align*}
by Proposition~\ref{prop:E61_ambiguity} and
\begin{align*}
\dfrac{ a_{2,4}(z)d_{5,4}(z)}{ a_{1,4}((-q)^{-3}z)a_{4,4}((-q)^{1}z)}
& =  \dfrac{[19][5][23][1]}{[7][17][11][13]} \times \dfrac{[1][3][21][23]}{[9][11][13][15]}.
\times \dfrac{[7][11]}{[-1][19]} \times \dfrac{[9][11][13]^2[15][17]}{[1][3][5][21][23][25]} \\
& = \dfrac{[23][1]}{[-1][25]} = \dfrac{(z-(-q)^{-1})}{(z-(-q)^{1})}.
\end{align*}
Thus we have
\begin{align*}
& \dfrac{  (z+q^1)(z+q^3)^2(z+q^5)^{2\text{ or } 3}(z+q^7)^4(z-q^{9})^3(z-q^{11})^2 (z+q^{13}) }
{ (z+q^3)^2(z+q^5)^3(z+q^7)^4(z+q^9)^3(z+q^{11})^2} \times  \dfrac{(z-(-q)^{-1})}{(z-(-q)^{1})}  \\
& = \dfrac{  (z+q^{-1})(z+q^3)^2(z+q^5)^{2\text{ or } 3}(z+q^7)^4(z-q^{9})^3(z-q^{11})^2 (z+q^{13}) }
{ (z+q^3)^2(z+q^5)^3(z+q^7)^4(z+q^9)^3(z+q^{11})^2}
\in \ko[z^{\pm 1}].
\end{align*}
which implies that the order of degree $(-q)^6$ should be $3$.
\end{proof}

By applying generalized Schur--Weyl duality
\[
\begin{tikzpicture}[scale=2.7,baseline=0]
\node (U) at (0,0) {$\Rep(R^{E_6})$};
\node (C1) at (-1,-0.6) {$ \Ca_{E_6^{(1)}} \supset \mathcal{C}^{(1)}_{Q}$};
\node (C2) at (1,-0.6) {$ \mathcal{C}^{(2)}_{Q} \subset \Ca_{E_6^{(2)}}$};
\draw[<-] (C1) -- node[midway,sloped,above] {\tiny $\mathcal{F}^{(1)}_Q$}  (U);
\draw[<-] (C2) -- node[midway,sloped,above] {\tiny $\mathcal{F}^{(2)}_{Q}$}  (U);
\draw[<->,dotted] (C1) -- (C2);
\end{tikzpicture}
\]
we have the twisted analogue of Proposition~\ref{prop: R}:

\begin{proposition} \label{prop: R2}
Let $Q$ be a Dynkin quiver of finite simply-laced type.
For a $[Q]$-minimal pair $(\al,\be)$ of a simple sequence $\us=(\al_1,\ldots,\al_k)$, $V^{(2)}_Q(\al_1) \otimes \cdots \otimes V^{(2)}_Q(\al_k)$
is simple and there exists a surjective homomorphism
\[
V_Q^{(2)}(\be) \otimes  V_Q^{(2)}(\al) \twoheadrightarrow V^{(2)}_Q(\al_1) \otimes \cdots \otimes V^{(2)}_Q(\al_k).
\]
\end{proposition}

In particular,~\eqref{eq: morphism E61} transfers to the following homomorphism in $\Ca^{(2)}_Q$ under the map in Proposition~\ref{prop: particular D43 E62}:
\begin{align}\label{eq: morphism E62}
 V(\varpi_1)_{(-q)^{-3}}  \otimes V(\varpi_3)_{\sqrt{-1}(-q)^1}  \twoheadrightarrow V(\varpi_2)_{-1} \otimes V(\varpi_4)_{\sqrt{-1}}.
 \end{align}

Next, we resolve the ambiguity in Proposition~\ref{prop:E62_ambiguity}.

\begin{proposition}
For type $E_6^{(2)}$, we have
\[
d_{3,3}(z)= (z^2-q^4)(z^2-q^8)^{2}(z^2-q^{12})^{3}(z^2-q^{16})^{3}(z^2-q^{20})^2(z^2-q^{24}).
\]
\end{proposition}

\begin{proof}
By~\eqref{eq: morphism E62} and~\cite[Lemma~C.15]{AK97}, we have
\begin{align}\label{eq: computation2}
\dfrac{ d_{1,3}(\sqrt{-1}(-q)^{-3}z)d_{3,3}((-q)^{1}z)}{d_{2,3}(iz)d_{3,4}(z)} \times
\dfrac{ a_{2,3}(\sqrt{-1}z)a_{3,4}(z)}{ a_{1,3}(\sqrt{-1}(-q)^{-3}z)a_{3,3}((-q)^{1}z)}\in \ko[z^{\pm 1}].
\end{align}

From Lemma~\ref{Lem: aij and dij}, one can compute that
\begin{align*}
& a_{2,3}(\sqrt{-1}z) \equiv \dfrac{\lrt{3}\lrt{21}\lrt{1}\lrt{23}}{\lrt{9}\lrt{15}\lrt{11}\lrt{13}},
&& a_{3,4}(z) \equiv \dfrac{\lrt{5}\lrt{19}\lrt{1}\lrt{23}}{\lrt{7}\lrt{17}\lrt{11}\lrt{13}},\\
& a_{1,3}(\sqrt{-1}z) \equiv \dfrac{\lrt{2}\lrt{22}}{\lrt{10}\lrt{14}},
&& a_{3,3}(z) \equiv \dfrac{\lrt{4}\lrt{20}\lrt{2}\lrt{22}\lrt{0}\lrt{24}}{\lrt{8}\lrt{16}\lrt{10}\lrt{14}\lrt{12}^2}.
\end{align*}

Since we have computed $d_{1,3}(z)$, $d_{2,3}(z)$ and $d_{3,4}(z)$,~\eqref{eq: computation2} can be written as follows:
\begin{align*}
&\dfrac{ d_{1,3}(\sqrt{-1}(-q)^{-3}z)d_{3,3}((-q)^{1}z)}{d_{2,3}(\sqrt{-1}z)d_{3,4}(z)}\\
&= \dfrac{(z^2-q^2)(z^2-q^6)^{2}(z^2-q^{10})^{2 \text{ or } 3}(z^2-q^{14})^{4}(z^2-q^{18})^3(z^2-q^{22})^2(z^2-q^{26})}
{ (z^2-q^6)^2(z^2-q^{10})^3(z^2-q^{14})^4(z^2-q^{18})^3(z^2-q^{22})^2   }
\end{align*}
by Proposition~\ref{prop:E62_ambiguity} and
\begin{align*}
&\dfrac{ a_{2,3}(\sqrt{-1}z)a_{3,4}(z)}{ a_{1,3}(\sqrt{-1}(-q)^{-3}z)a_{3,3}((-q)^{1}z)} \\
& = \dfrac{\lrt{3}\lrt{21}\lrt{1}\lrt{23}}{\lrt{9}\lrt{15}\lrt{11}\lrt{13}} \times \dfrac{\lrt{5}\lrt{19}\lrt{1}\lrt{23}}{\lrt{7}\lrt{17}\lrt{11}\lrt{13}}
 \times \dfrac{\lrt{7}\lrt{11}}{\lrt{-1}\lrt{19}} \times \dfrac{\lrt{9}\lrt{17}\lrt{11}\lrt{15}\lrt{13}^2}{\lrt{5}\lrt{21}\lrt{3}\lrt{23}\lrt{1}\lrt{25}} \\
& =\dfrac{\lrt{1}\lrt{23}}{\lrt{25}\lrt{-1}} = \dfrac{(z^2-q^{-2})}{(z^2-q^{2})}.
\end{align*}

Thus we have
\begin{align*}
& \dfrac{(z^2-q^{-2})(z^2-q^6)^{2}(z^2-q^{10})^{2 \text{ or } 3}(z^2-q^{14})^{4}(z^2-q^{18})^3(z^2-q^{22})^2(z^2-q^{26})}
{ (z^2-q^6)^2(z^2-q^{10})^3(z^2-q^{14})^4(z^2-q^{18})^3(z^2-q^{22})^2   } \in \ko[z^{\pm 1}],
\end{align*}
which implies our assertion.
\end{proof}

For types $E_7^{(1)}$ and $E_8^{(1)}$, we can remove the ambiguities in the previous version by applying the same arguments as above.
Hence there are roots of order $4,5,6$ also. To the best knowledge of the authors, this is the first such observation of a root of order strictly larger than $3$.

When we prove for $E_7^{(1)}$-cases, we employ the following homomorphisms
\begin{align*}
& V(3)_{(-q)^{-1}}  \otimes V(6)_{(-q)^{4}}  \twoheadrightarrow V(1)\otimes V(2)_{(-q)}\otimes V(7)_{(-q)^3}, \\
& V(3)_{(-q)^{-1}}  \otimes V(5)_{(-q)^{3}}  \twoheadrightarrow V(1)\otimes V(2)_{(-q)}\otimes V(6)_{(-q)^2}, \\
& V(1)_{(-q)^{-3}}  \otimes V(4)_{(-q)^1}  \twoheadrightarrow V(2)\otimes V(5),
\end{align*}
which can be obtained  from Proposition~\ref{prop: R} and Table~\ref{table:min_pairs}.

\begin{table}
\begin{tabular}{cccc}
\toprule
Minimal pair $\up$
& $\left({\scriptstyle\sprt{1111}{110}} ,{\scriptstyle\sprt{1223}{221}} \right)$
& $\left({\scriptstyle\sprt{1122}{110}} ,{\scriptstyle\sprt{1223}{221}} \right)$
& $\left({\scriptstyle\sprt{1224}{321}} ,{\scriptstyle\sprt{1011}{111}} \right)$
\\ \midrule
$\soc_Q(\up)$
& $\left({\scriptstyle\sprt{0111}{110}}, {\scriptstyle\sprt{1112}{110}} ,{\scriptstyle\sprt{1111}{111}}  \right)$
& $\left({\scriptstyle\sprt{0111}{110}}, {\scriptstyle\sprt{1112}{110}} ,{\scriptstyle\sprt{1122}{111}}  \right)$
& $\left({\scriptstyle\sprt{1112}{211}}, {\scriptstyle\sprt{1123}{221}} \right)$
\\ \bottomrule
\end{tabular}
\caption{The minimal pairs $\up$ and their socles $\soc_Q(\up)$ of $\Gamma_Q$ used to remove ambiguities.}
\label{table:min_pairs}
\end{table}

\bibliographystyle{alpha}
\bibliography{r_matrix}{}

\end{document}